\documentclass[11pt,a4paper]{article}
\usepackage[left=2.5cm, right=2.5cm, top=3.cm, bottom=3.cm]{geometry}

\usepackage[T1]{fontenc}
\usepackage[utf8]{inputenc}
\usepackage[english]{babel}

\usepackage{amsmath,amsfonts,amssymb,amsthm,mathtools,stmaryrd,dsfont}

\usepackage{makeidx}
\usepackage{graphicx}
\usepackage{xcolor}

\usepackage{hyperref}

\hypersetup{
	pdftitle={Fermionic semiclassical Lp estimates},
	bookmarks=true,
}

\usepackage{pgf,tikz,pgfplots}
\usepackage{mathrsfs}
\usetikzlibrary{arrows}
\usetikzlibrary[patterns]

\definecolor{stefou}{rgb}{0.0, 0.47, 0.44}
\definecolor{cristina}{rgb}{0.8, 0.36, 0.36}
\definecolor{djgreen}{rgb}{0.0, 0.29, 0.29}
\usepackage{multicol}

\usepackage{titletoc}
\usepackage{titlesec}
\usepackage{sectsty}

\usepackage{appendix}

\sectionfont{\centering\sc\MakeLowercase}

\numberwithin{equation}{section}

\usepackage{tocloft}
\makeatletter  % To fix a small offset with page number.
\renewcommand{\@pnumwidth}{2.3em} % {1.75em}
\makeatother
\theoremstyle{plain}
\newtheorem{thm}{Theorem}[section]
\newtheorem{prop}[thm]{Proposition}
\newtheorem{lemma}[thm]{Lemma}
\newtheorem{cor}[thm]{Corollary}
\newtheorem{defi}[thm]{Definition}

\newtheorem{assump}{Assumption}

\newtheorem{fait}[thm]{Fact}
\newtheorem{rmk}{Remark}
\newtheorem*{nota}{Notation}

\theoremstyle{plain}
\newtheorem{thm-princ}{Theorem}

\newcommand{\BR}{\mathcal{B}}
\newcommand{\CR}{\mathcal{C}}

\newcommand{\HR}{\mathcal{H}}
\newcommand{\OR}{\mathcal{O}}
\newcommand{\UR}{\mathcal{U}}
\newcommand{\VR}{\mathcal{V}}
\newcommand{\WR}{\mathcal{W}}

\newcommand{\N}{\mathbb{N}}
\newcommand{\Z}{\mathbb{Z}}
\newcommand{\R}{\mathbb{R}}
\newcommand{\C}{\mathbb{C}}

\newcommand{\prodscal}[2]{\left\langle#1,#2\right\rangle}
\newcommand{\abs}[1]{\left\lvert #1 \right\rvert}
\newcommand{\norm}[1]{\left\lVert #1 \right\rVert}

\DeclareMathOperator{\sgn}{sgn}
\DeclareMathOperator{\rk}{\operatorname{rank}}
\DeclareMathOperator{\tr}{\operatorname{Tr}}
\DeclareMathOperator{\id}{\operatorname{Id}}

\DeclareMathOperator{\ev}{ev}

\DeclareMathOperator{\dist}{dist}
\DeclareMathOperator{\loc}{loc}

\newcommand{\indicatrice}[1]{\mathds{1}\left(#1\right)}

\newcommand{\normLp}[3]{\lVert #1 \rVert_{L^{#2}#3}}

\newcommand{\schatten}{\mathfrak{S}}
\newcommand{\normSch}[3]{\left\lVert #1 \right\rVert_{\mathfrak{S}^{#2}#3}}

\newcommand{\test}[1]{\mathcal{C}_c^{\infty} (#1)}
\newcommand{\schwartz}{\mathscr{S}}
\newcommand{\schwartzprime}{\mathscr{S}'}
\newcommand{\fourier}{\mathcal{F}}

\newcommand{\crochetjap}[1]{\left\langle#1\right\rangle}

\DeclareMathOperator{\supp}{supp}
\DeclareMathOperator{\op}{Op}

\newcommand{\w}{\text{w}}
\newcommand{\qr}{\text{R}}

\newcommand{\comm}[2]{\left[ #1, #2\right]}

\newcommand{\bra}[1]{\left| #1 \right\rangle}
\newcommand{\ket}[1]{\left\langle #1 \right|}

\author{Ngoc Nhi Nguyen\thanks{Laboratoire de Mathématiques d'Orsay, CNRS, Université Paris-Saclay, 91405 Orsay, France.} \thanks{Current affilation: Università degli Studi di Milano, Dipartimento di Matematica, Via Cesare Saldini 50, 20133 Milano. Email: \href{mailto:ngoc.nguyen@unimi.it}{\texttt{ngoc.nguyen@unimi.it}}}}

\title{\scshape\LARGE\bfseries Fermionic semiclassical Lp estimates}

\date{\footnotesize Date: October 2, 2023. To appear in Journal of Functional Analysis.}
\begin{document}
	
	\maketitle

	\begin{abstract}
		We generalize the semiclassical $L^p$ estimates of Koch, Tataru and Zworski in the setting of Schr\"odinger operators with confining potentials to density matrices.  This is motivated by the problem of the concentration of free fermionic particles in a trapping potential. Our proof relies on semiclassical and many-body tools.
		As an application, we provide bounds on spectral clusters. We also discuss the optimality of the one-body and many-body bounds through explicit examples of quasimodes.
	\end{abstract}

	\tableofcontents

	\section{Introduction}\label{sec:intro}

	This article is devoted to estimates related to the spatial concentration of orthonormal families of eigenfunctions of Schr\"odinger operators	
	\begin{equation*}
		P := -h^2 \Delta + V
	\end{equation*}
	in the semiclassical regime $h\to0$. Here $-\Delta:=\sum_{j=1}^d\partial_{x_j}^2$ is the standard Laplacian, $V:\R^d\to\R$ is the potential and $d\geq 1$ is the spatial dimension. We work in the convenient setting where $V$ is smooth and $V(x)\to\infty$ as $\abs{x}\to\infty$, so that $P$ is a well-defined self-adjoint, bounded from below operator on $L^2(\R^d)$ with a discrete spectrum going to $+\infty$. This model trappes quantum systems, where particles are confined to live in a essentially bounded region of space. We are interested in the concentration properties of $L^2$ normalized eigenfunctions $\{u_h\}$ of $P$ associated to an eigenvalue $E$ (we have in mind the case where $E$ is essentially $h$-independent or bounded in $h$). More precisely, we aim to determine how ``concentrate'' such an eigenfunction may be. To do so, we choose to measure concentration through the possible growth of the norms $\normLp{u_h}{q}{(\R^d)}$ as $h\to 0$, for various choices of $q\in[2,\infty]$. Indeed, in the extreme case where $u_h$ is essentially constant (meaning that it is close to an $h$-independent function), then the norms $\normLp{u_h}{q}{(\R^d)}$ do not grow in $h$. On the contrary, if all the $L^2$-mass of $u_h$ is concentrated in a region $\Omega_h\subset\R^d$ with $\abs{\Omega_h}\to 0$ as $h\to 0$, then we typically have $\normLp{u_h}{q}{(\R^d)}\sim\abs{\Omega_h}^{1/q-1/2}\to\infty$ as $h\to 0$, when $q>2$. In the following, we will consider estimates of the type
	\begin{equation}\label{eq-intro:Lp-est}
		\normLp{u_h}{q}{(\R^d)}\leq C h^{-s},
	\end{equation}
	for some $s>0$ and $C>0$ depending only on $V$ and $E$, which we interpret as a measure of the highest ``rate'' of concentration of eigenfunctions of $P$ as $h\to 0$. 
	Notice that Sobolev embeddings imply the previous estimate with $s=d(\tfrac 12-\tfrac 1q )$, and we will see several cases where this exponent can be improved. Of particular importance is the determination of the optimal value of the exponent $s$, which amounts to construct explicit examples of $u_h$ for which the upper bound \eqref{eq-intro:Lp-est} is also a lower bound (with possibly a different value of $C$).
	Ideally, it would be much clearer if we had a pointwise description of the function $u_h$. But it is very hard to obtain for general potentials in any dimensions. That is why we settle for this rawer version of the concentration with the estimation of $L^q$ norms. 
	
	This strategy to study concentration of functions via $L^q$ norms was invented by Sogge, first in the context of spherical harmonics \cite{sogge1985oscillatory}, and then in the context of general compact Riemannian manifolds without boundary (where $P$ is replaced by $-h^2\Delta_g$, the Laplace-Beltrami operator) \cite{sogge1988concerning}. He not only considered eigenfunctions but more generally functions $u_h$ in spectral clusters, i.e.\ that satisfy $u_h=\indicatrice{\abs{P-E}\leq h}u_h$. Furthermore, he managed to find the optimal exponent $s=s(q)$ on any manifold, and proved that for high values of $q$, the highest rate of concentration was attained for specific functions concentrating around a point (generalizing the zonal spherical harmonics) while for low values of $q$, it was attained for functions concentrating around a curve (generalizing the Gaussian beams on spheres). These results were later extended to the case of Schr\"odinger operators with confining potentials on $\R^d$ (which is the case that we consider here) by Koch and Tataru in \cite{koch2005p}, and their method was revisited from the point of view of semiclassical analysis by Koch, Tataru and Zworski in \cite{koch2007semiclassical}. This last article also treats the more general case of quasimodes $u_h$, i.e.\ that satisfy $(P-E)u_h=\OR_{L^2}(h)$, and we will follow on their approach. Notice that many works were devoted to the improvement of Sogge's estimates for eigenfunctions, in specific geometries (typically with negative curvature, see for instance \cite{hassel2015improvedcrit,hezani2016Lpnorms,sogge2017improvedcrit,blair2019log}). We will not pursue this direction here.
	The case of confining potentials is more complicated than the one of compact manifolds without potentials, due to the presence of a transition region between the classically allowed region $\{V\leq E\}$ and the classically forbidden region $\{V>E\}$. Indeed, Koch and Tataru discovered specific concentration phenomena in the transition region that did not appear in Sogge's work. This can be first understood in the one dimensional case, as we will explain below.
	
	In this article, we investigate the more general situation of concentration of orthonormal families of quasimodes. This is motivated by the study of fermionic systems in quantum mechanics, where a well-known example of systems of $N$ uncorrelated fermionic states are Slater determinants $u_h^1\wedge\cdots\wedge u_h^N(x^1,\ldots,x^N)=\frac{1}{\sqrt{N!}}\det((u_h^j(x_i))_{1\leq i,j\leq N})$, which are associated to $N$ orthonormal functions $\{u^j_h\}_{1\leq j\leq N}$ in $L^2(\R^d)$.
	 The orthonormality is a manifestation of Pauli's exclusion principle, which states that two fermions cannot occupy the same quantum state. Intuitively, it means that two fermions cannot concentrate in the same region in space. Hence, while a single particle may be localized in a small region, many particles will tend to delocalize by this ``repulsion'' induced by Pauli's principle. To measure quantitatively the concentration of several particles, it is useful to introduce the spatial density of particles
	\begin{equation*}
		\rho_h=\sum_{j=1}^N \abs{u^j_h}^2 ,
	\end{equation*}
	and estimate the growth in $h$ of its $L^{q/2}$-norms, if each of the $u^j_h$ is a quasimode of $P$. For $N=1$, we recover the question mentioned above. This quantity is interesting because it provides, up to a normalization, the spatial repartition of the fermionic system. Note that this density corresponds to the density of the one-body operator of the Slater determinant $\Psi_N=u_h^1\wedge\cdots\wedge u_h^N\in L^2((\R^d)^N)$, defined by its associated integral kernel
	\begin{equation*}
		(x,y)\in\R^d\times\R^d\mapsto N\int_{\R^{d(N-1)}} \Psi_N(x,x^2,\ldots,x^N)\overline{\Psi_N(y,x^2,\ldots,x^N)} dx^2\ldots dx^N.
	\end{equation*}
	Actually, as we will see below, we instead look at more general densities of the form \eqref{eq-intro:density}. As for the one-body case, the pointwise expression of these objects is not always easy to study, even in the case without interaction\footnote{The interacting case is even more complicated to deal with. Moreover, it is far from obvious that we can deduce the same type of $L^{q/2}$-estimates. Indeed, one can expect the interactions affect the concentration of the particles.
	}. 
	That is why we study instead its $L^{q/2}$-norms. We can also see the measure of the spatial concentration as the quotient of the $L^{q/2}$-norms by the $L^2$ one.
	One can estimate trivially using the triangle inequality and the $N=1$ estimate \eqref{eq-intro:Lp-est},
	\begin{equation*}
		\normLp{\rho_h}{q/2}{(\R^d)}\leq \sum_{j=1}^N \normLp{u^j_h}{q}{(\R^d)}^2 \leq C^2 h^{-2s} N.
	\end{equation*}
	Our goal in this work is to prove estimates of the type\footnote{See below Theorems \ref{thm-intro:res-spectral-clusters} and \ref{thm-intro:red-microloc} for more precise and general statements. We have actually also an other factor $1/\log(1/h)$ at some power $t\geq 0$, but we omit it in the first part of the introduction.
}
	\begin{equation}\label{eq-intro:Lp-est-manybody}
		\normLp{\rho_h}{q/2}{(\Omega)} \leq Ch^{-2s} N^\theta,
	\end{equation}
	for some $\theta\in[0,1]$, in some regions $\Omega\subset\R^d$. Notice that this estimate reduces to \eqref{eq-intro:Lp-est} in the case $N=1$, and that it is a strict improvement of \eqref{eq-intro:Lp-est} only if $\theta<1$ by the above argument.
	In this paper, we have made the choice to prove many-body estimates so that we recover the best possible exponent $s$ (in our case the existing exponents $s$ in \cite{koch2007semiclassical}) for $N=1$, and then we try to obtain the smallest $\theta$ possible with respect to this constraint. However, one can have different values of $s$ and $\alpha$, as pointed out in Remark \ref{rmk:ELp-gene-matdens-bis} (when we can have a better $\alpha$ but a worse $s$). An other approach would be to prove the estimates for a large $N_h$ (for instance in power of $h^{-1}$) with the smallest exponent $\theta$ possible, and then interpolate these estimates with for instance the one obtained for smaller range $N=1$. It is worth to mentioning in our proof we do not really prove the case $N=1$, except for the exponent $q=\infty$. But in this case, it is because the proof for any given $N$ does not need anything new compared to the existing estimates \eqref{eq-intro:Lp-est}. 
	In the case of the Laplace-Beltrami operator on compact manifolds, it was done in \cite{frank2017spectral}, where the sharp exponent $\theta=\theta(q)$ was found. Here, we generalize their work to the case of confining potentials.
	From the point of view of physics, the statistical properties of systems of non-interacting trapped fermions and in particular of their possible scales of concentration has attracted some attention recently \cite{dean2015universal,dean2018wigner,dean2019noninteracting,dean2021impurities}, and our work goes in a similar direction.

	The fact that enough fermions tend to delocalize can be understood by the pointwise Weyl law, which informally states that 
	\begin{equation}\label{eq-intro:weyl-law}
		\rho_h(x) \sim_{h\to 0} \frac{\abs{B_{\R^d}(0,1)}}{(2\pi h)^d} (E-V(x))_+^{d/2},
	\end{equation}
	when the $u^j_h$ are chosen to be an orthonormal family of eigenfunctions associated to all the eigenvalues less that $E$ of $P$. For this choice of $\{u^j_h\}_j$, the $L^{q/2}$-norms of $\rho$ are of the same order $Ch^{-d}$ for all $q$, which underlines delocalization. Actually, this delocalization also occurs for a much lower number of functions. Indeed, if one does not consider all the eigenvalues less than $E$ of $P$, but only the eigenvalues between $E-h$ and $E>\min V$, one can show (for $d\geq 2$, see Section \ref{subsec:optim-manybody}) that all the $L^{q/2}$-norms are also of the same order ($Ch^{-(d-1)}$ in this case), so that delocalization is also true for this much smaller spectral window. We will see below that this example is very important to prove the sharpness of the exponent $\theta$ that we obtain in our estimates of the type \eqref{eq-intro:Lp-est-manybody}. This is why \eqref{eq-intro:Lp-est-manybody} measures the transition between the localization for small $N$ and the delocalization for $N$ large enough: when $N=1$, it is saturated by concentrated functions while for $N$ large, it is saturated by a delocalized system of functions. On compact manifolds with $V=0$, \eqref{eq-intro:weyl-law} was made rigorous by Avakumovic \cite{avakumovic1956uber}, Levitan \cite{levitan1952asymp} and H\"ormander \cite[Thm. 1.1]{hormander1968spectral}. For confining potentials, this asymptotic fails close to the transition region $\{V=E\}$ and a pointwise Weyl law was proved for $V(x)=\abs{x}^2$ in \cite{karad1989} and more recently for general potentials in \cite{deleporte2021universality}.
	
	To understand what happens in the transition region $\{V=E\}$, and also to illustrate the transition between localization and delocalization, it is useful to consider the case of the harmonic oscillator $V(x)=x^2$ in $d=1$, for which many explicit computations are available. For instance, asymptotics as $h\to 0$ of individual eigenfunctions $u_h$ associated to an eigenvalue $E>0$ (independent of $h$) are very well understood using Wentzel-Kramers-Brillouin (WKB) methods (see for instance \cite{olver1997asymptotics}) as depicted in Figure \ref{fig:oh-hr-eigenfunction}: in the classically allowed region $\{V<E\}$, $u_h$ has size 1 (and oscillates, which is not measured by $L^q$-norms) and in the classically forbidden region $\{V>E\}$, it is exponentially decaying (both in $h$ and $\abs{x}$). An interesting phenomenon appears in the transition region $\{V=E\}$, since $u_h$ has size $h^{-1/6}$ in a neighborhood of size $h^{2/3}$ of this region. One can thus see a concentration phenomenon which does not happen in the absence of a potential. Notice also that the concentration is only visible at the level of $L^q$-norms for large $q$ because $\normLp{u_h}{q}{(\R^d)}\sim1$ for $2\leq q\leq 4$ and $\normLp{u_h}{q}{(\R^d)}\sim h^{-\frac 16+\frac 2{3q}}$ for $q\geq 4$. Asymptotics of $\rho_h$, when $u^j_h$ fill all the energy levels up to $E$, are also well-known by the same method as depicted in Figure \ref{fig:oh-spectral-proj}: in the classically allowed region $\{V<E\}$, $\rho_h$ has size $h^{-1}$ and in the classically forbidden region $\{V>E\}$, it is also exponentially decaying. In the transition region, it displays some concentration, but contrary to the case of individual eigenfunctions, it is too small compared to the bulk $\{V<E\}$, so it is invisible in the $L^{q/2}$-norms. In this case, all the  $L^{q/2}$-norms are of the same order $h^{-1}$.
	Of course, such a precise pointwise information is very specific to the one-dimensional case and one cannot hope to extend it to higher dimensions. The results of \cite{koch2005p,koch2007semiclassical} cover the higher dimension case using $L^q$-norms, at the level of eigenfunctions/quasimodes. We extend their results to the case of several functions. These one-dimensional examples also show the different behavior according to the different regions $\{V<E\}$, $\{V>E\}$ and $\{V=E\}$, and the higher dimensional results will also take into account these differences.
	
	%%%%%%%%%%%%%%% DESSIN OH %%%%%%%%%%%%%%%
	
	\begin{figure}[!h]
		\centering
		\includegraphics[scale=0.2]{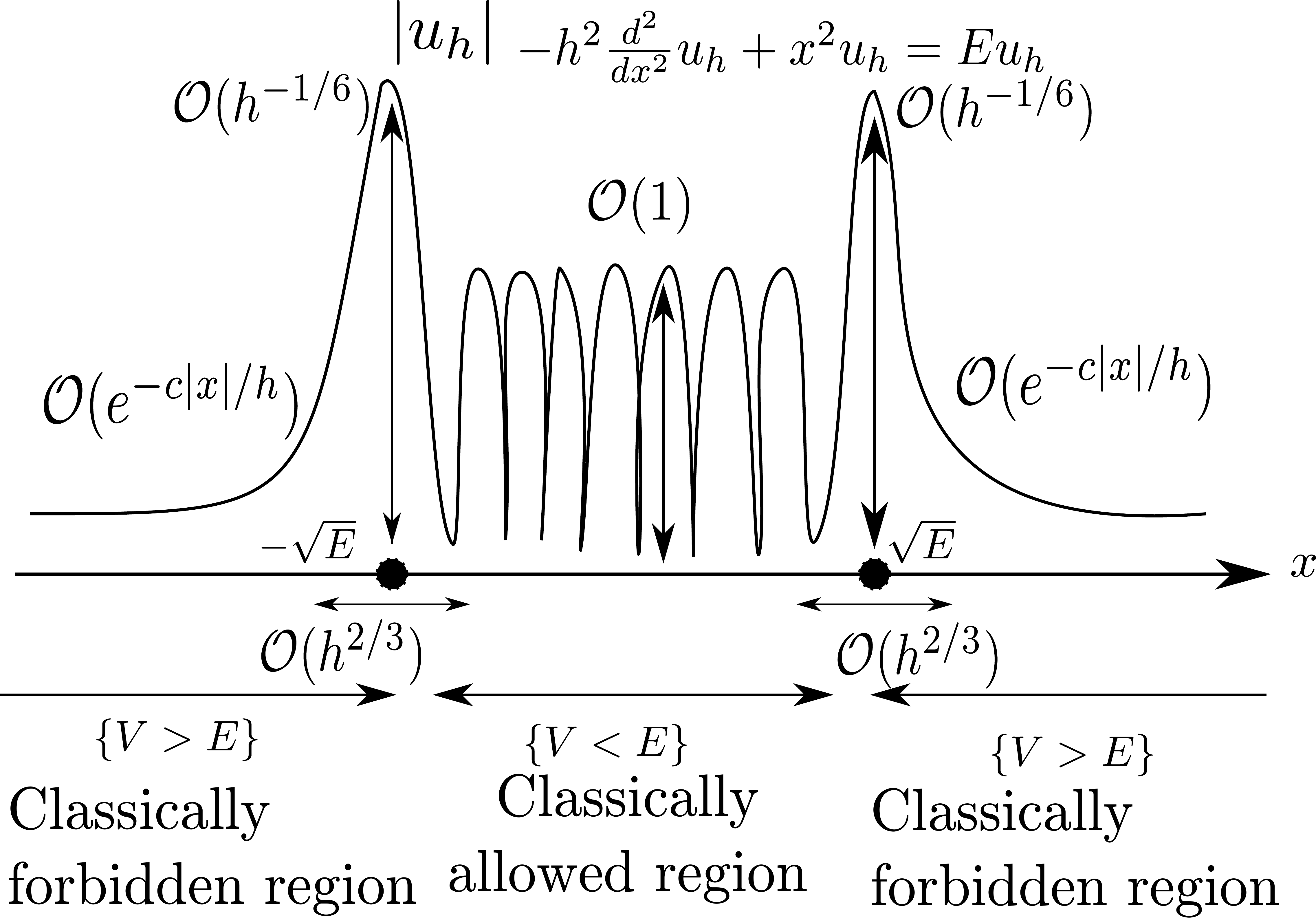}
		\caption{Eigenfunction of the scalar harmonic oscillator associated to the eigenvalue $E$.}
		\label{fig:oh-hr-eigenfunction}
	\end{figure}

	\begin{figure}[!h]
		\centering
		\includegraphics[scale=0.2]{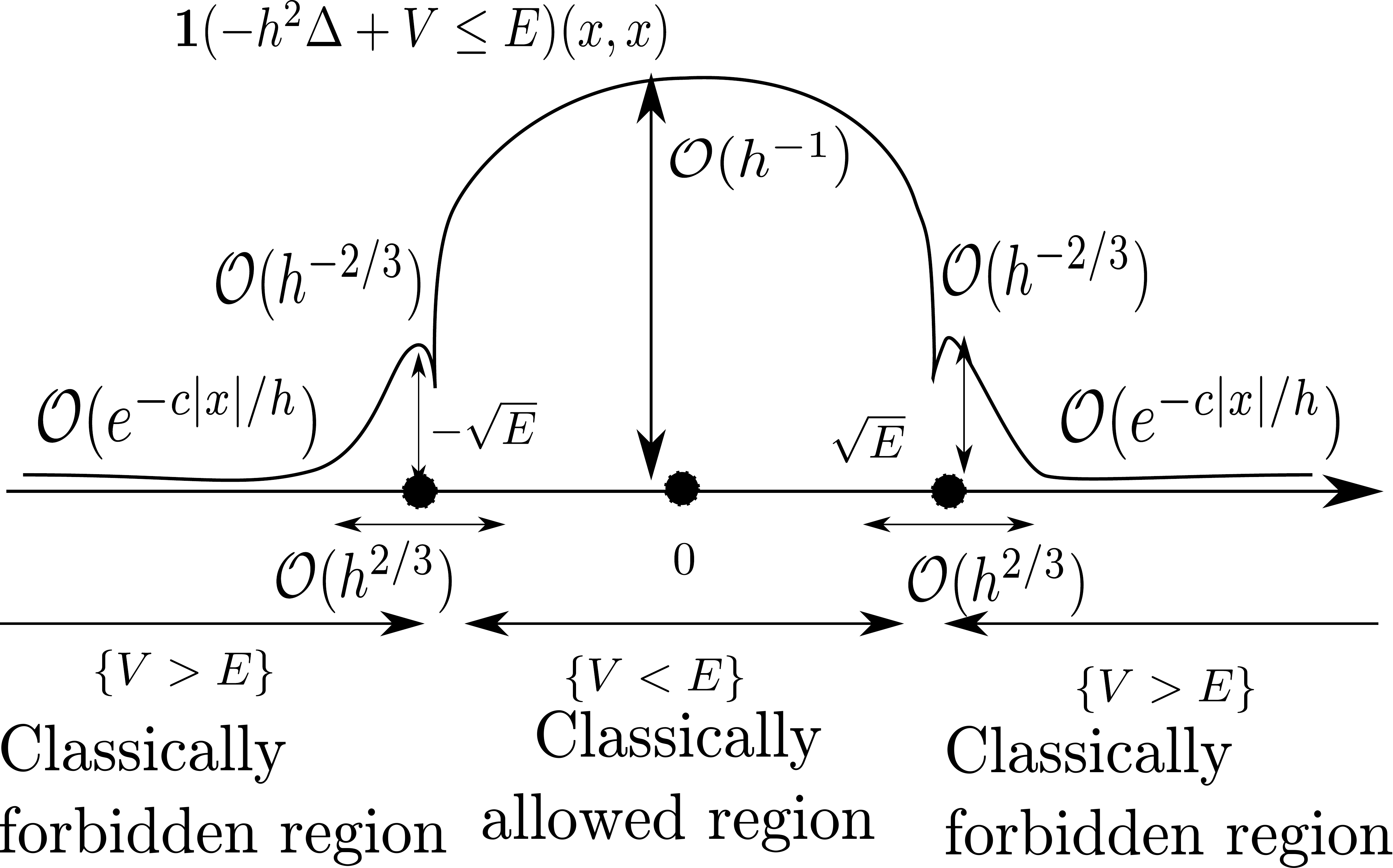}
		\caption{Spectral projector of the scalar harmonic oscillator.}
		\label{fig:oh-spectral-proj}
	\end{figure}
	
	%%%%%%%%%%%%%%% fin DESSIN OH %%%%%%%%%%%%%%%

	Let us now summarize (in a simplified way) our main results, which precise statements are in Theorem \ref{thm:spectral-clusters}.
	First, we show that for any $E>\min V$ and for any $\varepsilon\in(0,E-\min V)$, one has
	\begin{equation*}
		\normLp{\rho_h}{q/2}{(\{V<E-\varepsilon\})} \leq Ch^{-2s}N^\theta
		,
	\end{equation*}
	for any orthonormal systems $\{u^j_h\}_{1\leq j\leq N}$ of eigenfunctions associated to eigenvalues in $[E-h,E]$, for any $N$, with sharp values of $s$ and $\theta$ (which are the same as on compact manifolds without potential). The sharp value of $s$ is obtained for $N=1$, while the sharp value of $\theta$ is obtained choosing the maximal number of such $\{u^j_h\}_{j}$. This case is the same as what happens on compact manifolds since we are far from the transition region. 
	Around the transition region, we obtain a similar estimate
	\begin{equation*}
		\normLp{\rho_h}{q/2}{(\{\abs{V-E}\leq \varepsilon\})} \leq Ch^{-2s}N^\theta
	,
	\end{equation*}
	with different values of exponents $s$ and $\theta<1$, and under the important assumption that $\nabla_x V\neq 0$ on $\{V=E\}$. These estimates are typically not sharp, even for $N=1$, as noticed in \cite{koch2005p}, because there are obtained by summing rescaled estimates on multiple scales intermediate between a neighborhood of size $h^{2/3}$ of the transition region and the bulk. It is rather the estimates on each of these individual intermediate scales that are sharp (that is, the value of $s$ is sharp), as proved in \cite{koch2005p}. The sharpness of the exponent $\theta$ in these scales is an open question. Finally, we obtain estimates of the type
	\begin{equation*}
		\normLp{\rho_h}{q/2}{(\R^d)} \leq Ch^{-2s}N^\theta
		,
	\end{equation*}
	without any assumption on $E$ or on the behavior of $V$ on $\{V=E\}$. This is useful for instance in the case where $E=\min V$, in which there is no bulk. The exponent $s$ is also sharp using again $N=1$ (the saturation happening for the ground state of $P$), while the sharpness of the exponent $\theta$ is also open.

	Let us now comment on the methods of proof and detail the structure of the paper.
	As we already said, we use the strategy of \cite{koch2007semiclassical} based on microlocal analysis, mixed with the many-body tools of \cite{frank2017spectral}. First, we notice (as we will see in Section \ref{sec:app-spectral-clusters}) that it is enough to estimate functions $u_h$ or $u^j_h$ that are microlocalized, meaning that they ``live'' in a compact region in the phase space $\R^d\times\R^d$. This is because spectral localization implies microlocalization for elliptic operators (see for instance \eqref{eq:spectral-loc}). By compactness, it is thus enough to treat functions that are microlocalized around a point. Then, the properties of the classical symbol of $P$,
	\begin{equation*}
		p_E(x,\xi)=\abs{\xi}^2+V(x)-E,\quad (x,\xi)\in\R^d\times\R^d,
	\end{equation*} 
	at this point intervene.
	In Section \ref{sec:elliptic-est}, we will treat elliptic points where $p_E\neq 0$. There, the main tools are Sobolev embeddings in the one-body case and the Kato-Seiler-Simon inequalities in the many-body case. In the region where $p_E=0$, several cases are distinguished:
that satisfy some geometric assumptions $(cond)$ with respect to the energy level sets of $p_E(x,\xi)=p(x,\xi)-E$. That can be points out of $\{p_E=0\}$ or in the level set $\{p_E=0\}$ under one of the three conditions stated (see below Table \ref{table:geom-conditions} for the references to corresponding geometric assumptions)% fin corr rmk 6
	\begin{itemize}
		\item In Section \ref{sec:gene}, we give a general estimate which is valid for any potential $V$ and any energy $E$. The proof relies on adding an artificial time variable and use many-body Strichartz estimates in the spirit of \cite{frank2014, frank2017restriction}.
		\item In Section \ref{sec:sogge-est}, we treat the bulk case for which $p_E=0$ but $V\neq E$. The proof relies on seeing one of the $d$ space variables as a time variable and again use many-body Strichartz estimates.
		\item In Section \ref{sec:TP-est}, we treat the turning point region $p_E=0$ and $V=E$, under the no-ndegeneracy assumption that $\nabla_xV\neq 0$ on this set. In a neighborhood of size $h^{2/3}$ of this region, we use a $H^1$-estimate of \cite{koch2007semiclassical} together with the Kato-Seiler-Simon inequality. The remaining region is split into multiple scales $2^{j}h^{2/3}$, and each of them is treated using the estimates of Section \ref{sec:sogge-est} by rescaling.
	\end{itemize}

		\begin{table}[!h]
		\centering
		\begin{tabular}{|c|c|c|}
			\hline
			\multicolumn{2}{|c|}{\bf Conditions on the symbol $p_E:=p(x,\xi)-E$}& Statement \\
			\hline\hline
			{\it Elliptic}  & (ellip) & $p_E\neq 0$  \\
			\hline
			{\it General}  & (gene) & Assumption \ref{cond:gene}  \\
			\hline
			{\it Sogge} (for the bulk)  & (Sogge)& Assumption \ref{cond:sogge} \\
			\hline
			{\it Turning points} & (TP)& Assumption \ref{cond:TP} \\
			\hline
		\end{tabular}
		\caption{Geometric conditions and their abbreviation in the article.}
		\label{table:geom-conditions}
	\end{table}
	
	In Section \ref{sec:app-spectral-clusters}, we gather all the previous estimates to obtain our main results on spectral clusters. Finally, we discuss their optimality in Section \ref{sec:optim}.

	In the following, we will consider a more general description of many-body states than orthonormal functions. Actually, many more fermionic states are described by a one-body operator than Slater determinants.
	Namely, we will consider one-body density matrices (nonnegative compact operators on $L^2(\R^d)$). Such an operator $\gamma$ can be diagonalized in an orthonormal basis $\{u^j\}_{j}$ with associated eigenvalues $\{\lambda_j\}_j\subset\R_+$, and each $\lambda_j$ is interpreted as the average number of particles described by the state $\gamma$ which have wavefunction $u^j$.  In this formalism, the case described above of $N$ orthonormal functions $\{u^j\}_{1\leq j\leq N}$ corresponds to $\gamma=P_N$, the orthogonal projection on the space generated by the $\{u^j\}_{1\leq j\leq N}$. The $N^\theta$ factor in the right-side of \eqref{eq-intro:Lp-est-manybody} is then interpreted as the Schatten norm $\normSch{P_N}{\alpha}{}$ where $\alpha=1/\theta$ (see below for the definition of Schatten spaces). Furthermore, to any one-body density matrix $\gamma$, one can associate a density of particles
	\begin{equation}\label{eq-intro:density}
		\rho_\gamma = \sum_{j=1}^{+\infty} \lambda_j \abs{u^j}^2,
	\end{equation}
	which measures the spatial repartition of the particles described by $\gamma$. We will prove estimates similar to \eqref{eq-intro:Lp-est-manybody}, where $\rho_h$ in the left-side is replaced by $\rho_\gamma$ and $N^\theta$ in the right-side is replaced by $\normSch{\gamma}{\alpha}{}$.

	The exponents $s$ and $\alpha$ have actually a dependence on $d$ and $q$, that we sum up in Table \ref{table:ref-exponents}.
	As mentioned above, we will consider estimates on microlocalized objects. In the one-body setting, it means that one estimates $\normLp{\chi^\w u}{q}{}$ instead of $\normLp{u}{q}{}$ for a fixed $\chi\in\test{\R^d\times\R^d}$, where $\chi^\w$ denotes the Weyl quantization of the localization function $\chi$ (see below for the definition). In the many-body setting, it means that one estimates $\normLp{\rho_{\chi^\w\gamma\chi^\w}}{q/2}{}$ instead of $\normLp{\rho_\gamma}{q/2}{}$. Furthermore, we also mentioned that one could more generally estimate quasimodes $u_h$ (meaning that $\normLp{u_h}{2}{(\R^d)}=1$ and $(P-E)u_h=\OR_{L^2}(h)$). An equivalent way to consider estimates for microlocalized quasimodes is to replace \eqref{eq-intro:Lp-est} by
	\begin{equation*}
		\normLp{\chi^\w u_h}{q}{}\leq Ch^{-s}\left(\normLp{u_h}{2}{(\R^d)}+\frac 1h\normLp{(P-E)u_h}{2}{(\R^d)}\right).
	\end{equation*}
	The generalization to the many-body setting is given by estimates of the type
	\begin{equation}\label{eq-intro:Lp-est-micr-manybody}
		\normLp{\rho_{\chi^\w\gamma\chi^\w}}{q/2}{}\leq Ch^{-2s}\left(\normSch{\gamma}{\alpha}{}+\frac{1}{h^2}\normSch{(P-E)\gamma(P-E)}{\alpha}{}\right).
	\end{equation}
	The advantage of such a formulation is that the microlocalization is imposed by $\chi^\w$ and the property to be a quasimode is related to the choice of the norm in the right-side. Hence, we may prove \eqref{eq-intro:Lp-est-micr-manybody} for general $\gamma$, the restriction to be a microlocalized quasimode being included in the form of the inequality.
	We will prove such estimates in Sections \ref{sec:elliptic-est} to \ref{sec:TP-est}, with different values of $s$ and $\alpha$ according to the properties of $p_E$ on $\supp\chi$.

	Let us group and summarize more precisely the main results informally mentioned before.
	\medskip
	\begin{table}[!h]
		\centering
		\begin{tabular}{|c|c|c|c|}
			\hline
			&  \multicolumn{3}{|c|}{\bf Exponents}\\
			\hline
			\bf Condition & $s_{\rm cond}$ & $t_{\rm cond}$ & $\alpha_{\rm cond}$ \\
			\hline
			Comparison when $d\geq 3$ & Figure \ref{fig:comp-exp-s} & Figure \ref{fig:comp-exp-t} & Figure \ref{fig:comp-exp-alpha}\\
			\hline\hline
			(ellip) &  \multicolumn{3}{|l|}{defined in Eq. \eqref{eq-def:exp-s-alpha-ELp-ellip}, see Figure \ref{fig:exp-s-alpha_ellip}}  \\
			%	  &\multicolumn{3}{|l|}{}\\
			\hline
			(gene) &  \multicolumn{3}{|l|}{defined in Theorems \ref{thm:ELp-gene-1body} and \ref{thm:ELp-gene-matdens}}\\
			\hline
			(Sogge)& defined in Eq. \eqref{eq-def:s-ELp-sogge}  & $t_{\rm Sogge}=0$ &  defined in Eq. \eqref{eq-def:alpha-ELp-sogge-matdens}\\
			\hline
			(TP)& \multicolumn{3}{|l|}{defined in Theorems \ref{thm:ELp-TP-1body} and \ref{thm:ELp-TP-matdens}}\\
			\hline
		\end{tabular}
		\caption{Reference to all appearing exponents.}
		\label{table:ref-exponents}
	\end{table}
	\begin{thm-princ}[Spectral cluster $L^q$ estimates, see Theorem \ref{thm:spectral-clusters} and Section \ref{subsec:optim-manybody}]
		\label{thm-intro:res-spectral-clusters}
		Let the symbol $p(x,\xi)=|\xi|^2+V(x)$ with a confining\footnote{to be defined in Section \ref{sec:app-spectral-clusters}} potential $V:\R^d\to\R$, $E\in\R$, $\varepsilon>0$ be a small error. Let us denote by $P$ by the Schr\"odinger operator $-h^2\Delta+V$ and by $\Pi_{E,h}$ the spectral projector on the window $[E-h, E+h]$
		\begin{equation*}
			\Pi_{E,h}:=\indicatrice{P\in[E-h,E+h]}.
		\end{equation*}
		\item[\quad\underline{\emph{Upper bounds.}}] 
		Assume that	$\Omega=\R^d$, or $\Omega=\{V>E+\varepsilon\}$ in the classical forbidden region, $\Omega=\{V<E-\varepsilon\}$ in the bulk, or in a neighborhood of the turning points $\Omega=\{|V-E|<\varepsilon \}$ under the additional assumption that all points $x\in\R^d$ in $\{V=E\}$ must satisfy the condition $\nabla_xV(x)\neq 0$. 
		Then, there exist
		\begin{itemize}
			\item $h_0=h_0(E,\varepsilon)>0$,
			\item fixed exponents $s\in[-\infty,\infty)$, $t\geq 0$ and $\alpha\in[1,\infty]$ (to be defined below),
			\item a multiplicative constant $C=C(\Omega,V,E,\varepsilon)>0$,
		\end{itemize}
		such that for any $2\leq q\leq\infty$, any $h\in(0,h_0]$ and any $\gamma_h$ such that $\gamma_h=\Pi_{E,h}\gamma_h=\gamma_h\Pi_{E,h}$
		\begin{equation*}\label{eq-intro-eng:res_spec-clust-ELp-matdens}
		\fbox{$
			\normLp{\rho_{\gamma_h}}{q/2}{(\Omega)}
			\leq C h^{-2s(q,d)}\log(1/h)^{2t(q,d)} \normSch{\gamma_h}{\alpha(q,d)}{(L^2(\R^d))}.
			$}
		\end{equation*}
		with
		\begin{equation*}
			(s(q,d),t(q,d),\alpha(q,d))=
			\begin{cases}
			(s_{\rm gene}(q,d),t_{\rm gene}(q,d),\alpha_{\rm gene}(q,d)) &\text{for any}\ \Omega\subset\R^d_x, \\
			(-\infty,0,\infty) &\text{if}\ \Omega=\{V>E+\varepsilon\}, \\
			(s_{\rm Sogge}(q,d),t_{\rm Sogge}(q,d),\alpha_{\rm Sogge}(q,d))&\text{if}\ \Omega=\{V<E-\varepsilon\},\\
			(s_{\rm TP}(q,d),t_{\rm TP}(q,d),\alpha_{\rm TP}(q,d)) &\text{if}\ \Omega=\{|V-E|<\varepsilon\} 
			%\\&\quad\text{ with the additionnal turning point assumption}
			.
			\end{cases}
		\end{equation*}
		\item[\quad\underline{\emph{Optimality.}}] 
		Moreover, under additive assumption of ``flatness''\footnote{see Assumption \eqref{eq:prop:optim-highr:ELp-gene-1body}} around the energy $E\in\R$, one has the optimality of the exponent $\alpha_{\rm Sogge}$ in a classically allowed region.
		There exist a multiplicative constant $C=C(d,h_0,V,\varepsilon)>0$, $h_0>0$, a sequence of energies $E_h\in\R$ in a neighborhood of $E$ and of density matrice $\gamma_h$ such that for any $2\leq q\leq\infty$, any $h\in(0,h_0]$
		\begin{equation*}
			\fbox{$
				\normLp{\rho_{\gamma_h}}{q/2}{(\{V<E_h-\varepsilon\})}\geq C h^{-2s_{\rm Sogge}(q,d)}\log(1/h)^{2t_{\rm Sogge}(q,d)} \normSch{\gamma_h}{\alpha_{\rm Sogge}(q,d)}{(L^2(\R^d))}
				.
			$}
		\end{equation*}
	\end{thm-princ}
	%%$
	\medskip
	We also sum up Theorems \ref{thm:ELp-elliptic-matdens}, \ref{thm:ELp-gene-matdens}, \ref{thm:ELp-sogge-matdens} and \ref{thm:ELp-TP-matdens} in the case of Schr\"odinger operators. Actually, the results are stated for a more general class of pseudodifferential operators.
	\begin{thm-princ}[Microlocalized $L^q$ estimates]\label{thm-intro:red-microloc}
		Let the symbol $p(x,\xi)=|\xi|^2+V(x)$ with a confining\footnote{We will later that the condition of ``almost polynomial growth'' (Definition \ref{cond:am-potential-pol-growth}) is enough.} potential $V:\R^d\to\R$, $E\in\R$, $\varepsilon>0$ be a small error. Let us denote by $P$ by the operator $-h^2\Delta+V$.
		For any point  $(x_0,\xi_0)\in\R^d\times\R^d$ that satisfies one of the geometric conditions in Table \ref{table:geom-conditions}: $(cond)=(ellip)$, $(gene)$, $(Sogge)$ or $(TP)$, there exist fixed $s_{\rm cond}\geq 0$, $t_{\rm cond}\geq 0$ and $\alpha_{\rm cond}\in[1,\infty]$ associated to $(cond)$ (see Table \ref{table:ref-exponents} above) and
		\begin{itemize}
			\item an open bounded neighborhood $\UR\times\VR$ of $(x_0,\xi_0)$,
			\item $h_0>0$,
		\end{itemize}
		such that for any $\chi\in\test{\R^d\times\R^d}$ supported into $\UR\times\VR$, there exists a multiplicative constant $C=C(d,\chi,h_0)>0$, such that for any $2\leq q\leq\infty$, any $h\in(0,h_0]$ and any non-negative self-adjoint operator $\gamma$ on $L^2(\R^d)$, we have
		\begin{equation*}\label{eq-intro-eng:res_microloc-ELp-matdens}
		\fbox{$
			\begin{split}
			\normLp{\rho_{\chi^\w(x,hD)\gamma\chi^\w(x,hD)}}{q/2}{(\R^d)}
			&
			\leq C h^{-2s_{\rm cond}(q,d)}\log(1/h)^{2t_{\rm cond}(q,d)} 
			\times\\&\quad\times
			\left(\normSch{\gamma}{\alpha_{\rm cond}(q,d)}{(L^2(\R^d))}+\frac 1{h^2}\normSch{(P-E)\gamma(P-E)}{\alpha_{\rm cond}(q,d)}{(L^2(\R^d))}\right).
			\end{split}
			$}
		\end{equation*}
	\end{thm-princ}

\bigskip
	Implicitly, we can deduce $L^q$ bounds for density of family of non-negative bounded operators $\{\gamma_h\}_{h\in (0,h_0]}$ with an integral kernel, that are ``many-body quasimode'' of $P$ in nuclear type norm for any $h\in(0,h_0]$
	\begin{equation*}\forall \alpha\geq 1,\quad
		\gamma_h=\OR_{\schatten^\alpha}(1) \quad\text{ and }\quad
		(P-E)\gamma_h, \: \gamma_h(P-E) = \OR_{\schatten^\alpha}(h)
		,
	\end{equation*}
	with a spectral or a phase-space localization assumption
	\begin{itemize}
		\item $\gamma_h=\Pi_{E,h}\gamma_h\Pi_{E,h}+\OR_{\schatten^\alpha_{P,E}}(h^\infty)$,
		\item or such that there exists a compact $K\subset\R^d\times\R^d$ such that for any $\chi\in\test{\R^d\times\R^d}$ supported in $K$, $\gamma_h=\chi^\w(x,hD)\gamma_h\chi^\w(x,hD)+\OR_{\schatten^\alpha_{P,E}}(h^\infty)$.
	\end{itemize}
	Here, we denote $\gamma=\OR_{\schatten^\alpha_{P,E}}(h^\infty)$ if the weighted norm $\normSch{\gamma}{\alpha}{(L^2(\R^d))}+\frac 1{h^2}\normSch{(P-E)\gamma(P-E)}{\alpha}{}$ or $\normSch{(1+(P-E)^2/h^2)^{1/2}\gamma(1+(P-E)^2/h^2)^{1/2}}{\alpha}{}$
	are $\OR(h^\infty)$. Note that one can also replace $\OR_{\schatten^\alpha_{P,E}}(h^\infty)$ with $\OR_{\schatten^\alpha_{P,E}}(h^N)$ for a large enough fixed $N>0$.
	For now, the question of getting rid of the localization property is still opened, even in the one-body case.

	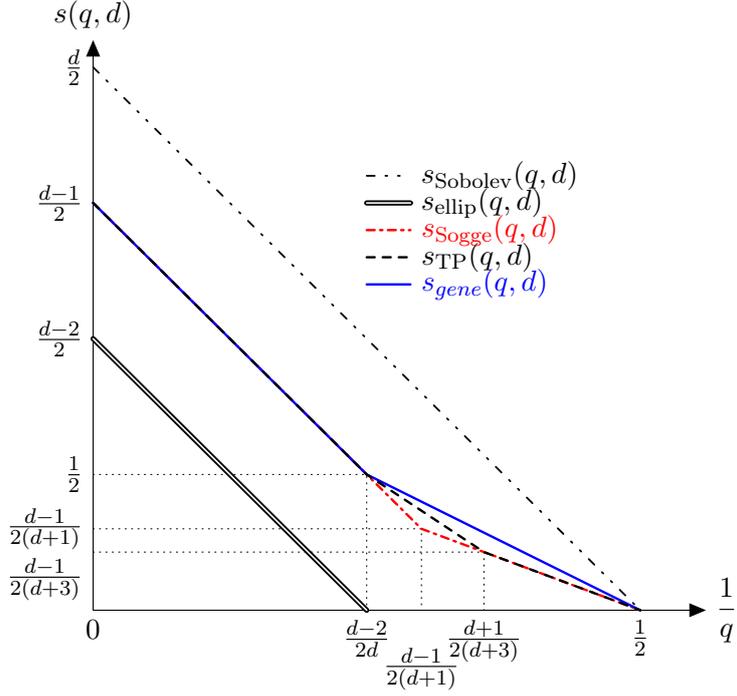
\begin{figure}[h!]
		\centering
		%%%%%%%%%%%%%%%%%%%%%% DESSIN ktz d>2 %%%%%%%%%%%%%%%%%%%%%%%%%%%%
		\begin{center}\begin{tikzpicture}[line cap=round,line join=round,>=triangle 45,x=12 cm,y=3 cm,scale=1.2]
			d=4
			\draw[->] (0.,0.) -- (0.56,0.);
			\draw (0.56,0.) node[right] {$\displaystyle{\frac{1}{q}}$};
			\draw[->] (0.,0.) -- (0.,2+0.1);
			\draw (0.,2+0.1) node[above] {$\displaystyle{s(q,d)}$};
			%
			%% abscisses
			%
			\draw (0,0) node[below]{$0$};
			\draw (1/4,0.) node[below]{$\frac{d-2}{2d}$};
			\draw (3/10,-0.1) node[below]{$\frac{d-1}{2(d+1)}$}; %%-
			\draw (5/14,0.) node[below]{$\frac{d+1}{2(d+3)}$}; 
			\draw (0.5,0.) node[below]{$\frac{1}{2}$};
			%
			%% ordonnées
			\draw (0,0.5) node[left]{$\frac{1}{2}$};
			\draw (0,3/10+0.) node[left]{$\frac{d-1}{2(d+1)}$};
			\draw (0.,3/14-0.1) node[left]{$\frac{d-1}{2(d+3)}$}; %%-
			\draw (0.,3/2) node[left]{$\frac{d-1}{2}$};
			%
			%% points
			%
			\draw[dotted] (0., 0.5)-|(1/4, 0.); % ep KT
			\draw[dotted] (0., 3/10)-|(3/10, 0.); % ep sogge
			\draw[dotted] (0., 3/14)-|(5/14, 0.); % ep sogge
			%
			%% courbes
			%
			% SOBOLEV
			\draw (0.,2) node[left]{$\frac{d}{2}$};
			\draw[line width=0.7pt,loosely dashdotdotted](0,2)--(0.5,0);
			%  ellip
			\draw (0,1) node[left]{$\frac{d-2}2$};
			\draw[line width=0.7pt,double](0,1)--(1/4,0);
			% sogge
			\draw[ line width=0.9pt, color=red,dash dot] (0.,3/2) --(1/4,0.5)-- (3/10,3/10)--(0.5,0.);
			% gene
			\draw[ line width=0.9pt, color=blue] (0.,3/2) -- (1/4,0.5)--(0.5,0.); % spec
			\draw[dashed, line width=0.9pt] (0.,3/2) -- (1/4,0.5)--(5/14,3/14)--(0.5,0.);
			%
			% legende
			%
			% sobolev
			\draw[line width=0.7pt,loosely dashdotdotted] (1/4,3/2+0.1)-|(1/4+.04,3/2+0.1)  node[right]{$ s_{\text{Sobolev}}(q,d)$};
			% ellip
			\draw[line width=0.7pt,double] (1/4,3/2)-|(1/4+.04,3/2)  node[right]{$ s_{\text{ellip}}(q,d)$};
			% sogge
			\draw[line width=0.9pt, color=red,dash dot] (1/4,3/2-0.1)-|(1/4+.04,3/2-0.1)  node[right]{$s_{\text{Sogge}}(q,d)$};
			% spe
			\draw[line width=0.9pt, dashed]
			(1/4,3/2-0.2)-|(1/4+.04,3/2-0.2) node[right]{$s_{\text{TP}}(q,d)$};
			% gene
			\draw[line width=0.9pt, color=blue]
			(1/4,3/2-0.3)-|(1/4+.04,3/2-0.3)  node[right]{$ s_{gene}(q,d)$};
			\end{tikzpicture}\end{center}
		%%%%%%%%%%%%%%%%%%%%%% FIN DESSIN kzt s(q,d) d>2 %%%%%%%%%%%%%%%%%%%%%%%%%%%%
		\caption{Exponent $s(q,d)$ of microlocalized estimates when $d\geq 2$.}
		\label{fig:comp-exp-s}
	\end{figure}

% CORRECTION Rmk 6

	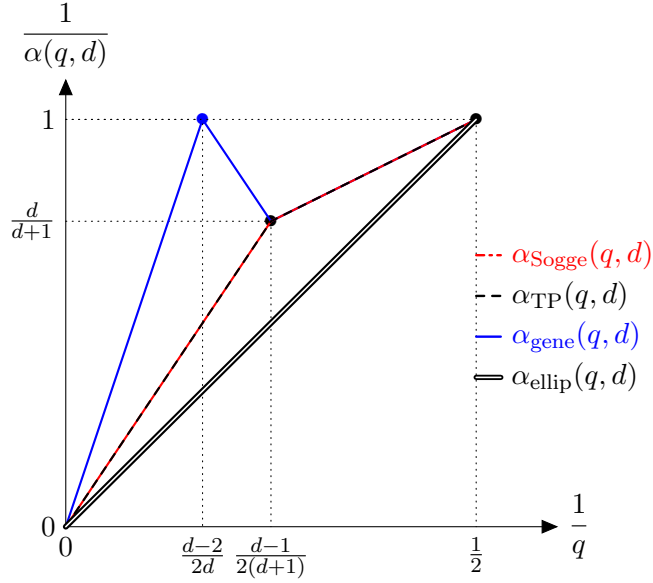
\begin{figure}[!h]
	%%%%%%%%%%%%%%%%%%%%%%  DESSIN alpha(q,d) d>2 %%%%%%%%%%%%%%%%%%%%%%%%%%%%
	\begin{center}\begin{tikzpicture}[line cap=round,line join=round,>=triangle 45,x=6.0cm,y=3.0 cm,scale=1.8]
		%d=4
		\draw[->] (0.,0.) -- (0.6,0.);
		\draw (0.6,0.) node[right] {$\displaystyle{\frac{1}{q}}$};
		\draw[->] (0.,0.) -- (0.,1.1);
		\draw (0.,1.1) node[above] {$\displaystyle{\frac 1{\alpha(q,d)}}$};
		%
		%% abscisses
		%
		\draw (0.,0) node[below]{0};
		\draw (0.5,0.) node[below]{$\frac{1}{2}$};
		%
		%% ordonnées
		%
		\draw (0.,0) node[left]{0};
		\draw (0.,1) node[left]{1};
		%
		%% points
		% point kt
		\draw (1/6,0.) node[below]{$\frac{d-2}{2d}$};
		\draw[color=blue] (1/6,1) node {$\bullet$};
		\draw[dotted] (0, 1)-|(1/6, 0);
		% point sogge
		\draw (1/4,0.) node[below]{$\frac{d-1}{2(d+1)}$};
		\draw (0.,3/4) node[left]{$\frac{d}{d+1}$};
		\draw[color=black] (1/4,3/4) node {$\bullet$};
		\draw[dotted] (0, 3/4)-|(1/4,0);
		% point (1/2,1)
		\draw[dotted] (0, 1)-|(0.5, 0);
		\draw[color=black] (0.5,1) node {$\bullet$};
		%% courbes
		% gene
		\draw[line width=0.8pt, color=blue] (0.,0) --(1/6,1)--(1/4,3/4)--(0.5,1);
		% sogge
		\draw[color=red,line width=0.8pt] (0.,0) --(1/4, 3/4)--(0.5,1);
		% tp
		\draw[dashed,line width=0.8pt] (0.,0) --(1/4, 3/4)--(0.5,1);
		% ellip
		\draw[double,line width=0.8pt] (0.,0) --(0.5,1);
		%%
		% legend
		%
		% sogge
		\draw[line width=0.9pt, color=red,dash dot] (1/2,2/3)-|(1/2+.03,2/3)  node[right]{$\alpha_{\text{Sogge}}(q,d)$};
		% spe
		\draw[line width=0.9pt, dashed]
		(1/2,2/3-0.1)-|(1/2+.03,2/3-0.1) node[right]{$\alpha_{\text{TP}}(q,d)$};
		% gene
		\draw[line width=0.9pt, color=blue]
		(1/2,2/3-0.2)-|(1/2+.03,2/3-0.2)  node[right]{$ \alpha_{\text{gene}}(q,d)$};
		% forbidden
		\draw[line width=0.7pt,double](1/2,2/3-0.3)-|(1/2+.03,2/3-0.3) node[right]{$\alpha_{\text{ellip}}(q,d)$};
		\end{tikzpicture}\end{center}
	%%%%%%%%%%%%%%%%%%%%%% FIN DESSIN alpha(q,d) %%%%%%%%%%%%%%%%%%%%%%%%%%%%
	\caption{Schatten exponent $\alpha(q,d)$ when $d\geq 3$.}
	\label{fig:comp-exp-alpha}
\end{figure}

	\section*{Acknowledgements}
	
	I would like to thank my PhD advisor Julien Sabin for his advices and support. This project has also been partially supported (during the correction period) by the European Research Council (ERC) through the Starting Grant {\sc FermiMath}, grant agreement nr. 101040991. Finally,
	would like to thank the anonymous referee for her or his comments, which led to improvements of this presentation's paper.
	
	\section{Review of semiclassical analysis and density matrices}\label{sec:useful-review}
	
	Before going to the main results and their proofs, we recall the results of semiclassical analysis and density matrix analysis that we will use. We refer to \cite{dimassi1999spectral}, \cite{zworski2012semiclassical} and \cite{simon2005trace} for further details.
	
	\subsection{Symbol classes and quantization}
	
	Let us recall the definitions of order functions and symbol classes. 
	In the following, we will fix $d\in\N^*$ and we will use the notation $\crochetjap{x}:=(1+\abs{x}^2)^{1/2}$, for $x\in\R^d$. 
	
	\begin{defi}[Order functions {\cite[Sec. 4.4.1]{zworski2012semiclassical}}]
		A function $m\in\CR^\infty(\R^d\times\R^d,[0,\infty[)$ is called an \emph{order function} on $\R^d\times\R^d$ if there exist $C,N>0$ such that
		\begin{equation*}
			\forall(x,\xi),(y,\eta)\in\R^d\times\R^d  \quad m(x,\xi)\leq C ( 1+ \abs{x-y}+\abs{\xi-\eta})^N m(y,\eta)  .
		\end{equation*}
	\end{defi}

	\begin{rmk}
		Relevant examples of order functions are $(x,\xi)\mapsto\crochetjap{x}^{k}\crochetjap{\xi}^{\ell}$ for any $k,\ell\in\R$.
	\end{rmk}

	\begin{defi}[Symbols {\cite[Sec. 4.4.1]{zworski2012semiclassical}}]
		Let $m$ be an order function on $\R^d\times\R^d$.
		A function $a\in\CR^\infty(\R^d\times\R^d)$ is a \emph{symbol} in the class $S(m)$ if for all $\alpha\in\N^d\times\N^d$, there exists $C_\alpha>0$ such that
		\begin{equation*}
			\forall (x,\xi)\in\R^d\times\R^d,\quad \abs{\partial^\alpha_{x,\xi} a(x,\xi)}\leq C_\alpha m(x,\xi) .
		\end{equation*}
	\end{defi}

	\begin{rmk}
		In the following, we will consider symbols $a$ which will depend on an external parameter (which can be $h$). In that case, it will be important that the constants $C_\alpha$ are independent of the parameter.
	\end{rmk}
	
	\begin{rmk}
		Below, we will encounter symbols belonging to the Schwartz space $\schwartz(\R^d\times\R^d)$, which is equivalent to belong to the symbol classes $S(\crochetjap{x,\xi}^{-k})$ for all $k\in\N$.
	\end{rmk}

	\begin{nota}
		Let $N\in\R$ and $a\in S(m)$.
		\begin{itemize}
		\item
		We write $a=\OR_{S(m)}(h^N)$ if for any $\alpha\in\N^{2d}$, there exists $C_{\alpha,N}>0$, independent of $h$, such that
		\begin{equation*}\forall (x,\xi)\in\R^d\times\R^d,\quad
			\abs{\partial^\alpha_{x,\xi}a(x,\xi)}\leq C_{\alpha,N}h^Nm(x,\xi).
		\end{equation*}
		\item
		We denote $a=\OR_{S(m)}(h^\infty)$ if $a=\OR_{S(m)}(h^N)$ for any $N\in\N$.
	\end{itemize}
		Similarly, for $a\in\schwartz(\R^d\times\R^d)$, we write $a=\OR_{\schwartz}(h^N)$ (resp. $a=\OR_{\schwartz}(h^\infty)$) if $a=\OR_{S(\crochetjap{x,\xi}^{-k})}(h^N)$ (resp. $a=\OR_{S(\crochetjap{x,\xi}^{-k})}(h^\infty)$) for all $k\in\N$.
	\end{nota}

	%Our main results only apply in the case of Schr\"odinger operators, a polynomial growth condition to $V\in\CR^\infty(\R^d,\R)$ as below.
	It will be important for us that the classical symbol $p(x,\xi)=\abs{\xi}^2+V(x)$ is a symbol in the sense of the above definition. This motivates the following definition of the class of potentials that we consider.
	
	\begin{defi}\label{cond:am-potential-pol-growth}
		A potential $V\in\CR^\infty(\R^d,\R)$ has \emph{at most polynomial growth} if there exists $k\in\N^*$ such that
		\begin{equation}
			\label{eq:cond-potential-symb}
			\forall\alpha\in\N^d,\exists C_\alpha>0,\forall x\in\R^d,\:\abs{\partial^\alpha V(x)}\leq C_\alpha\crochetjap{x}^k 
			.
		\end{equation}
	\end{defi}
	
	\begin{rmk}
		In the above definition, \eqref{eq:cond-potential-symb} implies that $p(x,\xi)=\abs{\xi}^2+V(x)$ is in the symbol class $S(m)$ for $m(x,\xi):=\crochetjap{\xi}^2\crochetjap{x}^k$.
	\end{rmk}

	We will always assume that the potential $V$ is bounded from below. That ensures that the operator $P=h^2\Delta+V$ is also bounded from below.
	\begin{defi}[Boundedness from below]\label{cond:bound-from-below}
		The potential $V$ is bounded from below, more exactly that there exists $C\in\R$ such that for any $x\in\R$, there exists $C\in\R$ such that for any $x\in\R^d$, one has $V(x)\geq C$. 
	\end{defi}
	
	\begin{defi}[Quantization, {\cite[Thm. 4.16]{zworski2012semiclassical}}]\label{def:quantizations}
		Let $m$ be an order function on $\R^d\times\R^d$ and $a\in S(m)$. Let $h>0$.
		Let $\mathfrak{t}\in [0,1]$. The $\mathfrak{t}$-quantization of $a$, denoted by $\op_h^\mathfrak{t}(a)$, is the linear continuous operator $\schwartz(\R^d)\to\schwartz(\R^d)$ defined by the formula
		\begin{equation*}
			\op_h^\mathfrak{t}(a)u(x) = \frac1{(2\pi h)^d} \int_{\R^d}\int_{\R^d} e^{i\frac{\xi\cdot(x-y)}h}  a(\mathfrak{t}x+(1-\mathfrak{t})y,\xi) u(y)d\xi dy
		\end{equation*}
		for any $x\in\R^d$ and any $u\in\schwartz(\R^d)$.
		For $\mathfrak{t}=1/2$, $\op_h^{1/2}(a)$ is called the \emph{Weyl quantization} of $a$ and is also denoted by $a^\w(x,hD)$.
		For $\mathfrak{t}=1$, $\op_h^1(a)$ is called the \emph{right quantization} of $a$ and is also denoted by $a^\qr(x,hD)$.		
	\end{defi}

	\subsection{Semiclassical pseudodifferential calculus}
	
	In the following, we list some standard operations on pseudodifferential operators. We only state them for the Weyl quantization to keep a light notation, but they are still valid for other quantizations (in Definition \ref{def:quantizations}).
	
	\begin{prop}[Composition of pseudodifferential  operators {\cite[Thm. 4.12 and 4.18]{zworski2012semiclassical}}]\label{prop:SA-comp-op}
		Let $m_1$ and $m_2$ two order functions on $\R^d\times\R^d$. Let $a\in S(m_1)$ and $b\in S(m_2)$.
		\begin{itemize}
			\item[1)]
			Then, there exists a symbol in $S(m_1m_2)$, that we denote by $a_\# b$, such that
			\begin{equation*}
				a^\w(x,hD)b^\w(x,hD)=(a_\# b)^\w(x,hD).
			\end{equation*}
			\item[2)]
			Furthermore, there exists a unique family $\{c_j\}_{j\in\N}\subset S(m_1m_2)$ supported in $\supp a \cap\supp b$ such that for any $N\in\N^*$, there exists $r_N\in S(m_1m_2)$ such that
			\begin{equation*}
				a_\# b = \sum_{j=0}^{N-1} h^jc_j+ h^Nr_N.
			\end{equation*}
			Moreover, we have $c_0=ab$.
		\end{itemize}
	\end{prop}
%% c.f. DETAILS phd

	This result has two important corollaries.
	
	\begin{cor}[Disjoint supports {\cite[Thm. 4.12]{zworski2012semiclassical}}]\label{cor:SA-comp-op-disj-spt}
		Let $a\in S(m_1)$ and $b\in S(m_2)$ be such that $\supp a\cap\supp b=\emptyset$. Then, 
		\begin{equation*}
			a_\# b=\OR_{S(m_1m_2)}(h^\infty).
		\end{equation*}
	\end{cor}

	\begin{cor}[Commutator]\label{cor:SA-comm}
		Let $a\in S(m_1)$ and $ b \in S(m_2)$. Then, there exists $r\in S(m_1m_2)$ such that for any $h>0$
		\begin{equation*}
			\comm{a^\w(x,hD)}{b^\w(x,hD)} = hr^\w(x,hD).
		\end{equation*}
	\end{cor}

%% c.f. DETAILS PhD

	The following proposition quantifies the difference between two quantizations of the same symbol.

%% DETAILS phd

	\begin{prop}[Change of quantization {\cite[Thm. 4.13]{zworski2012semiclassical}}]\label{cor:SA-quantif-change}
		Let $a\in S(m)$ and $\mathfrak{t},\mathfrak{s}\in[0,1]$. Then, there exists $\tilde{a}_{t,s}\in S(m)$ such that for any $h>0$
		\begin{equation*}
			\op_h^\mathfrak{t}(a)-\op_h^\mathfrak{s}(a) = h \op_h^\mathfrak{t}\left(\tilde{a}_{\mathfrak{t},\mathfrak{s}} \right) .
		\end{equation*}
%%% DETAILS
	\end{prop}

	Let us now recall the definition of locally elliptic symbols.

	\begin{defi}[Elliptic symbol]\label{def:elliptic-symb}
		Let $m$ be an order function on $\R^d\times\R^d$. A symbol $a\in S(m)$ is \emph{elliptic} on $U\subset\R^d\times\R^d$ if there exists $C>0$ such that $\abs{a(x,\xi)}\geq m(x,\xi)/C$ for all $(x,\xi)\in U$.		
	\end{defi}
	
	The following lemma gives local left and right inverses for quantization of locally elliptic symbols. These microlocal equalities will be very useful in the proof of Theorem \ref{thm:ELp-TP-matdens}.

	\begin{lemma}[{\cite[Lem. 2.1]{koch2007semiclassical}}]\label{lemma:SA-inversion}
		Let $\chi\in S(1)$, $m$ be an order function and $a\in S(m)$ elliptic on $\supp\chi$.
		Then, for any $\mathfrak{t}\in[0,1]$, there exist $b_{\mathfrak{t}}\in S(1/m)$, $r_{1,\mathfrak{t}},r_{2,\mathfrak{t}}\in S(1)$ such that
		\begin{equation*}
			\begin{split}
			\op_h^{\mathfrak{t}}(b_{\mathfrak{t}}) \op_h^{\mathfrak{t}}(a) \op_h^{\mathfrak{t}}(\chi) &= \op_h^{\mathfrak{t}}(\chi) +\op_h^{\mathfrak{t}}(r_{1,\mathfrak{t}}),\\
			 \op_h^{\mathfrak{t}}(a)\op_h^{\mathfrak{t}}(b_{\mathfrak{t}}) \op_h^{\mathfrak{t}}(\chi) &= \op_h^{\mathfrak{t}}(\chi) +\op_h^{\mathfrak{t}}(r_{2,\mathfrak{t}}),
			\end{split}
		\end{equation*}
		where $r_{1,\mathfrak{t}},r_{2,\mathfrak{t}} =\OR_{S(1)}(h^\infty)$.
		If $\chi\in\test{\R^d\times\R^d}$ then $r_1,r_2 =\OR_{\schwartz}(h^\infty)$.
	\end{lemma}
	% fin lemme \ref{lemma:SA-inversion}

	For any real elliptic $p\in S(m)$ (when $m\geq 1$), the operator $P=p^\w(x,hD)$ is self-adjoint on a suitable domain (\cite[Sec. 10.1.2]{zworski2012semiclassical}) so that $f(P)$ is well-defined by functional calculus, for any $f\in\test{\R}$.
	The next theorem states that such a $f(P)$ is actually a pseudodifferential operator and provides us information on its associated symbol.
	This result is crucial for justifying the application of microlocalized estimates (in Sections \ref{sec:elliptic-est}, \ref{sec:gene}, \ref{sec:sogge-est} and \ref{sec:TP-est}) to spectral clusters in Section \ref{sec:app-spectral-clusters}.	
		
	\begin{thm}[{\cite[Thm. 8.7]{dimassi1999spectral}}]\label{thm:funct-calculus}
		Let $m$ be an order function on $\R^d\times\R^d$ such that $m\geq 1$, $p\in S(m)$ be a symbol such that $p+i$ is elliptic on $\R^d\times\R^d$ and $P:=p^\w(x,hD)$. Let $f\in\test{\R}$.
		Then, there exists $a\in \cap_{k\in\N}S(m^{-k})$ such that $f(P)=a^\w(x,hD)$. 
		Moreover, there exist functions  $ \{a_j\}_{j\in\N} \subset \cap_{k\in\N}S(m^{-k})$ supported in $\supp(f\circ p)$ such that for all $N\geq 1$ there exists $r_N\in\cap_{k\in\N}S(m^{-k})$ such that
		\[ a =\sum_{j=0}^{N-1} h^j a_j + h^N r_N  .\]
		In particular, the principal symbol $a_0$ is equal to $f\circ p$.
	\end{thm}

	\subsection{Semiclassical bounds}

	In this section, almost all the results are stated accordingly to symbols $a:\R^d\times\R^d\to\R$.
	But, for the reader's convenience we voluntarily write the dimension $n$ instead of $d$ in Lemma \ref{lemma:prop-F(t)}, Theorem \ref{thm:SStrichartz-bounds}, Theorem \ref{thm:complex-interpol-schatten} and Theorem \ref{thm:SStrichartz-matdens} to draw the attention that they can be different objects. Actually, we essentially will apply these results to $n=d$ (see Section \ref{sec:gene}) and $n=d-1$ (see Section \ref{sec:sogge-est}).

		We state now a natural property on quantizations of Schwartz functions, that is very useful to prove that the density $\rho_{\chi^\w\gamma\chi^\w}$ is well-defined (Lemma \ref{lemma:pre-mercer_loc}) and a corollary of Mercer theorem (Remark \ref{rmk:mercer-thm}).
		\begin{lemma}\label{lemma:cond-traceclass}
			If $a\in\schwartz(\R^d\times\R^d)$, then the integral kernel $\op_h^\mathfrak{t}(\chi)$ is also in $\schwartz(\R^d\times\R^d)$. As a consequence, this operator is trace-classe and Hilbert-Schmidt.
		\end{lemma}
	
	When $a$ is in the Schwartz space $\schwartz(\R^d\times\R^d)$, the operator $a^\w(x,hD)$ not only preserves continuously $\schwartz(\R^d)$ (it is still valid for any symbol $a\in S(m)$), but it has the good property of extending to a regularizing operator.
	
	\begin{prop}[{\cite[Thm. 4.1]{zworski2012semiclassical}}]\label{prop:quantiz-schwartz}
		Let $a\in\schwartz(\R^d\times\R^d)$. Then, for any $h>0$,  the operator $a^\w(x,hD)$ maps continuously $\schwartzprime(\R^d)$ to $\schwartz(\R^d)$.
	\end{prop}	
	
	Let us recall the Calderon-Vaillancourt theorem, which implies the $L^2$-boundness of the quantizations of symbols in $S(1)$.
	
	\begin{prop}[Calderon-Vaillancourt {\cite[Thm. 4.23]{zworski2012semiclassical}}]\label{thm:SA-Cald-Vaill}
		Let $a\in S(1)$. Then, for any $h>0$, the operator $a^\w(x,hD)$ extends to a bounded linear operator on $L^2(\R^d)$, with operator norm bounded uniformly in $h\in(0,1]$.
%% DETAILS thm CV
%		Indeed, there exist $M>0$ and $C_d>0$ such that for any $h>0$
%		\begin{equation*}
%			\norm{a^\w(x,hD)}_{L^2(\R^d)\to L^2(\R^d)}\leq C_d\sum_{\abs{\alpha}\leq Md}h^{\abs{\alpha}/2}\normLp{\partial^\alpha a}{\infty}{(\R^d)}.
%		\end{equation*}
	\end{prop}
	
	Let us state now basic semiclassical $L^q$ estimates, from which one can deduce a semiclassical version of Sobolev embedding for microlocalized functions.
	
	\begin{lemma}[Basic $L^q$ estimates, {\cite[Lemma 2.2]{koch2007semiclassical}}]\label{lemma:SA-basic-ELp}
		Let $a\in\schwartz(\R^d\times\R^d)$. Then, there exists $C>0$ such that for any $h>0$ and any $ 1\leq p\leq q\leq \infty$
		\begin{equation*}
			\normLp{a^\w(x,hD) u}{q}{(\R^d)}\leq Ch^{-d\left(\frac 1p-\frac 1q\right)} \normLp{u}{p}{(\R^d)} .
		\end{equation*}
	\end{lemma}
	% fin lemme\ref{lemma:SA-basic-ELp}
	
	The exponent $d(1/p-1/q)$ on the semiclassical parameter in the previous estimate can be indeed improved in the elliptic setting for $p=2$.
	
	\begin{lemma}[One-body elliptic estimates, {\cite[Thm. 3]{koch2007semiclassical}}]\label{lemma:SA-elliptic-1body}
			Let $d\geq 1$.
			Let $m$ be an order function on $\R^d\times\R^d$, $p\in S(m)$ and $P:=p^\w(x,hD)$ (or any other quantization). Let $(x_0,\xi_0)\in\R^d\times\R^d$ such that
			\begin{equation*}
			p(x_0,\xi_0)\neq 0.
		\end{equation*}
		Then, there exists a neighborhood $\VR$ of $(x_0,\xi_0)$ and $h_0>0$, such that for all $\chi\in\test{\R^d\times\R^d}$ supported in $\VR$, for any $0< h\leq h_0$, there exists $C>0$ such that for all $\mathfrak{t}\in[0,1]$ and for all $2\leq q\leq\infty$, 
		\begin{equation*}
			\normLp{\op_h^\mathfrak{t} (\chi) u}{q}{(\R^d)} \leq C h^{1-d\left(\frac 12 -\frac 1q\right)}\left(\normLp{u}{2}{(\R^d)}+\frac 1 h \normLp{Pu}{2}{(\R^d)}\right)
			.
		\end{equation*}
		Equivalently, for all $2\leq q\leq \infty$
		\begin{equation}\label{eq:SA-elliptic}
			\op_h^\mathfrak{t}(\chi) (1+P^*P/h^2)^{-1/2} = \OR\left(h^{1-d\left(\frac 12 -\frac 1q\right)}\right): L^2(\R^d)\to L^q(\R^d)
			.
		\end{equation}
	\end{lemma}
	% fin lemme \ref{lemma:SA-elliptic-1body}
	
	\begin{rmk}\label{rmk:comp-s-ellip}
		The bound of Lemma \ref{lemma:SA-elliptic-1body} is one power of $h$ better than the one in Lemma \ref{lemma:SA-basic-ELp} (for $p=2$, c.f. Figure \ref{fig:comp-exp-s}), thanks to the term involving the operator $P$ on the right-side. This is particularly relevant for quasimodes $u$ of $P$, since they satisfy $Pu=\OR_{L^2}(h)\normLp{u}{2}{}$.
	\end{rmk}
	
	Here, we state the integrated form of Weyl's law, which gives an asymptotic of the number of eigenfunctions of a pseudodifferential operator in a fixed interval as $h\to 0$.
	
	\begin{prop}[Integrated Weyl law, {\cite[Chap. 9]{dimassi1999spectral}}]\label{prop:int-weyl-law}
		Let $m$ be an order function such that $m(x,\xi)\to+\infty$ when $\abs{(x,\xi)}\to+\infty$ and let $p\in S(m)$ be real valued such that $p+i$ is elliptic on $\R^d\times\R^d$. Let $a<b$ be two real numbers. For any $h>0$, define $P=p^\w(x,hD)$ and denote by $N_h([a,b])$ the number of eigenvalues of $P$ in the interval $[a,b]$.
		Then, we have
		\begin{equation*}
			N_h\left([a,b]\right) = \frac{1}{(2\pi h)^d} \left[
			\abs{p^{-1}([a,b])}+o_{h\to 0}(1)\right].
		\end{equation*}
	\end{prop}

	We now review well-known results on quantum dynamics and their propagators.
	
	\begin{lemma}[Properties of the propagator $F(t,t_0)$, {\cite[Thm. 10.1]{zworski2012semiclassical}}]\label{lemma:prop-F(t)}
		Let $n\in\N^*$ and $h>0$.
		Let $t_0\in\R$ and $a\in\CR^\infty(\R_t,S_{\R^n\times\R^n}(1))$. %$a\in\CR^\infty(\R_t,S_{\R^{d-1}\times\R^{d-1}}(1))$.
		The equation
		\begin{equation*}
			\begin{cases}
			[hD_t-a^\w(t,x,hD_{x})] F(t,t_0)=0, \quad t\in\R,
			\\ F(t_0,t_0)= \id,
			\end{cases}
		\end{equation*}
		has a unique solution $\{F(t,t_0)\}_{t\in\R}$ in $\CR(\R,\BR(L^2(\R^n)))$, which is a family of unitary operators.
		 Furthermore,
		\begin{itemize}
			\item[(i)] For any compact $J\subset\R$ and any $k,s\in\N$, there exists $C>0$ (independent of $h\in(0,h_0)$) such that for any $t,t_0\in J$,
			\begin{equation}\label{eq:prop-F(t)}
				\norm{(hD_t)^k F(t,t_0)}_{H_h^s(\R^n)\to H_h^s(\R^n)}\leq C
				.
			\end{equation}
			\item[(ii)] For any $\psi_1\in\test{\R}$, the operator $\psi_1(t)F(t,t_0)$ maps continuously $H_h^s(\R^n)$ into $H_h^s(\R^{n+1})$ for all $s\in\N$ (with an operator norm independent of $h$).
			\item[(iii)] Let us define the operator $T_F$, which acts on functions on $\R^{n+1}$, by
			\begin{equation*}
				T_F: u(t,x) \mapsto \psi_1(t)\int_{t_0}^t( F(t,s) u(s))(x) \: ds .
			\end{equation*}
			Then, $T_F$ maps continuously $H_h^s(\R^{n+1})$ into $H_h^s(\R^{n+1})$ for all $s\in\N$ (with a bound independent of $h$).
		\end{itemize}
	\end{lemma}
	% fin enonce \ref{lemma:prop-F(t)} 
	
	\begin{rmk}
		The proof of (i) is done in \cite[Thm. 10.1]{zworski2012semiclassical} in the case $k=0$. The bounds for higher values of $k$ can be obtained by induction using the equation satisfied by $F(t,t_0)$.
		The proofs of (ii) and (iii) follow from (i) by elementary arguments. But for more completeness, we detail these points in the proof below.
	\end{rmk}

	\begin{proof}[Proof of Lemma \ref{lemma:prop-F(t)}]
	\item[\quad(i)] Assume that we have the induction until the index $k\in\N$ for \eqref{eq:prop-F(t)}. We now check that the relation is still valid for $k+1$. By using the evolution equation satisfied by the propagator, one has
		\begin{align*}
			(hD_t)^{k+1} F(t,t_0)
			&=  (hD_t)^k (hD_t F(t,t_0))
			= (hD_t)^k \left( a^\w(t,x,hD_x) F(t,t_0)\right)
			\\&= \sum_{\ell+\ell'=k} h^\ell
					D_t^\ell a^\w(t,x,hD_x)
					(hD_t)^{\ell'}F(t,t_0)
			.
		\end{align*}
		Moreover, any operator $D_t^\ell a^\w(t,x,hD_x)$ can be written as the Weyl quantization of a symbol $b_{\ell,t}$, that satisfies the relation $b_{\ell,t}(x/2,\xi)=\partial_t^\ell(a_t(x/2,\xi))$ for any $(x,\xi)\in\R^n\times\R^n$. By the definition of the function $a$, for any $\ell\in\N$ and any $\alpha\in\N^n\times\N^n$, there exists $C_{\ell,\alpha}>0$ such that
		\begin{equation*}
			\forall (x,\xi)\in\R^n\times\R^n\quad
			\abs{\partial_t^k\partial_{x,\xi}^\alpha a_t(x,\xi)} \leq C_{\ell,\alpha}
			.
		\end{equation*}
		As a consequence $b_{\ell,t}\in S(1)$.
		Eventually, one gets the desired uniform bound \eqref{eq:prop-F(t)} for the index $k+1$ by the compacity of the interval $J$, the Calderon-Vaillancourt theorem and the induction relation.
		\item[\quad(ii)] Let $\psi_1\in\test{\R}$ and let $s\in\N$. Let $t_0\in J$ and $v\in L^2(\R^n)$. We write the $H_h^s$-norm of $\psi_1(t)F(t,t_0)v$ as a finite sum of 
		\begin{equation*}
			\normLp{(hD_t)^k(hD_x)^\ell \left(\psi_1(t)F(t,t_0)v(x)\right)}{2}{(\R_{t,x}^{n+1})}
		\end{equation*}
		on all multiindexes $k\in\N$ and $\ell\in\N^n$ such that $\abs{k}+\abs{\ell}=s$. By using the bound \eqref{eq:prop-F(t)}, that the intervals $\supp\psi_1$ and $J$ are compact, there exists $C>0$ such that one has, for any $(k,\ell)\in\N\times\N^n$ as above, any $t,t_0\in J$ and any $v\in L^2(\R^n)$
		\begin{align*}
			\norm{(hD_t)^k(hD_x)^\ell \left(\psi_1(t)F(t,t_0)v(x)\right)}_{L^2_{t,x}(\R^{n+1})}
			&
			\leq \norm{(hD_x)^\ell(hD_t)^k \left(\psi_1(t)F(t,t_0)v(x)\right)}_{L^2_t(\supp\psi_1,L^2_x(\R^n))}
			\\&
			\leq \norm{(hD_t)^k \left(\psi_1(t)F(t,t_0)v(x)\right)}_{L^2(\R,H_h^s(\R^n))}
			\\&\leq	C\norm{v}_{H_h^s(\R^n)}
		.
		\end{align*}
		Finally, for any $s\in\N$, there exists $C>0$ such that for any $t_0\in J$
		\begin{equation*}
			\norm{\psi_1(t)F(t,t_0)}_{H_h^s(\R^n)\to H_h^s(\R^{n+1})} \leq C.
		\end{equation*}
	\item[\quad(iii)] Let $s\in\N$ and $u\in L^2(\R^{n+1})$.
		One has
		\begin{align*}
			\norm{T_F u}_{H_h^s(\R^{n+1})}
			\leq \sum_{0\leq k+\ell\leq s} \norm{(hD_t)^k \left(\int_{t_0}^t \psi_1(t) F(t,r) u(r) dr\right)}_{L^2_t(\R,H_h^{\ell}(\R^n))}
			.
		\end{align*}
		Let $k,\ell\in\N$ such that $k+\ell\leq s$.
		We apply the following equality (that can be proved by induction on $k\in\N$)
		\begin{align*}
			\partial_t^k\left(\int_{t_0}^t v(t,r) dr\right)
			=\int_{t_0}^t \partial_t^k v(t,r) dr +\sum_{k_1+k_2=k-1}\partial_t^{k_1}(t\mapsto \partial_1^{k_2}v(t,t)) -k\partial_t^{k-1}v(t,t_0)
			,
		\end{align*}
		to the function $v(t,r)=\psi_1(t)F(t,r)u(r,\cdot)$ and we have then
		\begin{equation}\label{eq-demo:sogge-prop-T_F}
		\begin{split}
			&\norm{(hD_t)^k \left(\int_{t_0}^t  \psi_1(t)F(t,r) u(r) dr\right)}_{L^2_t(\R,H_h^{\ell}(\R^n))}
			\\&\quad
			\lesssim \sum_{0\leq k+\ell\leq s} \left(
			\norm{\int_{t_0}^t (hD_t)^k  \psi_1(t)F(t,r) u(r) dr}_{L^2_t(\R,H_h^{\ell}(\R^n))}
			\right.
			\\&\qquad\qquad\qquad\qquad
			+ h\sum_{k_1+k_2=k-1}\norm{(hD_t)^{k_1}(t\mapsto (hD_1)^{k_2}\psi_1(t) F(t,t)u(t))}_{L^2_t(\R,H_h^{\ell}(\R^n))}
			\\&\qquad\qquad\qquad\qquad\qquad\left.
			+ h \norm{(hD_t)^{k-1} \psi_1(t)F(t,t_0) u(t_0)}_{L^2_t(\R,H_h^{\ell}(\R^n))}\right)
			.
		\end{split}
		\end{equation}
		First, notice that
		\begin{align*}
			\norm{\int_{t_0}^t (hD_t)^k  \psi_1(t)F(t,r) u(r) dr}_{L^2_t(\R,H_h^{\ell}(\R^n))}\leq \sum_{m=1}^k\norm{\int_{t_0}^t (hD_t)^m F(t,r) u(r) dr}_{L^2_t(\supp\psi_1,H_h^{\ell}(\R^n))}.
		\end{align*}
		 By the bound \eqref{eq:prop-F(t)} applied to $J=\supp\psi_1$, there exists $C>0$ such that for any $t,r\in\supp\psi_1$
		\begin{equation*}
			\norm{\partial_t^k F(t,r)u(r)}_{H_h^\ell(\R^n)}\leq C\norm{u(r)}_{H_h^\ell(\R^n)},
		\end{equation*}
		so that for any $m\leq k$
		\begin{align*}
			&\norm{\int_{t_0}^t (hD_t)^m F(t,r) u(r) dr}_{L^2_t(\supp\psi_1,H_h^{\ell}(\R^n))}
			\\&\qquad
			\leq \abs{\supp\psi_1}^{1/2} \sup_{t\in\supp\psi_1}\abs{\int_{t_0}^t\sup_{\tau\in\supp\psi_1}\norm{(hD_\tau)^mF(\tau,r)u(r)}_{H_h^\ell(\R^n)} dr}
			\\&\qquad
			\leq C\int_{\supp\psi_1}\norm{u(r)}_{H_h^\ell(\R^n)} dr
			\leq C \norm{u}_{H_h^\ell(\R^{n+1})}
			\\&\qquad \leq C \norm{u}_{H_h^s(\R^{n+1})}
			.
		\end{align*}
		Finally, we obtain that for any $s\in\N$, there exists $C>0$ such that for any $ u\in L^2(\R^{n+1})$ 
		\begin{equation*}
			\norm{T_F u}_{H_h^s(\R^{n+1})}\leq C \norm{u}_{H_h^s(\R^{n+1})},
		\end{equation*}
		which is the desired estimate.
	\end{proof}	% end proof \ref{lemma:prop-F(t)} 

	Let us now give a statement of semiclassical dispersive estimates, which are crucial ingredients in the proof of our results.
	The obtention of these dispersive bounds is based on the semiclassical parametrix construction of the propagator $F(t,r)$, using WKB method. The decay estimates then follows from the stationary phase formula. More precisely, this propagator is approximated by a Fourier integral operator. It is done in \cite[Thm. 10.4 and 10.8]{zworski2012semiclassical} or in \cite[Chapter 10]{dimassi1999spectral}.
	
	\begin{thm}[Semiclassical dispersive bounds]\label{thm:SStrichartz-bounds}
		Let $n\in\N$ be such that $n\geq 1$. Let $a=a_t(x,\xi)\in \CR^\infty(\R_t,S_{\R^n\times\R^n}(1))$. Let $r\in\R$. Let $\{F(t,r)\}_{t\in\R}$ be the propagator of the Schr\"odinger evolution equation
		\begin{equation*}
			\left\lbrace\begin{array}{lll}
			[hD_t-a^\w(t,x,hD_x)]F(t,r)=0 \quad t\in\R
			\\ F(r,r)= \id
			.
			\end{array}\right.
		\end{equation*}
		For any $\psi\in\test{\R}$ and $\chi\in\test{\R^n\times\R^n}$, let us define the microlocalized propagator $U(t,r)$ of the previous equation by
		\begin{equation*}
			U(t,r):= \psi(t-r)F(t,r)\chi^\w(x,hD).
		\end{equation*}
		Let $(x_0,\xi_0)\in\R^n\times\R^n$ and $I\subset\R$ a compact interval of $\R$, such that for all $t\in I$
		\begin{equation}\label{cond:SStrichartz-bounds}
			\partial_\xi^2 a_t(x_0,\xi_0) \text{ is non-singular}
			.
		\end{equation}
		Then, for every open interval $J$ such that $J\subset I$, there exist $\delta>0$ independent of $h$ and a neighborhood $U\times V$ of $(x_0,\xi_0)$ such that
		for every $\psi\in\test{\R}$ supported in $(-\delta,\delta)$ and $\chi\in\test{\R^n\times\R^n}$ supported in $U\times V$, we have the uniform bounds for all $t,s\in\R$ 
		\begin{equation}\label{eq:SStrichartz-bounds}
			\left\lbrace\begin{array}{lll}
			\sup_{r\in J}\norm{U(t,r)U(s,r)^*}_{L^2(\R^n)\to L^2(\R^n)} &\leq C
			\\
			\sup_{r\in J}\norm{U(t,r)U(s,r)^*}_{L^1(\R^n)\to L^\infty(\R^n)} &\leq Ch^{-n/2}{(h+\abs{t-s})}^{-n/2}
			.
			\end{array}\right.
		\end{equation}
	\end{thm}
	% fin \ref{thm:SStrichartz-bounds}

	\subsection{Density matrices}\label{sub-sec:density-mat}
	
	We finally review some definitions and standard results on Schatten spaces.
	
	\begin{defi}[Schatten spaces]\label{def:Sch-spaces}
		Let $\alpha\geq 1$.
		For any Hilbert spaces $\HR$ and $\HR'$, we define the Schatten space $\schatten^\alpha(\HR,\HR')$ for any $\alpha\in[1,+\infty)$ the set 
		\begin{equation*}
			\schatten^\alpha(\HR,\HR')=\{A:\HR\to\HR'\text{ compact operator } \: :\: \tr_\HR((A^*A)^{\alpha/2})<\infty\}.
		\end{equation*}
		Endowed with the norm
		\begin{equation*}
			\normSch{A}{\alpha}{(\HR,\HR')} := \left(\tr_\HR((A^*A)^{\alpha/2})\right)^{1/\alpha},
		\end{equation*}
		it is a Banach space.
		Let us call $\schatten^\infty(\HR,\HR')$, the space of compact operators. In the following, when it appears, the $\schatten^\infty$-norm denotes the operator norm of compact operators that maps $L^2(\HR)$ into $L^2(\HR')$.
	\end{defi}

	We first state below the good properties of a microlocalized operator of the form $\chi^\w\gamma\chi^\w$.
	\begin{lemma}\label{lemma:pre-mercer_loc}
		Assume that $\gamma$ is a non-negative operator on $L^2(\R^d)$. Then, if $\chi\in\schwartz(\R^d\times\R^d)$ and $\gamma$ is bounded, then the operator $\chi^\w(x,hD)\gamma\chi^\w(x,hD)$ is non-negative, compact, and trace-class. In particular, its kernel is continuous, bounded and is in $L^p(\R^d\times\R^d)$ for any $p\geq 1$.
	\end{lemma}
	A first consequence is that the $L^q$ norm of the restriction of the integral kernel to the diagonal $\rho_{\chi^\w\gamma\chi^\w}$ of the integral kernel of $\chi^\w\gamma\chi^\w$ is always well-defined.

	\medskip
		In the proofs of the $L^q$ bounds
		\begin{equation}\label{eq:normal-form}
			\normLp{\rho_{A\gamma A^\ast}}{q}{}\leq C  h^{-s}\log(1/h)^{t}\normLp{W}{2(q/2)'}{}\left(\normSch{\gamma}{\alpha}{}+\frac 1{h^2}\normSch{(P-E)\gamma(P-E)}{\alpha}{}\right),
		\end{equation}
		for a compact operator $A$ mapping $L^2(\R^d)$ into $L^{2(q/2)'}(\R^d)$,
		we will see that it is enough to prove a kind of ``dual form''
		\begin{equation}\label{eq:almost-dual-form}
			\normSch{W A\sqrt{\gamma}}{2}{}\leq C  h^{-s}\log(1/h)^{t}\normLp{W}{2(q/2)'}{}\left(\normSch{\gamma}{\alpha}{}+\frac 1{h^2}\normSch{(P-E)\gamma(P-E)}{\alpha}{}\right),
		\end{equation}
		for any $ W\in L^{2(q/2)'}\cap\CR^0(\R^d)$.
		\begin{nota}
			Here and in the following, for an exponent $p\in[1,\infty]$, we define its conjugated exponent $p':=p/(p-1)$. 
		\end{nota}
		Notice that the implication \eqref{eq:almost-dual-form}$\Longrightarrow$\eqref{eq:normal-form} is similar to \cite[Lem. 3]{frank2017restriction} but with the Hilbert-Schmidt norm instead of the $\schatten^{2\alpha'}$ norm and a weighted Schatten norm in the right-hand side. The conjugate exponent of the exponent $\alpha$ will appear naturally when we write H\"older inequality on $\normSch{W A_j\sqrt{\gamma}}{2}{}$ for well-chosen $A_j$ such that $A=\sum_j A_j$
		\begin{equation*}\label{eq:main-dual-form}
			\normSch{W A_j}{2\alpha'}{}\leq C  h^{-s}\log(1/h)^{t}\normLp{W}{2(q/2)'}{}.
		\end{equation*}
		In this paper, $A$ is essentially $\chi^\w$ or $\indicatrice{P\in[E-h,E+h]}$.

	We state now a version of Mercer theorem (originally in \cite{mercer1909} for kernels on compact sets), that allows to write the implication $\eqref{eq:almost-dual-form}\Longrightarrow\eqref{eq:normal-form}$.
	\begin{thm}[Mercer theorem]\label{lemma:mercer-thm}
		Let a bounded non-negative operator $\gamma$ on $L^2(\R^d)$ associated to an integral kernel continuous on $\Omega$. Assume that the restriction of the kernel to the diagonal $\rho_\gamma\in L^1(\R^d)$ and that the kernel is square integrable. Then, there exists a orthonormal basis of continuous eigenfunctions $\{u_j\}_{j\in\N}$ of $\gamma$ with corresponding non-negative eigenvalues $\{\lambda_j\}_{j\in\N}$ such that 
		\begin{equation*}\forall x,y\in\R^d,\quad			\gamma(x,y)=\sum_{j\in\N}\lambda_ju_j(x)\overline{u_j(y)},
		\end{equation*}
		with a convergence of the series on $L^2$-norm and an uniform convergence in all compacts of $\R^d$.
	\end{thm}
	\begin{rmk}\label{rmk:mercer-thm}
		A consequence of Lemma \ref{lemma:mercer-thm} is that for any function $\chi\in\schwartz(\R^d\times\R^d)$ and any non-negative bounded $\gamma$, the operator $\chi^\w(x,hD)\gamma\chi^\w(x,hD)$ is non-negative and trace-class.
		It naturally follows from Mercer theorem that for any continuous functions $W:\R^d\to\R$ and $\psi\in L^\infty(\R^d,\R)$ 
		\begin{equation*}
			\tr_{L^2}(W\psi\chi^\w\gamma\chi^\w \psi W)=\int_{\R^d}(W\psi\chi^\w\gamma\chi^\w \psi W)(x,x)dx =\int_{\R^d}\rho_{\psi\chi^\w\gamma\chi^\w\psi}(x)W(x)^2dx.
		\end{equation*}
		Then, for $q=2$
		\begin{equation*}
			\normLp{\psi^2\rho_{\chi^\w\gamma\chi^\w}}{1}{(\R^d)}
			= \tr_{L^2}(\psi\chi^\w\gamma\chi^\w\psi)=\normSch{\psi\chi^\w\sqrt{\gamma}}{2}{}^2.
		\end{equation*}
		As well, for $q\in(2,\infty)$
		\begin{align*}
			&\normLp{\psi^2\rho_{\chi^\w\gamma\chi^\w}}{q/2}{(\R^d)}
			\\&\quad
			\leq \sup_{W\in L^{2(q/2)'}\cap\CR^0(\R^d)}\frac{\int_{\R^d}\rho_{\psi\chi^\w\gamma\chi^\w\psi}(x)W(x)^2 dx}{\normLp{W}{2(q/2)'}{(\R^d)}^2}
			&
			= \sup_{W\in L^{2(q/2)'}\cap\CR^0(\R^d)}\frac{\tr_{L^2}(W\psi\chi^\w\gamma\chi^\w\psi W)}{\normLp{W}{2(q/2)'}{(\R^d)}^2}
			\\&\quad
			\leq \sup_{W\in L^{2(q/2)'}\cap\CR^0(\R^d)}\frac{\normSch{W\psi\chi^\w\sqrt{\gamma}}{2}{}^2}{\normLp{W}{2(q/2)'}{(\R^d)}}
			.
		\end{align*}
		In the following $\psi=1$ or a localization function on a region of $\R^d$.
	\end{rmk}

	We can now state the Kato-Seiler-Simon inequalities, which are very useful tools in the many-body setting.
	
	\begin{lemma}[Kato-Seiler-Simon,{\cite[Thm. 4.1]{simon2005trace}}]\label{lemma:Kato-Seiler-Simon_schatten}
		Let $2\leq\alpha<\infty$.
		 Then, for all functions $f,g\in L^\alpha(\R^d)$, the operator $f(x)g(-i\nabla)$ 
		is in $\schatten^\alpha(L^2(\R^d))$ and
		\begin{equation*}
			\normSch{f(x)g(-i\nabla)}{\alpha}{\left(L^2(\R^d)\right)}\leq (2\pi)^{-d/\alpha} \normLp{f}{\alpha}{(\R^d)}\normLp{g}{\alpha}{(\R^d)} 
			.
		\end{equation*}
	\end{lemma}
	% fin enonce KSS
	
	As a corollary, this implies a version of semiclassical Sobolev embedding estimates $H_h^m \xhookrightarrow{} L^q $ for operators.
	
	\begin{lemma}[Semiclassical Schatten Sobolev estimates]\label{lemma:Kato-Seiler-Simon_dual}
		Let $m\geq 0$ and $q> 2$.
		Then, if $\frac 1q >\frac 12-\frac md$, we have 
		for all $W\in L^{2(q/2)'}(\R^d)$
		\begin{equation*}
			\normSch{W(1-h^2\Delta)^{-m/2}}{2(q/2)'}{(L^2(\R^d))}
			\leq C h^{-d\left(\frac{1}{2}-\frac{1}{q}\right)}\normLp{W}{2(q/2)'}{(\R^d)}
			.
		\end{equation*}
	\end{lemma}
	% fin \ref{lemma:Kato-Seiler-Simon_dual}
	
	% c.f. DETAILS PhD
	% on weak schatten spaces

	We will use the following complex interpolation result in Schatten spaces, which can be found in \cite[Prop. 1]{frank2017restriction} (see also \cite[Thm. 2.9]{simon2005trace}). 
	
	\begin{thm}[Complex interpolation in Schatten spaces]\label{thm:complex-interpol-schatten}
		Let $n\geq 1$.
		Let $a_0<a_1$ be two real numbers. Let $T$ be an application which maps the strip $S=\{z\in\C,\:a_0\leq \Re z \leq a_1\}$ into bounded operators on $L^2(\R^{n+1})$. Moreover, let us assume that the family of operators $\{T_z\}_{z\in S}$ is analytic in the sense of Stein i.e.
		\[ z\in S \mapsto \prodscal{f}{T_z g}_{L^2(\R^{n+1})} \text{ is continuous for all simple functions }f,g .\]
		and 
		\[ z\in \overset{\circ}{S}\mapsto \prodscal{f}{T_z g}_{L^2(\R^{n+1})} \text{ is analytic for all simple functions }f,g .\]
		If there exist $C_0,C_1,b_0,b_1>0$,
		$1\leq p_0, p_1,q_0, q_1 \leq\infty$, and $1\leq r_0,r_1<\infty$ such that for all $\sigma\in\R$, one has for all simple functions $W_1, W_2$ on $\R^{n+1}$
		\begin{equation*}
			\forall j=0,1,\quad \normSch{W_1 T_{a_j+i\sigma} W_2}{r_j}{(L^2(\R^{n+1}))} \leq C_je^{b_j\abs{\sigma}} \norm{W_1}_{L_t^{p_j}L_x^{q_j}(\R^{n+1})}\norm{W_2}_{L_t^{p_j}L_x^{q_j}(\R^{n+1})}
			.
		\end{equation*}
		Then, for all $0\leq\theta\leq 1$
		\begin{equation*}
			\normSch{W_1 T_{a_\theta} W_2}{r_\theta}{(L^2(\R^{n+1}))} \leq C_0^{1-\theta}C_1^\theta \norm{W_1}_{L_t^{p_\theta}L_x^{q_\theta}(\R^{n+1})}\norm{W_2}_{L_t^{p_\theta}L_x^{q_\theta}(\R^{n+1})} 
			,
		\end{equation*}
		where $a_\theta$, $r_\theta$, $p_\theta$ and $q_\theta$ are defined by
		\begin{equation*}
			a_\theta = (1-\theta)a_0+ \theta a_1,\: \frac{1}{r_\theta}=\frac{1-\theta}{r_0}+\frac{\theta}{r_1},\: \frac{1}{p_\theta}=\frac{1-\theta}{p_0}+\frac{\theta}{p_1} \text{ and } \frac{1}{q_\theta}=\frac{1-\theta}{q_0}+\frac{\theta}{q_1} .
		\end{equation*}
		Actually, when one of the exponents $r_j=\infty$, the $\schatten^\infty$ norm on the left-hand side of the two previous bounds can be replaced by the operator norm.
	\end{thm}
	% fin enonce \ref{thm:complex-interpol-schatten}

	\subsection{Strichartz estimates for density matrices}\label{subsec:strichartz-est-matdens}
	
	In this section, we provide Strichartz estimates in Schatten spaces, which will be a key ingredient for our proof. They generalize the one-body Strichartz estimates \cite[Prop. 4.3]{koch2007semiclassical}, which were also the key ingredient of the proof of Koch-Tataru-Zworski. Such many-body Strichartz estimates were discovered in \cite{frank2014}, and later generalized in \cite{frank2017restriction}. Our proof is inspired by the one in \cite{frank2017restriction}, and provides a way to obtain the full range of Strichartz estimates in Schatten spaces under the general assumption that the propagator satisfies dispersive estimates such as the one in Theorem \ref{thm:SStrichartz-bounds}. In the one-body case, the fact that Strichartz estimates follow abstractly from dispersive bounds were discovered by Ginibre and Velo \cite{ginibre1992smoothing}, and we generalize the results to the many-body case. Interestingly, our many-body proof uses complex interpolation in the spirit of the original proof of Strichartz \cite{strichartz1977restrictions} rather than the direct approach using the Hardy-Littlewood-Sobolev inequality of \cite{ginibre1992smoothing}.
	
	\begin{thm}\label{thm:SStrichartz-matdens}
		Assume the same hypotheses as in Theorem \ref{thm:SStrichartz-bounds}.
		Let $2\leq q\leq\frac{2(n+1)}{n-1}$. Let $p(q):=\frac{2\left(\frac q2\right)'}{1+\left(\frac 2{q-2}-\frac n2\right)_-}$.
		Then, there exist $C>0$ and $h_0>0$,
		such that for any $0<h\leq h_0$, we have, for any $W\in L^{p(q)}_tL^{2(q/2)'}_x(\R^{n+1})$
		\begin{equation*}
			\begin{split}
			\sup_{r\in J} \normSch{WU(t,r)}{2\left(\frac{2q}{q+2}\right)'}{(L^2(\R^n),L^2(\R^{n+1}))}
			&\leq C \norm{W}_{L^{p(q)}_tL^{2(q/2)'}_x(\R^{n+1})}
			\times\\&\quad\times
			\begin{cases}
			h^{-\frac n2\left(\frac 12-\frac 1q\right)} &\text{if}\  2\leq q<\frac{2(n+1)}{n-1},\\
			\log(1/h)^{\frac 1{n+1}}h^{-\frac 1{n+1}} &\text{if}\  q=\frac{2(n+1)}{n-1}
			.
			\end{cases}
			\end{split}
		\end{equation*}
	\end{thm}
	% fin \ref{thm:SStrichartz-matdens}

	\begin{proof}[\underline{Proof of Theorem \ref{thm:SStrichartz-matdens}}]
%% DETAILS V1
	Fix $r\in J$.
	Let $z\in\C$.
	For all $t,s\in\R$, let us define the operator $T_z(t,s,r)$ on $L^2 (\R^n)$ by
	\begin{equation*}\label{def:gene-matdens-T_z(t,s)}
		T_z(t,s,r) := (t-s+i0)^z U(t,r)U(s,r)^*.
	\end{equation*}
	Let the operator $T_z(r)$ acting on functions on $\R^{n+1}$ be defined by
	\begin{equation*}
		\forall (f,g)\quad
		\prodscal{f}{T_z(r)g}_{L^2(\R^{n+1})}
		= \int_{\R\times\R} \prodscal{f(t)}{T_z(t,s,r)g(s)}_{L^2_x(\R^n)} dt ds.
	\end{equation*}
	Defining $A=WU(t,r): L^2(\R^n)\to L^2(\R^{n+1})$, we notice that we have $AA^*= WT_0(r)\bar{W}$ so that the bound in the theorem will follow from estimating $WT_0(r)\bar{W}$ in $\schatten^{\big(\frac{2q}{q+2}\big)'}(L^2(\R^{n+1}))$.
	We will use the following properties of the distribution $m_z(t) =(t+i0)^z$, which can be found in \cite[Chap. I, Sec. 3.6]{gelf'and1964generalized}:
	\begin{itemize}
		\item[$\diamond$] the family
		$\{m_z\}_{z\in\C}\subset\schwartzprime(\R)$ is analytic and for any $z\in\C$ each $m_z$ admits a Fourier transform with this expression
		\begin{equation}\label{eq:gene-matdens-(t+i0)^z}
			\fourier((t+i0)^z))(\omega) = \frac{ \sqrt{2\pi} e^{iz\pi/2} }{\Gamma(-z)}\omega_+^{-z-1} .
		\end{equation}
		\item[$\diamond$] Let $z=-1+i\sigma$ with $\sigma\in\R$, then $\hat{m}_z$ is bounded and
		\begin{equation}\label{eq:gene-matdens-(t+i0)^{-1+is}}
			\normLp{\hat{m}_{-1+i\sigma}}{\infty}{(\R)} \leq\sqrt{\frac 2\pi }e^{\sigma\pi/2} 
			.
		\end{equation}
		\item[$\diamond$]  When the real part of $z$ is strictly greater than $-1$, $m_z$ is in $L^1_{\loc}(\R)$ and
		\begin{equation*}\label{eq:gene-matdens-(t+i0)^z-bis}
			\abs{m_z(t)} \leq \abs{t}^{\Re z} .
		\end{equation*}
	\end{itemize}
	We will obtain bounds on $W_1 T_0(r) W_2$ for all simple functions $W_1$ and $W_2$ using Theorem \ref{thm:complex-interpol-schatten}, estimating the operator $W_1 T_z(r) W_2$ in 
	\begin{itemize}
		\item the operator norm for $\Re z= -1$
		,
		\item the $\schatten^2$-norm for $\Re z=\beta\geq\frac{n-1}{2}$
		.
	\end{itemize}

	\subparagraph[Step 1]{Step 1. Operator norm bounds.}\label{proof:step1:gene-matdens-complex-interpol}
	
	Let us prove that 
	there exists $C>0$ such that for any simple functions $W_1,W_2$ on $\R^{n+1}$ and for any $\sigma\in\R$
	\begin{equation}\label{eq:gene-matdens-sch-infty}
		\sup_{r\in J}
		\norm{W_1T_{-1+i\sigma}(r) W_2}_{L^2(\R^{n+1})\to L^2(\R^{n+1})} \leq C e^{\sigma\pi/2}\normLp{W_1}{\infty}{(\R^{n+1})}\normLp{W_2}{\infty}{(\R^{n+1})}.
	\end{equation}
	
	Let $\sigma\in\R$ and $F,G$ be functions $\test{\R^{n+1}}\subset L^2(\R^{n+1})$. We can write
		\begin{align*}
			&
			\prodscal{F}{T_{-1+i\sigma}(r) G}_{L^2_{t,x}}
			\\&\qquad
			= \int_{\R\times \R}\prodscal{F(t)}{ T_{-1+i\sigma}(t,s,r)  G(s)}_{L^2_x}  dt ds
			\\&\qquad
			= \int_{\R\times \R} \prodscal{ F(t)}{  (t-s +i 0)^{-1+i\sigma} U(t,r) U(s,r)^*  G(s)}_{L^2_x}  dt ds
			\\&\qquad
			= \int_{\R\times \R} (t-s +i 0)^{-1+i\sigma} \prodscal{U(t,r)^* F(t)}{U(s,r)^* G(s)}_{L^2_x}  dt ds
			.
		\end{align*}
		Define the functions $f,g$ by
		\begin{align*}
			f(t,x;r) := (U(t,r)^* F(t,\cdot))(x)
			\quad\text{ and }\quad 
			g(t,x;r) := (U(t,r)^* G(t,\cdot))(x)
			.
		\end{align*}
		By the $L^2\to L^2$ bound of \eqref{eq:SStrichartz-bounds} 
		\begin{equation*}
			\sup_{t\in\R}\sup_{r\in J}\norm{U(t,r)^*}_{L^2_x\to L^2_x} \lesssim 1
			,
		\end{equation*}
		so that the previous functions satisfy the bounds
		\begin{align*}
			\sup_{r\in J}\norm{f(t,x;r)}_{ L^2_{t,x}(\R^{n+1})}
			&\lesssim \norm{F}_{L^2_{t,x}(\R^{n+1})},
			\\\sup_{r\in J}\norm{g(t,x;r)}_{L^2_{t,x}(\R^{n+1})}
			&\lesssim \norm{G}_{L^2_{t,x}(\R^{n+1})}.
		\end{align*}
		We now write everything with the Fourier transform in the time variable, with $m_z(t):=(t+i0)^z$
		\begin{align*}
			&
			\prodscal{F}{T_{-1+i\sigma}(r) G}_{L^2_{t,x}}
			\\&\quad
			= \int_{\R\times\R}  m_{-1+i\sigma}(t-s) \prodscal{f(t;r)}{g(s;r)}_{L^2_x}  dt ds
			\\&\quad
			= \sqrt{2\pi} \int_\R \hat{m}_{-1+i\sigma}(\omega) \prodscal{\hat{f}(\omega;r)}{\hat{g}(\omega;r)}_{L^2_x} d\omega .
		\end{align*}
		Hence, by the Cauchy-Schwarz inequality
		\begin{align*}
			\forall r\in J\quad &
			\abs{ \prodscal{F}{ T_{-1+i\sigma}(r)  G}_{L^2_{t,x}} }
			\\&\quad
			\leq \sqrt{2\pi} \norm{\hat{m}_{-1+i\sigma}}_{L^\infty(\R)} \norm{\fourier_t(f)}_{L^2_{\omega,x}(\R^{n+1})}\norm{\fourier_t(g)}_{L^2_{\omega,x}(\R^{n+1})}
			\\&\quad
			\leq\sqrt{2\pi} \norm{\hat{m}_{-1+i\sigma}}_{L^\infty(\R)} \norm{f}_{L^2_{t,x}(\R^{n+1})}\norm{g}_{L^2_{t,x}(\R^{n+1})}
			\\&\quad
			\leq C \norm{\hat{m}_{-1+i\sigma}}_{L^\infty(\R)} \norm{F}_{L^2_{t,x}(\R^{n+1})}\norm{G}_{L^2_{t,x}(\R^{n+1})}
			.
		\end{align*}
		Finally,
		$\hat{m}_{-1+i\sigma}\in L^\infty(\R)$ and we have the bound \eqref{eq:gene-matdens-(t+i0)^{-1+is}}. Hence, we deduce a bound on $T_{-1+i\sigma}:L^2(\R^{n+1})\to L^2(\R^{n+1})$, from which we deduce \eqref{eq:gene-matdens-sch-infty}.
%	\end{proof}
	% fin demo \ref{lemma:gene-matdens-sch-infty}
	
	\subparagraph[Step 2]{Step 2. Schatten $\schatten^2$-bounds.}\label{proof:step2:gene-matdens-complex-interpol}
	
	Let $\beta\geq \tfrac{n-1}2$.
	Let us prove that 
	there exists $C>0$ such that for any $z\in\C$ with $\Re z=\beta$, and any simple funtions $W_1,W_2$ on $\R^{n+1}$
	\begin{equation}\label{eq:gene-matdens-sch-2}
		\begin{split}
		\sup_{r\in J}&
		\normSch{W_1 T_z(r) W_2}{2}{(L^2(\R^{n+1}))}
		\\&\quad
		\leq C
		\norm{W_1}_{L_t^{\frac{2}{1+\left(\beta-\frac n2\right)_-}} L_x^{2}(\R^{n+1})} \norm{W_2}_{L_t^{\frac{2}{1+\left(\beta-\frac n2\right)_-}} L_x^{2}(\R^{n+1})}
		\times\\&\qquad\times
		\begin{cases}
		%h^{\beta-n-1/2}&\text{if}\  0\leq \beta<  \frac{n-1}2,\\
		\log(1/h)^{1/2}h^{-n/2} &\text{if}\  \beta=\frac{n-1}2,
		\\h^{-n/2}&\text{if}\  \frac{n-1}2 <\beta\leq \infty
		.
		\end{cases}
		\end{split}
	\end{equation}

		By the $L^1\to L^\infty$-bound of \eqref{eq:SStrichartz-bounds}, the integral kernel $T_z(t,s,r)(x,y)$ of $T_z(t,s,r)$ satisfies
		\begin{align*}
			\forall t,s\in\R\quad
			\sup_{r\in J}\norm{T_z(t,s,r)(x,y)}_{L_{x,y}^\infty(\R^n\times\R^n)}
			&= \sup_{r\in J}\norm{T_z(t,s,r)}_{L^1(\R^n)\to L^\infty(\R^n)}
			\\&\lesssim h^{-n/2} \abs{t-s}^{\Re z}(h+\abs{t-s})^{-n/2} 
			.
		\end{align*}
		Thus, we obtain a bound on the $\schatten^2$-norm of $W_1T_z(r) W_2$ for any $ \beta:=\Re z\geq 0$
		\begin{align*}
			\forall r\in J,\quad&
			\normSch{W_1 T_{z}(r) W_2}{2}{(L^2(\R^{n+1}))}^2
			\\&= \int_{\R^{n+1}}\int_{\R^{n+1}} 
			\abs{W_1(t,x) T_z(t,x,s,y;r) W_2(s,y) }^2 dt dx ds dy
			\\&
			\lesssim h^{-n} \int_\R \int_\R \indicatrice{\abs{t-s}<2\delta} \frac{\abs{t-s}^{2\beta}}{(h+\abs{t-s})^n}
			\normLp{W_1(t)}{2}{(\R^n)}^2 \normLp{W_2(s)}{2}{(\R^n)}^2  dt ds 
			\\&\lesssim h^{-n}
			\begin{cases}
				\norm{W_1}_{L_{t,x}^{2}(\R^{n+1})}^2\norm{W_2}_{L_{t,x}^{2}(\R^{n+1})}^2 &\text{if}\  \beta\geq \frac n2,\\
				\norm{W_1}_{L_t^{\frac{2}{1+\beta-\frac n2}} L_x^2(\R^{n+1})}^2 \norm{W_2}_{L_t^{\frac{2}{1+\beta-\frac n2}} L_x^2(\R^{n+1})}^2
				&\text{if}\ \frac {n-1}2<\beta< \frac n2,\\
				\log(1/h) \norm{W_1}_{L_t^{\frac{2}{1+\beta-\frac n2}} L_x^{2 }(\R^{n+1})}^2 \norm{W_2}_{L_t^{\frac{2}{1+\beta-\frac n2}} L_x^{2}(\R^{n+1})}^2
				&\text{if}\ \beta=\frac{n-1}2
				.
			\end{cases}
		\end{align*}
		In the first line, we used $\abs{t-s}^{2\beta}(h+\abs{t-s})^{-n}\lesssim 1$ for $\abs{t-s}<2\delta$. 
		In the second line, we used $\abs{t-s}^{2\beta}(h+\abs{t-s})^{-n}\lesssim \abs{t-s}^{2\beta-n}$ and the Hardy-Littlewood-Sobolev 
		inequality (see for instance \cite[Thm. 4.3]{lieb-loss2001analysis} applied to the functions $\abs{W_1}^2$ and $\abs{W_2}^2$, and to the exponents $p=r=2/(2\beta-n)$ and $\lambda=n-2\beta$).
			In the third line, we used the Young inequality 
		(see for instance \cite[Thm. 4.2]{lieb-loss2001analysis} applied to $(f,g,h)=(\abs{W_1}^2,\abs{W_2}^2,t\mapsto\indicatrice{\abs{t}<2\delta})\abs{t}^{2\beta}(h+\abs{t})^{-n})$ and to the corresponding exponents $p=q=1/(1+\beta-n/2)=2$ and $r=1$)
		and that
		\begin{equation*}
		\int_{-2\delta}^{2\delta } \frac{\abs{t}^{n-1}}{(h+\abs{t})^n}dt\lesssim \log(1/h)
		.
		\end{equation*}
		That ends the proof of \eqref{eq:gene-matdens-sch-2}.

	\subparagraph[Step 3]{Step 3. Conclusion.}\label{proof:step3:gene-matdens-complex-interpol}
	
	Interpolating $z=0$ between $\Re z=-1$ and $\Re z=\beta\geq\frac{n-1}2$,
	%% DETAILS
	%	\\ $0 = (1-\theta)(-1)+ \theta \beta,$ that is $\theta=(\beta+1)^{-1}$.
	by Theorem \ref{thm:complex-interpol-schatten}, we get
	\begin{align*}
		\sup_{r\in J}
		\normSch{W_1 T_0(r) W_2}{2(\beta+1)}{(L^2(\R^{n+1}))}
		&\lesssim 
		\norm{W_1}_{L_t^{\frac{2 (\beta+1)}{1+\left(\beta-\frac n2\right)_-}} L_x^{2 (\beta+1) }(\R^{n+1})}  \norm{W_2}_{L_t^{\frac{2 (\beta+1)}{1+\left(\beta-\frac n2\right)_-}} L_x^{2 (\beta+1) }(\R^{n+1})}
		\times\\&\quad\times
		\begin{cases}
			h^{-\frac{n}{2(\beta+1)}} &\text{if}\  \beta>\frac{n-1}2,\\
			\log(1/h)^{\frac 1{n+1}}h^{-\frac 1{n+1}} &\text{if}\  \beta=\frac{n-1}2
			.
			\end{cases}
	\end{align*}
	Defining $q\geq 2$ such that $2(q/2)'=2(\beta+1)$, we have the desired estimates for all $ 2 \leq q \leq\tfrac{2(n+1)}{n-1}$.
	That ends the proof of Theorem \ref{thm:SStrichartz-matdens}.
	\end{proof}
	% end proof \ref{thm:SStrichartz-matdens}
	
	%% DETAILS PROOF ESTIM L^p GENE \ref{rmk:SStrichartz-matdens}

	\subsection{Relations between various estimates on quasimodes}\label{sec:abstract-thm}
	
	Below, we will see several estimates of type \eqref{eq-intro:Lp-est-micr-manybody} depending on how the phase space localization is made. Here, we explain how to relate these different estimates.
	
	Let $d\geq 1$. Let $m$ be an order function on $\R^d\times\R^d$, $p\in S(m)$ and $P:= p^\w(x,hD)$ (or any other quantization).
	In the following, we will consider parameters $q\in[2,\infty]$, $s$, $t\geq 0$ and $\alpha\geq 1$ satisfying
	\begin{equation}\label{eq:worse-than-ellip}
		s\geq d\left(\frac 12 -\frac 1q\right) -1
		\quad\text{and}\quad
		\alpha\leq \frac q2.
	\end{equation}

	\begin{rmk}
		The previous assumption states that an estimate with a bound
		\begin{equation*}
			\normLp{\rho_{\chi^\w\gamma\chi^\w}}{q/2}{}\leq Ch^{-2s}\log(1/h)^{2t}\normSch{(1+P^*P/h^2)^{1/2}\gamma(1+P^*P/h^2)^{1/2}}{\alpha}{}
		\end{equation*}
		for $\gamma$ a bounded non-negative operator and $\chi\in\test{\R^d\times\R^d}$,
		 with $(q,s,t,\alpha)$ satisfying \eqref{eq:worse-than-ellip}, is worse than the elliptic one, i.e.\ for which there is equality case of \eqref{eq:worse-than-ellip}. Such an elliptic estimate is proved in Theorem \ref{thm:ELp-elliptic-matdens}.
		We should insist on the fact that the density $\rho_{\chi^\w\gamma\chi^\w}$ is well-defined, thanks to the assumptions of $\gamma$ and $\chi$ (see Lemma \ref{lemma:pre-mercer_loc} below).
		Furthermore, these results will be applied in this paper to values $(s,\alpha)$ that always satisfy the relation \eqref{eq:worse-than-ellip}, as evidenced by Figures \ref{fig:comp-exp-s} and \ref{fig:comp-exp-alpha}.
	\end{rmk}

	Notice also the equivalence of the following weighted $L^2$-norms
	\begin{equation*}
		\normLp{(1+P^*P/h^2)^{1/2}u}{2}{(\R^d)} \leq \normLp{u}{2}{(\R^d)}+\frac 1 h \normLp{Pu}{2}{(\R^d)} \leq \sqrt{2}\normLp{(1+P^*P/h^2)^{1/2}u}{2}{(\R^d)}  .
	\end{equation*}
	can be extended to density matrices.
	\begin{lemma}\label{lemma:equiv-norm-P-matdens}
		Let $m$ be an order function. Let $p\in S(m)$ and $P:=p^\w(x,hD)$. For any $\alpha\geq 1$ and non-negative density matrix $\gamma$ on $ L^2(\R^d)$
		\begin{equation*}
		\begin{split}
		\normSch{(1+P^*P/h^2)^{1/2}\gamma(1+P^*P/h^2)^{1/2}}{\alpha}{}
		&\leq \normSch{\gamma}{\alpha}{}+\frac 1{h^2}\normSch{P^*\gamma P}{\alpha}{}
		\\&
		\leq 2\normSch{(1+P^*P/h^2)^{1/2}\gamma(1+P^*P/h^2)^{1/2}}{\alpha}{}	.
		\end{split}
	\end{equation*}
\end{lemma}

	\begin{proof}[Proof of Lemma \ref{lemma:equiv-norm-P-matdens}]
		Let $\gamma$ be a bounded self-adjoint non-negative operator on $L^2(\R^d)$ and let $\alpha\geq 1$.
		On the one hand, by cyclicity of the trace and the triangle inequality
		\begin{align*}
			\normSch{(1+P^*P/h^2)^{1/2}\gamma(1+P^*P/h^2)^{1/2}}{\alpha}{}
			&= \normSch{\sqrt{\gamma}(1+P^*P/h^2)\sqrt{\gamma}}{\alpha}{}
			\\&= \normSch{\gamma+\sqrt{\gamma}\frac{P^*P}{h^2}\sqrt{\gamma}}{\alpha}{}
			\\&\leq \normSch{\gamma}{\alpha}{} +\frac{1}{h^2}\normSch{\sqrt{\gamma}P^*P\sqrt{\gamma}}{\alpha}{}
			%
			%\\&\leq
			=\normSch{\gamma}{\alpha}{}+\frac{1}{h^2}\normSch{P^*\gamma P}{\alpha}{}
			.
		\end{align*}
		On the other hand, since $0\leq \gamma\leq\sqrt{\gamma}(1+P^*P/h^2)\sqrt{\gamma}$ and $0\leq \sqrt{\gamma} P^*P/h^2\sqrt{\gamma} \leq \sqrt{\gamma}(1+P^*P¨/h^2)\sqrt{\gamma}$, we have that
		\begin{align*}
			\normSch{\gamma}{\alpha}{}+\frac{1}{h^2}\normSch{P^*\gamma P}{\alpha}{}
			&=
			\normSch{\gamma}{\alpha}{}+\frac 1{h^2}\normSch{\sqrt{\gamma}P^*P\sqrt{\gamma}}{\alpha}{}
			\\&\leq  2\normSch{\sqrt{\gamma}(1+P^*P/h^2)\sqrt{\gamma}}{\alpha}{}
			= 2\normSch{(1+P^*P/h^2)^{1/2}\gamma(1+P^*P/h^2)^{1/2}}{\alpha}{}.
		\end{align*}
	\end{proof}
	% end proof \ref{lemma:equiv-norm-P-matdens}

	The following is an assumption on $S\subset\R^d\times\R^d$, $q\in[2,\infty]$, $s,t\geq 0$ and $\alpha\geq 1$.
	
	\begin{assump}[Microlocalization around points]\label{ass:microloc}
		The parameters $S\subset\R^d\times\R^d$, $q\in[2,\infty]$, $s,t\geq 0$ and $\alpha\geq 1$ satisfy Assumption \ref{ass:microloc} if 
		they satisfy the hypothesis \eqref{eq:worse-than-ellip} and if
		for all $(x_0,\xi_0)\in S$, there exist a neighborhood $\VR$ of $(x_0,\xi_0)$ and $h_0>0$, such that for all $\chi\in\test{\R^d\times\R^d}$ supported in $\VR$, there exists $C>0$ such that for any $0< h\leq h_0$ and any bounded non-negative operator $\gamma$ on $L^2(\R^d)$
		\begin{equation*}
			\normLp{\rho_{\chi^\w\gamma\chi^\w}}{q/2}{(\R^d)}\leq C \log(1/h)^{2t}h^{-2s}\normSch{(1+P^*P/h^2)^{1/2}\gamma (1+P^*P/h^2)^{1/2}}{\alpha}{}.
		\end{equation*}
	\end{assump}
	% end statement \ref{ass:microloc}

	\begin{thm}[Microlocalization in a compact]\label{thm:abstract-microloc-extend}
		Let $S\subset\R^d\times\R^d$, $q\in[2,\infty]$, $s,t\geq 0$ and $\alpha\geq 1$ be such that Assumption \ref{ass:microloc} helds.
		Then, for all $\chi\in\test{\R^d\times\R^d}$ supported in $S$, there exists $C>0$ and $h_0>0$ such that for any $0<h\leq h_0$ and any bounded non-negative operator $\gamma$ on $L^2(\R^d)$
		\begin{equation*}
			\normLp{\rho_{\chi^\w\gamma\chi^\w}}{q/2}{(\R^d)}\leq C \log(1/h)^{2t}h^{-2s}\normSch{(1+P^*P/h^2)^{1/2}\gamma (1+P^*P/h^2)^{1/2}}{\alpha}{}.
		\end{equation*}
	\end{thm}
	% end statement \ref{thm:abstract-microloc-extend}
	
	\begin{proof}[\underline{Proof of Theorem \ref{thm:abstract-microloc-extend}}]
		Since $\supp\chi$ is compact and is contained on $S$, there exist open sets $\{\VR_j\}_{j=1}^M$ given by Assumption \ref{ass:microloc} such that
		\[ \supp\chi \subset \bigcup_{j=1}^M \VR_j .\]
		Moreover, one can find a partition of unity $1=\sum_{j=1}^M \varphi_j$ on $\supp\chi$ with $\supp\varphi_j\subset \VR_j$.
		Let us treat separatly the different possible cases $q=2$, $q\in(2,\infty)$ and $q=\infty$.
		Note that 
		\begin{align*}
			\rho_{\chi^\w\gamma\chi^\w} &= \sum_{j=1}^M \rho_{(\varphi_j\chi)^\w\gamma(\chi\varphi_j)^\w} +2 \sum_{1\leq \ell< k\leq M}\rho_{(\varphi_\ell\chi)^\w\gamma(\chi\varphi_k)^\w}
			.
		\end{align*}
		By assumption, one has the bound for the $L^\infty$ norm of $\rho_{(\varphi_\ell\chi)^\w\gamma(\chi\varphi_k)^\w}$ when $\ell=k$. Let us show that it is also true when $k\neq\ell$. 
		\begin{align*}
			(\varphi_\ell&\chi)^\w\gamma(\chi\varphi_k)^\w
			\\&= 	(\varphi_\ell\chi)^\w(1+P^*P/h^2)^{-1/2}(1+P^*P/h^2)^{1/2}\gamma(1+P^*P/h^2)^{1/2}(1+P^*P/h^2)^{-1/2}(\chi\varphi_k)^\w
			\\&\leq \norm{(1+P^*P/h^2)^{1/2}\gamma(1+P^*P/h^2)^{1/2}}_{L^2\to L^2} (\varphi_\ell\chi)^\w(1+P^*P/h^2)^{-1}(\chi\varphi_k)^\w
			.
		\end{align*}
		Furthermore, one has for any $x\in\R^d$, by the one-body $L^\infty$ version of Assumption \ref{ass:microloc}
		\begin{align*}
			&\rho_{(\varphi_\ell\chi)^\w(1+P^*P/h^2)^{-1}(\chi\varphi_k)^\w}(x)
			\\&\quad
			= \int_{\R^d}((\varphi_\ell\chi)^\w(1+P^*P/h^2)^{-1/2})(x,y)\overline{((\chi\varphi_k)^\w)(1+P^*P/h^2)^{-1/2}}(x,y)  dy
			\\&\quad
			\leq \norm{(\varphi_\ell\chi)^\w(1+P^*P/h^2)^{-1/2})}_{L^\infty_x L^2_y}\norm{(\varphi_k\chi)^\w(1+P^*P/h^2)^{-1/2})}_{L^\infty_x L^2_y}
			\\&\quad
			\leq \norm{(\varphi_\ell\chi)^\w(1+P^*P/h^2)^{-1/2})}_{L^2\to L^\infty}\norm{(\varphi_k\chi)^\w(1+P^*P/h^2)^{-1/2})}_{L^2\to L^\infty}
			\\&\quad\leq C\log(1/h)^{2t}h^{-2s}\norm{(1+P^*P/h^2)^{1/2}\gamma(1+P^*P/h^2)^{1/2}}_{L^2\to L^2}
			.
		\end{align*}
		Then,
		\begin{equation*}
			\normLp{\rho_{(\varphi_\ell\chi)^\w\gamma(\chi\varphi_k)^\w}}{\infty}{(\R^d)}
			\leq C\log(1/h)^{2t}h^{-2s} \norm{(1+P^*P/h^2)^{1/2}\gamma(1+P^*P/h^2)^{1/2}}_{L^2\to L^2}
			.
		\end{equation*}
		By the triangle inequality, one has the desired bound for the $L^\infty$ norm.
		Furthermore, we have for all $j\in\{1,\ldots,M\}$, bounds on $\normSch{W(\chi\varphi_j)^\w\sqrt{\gamma}}{2}{}$,
		\begin{align*}
		&
		\normSch{W(\chi\varphi_j)^\w\sqrt{\gamma}}{2}{}
		\\&\quad
		\leq C \log(1/h)^{t}h^{-s}\normSch{(1+P^*P/h^2)^{1/2}\gamma (1+P^*P/h^2)^{1/2}}{\alpha}{}^{1/2} \normLp{W}{2(q/2)'}{(\R^d)}
		,
		\end{align*}
		for
		\begin{equation*}
			\begin{cases}
				 W=1 &\text{if}\ q=2,\\
				W\in L^{2(q/2)'}\cap \CR^0(\R^d)&\text{if}\ q\in(2,+\infty)
				.
			\end{cases}
		\end{equation*}
		Hence, by the triangle inequality, we deduce the bound on $\normSch{W\chi^\w\sqrt{\gamma}}{2}{}$ and then we recover the one on $\normLp{\rho_{\chi^\w\gamma\chi^\w}}{q/2}{(\R^d)}$ with the Mercer theorem (Remark \ref{rmk:mercer-thm}).
		That ends the proof of Theorem \ref{thm:abstract-microloc-extend}.
	\end{proof}
	% fin demo \ref{thm:abstract-microloc-extend}

	\begin{thm}[Microlocalization and localization in space]\label{thm:abstract-loc-space}
		Let $S\subset\R^d\times\R^d$, $q\in[2,\infty]$, $s,t\geq 0$ and $\alpha\geq 1$ be such that Assumption \ref{ass:microloc} holds.
		Then, for all $\chi\in\test{\R^d\times\R^d}$ and for all set $\Omega\subset\R^d$ such that
		\begin{equation*}
			\supp\chi \cap \Omega\times\R^d \subset \overset{\circ}{S} ,
		\end{equation*}
		there exist $C>0$ and $h_0>0$ such that for any $0<h\leq h_0$ and any bounded non-negative operator $\gamma$ on $L^2(\R^d)$ 
		\begin{equation*}
			\normLp{\rho_{\chi^\w\gamma\chi^\w}}{q/2}{(\Omega)}\leq C \log(1/h)^{2t}h^{-2s}\normSch{(1+P^*P/h^2)^{1/2}\gamma (1+P^*P/h^2)^{1/2}}{\alpha}{}.
		\end{equation*}
	\end{thm}
	% fin enonce \ref{thm:abstract-loc-space}

	\begin{proof}[\underline{Proof of Theorem \ref{thm:abstract-loc-space}}]
		Let $\tilde{\Omega}\subset\R^d$ an open bounded set such that $\Omega\subset\overline{\Omega}\subset\tilde{\Omega}$ and such that
		\[ \{(x,\xi)\in\supp\chi \: : \: x\in\tilde{\Omega}\} \subset S .\]
		Let $\chi_\Omega\in\test{\R^d,[0,1]}$ be a function supported in $\tilde{\Omega}$ such that $\chi_\Omega=1$ on $\Omega$.
%% fin CORRECTION rmk 22
		We have by the Mercer theorem (more precisely Remark \ref{rmk:mercer-thm}) when $q\in (2,\infty)$
		\begin{align*}
			\normLp{\rho_{\chi^\w\gamma\chi^\w}}{q/2}{(\Omega)}
			&\leq \normLp{\rho_{\chi_\Omega\chi^\w\gamma\chi^\w\chi_\Omega}}{q/2}{(\R^d)}
			%
% DETAILS V1, cf PhD
%			\\&\quad
%			= \sup_{W\in L^{2(q/2)'}\cap\CR^0(\R^d)}\frac{\int_{\R^d}\rho_{\chi_\Omega\chi^\w\gamma\chi^\w\chi_\Omega}(x) W(x)^2 dx}{\normLp{W}{2(q/2)'}{(\R^d)}^2}
%			%
%			\\&\quad
%			= \sup_{W\in L^{2(q/2)'}\cap\CR^0(\R^d)}\frac{\tr_{L^2}(W\chi_\Omega\chi^\w\gamma\chi_\Omega\chi^\w W)}{\normLp{W}{2(q/2)'}{(\R^d)}^2}
%			%
%			\\&\quad=
			\leq \begin{cases}
			\normSch{\chi_\Omega\chi^\w\sqrt{\gamma}}{2}{}^2&\text{if}\ q=2,
			\\
			\displaystyle\sup_{W\in L^{2(q/2)'}\cap\CR^0(\R^d)}\frac{\normSch{W\chi_\Omega\chi^\w\sqrt{\gamma}}{2}{}^2}{\normLp{W}{2(q/2)'}{(\R^d)}^2}&\text{if}\ q\in (2,\infty)
			.
			\end{cases}
		\end{align*}
		There exists $r\in\schwartz(\R^d\times\R^d)$ such that
		\begin{equation*}
			\chi_\Omega\chi^\w = (\chi_\Omega\chi)^\w +hr^\w.
		\end{equation*}
		On the one hand, by the H\"older and Kato-Seiler-Simon inequalities (Lemma \ref{lemma:Kato-Seiler-Simon_dual}) for $m\in\N$ such that $m>\frac d{2(q/2)'} $, for any $N\in\N$
		\begin{align*}&
			h\normSch{Wr^\w\sqrt{\gamma}}{2}{}
			\\&\quad
			\leq h \normSch{W\chi_\Omega(1-h^2\Delta)^{-m}}{2(q/2)'}{}\norm{(1-h^2\Delta)^{m}r^\w}_{L^2\to L^2}\normSch{\sqrt{\gamma}}{q}{}
			\\&\quad
			\leq C h^{1-d\left(\frac 12-\frac 1q\right)}\normLp{W}{2(q/2)'}{(\R^d)}\normSch{\gamma}{q/2}{}^{1/2}
			.
		\end{align*}	
		On the other hand, by Theorem \ref{thm:abstract-microloc-extend} applied to $S$ and $(q,s,t,\alpha)$, there exist $C>0$ and $h_0$ such that for any $0<h\leq h_0$ and any non-negative operator $\gamma$ on $L^2(\R^d)$
		\begin{align*}
			\sup_{W\in L^{2(q/2)'}\cap\CR^0(\R^d)}\frac{\normSch{W(\chi_\Omega\chi)^\w\sqrt{\gamma}}{2}{}^2}{\normLp{W}{2(q/2)'}{(\R^d)}^2}
			&=
			\normLp{\rho_{(\chi_\Omega\chi)^\w\gamma(\chi\chi_\Omega)^\w}}{q/2}{(\R^d)} \\&\leq C\log(1/h)^{2t}h^{-2s}\normSch{(1+P^*P/h^2)^{1/2}\gamma (1+P^*P/h^2)^{1/2}}{\alpha}{}.
		\end{align*}
		Finally, by the triangle inequality, we get the desired inequality.		
	\end{proof}
	% end proof \ref{thm:abstract-loc-space}

	\begin{rmk}\label{rmk:abstract-loc-space}
		The above proof shows that the result of Theorem \ref{thm:abstract-loc-space} also holds when $\normLp{\rho_{\chi^\w\gamma\chi^\w}}{q/2}{(\Omega)}$ is replaced by
		$	\normLp{\rho_{\chi^\w\chi_\Omega\gamma\chi_\Omega\chi^\w}}{q/2}{(\R^d)}$ or $	\normLp{\rho_{\chi_\Omega\chi^\w\gamma\chi^\w\chi_\Omega}}{q/2}{(\R^d)}$.
	\end{rmk}
	% fin \ref{rmk:abstract-loc-space}

	%\newpage
	
	\section{Elliptic estimates}\label{sec:elliptic-est}

	In this section, we state and prove estimates in the elliptic region where $p\neq 0$. In the one-body case ($\rk\gamma=1$), one recovers Lemma \ref{lemma:SA-elliptic-1body}.

	\begin{thm}[Many-body elliptic estimates]\label{thm:ELp-elliptic-matdens}
		Let $d\geq 2$ and $2\leq q\leq \infty$. Let $m$ be an order function on $\R^d\times\R^d$ and $p\in S(m)$. Let $P:= p^\w(x,hD)$ (or any other quantization).
		Let $(x_0,\xi_0)\in\R^d\times\R^d$ be a point such that
		\begin{equation*}
			p(x_0,\xi_0)\neq 0.
		\end{equation*}
		Then, there exist a neighborhood $\VR$ of $(x_0,\xi_0)$ and $h_0>0$, such that for any $\chi\in\test{\R^d\times\R^d}$ with support contained in $\VR$, there  exists $C>0$ such that for any $0< h\leq h_0$, for any bounded non-negative operator $\gamma$ on $L^2(\R^d)$
		\begin{equation*}
			\normLp{\rho_{\chi^\w\gamma\chi^\w}}{q/2}{(\R^d)} \leq C h^{-2s(q,d)}\normSch{(1+P^*P/h^2)^{1/2}\gamma(1+P^*P/h^2)^{1/2}}{\alpha(q,d)}{}
			,
		\end{equation*}
		where the exponents $s$ and $\alpha$ are given by
		\begin{equation}\label{eq-def:exp-s-alpha-ELp-ellip}
			s(q,d)= d\left(\frac 1 2-\frac 1 q\right)-1,\quad\alpha(q,d)=\frac q 2.
		\end{equation}
	\end{thm}
	% fin enonce \ref{thm:ELp-elliptic-matdens}
	
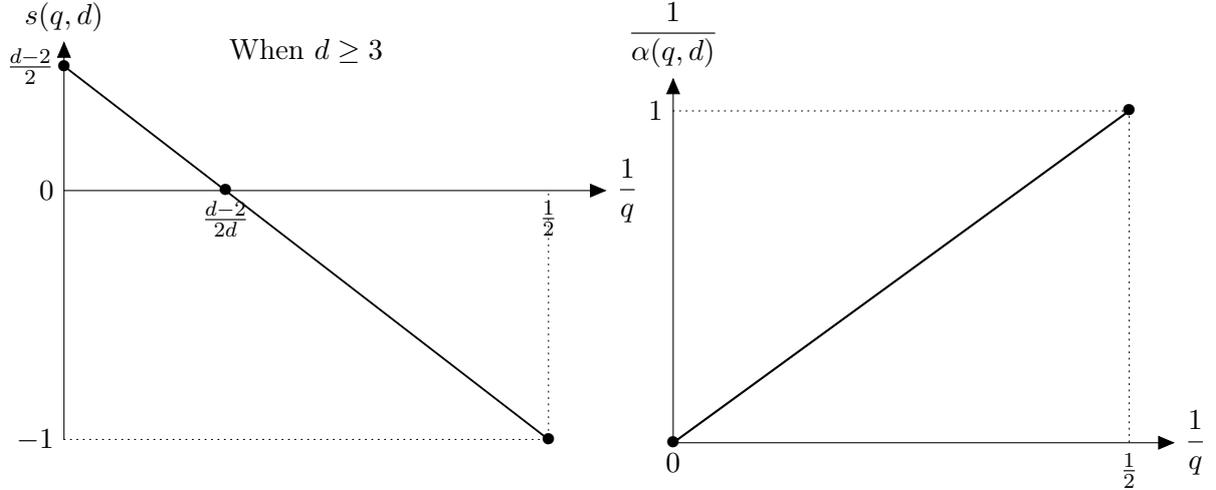
\begin{figure}[!h]
	\begin{multicols}{2}
		
		%%%%%%%%%%%%%%%%%%%%%% DESSIN ktz d>2 %%%%%%%%%%%%%%%%%%%%%%%%%%%%
		\begin{center}\begin{tikzpicture}[line cap=round,line join=round,>=triangle 45,x=8.5 cm,y=2.2 cm,scale=1.5]
			d=3
			\draw[->] (0.,0.) -- (0.56,0.);
			\draw (0.56,0.) node[right] {$\displaystyle{\frac{1}{q}}$};
			\draw[->] (0.,-1) -- (0.,1/2+0.1);
			\draw (0.,1/2+0.1) node[above] {$\displaystyle{s(q,d)}$};
			%
			%% abscisses
			%
			\draw (0,0) node[left]{$0$};
			\draw (1/6,0.) node[below]{$\frac{d-2}{2d}$};
			\draw (0.5,0) node[below]{$\frac 12$};
			%
			%% ordonnées
			\draw (0,-1) node[left]{$-1$};
			\draw (0,0.5) node[left]{$\frac{d-2}2$};
			%
			%% courbes
			% ellip
			\draw[line width=0.7pt](0,0.5)--(0.5,-1);
			%
			%% points
			%
			\draw (0,0.5) node{$\bullet$};
			\draw (1/6,0.) node{$\bullet$};
			\draw (0.5,-1)node{$\bullet$};
			\draw[dotted] (0, -1)-|(0.5, 0.);
			\draw (1/4,0.56) node{When $d\geq 3$} ;
			\end{tikzpicture}\end{center}
		%%%%%%%%%%%%%%%%%%%%%% FIN DESSIN kzt s(q,d) d>2 %%%%%%%%%%%%%%%%%%%%%%%%%%%%
		
		%%%%%%%%%%%%%%%%%%%%%%  DESSIN alpha(q,d) d>2 %%%%%%%%%%%%%%%%%%%%%%%%%%%%
		\begin{center}\begin{tikzpicture}[line cap=round,line join=round,>=triangle 45,x=6.0cm,y=2.2 cm,scale=2.]
			d=3
			\draw[->] (0.,0.) -- (0.55,0.);
			\draw (0.55,0.) node[right] {$\displaystyle{\frac{1}{q}}$};
			\draw[->] (0.,0.) -- (0.,1.1);
			\draw (0.,1.1) node[above] {$\displaystyle{\frac 1{\alpha(q,d)}}$};
			%
			%% abscisses
			%
			\draw (0,0.) node[below]{$0$}; %
			\draw (0.5,0.) node[below]{$\frac{1}{2}$};
			%
			%% ordonnées
			%
			\draw (0.,1) node[left]{1};
			%
			%% points
			\draw[color=black] (0.,0) node {$\bullet$};
			\draw[color=black] (0.5,1) node {$\bullet$};
			\draw[dotted] (0, 1)-|(0.5, 0.);
			%
			%% courbes
			%
			% ellip
			\draw[line width=0.8pt] (0.,0) --(0.5,1);
			\end{tikzpicture}\end{center}
		%%%%%%%%%%%%%%%%%%%%%% FIN DESSIN alpha(q,d) %%%%%%%%%%%%%%%%%%%%%%%%%%%%
	\end{multicols}
	
	\caption{Exponent $s(q,d)$ and $\alpha(q,d)$ for elliptic estimates.}
	\label{fig:exp-s-alpha_ellip}
\end{figure}
	
	\begin{proof}[\underline{Proof of Theorem \ref{thm:ELp-elliptic-matdens}}]
		There exists a neighborhood $\VR$ of $(x_0,\xi_0)$ where $p$ is non-zero. 
		We have by Mercer theorem
		\begin{align*}
			&\normLp{\rho_{\chi^\w\gamma\chi^\w}}{q/2}{(\R^d)}
			%
% DETAILS cf Phd thesis
%			\\&\quad
%			= \sup_{W\in L^{2(q/2)'}\cap\CR^0(\R^d)} \frac{\int_{\R^d} \rho_{\chi^\w\gamma\chi^\w}(x)\abs{W(x)}^2  dx}{\normLp{W}{2(q/2)'}{(\R^d)}^2}
%			%
			\\&\quad
			\leq \sup_{W\in L^{2(q/2)'}\cap\CR^0(\R^d)}\frac{ \tr_{L^2}(W\chi^\w\gamma\chi^\w W) } {\normLp{W}{2(q/2)'}{(\R^d)}^2}
			\\&\quad
			\leq \sup_{W\in L^{2(q/2)'}\cap\CR^0(\R^d)}\frac{ \normSch{W\chi^\w(1+P^*P/h^2)^{-1/2}}{2(q/2)'}{}^2 } {\normLp{W}{2(q/2)'}{(\R^d)}^2} \normSch{(1+P^*P/h^2)^{1/2}\gamma(1+P^*P/h^2)^{1/2}}{q/2}{}
			.
		\end{align*}
		It remains to prove for any $2\leq q\leq\infty$
		\begin{equation*}
			\normSch{W\chi^\w(1+P^*P/h^2)^{-1/2}}{2(q/2)'}{} \lesssim h^{1-d\left(\frac 12-\frac 1q\right)}\normLp{W}{2(q/2)'}{(\R^d)}
		,
		\end{equation*}
		i.e.\ for any $2\leq \alpha\leq\infty$
		\begin{equation*}
			\normSch{W\chi^\w(1+P^*P/h^2)^{-1/2}}{\alpha}{} \lesssim h^{1-d /\alpha}\normLp{W}{\alpha}{(\R^d)}
		.
		\end{equation*}
		Let us show the previous bound with $\alpha=2$ and $\alpha=\infty$.
		The proof of Lemma \ref{lemma:SA-elliptic-1body} indeed shows \eqref{eq:SA-elliptic}, that we recall:
		\begin{equation*}
			\forall 2\leq q\leq\infty,\quad
			\chi^\w(1+P^*P/h^2)^{-1/2} =\OR\left(h^{1-d\left(\frac 12-\frac 1q\right)}\right): L^2(\R^d)\to L^q(\R^d)
			.
		\end{equation*}
		\item[\quad]
		The case $\alpha=\infty$ is given by the one function's estimate \eqref{eq:SA-elliptic} applied to $q=2$
		\begin{align*}
			\norm{W\chi^\w(1+P^*P/h^2)^{-1/2}}_{L^2\to L^2}
			&= \norm{W\chi^\w(1+P^*P/h^2)^{-1/2}}_{L^2\to L^2}
			\\&\leq \norm{\chi^\w(1+P^*P/h^2)^{-1/2}}_{L^2\to L^2} \normLp{W}{\infty}{(\R^d)}
			\\&\lesssim h \normLp{W}{\infty}{(\R^d)}
		.
		\end{align*}
		\item[\quad]
		Suppose that $\alpha=2$.
		We write the $\schatten^2$ norm with respect to the integral kernel and use the one function's estimate \eqref{eq:SA-elliptic} applied to $q=\infty$
		\begin{align*}
			\normSch{W\chi^\w(1+P^*P/h^2)^{-1/2}}{2}{}^2 
			&= \int_{\R^d}\int_{\R^d} \abs{\left(W\chi^\w(1+P^*P/h^2)^{-1/2}\right) (x,y)}^2 dx dy
			\\&= \int_{\R^d}\int_{\R^d} \abs{W(x)}^2 \abs{\left(\chi^\w(1+P^*P/h^2)^{-1/2}\right) (x,y)}^2 dx dy
			\\&\leq \normLp{W}{2}{(\R^d)}^2 \sup_{x\in\R^d}\normLp{\left(\chi^\w(1+P^*P/h^2)^{-1/2}\right) (x,\cdot)}{2}{(\R^d)}^2
			\\&\leq \normLp{W}{2}{(\R^d)}^2 \norm{\left(\chi^\w(1+P^*P/h^2)^{-1/2}\right) (x,y)}_{L^\infty_x L^2_y(\R^d\times\R^d)}^2
			\\&\leq \normLp{W}{2}{(\R^d)}^2 \norm{\chi^\w(1+P^*P/h^2)^{-1/2}}_{L^2\to L^\infty}^2
			\\&\lesssim h^{2-d}\normLp{W}{2}{(\R^d)}^2.
		\end{align*}
		We may thus write
		\begin{equation*}
			\normSch{W\chi^\w(1+P^*P¨/h^2)^{-1/2}}{2}{} \lesssim h^{1-d/2} \normLp{W}{2}{(\R^d)} .
		\end{equation*}
		Then, by interpolation between the two previous bounds we get the bounds for all the exponents $2\leq \alpha\leq\infty$.
		Finally, for any $2\leq q\leq\infty$
		\begin{equation*}
			\normLp{\rho_{\chi^\w\gamma\chi^\w}}{q/2}{(\R^d)} \leq C h^{2-2d\left(\frac 12-\frac 1q\right)}\normSch{(1+P^*P/h^2)^{1/2}\gamma(1+P^*P/h^2)^{1/2}}{q/2}{}.
		\end{equation*}
	\end{proof}
	% fin demo \ref{thm:ELp-elliptic-matdens}

	\section{More general L$^p$ estimates}\label{sec:gene}
	
	We now turn to the region $p=0$. We give a general first estimate which holds under the sole assumption that $\partial_\xi^2 p$ is not degenerate. This is particularly useful in the context of Schr\"odinger operators, because this assumption holds without any hypothesis on the potential $V$. In the one-body case ($\rk\gamma=1$), we recover \cite[Thm. 6]{koch2007semiclassical} (up to logarithmic factors which appear in few cases).

	\subsection{Statement of the result}
	
	Let $d\geq 1$.
	Let $m$ be an order function on $\R^d\times\R^d$, $p\in S(m)$ be real-valued and $P:= p^\w(x,hD)$ (the following theorem are also true for any other quantization $P$ of $p$).
	\begin{assump}\label{cond:gene}
		A point $(x_0,\xi_0)\in\R^d\times\R^d$ satisfies the \emph{general non-degeneracy condition} for the symbol $p$ if
		\begin{equation*}
			\partial_\xi^2p(x_0,\xi_0) \text{ is non-degenerate}.
		\end{equation*}
	\end{assump}

	\begin{rmk}
		For Schr\"odinger operators $p(x,\xi)=\xi^2+V(x)-E$ with $V\in\CR^\infty(\R^d,\R)$ bounded from below and satisfying Definition \ref{cond:am-potential-pol-growth}, the previous assumption is satisfied for all $(x_0,\xi_0)\in\R^d\times\R^d$.
	\end{rmk}

	Recall first the one-body result.
		
	\begin{thm}[General one-body estimates, {\cite[Thm. 6]{koch2007semiclassical}}]\label{thm:ELp-gene-1body}
		%Let $d\geq 1$.
		Let $(x_0,\xi_0)\in\R^d\times\R^d$ be a point satisfying Assumption \ref{cond:gene}.
		Then, there exist a neighborhood $\VR$ of $(x_0,\xi_0)$ and $h_0>0$, such that for any $\chi\in\test{\R^d\times\R^d}$ with support contained in $\VR$, there exists $C>0$ such that for any $0< h\leq h_0$, for any $2\leq q\leq\infty$ and $u\in L^2(\R^d)$,
		\[ \normLp{\chi^\w u}{q}{(\R^d)} \leq C \log(1/h)^{t(q,d)} h^{-s(q,d)} \left(\normLp{u}{2}{(\R^d)}+\frac 1h\normLp{Pu}{2}{(\R^d)}\right) \]
		where $s(q,d)$ and $t(q,d)$ and  are given by the formulas
		\begin{itemize}
			\item
			when $d=1$:
			\begin{equation}\label{eq-def:s,t-d=1-ELp-gene-1body}
				t(q,1)=0 \quad\text{ and }\quad s(q,1)=\frac 12\left(\frac 12-\frac 1q\right),
			\end{equation}
			\item 
			when $d=2$:
			\begin{equation}\label{eq-def:s,t-d=2-ELp-gene-1body}
				t(q,2) = \begin{cases}
							0 &\text{if}\ 2\leq q<\infty,\\
							\frac 1 2 &\text{if}\  q=\infty.
						\end{cases}
			\quad\text{ and }\quad
			s(q,2) = \frac 1 2-\frac 1 q
			,
			\end{equation}
			\item
			when $d\geq 3$: $t(q,d)=0$ and
			\begin{equation}\label{eq-def:s-ELp-gene-1body}
				s(q,d) = 
				\begin{cases}
						\frac d 2\left(\frac 1 2-\frac 1 q\right) &\text{if}\ 2\leq q\leq\frac{2d}{d-2},\\
					d\left(\frac 1 2-\frac 1q\right)-\frac 12 &\text{if}\ \frac{2d}{d-2}\leq q\leq\infty.
				\end{cases}
			\end{equation}
		\end{itemize}
		Equivalently, one has for all $2\leq q\leq\infty$ 
		\begin{equation*}%\label{eq:ELp-gene-1body}
			\chi^\w{(1+P^*P/h^2)^{-1/2}} =\OR \left(\log(1/h)^{t(q,d)}h^{-s(q,d)}\right) : L^2(\R^d)\to L^q(\R^d).
		\end{equation*}
	\end{thm}
	% fin enonce \ref{thm:ELp-gene-1body}
	
	\begin{rmk}\label{rmk:comp-s-gene}
		The exponent $s_{\text{gene}}$ defined on Theorem \ref{thm:ELp-gene-1body} is larger or equal to the one in Sobolev estimates $s_{\text{Sobolev}}(q,d)=d(1/2-1/q)$ ($s_{\text{gene}}(q,d)>s_{\text{Sobolev}}(q,d)$ for $2<q\leq\infty$ and they are equal for $q=2$). It is also stricly smaller than the exponent of the elliptic estimates $s_{\text{ellip}}(q,d)=d(1/2-1/q)-1$ (see Figure \ref{fig:comp-exp-s}).
	\end{rmk}

	\begin{thm}[General many-body estimates]\label{thm:ELp-gene-matdens}  
		%Let $d\geq 1$.
		Let $(x_0,\xi_0)\in\R^d\times\R^d$ be a point satisfying Assumption \ref{cond:gene}. Then, there exist a neighborhood $\VR$ of $(x_0,\xi_0)$ and $h_0>0$, such that for any $\chi\in\test{\R^d\times\R^d}$ with support contained in $\VR$, there exists $C>0$ such that for any $2\leq q\leq \infty$, for any $0< h\leq h_0$, for any bounded non-negative operator $\gamma$ on $L^2(\R^d)$
		\[ \normLp{\rho_{\chi^\w\gamma\chi^\w}}{q/2}{(\R^d)} \leq C \log(1/h)^{2t(q,d)} h^{-2s(q,d)}\normSch{(1+P^2/h^2)^{1/2}\gamma(1+P^2/h^2)^{1/2}}{\alpha(q,d)}{} \]
		where $s(q,d)$ is given by the formula \eqref{eq-def:s-ELp-gene-1body} and $t(q,d)$, $\alpha(q,d)$ are given by
	\begin{itemize}
	\item
		when $d=1$: 
		\begin{equation}\label{eq-def:t,alpha-d=1-ELp-gene-matdens}
		t(q,1)=0 \quad\text{ and }\quad \alpha(q,1) = \frac q2.
		%\normLp{\rho_{\chi^\w\gamma\chi^\w}}{q/2}{(\R^d)} \leq C h^{2-2d\left(\frac 12-\frac 1q\right)}\normSch{(1+P^*P/h^2)^{1/2}\gamma(1+P^*P/h^2)^{1/2}}{q/2}{}.
		\end{equation}
	\item 
		when $d=2$:
		\begin{equation}\label{eq-def:t-d=2-ELp-gene-matdens}
			t(q,2) = 
			\begin{cases}
				0 &\text{if}\ 2\leq q< 6,\\
				\frac 1 2-\frac 1 q &\text{if}\ 6\leq q\leq\infty
				,
			\end{cases}
		\end{equation}
		and
		\begin{equation}\label{eq-def:alpha-d=2-ELp-gene-matdens}
			\alpha(q,2) =  \begin{cases}
			\frac{2q}{q+2} &\text{if}\ 2\leq q\leq 6,
			\\ \frac q 4 &\text{if}\ 6\leq q\leq\infty
			,
			\end{cases}
		\end{equation}
	\item
		when $d\geq 3$:
		\begin{equation}\label{eq-def:t-d>2-ELp-gene-matdens}
			t(q,d) =\begin{cases}
			0&\text{if}\ 2\leq q <\frac{2(d+1)}{d-1} 
			\\ \frac{d}{q}-\frac{d-2}2  &\text{if}\ \frac{2(d+1)}{d-1}\leq q\leq\frac{2d}{d-2}
			\\0&\text{if}\ \frac{2d}{d-2}\leq q\leq \infty
			,
			\end{cases}
		\end{equation}
		and 
		\begin{equation}\label{eq-def:alpha-ELp-gene-matdens}
			\alpha(q,d)= \begin{cases}
			\frac{2q}{q+2} & \text{if}\  2 \leq q \leq\frac{2(d+1)}{d-1}
			\\\frac{2q}{d(q-2)} & \text{if}\  \frac{2(d+1)}{d-1}\leq q \leq \dfrac{2d}{d-2}
			\\\frac{(d-2)}{2d} q & \text{if}\  \frac{2d}{d-2}\leq q\leq\infty
			.
			\end{cases}
		\end{equation}
	\end{itemize}
	\end{thm}
	% fin enonce \ref{thm:ELp-gene-matdens} 

	We write the different end-points with what we obtain the estimates by interpolation.
	\begin{table}[h!]
		\centering
		\begin{tabular}{|c|c|c|c|c|}
			\hline
			End-point $q$  & 2 & $\infty$ & $ \frac{2(d+1)}{(d-1)}$  & $ \frac{2d}{(d-2)}$\\
			\hline
			Refering name  & (4.1) & (4.2) & (4.3) & (4.4)  \\
			\hline
		\end{tabular}
		\caption{References of the end-points' labels in the general case.}
		\label{table:endpoint-gene}
	\end{table}
	
	\begin{multicols}{2}
			%%%%%%%%%%%%%%%%%%% DEBUT DESSIN $s(q,d)$ et $t(q,d)$ d=1 %%%%%%%%%%%%%%%%%%%%%%
		\begin{center}\begin{tikzpicture}[line cap=round,line join=round,>=triangle 45,x=6 cm,y=4 cm,scale=2.1]
			d=2
			\draw (0.25,0.6) node{When $d=1$};
			\draw[->] (0.,0.) -- (0.55,0.);
			\draw (0.55,0.) node[right] {$\displaystyle{\frac{1}{q}}$};
			\draw[->] (0.,0.) -- (0.,0.6);

			%
			%% abscisses
			%
			\draw (0.5,0.) node[below]{$\frac{1}{2}$};
			%
			%% etiquettes
		\draw (0.5,-0.1) node[below]{$(4.1)$};
		\draw (0.,-0.1) node[below]{$(4.2)$};
			%
			%
			%% ordonnées
			\draw (0.,1/4) node[left]{$\frac 14$};
			%
			%% s(q,d) ktz
			\draw (0.,1/4) node {$\bullet$}; % L^infty
			\draw[line width=0.8pt] (0.,1/4) --(0.5,0.);
			%
			%
			% t(q,d) ktz
			%
			\draw (0,0.) node {$\bullet$};
			\draw (0,0.) node[below]{$0$};
			\draw (0.5,0) node {$\bullet$};
			\draw[line width=2.pt, dashed] (0,0)--(0.5,0.);
			%
			%% legendes
			%
			\draw[line width=0.8pt] (1/3,0.4) -- (1/3+1/15,0.4);
			\draw (1/3+1/15,0.4) node[right]{$s(q,1)$};
			\draw[line width=1.pt,dashed] (1/3,0.3) -- (1/3+1/15,0.3);
			\draw (1/3+1/15,0.3) node[right]{$t(q,1)$};
			\end{tikzpicture}\end{center}
		%%%%%%%%%%%%%%%%%%%%%% FIN DESSIN $s(q,d)$ et $t(q,d)$ d=1 %%%%%%%%%%%%%%%%%%%%%%%%%%%%

		%%%%%%%%%%%%%%%%%%%%%%  DESSIN alpha(q,d) d=1 %%%%%%%%%%%%%%%%%%%%%%%%%%%%
		\begin{center}\begin{tikzpicture}[line cap=round,line join=round,>=triangle 45,x=6.0cm,y=3 cm,scale=1.5]
			d=3
			\draw[->] (0.,0.) -- (0.55,0.);
			\draw (0.55,0.) node[right] {$\displaystyle{\frac{1}{q}}$};
			\draw[->] (0.,0.) -- (0.,1.1);
			\draw (0.,1.1) node[above] {$\displaystyle{\frac 1{\alpha(q,1)}}$};
			%
			%% abscisses
			%
			\draw (0,0.) node[below]{$0$}; 
			\draw (0.5,0.) node[below]{$\frac{1}{2}$};
			%
			%% etiquettes
		\draw (0.5,-0.1) node[below]{$(4.1)$};
		\draw (0.,-0.1) node[below]{$(4.2)$};
			%
			%% ordonnées
			%
			\draw (0.,1) node[left]{1};
			%
			%% points
			%
			\draw[color=black] (0.,0) node {$\bullet$};
			\draw[color=black] (0.5,1) node {$\bullet$};
			\draw[dotted] (0, 1)-|(0.5, 0);
			%
			%% courbes
			\draw[line width=0.8pt] (0,0)--(0.5,1);
			\end{tikzpicture}\end{center}
		%%%%%%%%%%%%%%%%%%%%%% FIN DESSIN alpha(q,d) d=1 %%%%%%%%%%%%%%%%%%%%%%%%%%%%
	\end{multicols}

		% d=2
	\begin{multicols}{2}
		
		%%%%%%%%%%%%%%%%%%% DEBUT DESSIN $s(q,d)$ et $t(q,d)$ d=2 %%%%%%%%%%%%%%%%%%%%%%
		\begin{center}\begin{tikzpicture}[line cap=round,line join=round,>=triangle 45,x=6 cm,y=3 cm,scale=2.1]
			d=2
			\draw (0.25,0.6) node{When $d=2$};
			\draw[->] (0.,0.) -- (0.55,0.);
			\draw (0.55,0.) node[right] {$\displaystyle{\frac{1}{q}}$};
			\draw[->] (0.,0.) -- (0.,0.6);
			%\draw (0.,1.1) node[above] {$\displaystyle{s(q,d)}$};
			%
			%% abscisses
			%
			\draw (1/6,0.) node[below]{$\frac{d-1}{2(d+1)}=\frac 16$};
			\draw (0.5,0.) node[below]{$\frac{1}{2}$};
			%
			%% etiquettes
		\draw (0.5,-0.15) node[below]{$(4.1)$};
		\draw (0.,-0.15) node[below]{$(4.2)$};
		\draw (1/6,-0.15) node[below]{$(4.3)$};
			%
			%% ordonnées
			\draw (0.,0.5) node[left]{$\frac{d-1}{2}=\frac 12$};
			\draw (0.,1/3) node[left]{$\frac 13$};%{$\frac{1}{d+1}=\frac 13=\frac{d}{2(d+1)}$};
			%
			%% s(q,d) ktz
			\draw (0.,0.5) node {$\bullet$}; % L^infty
			\draw[line width=0.8pt] (0.,0.5) --(0.5,0.);
			%
			%
			% t(q,d) ktz
			%
			\draw (1/6,1/3) node {$\bullet$};
			\draw[dotted] (0., 1/3)-|(1/6, 0.);
			\draw[line width=2.pt, dashed] (0,0.5)--(1/6,1/3);
			\draw[)-,line width=2.pt, dashed] (1/6,0)--(0.5,0.);
			%
			%% legendes
			%
			\draw[line width=0.8pt] (1/3,0.4) -- (1/3+1/15,0.4);
			\draw (1/3+1/15,0.4) node[right]{$s(q,2)$};
			\draw[line width=1.pt,dashed] (1/3,0.3) -- (1/3+1/15,0.3);
			\draw (1/3+1/15,0.3) node[right]{$t(q,2)$};
			\end{tikzpicture}\end{center}
		%%%%%%%%%%%%%%%%%%%%%% FIN DESSIN $s(q,d)$ et $t(q,d)$ d=2 %%%%%%%%%%%%%%%%%%%%%%%%%%%%
		
		%%%%%%%%%%%%%%%%%%%%%%  DESSIN alpha(q,d) d=2 %%%%%%%%%%%%%%%%%%%%%%%%%%%%
		\begin{center}\begin{tikzpicture}[line cap=round,line join=round,>=triangle 45,x=6.0cm,y=3 cm,scale=1.5]
			d=3
			\draw[->] (0.,0.) -- (0.55,0.);
			\draw (0.55,0.) node[right] {$\displaystyle{\frac{1}{q}}$};
			\draw[->] (0.,0.) -- (0.,1.1);
			\draw (0.,1.1) node[above] {$\displaystyle{\frac 1{\alpha(q,2)}}$};
			%
			%% abscisses
			%
			\draw (0,0.) node[below]{$0$}; %$\frac{d-2}{2d}$
			\draw (1/6,0.) node[below]{$\frac{d-1}{2(d+1)}=\frac{1}{6}$}; %$\frac{d-1}{2(d+1)}$
			\draw[color=black] (0.5,1) node {$\bullet$};
			\draw (0.5,0.) node[below]{$\frac{1}{2}$};
			%
			%% etiquettes
		\draw (0.5,-0.15) node[below]{$(4.1)$};
		\draw (0.,-0.15) node[below]{$(4.2)$};
		\draw (1/6,-0.15) node[below]{$(4.3)$};
			%
			%% ordonnées
			%
			\draw (0.,1) node[left]{1};
			\draw (0.,2/3) node[left]{$\frac{d}{d+1}=\frac 23$};
			%
			%% points
			% point (1/2,1)
			\draw[dotted] (0, 1)-|(0.5, 0);
			\draw[color=black] (0.,0) node {$\bullet$};
			\draw[color=black] (1/6,2/3) node {$\bullet$};
			\draw[dotted] (0, 2/3)-|(1/6,0);
			%
			%% courbes
			%
			% ktz
			\draw[line width=0.8pt] (0.,0) --(1/6,2/3)--(0.5,1);
			\end{tikzpicture}\end{center}
		%%%%%%%%%%%%%%%%%%%%%% FIN DESSIN alpha(q,d) d=2 %%%%%%%%%%%%%%%%%%%%%%%%%%%%
	\end{multicols}

	 % d>2
	 \begin{multicols}{2}
	 	
	 	%%%%%%%%%%%%%%%%%%% DEBUT DESSIN $s(q,d)$ et $t(q,d)$ d>2 %%%%%%%%%%%%%%%%%%%%%%
	 	\begin{center}\begin{tikzpicture}[line cap=round,line join=round,>=triangle 45,x=6 cm,y=3 cm,scale=2.]
	 		d=3
	 		\draw (0.25,1) node{When $d\geq 3$};
	 		\draw[->] (0.,0.) -- (0.55,0.);
	 		\draw (0.55,0.) node[right] {$\displaystyle{\frac{1}{q}}$};
	 		\draw[->] (0.,0.) -- (0.,1.1);
	 		\draw (0,0.) node{$\bullet$};
	 		\draw (0,0.) node[below]{$0$};
	 		%
	 		%% abscisses
	 		%
	 		\draw (1/6,0.) node[below]{$\frac{d-2}{2d}$};
	 		\draw (1/4,0.) node[below]{$\frac{d-1}{2(d+1)}$};
	 		\draw (0.5,0.) node[below]{$\frac{1}{2}$};
	 		%
	 		%% etiquettes
	 	\draw (0.,-0.15) node[below]{$(4.2)$};
	 	\draw (1/6,-0.15) node[below]{$(4.4)$};
	 	\draw (1/4,-0.15) node[below]{$(4.3)$};
	 	\draw (0.5,-0.15) node[below]{$(4.1)$};
	 		%
	 		%% ordonnées
	 		\draw (0., 3/8) node[left]{$\frac{d}{2(d+1)}$};
	 		\draw (0,0.5) node[left]{$\frac{1}{2}$};
	 		\draw (0.,1.) node[left]{$\frac{d-1}{2}$};
	 		\draw (0.,1/4) node[left]{$\frac{1}{d+1}$};
	 		%
	 		%% s(q,d) ktz
	 		\draw (0.,1.) node {$\bullet$};
	 		\draw (1/6,0.5) node {$\bullet$};
	 		\draw[dotted] (0., 0.5)-|(1/6, 0.);
	 		\draw (1/4,3/8) node {$\bullet$};
	 		\draw[dotted] (0., 3/8)-|(1/4, 0.);
	 		\draw[line width=0.8pt] (0.,1) -- (1/6,0.5)--(0.5,0.);
	 		%
	 		%
	 		% t(q,d) ktz
	 		\draw (1/2,0.) node {$\bullet$};
	 		\draw (1/6,0.) node {$\bullet$};
	 		\draw (1/4,1/4) node {$\bullet$};
	 		\draw[dotted] (0., 1/4)-|(1/4, 0.);
	 		\draw[line width=1.pt, dashed] (0,0)--(1/6,0.)--(1/4,1/4);
	 		\draw[)-,line width=1.2pt, dashed] (1/4,0)--(0.5,0.);
	 		%
	 		%% legendes
	 		%
	 		\draw[line width=0.9pt] (1/3,0.8) -- (1/3+1/15,0.8);
	 		\draw (1/3+1/15,0.8) node[right]{$s(q,d)$};
	 		\draw[line width=1.pt,dashed] (1/3,0.7) -- (1/3+1/15,0.7);
	 		\draw (1/3+1/15,0.7) node[right]{$t(q,d)$};
	 		\end{tikzpicture}\end{center}
	 	%%%%%%%%%%%%%%%%%%%%%% FIN DESSIN $s(q,d)$ et $t(q,d)$ d>2 %%%%%%%%%%%%%%%%%%%%%%%%%%%%
	 	
	 	%%%%%%%%%%%%%%%%%%%%%%  DESSIN alpha(q,d) d>2 %%%%%%%%%%%%%%%%%%%%%%%%%%%%
	 	\begin{center}\begin{tikzpicture}[line cap=round,line join=round,>=triangle 45,x=6.0cm,y=3 cm,scale=2.0]
	 		d=3
	 		\draw[->] (0.,0.) -- (0.55,0.);
	 		\draw (0.55,0.) node[right] {$\displaystyle{\frac{1}{q}}$};
	 		\draw[->] (0.,0.) -- (0.,1.1);
	 		\draw (0.,1.1) node[above] {$\displaystyle{\frac 1{\alpha(q,d)}}$};
	 		\draw (0,0.) node[below]{$0$};
	 		%
	 		%% abscisses
	 		%
	 		\draw (1/6,0.) node[below]{$\frac{d-2}{2d}$};
	 		\draw (1/4,0.) node[below]{$\frac{d-1}{2(d+1)}$};
	 		\draw (0.5,0.) node[below]{$\frac{1}{2}$};
	 		%
	 		%% etiquettes
	 		%
	 	\draw (0.,-0.15) node[below]{$(4.2)$};
	 	\draw (1/6,-0.15) node[below]{$(4.4)$};
	 	\draw (1/4,-0.15) node[below]{$(4.3)$};
	 	\draw (0.5,-0.15) node[below]{$(4.1)$};
	 		%
	 		%% ordonnées
	 		%
	 		\draw (0.,1) node[left]{1};
	 		\draw (0.,3/4) node[left]{$\frac{d}{d+1}$};
	 		\draw (0.,2/3) node[left]{$\frac{d-1}{d}$};
	 		%
	 		%% points
	 		% point (1/2,1)
	 		\draw[color=black] (0.5,1) node {$\bullet$};
	 		\draw[dotted] (0, 1)-|(0.5, 0);
	 		\draw[color=black] (0.,0) node {$\bullet$};
	 		\draw[color=black] (1/6,1) node {$\bullet$};
	 		\draw[dotted] (0, 1)-|(1/6, 0);
	 		\draw[color=black] (1/4,3/4) node {$\bullet$};
	 		\draw[dotted] (0, 3/4)-|(1/4,0);
	 		%
	 		%% courbes
	 		%
	 		% ktz
	 		\draw[line width=0.8pt, color=black] (0.,0) --(1/6,1)--(1/4,3/4) ;%--(0.5,1);
	 		\draw[line width=0.8pt, color=black] (1/4,3/4)--(0.5,1);
	 		\end{tikzpicture}\end{center}
	 	%%%%%%%%%%%%%%%%%%%%%% FIN DESSIN alpha(q,d) d>2 %%%%%%%%%%%%%%%%%%%%%%%%%%%%
	 \end{multicols}
 	% d>2
	
	\subsection*{Some comments on improvement}
	
	\begin{rmk}\label{rmk:Linfty-gene-dim2-improvment}
		In dimension 2, in the case of Schrödinger operators for the symbols $p(x,\xi)=\abs{\xi}^2+V(x)$ satisfying Definition \ref{cond:am-potential-pol-growth}, Smith and Zworski \cite{smith2012pointwise} proved that Theorem \ref{thm:ELp-gene-1body} is true with $t(\infty,2)=0$ (which means that we can get rid of the logarithm in dimension $2$ for $q=\infty$). This implies that we can set $t(\infty,2)=0$ in Theorem \ref{thm:ELp-gene-matdens} as well, and by interpolation Theorem \ref{thm:ELp-gene-matdens} holds for $t(q,2)=\frac 2q$ for all $q\in[6,\infty]$.
	\end{rmk}

	%%%%%%%%%%%%%%%%%%% DEBUT DESSIN $s(q,d)$ et $t(q,d)$ d=2 %%%%%%%%%%%%%%%%%%%%%%
	\begin{center}\begin{tikzpicture}[line cap=round,line join=round,>=triangle 45,x=6 cm,y=3 cm,scale=2.1]
		d=2
		\draw (0.25,0.6) node{When $d=2$};
		\draw[->] (0.,0.) -- (0.6,0.);
		\draw (0.6,0.) node[right] {$\displaystyle{\frac{1}{q}}$};
		\draw[->] (0.,0.) -- (0.,0.6);
		%\draw (0.,1.1) node[above] {$\displaystyle{s(q,d)}$};
		%
		%% abscisses
		%
		\draw (1/6,0.) node[below]{$\frac{d-1}{2(d+1)}=\frac 16$};
		\draw (0.5,0.) node[below]{$\frac{1}{2}$};
		%
		%% ordonnées
		\draw (0.,0.5) node[left]{$\frac{d-1}{2}=\frac 12$};
		\draw (0.,1/3) node[left]{$\frac{1}{d+1}=\frac 13=\frac{d}{2(d+1)}$};
		%
		%% s(q,d) ktz
		\draw (0.,0.5) node {$\bullet$}; % L^infty
		\draw[line width=0.8pt] (0.,0.5) --(0.5,0.);
		%
		%
		% t(q,d) ktz
		\draw (0.,0.) node {$\bullet$}; % L^infty
		\draw (1/6,1/3) node {$\bullet$};
		\draw[dotted] (0., 1/3)-|(1/6, 0.);
		\draw[line width=2.pt, dashed] (0,0.)--(1/6,1/3);
		\draw[)-,line width=2.pt, dashed] (1/6,0)--(0.5,0.);
		%
		%% legendes
		%
		\draw[line width=0.8pt] (1/3,0.4) -- (1/3+1/15,0.4);
		\draw (1/3+1/15,0.4) node[right]{$s(q,d)$};
		\draw[line width=1.pt,dashed] (1/3,0.3) -- (1/3+1/15,0.3);
		\draw (1/3+1/15,0.3) node[right]{$t(q,d)$};
		\end{tikzpicture}\end{center}
	%%%%%%%%%%%%%%%%%%%%%% FIN DESSIN $s(q,d)$ et $t(q,d)$ d=2 %%%%%%%%%%%%%%%%%%%%%%%%%%%%

	%% DETAILS phd on Remark \ref{rmk:Linfty-gene-dim2-improvment}
				
	\begin{rmk}\label{rmk:gene-optim-endpoint}
		When $d\geq 3$, the Schatten exponent that we obtain in Theorem \ref{thm:ELp-gene-matdens} for the Keel-Tao endpoint $q=2d/(d-2)$ is $\alpha=1$. Frank and Sabin \cite[Lem. 2]{frank2016stein} proved that this Schatten exponent is sharp in the related context of the Strichartz estimates associated to the propagator $e^{it\Delta}$. We expect that this result extends to our context, however it is not straightforward to adapt their proof. Indeed, their strategy amounts to showing that the dual operator is not compact, while here the dual operator is compact. Hence, we would rather need to quantify the ``loss of compactness''` of our dual operator as $h\to0$, which is a very interesting problem.
% c.f.	DETAILS phD
	\end{rmk}
	
	\subsection{Notation for the proof of Theorem \ref{thm:ELp-gene-matdens}}\label{subsec:notation-ELp-gene-matdens}
	
	\begin{itemize}
		\item Let $\delta\in]0,1[$ and $I=[-\delta/2,\delta/2]$. 
		\item Let $J=[-1,1]$.
		\item Let $\psi\in\test{\R}$ such that $\supp\psi\subset \overset{\circ}{I}$ and $\psi\neq 0$.
		\item Let $\tilde{\psi}\in\test{\R}$ such that $\tilde{\psi}=1$ on $I$ and $\supp\tilde{\psi}\subset[-\delta,\delta]$.
		\item Let $\VR=U\times V$ a bounded open neighboorhood of $(x_0,\xi_0)$.
		\item Let $\chi\in\test{\R^d\times\R^d}$ be such that $\supp\chi\subset\VR$.
		\item Let  $\tilde{\chi}\in\test{\R^d\times\R^d,[0,1]}$ be such that $\tilde{\chi}=1$ on $\supp\chi$ and such that $\supp\tilde{\chi}\subset\VR$.
		\item Let $\chi_0\in\test{\R^d\times\R^d}$ such that  $\chi_0=1$ on a neighborhood of $(x_0,\xi_0)$.
	\end{itemize}

	We will add constraints on $\delta$ and $\VR$ along the proof.
	
	% fin \ref{subsec:notation-ELp-gene-matdens}
	
	\subsection{Proof of Theorem \ref{thm:ELp-gene-matdens}}
	
	\paragraph[q=2,infty]{
			End-points (4.1) and (4.2) of Table \ref{table:endpoint-gene}.
		}
	
	We start to give the extremal estimates for $q=2$ and $q=\infty$. 
	\begin{itemize}
		\item
		By the Mercer theorem (Remark \ref{rmk:mercer-thm}) and
		 the Calderon-Vaillancourt theorem (Theorem \ref{thm:SA-Cald-Vaill}), we have
		\[ \normLp{\rho_{\chi^\w\gamma\chi^\w}}{1}{(\R^d)} = \tr_{L^2} (\chi^\w\gamma\chi^\w) \lesssim \normSch{\gamma}{1}{} .\]
		\item
%% DETAILS V1
 		Similarly as in the proof of Theorem \ref{thm:abstract-microloc-extend}, by the one function $L^\infty$ estimate (c.f. Theorem \ref{thm:ELp-gene-1body})
		\begin{align*}
			\normLp{\rho_{\chi^\w\gamma\chi^\w}}{\infty}{(\R^d)}
			&\leq \norm{\left(1+P^2/h^2\right)^{1/2}\gamma\left(1+P^2/h^2\right)^{1/2}}_{L^2\to L^2} \normLp{\rho_{\chi^\w(1+P^2/h^2)^{-1}\chi^\w}}{\infty}{(\R^d)}
			\\&\lesssim  \norm{\left(1+P^2/h^2\right)^{1/2}\gamma\left(1+P^2/h^2\right)^{1/2}}_{L^2\to L^2}
			\begin{cases}
				h^{-1/2}&\text{if}\  d=1,\\
				\log(1/h)/h &\text{if}\  d=2,\\
				h^{-(d-1)} &\text{if}\  d>2.
			\end{cases}
		\end{align*}
	\end{itemize}
	We thus have the bounds for $q=2$ and $q=\infty$.
	For $d=1$, we interpolate between them and get
	\begin{equation*}
		\normLp{\rho_{\chi^\w\gamma\chi^\w}}{q/2}{(\R)}\lesssim h^{\frac 1q-\frac 12}\normSch{(1+P^2/h^2)^{1/2}\gamma(1+P^2/h^2)^{1/2}}{q/2}{(L^2(\R))},
	\end{equation*}
	which is exactly Theorem \ref{thm:ELp-gene-matdens} in the case $d=1$.
	
	\paragraph[q=ep KT]{%rmk 15
		End-point (4.4) of Table \ref{table:endpoint-gene}.	
}
	For $d\geq 3$, by the triangle inequality at the Keel-Tao endpoint,
	that is the one-body estimate for $q=\frac{2d}{d-2}$ (end-point (4.3)) of Theorem \ref{thm:ELp-gene-1body},
	we have
	\begin{equation*}
		\normLp{\rho_{\chi^\w\gamma\chi^\w}}{d/(d-2)}{(\R^d)} \lesssim h^{-1}\normSch{(1+P^2/h^2)^{1/2}\gamma(1+P^2/h^2)^{1/2}}{1}{(L^2(\R^d))}.
	\end{equation*}
	Interpolating this bound with the bound for $q=\infty$ proves Theorem \ref{thm:ELp-gene-matdens} in the case $2d/(d-2)\leq q\leq \infty$, $d\geq 3$. The next step is to get estimates for $2\leq q\leq 2(d+1)/(d-1)$ with $d\geq 2$. The remaining estimates in the range $2(d+1)/(d-1)\leq q\leq 2d/(d-2)$, $d\geq 3$ are then obtained by interpolating the estimates for $q=2(d+1)/(d-1)$ and $q=2d/(d-2)$.

\paragraph[q=ep Sogge]{
		End-point (4.3) of Table \ref{table:endpoint-gene} (and other points between (4.1) and (4.3)).	
	}
	We thus now fix $d\geq 2$ and $2\leq q\leq 2(d+1)/(d-1)$. The idea is to introduce a new variable $t\in\R$.
	%% c.f. THESE
	%% By Lemma \ref{lemma:gene-matden-reduc-dim}
	Then, we have by Mercer theorem (Remark \ref{rmk:mercer-thm})
	\begin{align*}
		\normLp{\psi}{2}{(I)}^2\normLp{\rho_{\chi^\w\gamma\chi^\w}}{q/2}{(\R^d)}
		&=  \sup_{W\in L^{2(q/2)'}\cap\CR^0(\R^d)} \frac{ \normLp{\psi}{2}{(I)}^2 \int_{\R^d}\rho_{\chi^\w\gamma\chi^\w}(x)\abs{W(x)}^2 dx}{\normLp{W}{2(q/2)'}{(\R^d)}^2 }
		\\&\leq  \sup_{W\in L^{2(q/2)'}\cap\CR^0(\R^d)}\frac{ \normLp{\psi}{2}{(I)}^2 \normSch{W(x)\chi^\w\sqrt{\gamma}}{2}{(L^2(\R^d))}^2 }{\normLp{W}{2(q/2)'}{(\R^d)}^2 }
		\\&\leq \sup_{W\in L^{2(q/2)'}\cap\CR^0(\R^d)}\frac{ \normSch{\psi(t)W(x)\chi^\w\sqrt{\gamma}}{2}{(L^2(\R^d),L^2(\R^{d+1}))}^2 }{\normLp{W}{2(q/2)'}{(\R^d)}^2 }
		.
	\end{align*}
	In the last inequality, we use that 
	%% Lemma \ref{lemma:gene-matden-reduc-dim}
	for any bounded operator $A$ on $L^2(\R^d)$ and any $\psi\in L^2(\R)$, we have (as can be seen by computing $(\psi(t)A)^*\psi(t)A$)
	\begin{equation*}
	\normSch{\psi(t)A}{2}{(L^2(\R^d),L^2(\R^{d+1}))}
	= \normLp{\psi}{2}{(\R)}\normSch{A}{2}{(L^2(\R^d))}.
	\end{equation*}
	%% fin Lemma \ref{lemma:gene-matden-reduc-dim}
	%
	We define the operator $B_h$ by $\frac 1h P$. Let $u\in L^2(\R^d)$.
	Then $v(t,x):=u(x)$ satisfies
	\begin{equation*}%\label{eq:gene-eq-evol}
		\begin{cases}
		(hD_t+P)v &= hB_h v, \\
		v(0,x)&= u(x). 
		\end{cases}
	\end{equation*}
	Let $R_h\in\op_h^\w(\schwartz)$ be such that
	\[ \chi_0^\w(x,hD) P = (\chi_0 p)^\w(x,hD) - h R_h .\]
	Defining the unitary operators $\{F(t)\}_{t\in\R}$ on $L^2(\R^d)$ such that
	\begin{equation*}%\label{eq:gene-propag}
		\begin{cases}
		(hD_t+ ({\chi_0}p)^\w)F(t) &= 0,\quad \forall t\in\R, \\
		F(0)= \id
		,
		\end{cases}
	\end{equation*}
	we have by the Duhamel formula
	\begin{align*}
		\chi^\w(x,hD)
		&= F(t)\chi^\w (x,hD) -i\int_0^t F(t)F(r)^* (\chi_0^\w(x,hD) B_h +R_h) \chi^\w(x,hD) dr 
		\\&= F(t)F(0)^*\chi^\w -i\int_0^t F(t)F(r)^* \tilde{\chi}^\w (x,hD) \: (\chi_0^\w(x,hD)  B_h+ R_h ) \chi^\w(x,hD) dr
		\\&\quad
		 -i \int_0^t F(t)F(r)^* (1-\tilde{\chi})^\w (x,hD) \: (\chi_0^\w(x,hD)  B_h+ R_h) \chi^\w(x,hD) dr
		,
	\end{align*}
	as an identity between bounded operators on $L^2(\R^d)$ for all $t\in\R$.
	%
	%% def U(t,r) gene
	\begin{defi}\label{def:gene-U(t,r)}
	For all $\varphi\in\test{\R^d\times\R^d}$, let $U_\varphi(t,r) =U_{\varphi,+}(t,r)-U_{\varphi,-}(t,r)$ where
	\begin{equation*}\label{eq-def:gene-U_+(t,r)}
	\begin{split}
		U_{\varphi,+}(t,r)
		&:= \indicatrice{r\geq 0}\indicatrice{t\geq r} \: \psi(t)\tilde{\psi}(t-r)F(t)F(r)^*\varphi^\w(x,hD)
	\end{split}
	\end{equation*}
	and
	\begin{equation*}\label{eq-def:gene-U_-(t,r)}
	\begin{split}
		U_{\varphi,-}(t,r)
		&:= \indicatrice{r< 0}\indicatrice{t\leq r} \: \psi(t)\tilde{\psi}(t-r)F(t)F(r)^*\varphi^\w(x,hD)
		.
	\end{split}
	\end{equation*}
	\end{defi}
	%% fin def U(t,r) gene \ref{def:gene-U(t,r)}
	%
	Let us also define $S$ by
	\begin{equation}\label{eq-def:gene-S}
		S
		=-i\int_0^t \psi(t) F(t)F(r)^* (1-\tilde{\chi}^\w(x,hD))(\chi_0^\w(x,hD) B_h +R_h)\chi^\w(x,hD) dr
		.
	\end{equation}
	By multipling by $\psi$ on the left of the previous Duhamel formula, we have
	 \begin{align*}
		 \psi(t)\chi^\w(x,hD) 
		 &= U_\chi(t,0)-i\int_J  U_{\tilde{\chi}}(t,r) (\chi_0^\w B_h\chi^\w +R_h\chi^\w)\: dr + S.
	 \end{align*}
	By the triangle inequality and H\"older inequality, we get for all $\alpha\geq 1$
	\begin{align*}
		&\normSch{\psi(t)W(x)\chi^\w\sqrt{\gamma}}{2}{(L^2(\R^d),L^2(\R^{d+1}))}
		\\&\quad
		\leq \normSch{ W(x)U_\chi(t,0)}{2\alpha'}{(L^2(\R^d),L^2(\R^{d+1}))}\normSch{\sqrt{\gamma}}{2\alpha}{(L^2(\R^d))}
		\\&\quad\quad
		+\abs{J}\sup_{r\in J}\normSch{W(x)U_{\tilde{\chi}}(t,r)}{2\alpha'}{(L^2(\R^d),L^2(\R^{d+1}))}
		\normSch{(\chi_0^\w B_h\chi^\w +R_h\chi^\w)\sqrt{\gamma} }{2\alpha}{(L^2(\R^d))} 
		\\&\quad\quad
		+\normSch{W(x)S}{2(q/2)'}{(L^2(\R^d),L^2(\R^{d+1}))}\normSch{\sqrt{\gamma }}{q}{(L^2(\R^d))}
		.
	\end{align*}
	There exists $\mathfrak{r}\in\schwartz(\R^d\times\R^d)$ such that $\comm{B_h}{\chi^\w}=\frac 1 h\comm{P}{\chi^\w} = \mathfrak{r}^\w$.
	Thus
	\begin{align*}
		&
		\normSch{(\chi_0^\w B_h\chi^\w +R_h\chi^\w)\sqrt{\gamma} }{2\alpha}{(L^2(\R^d))}
		\\&\quad
		\leq \normSch{\chi_0^\w \chi^\w B_h \sqrt{\gamma} }{2\alpha}{(L^2(\R^d))} +\normSch{\chi_0^\w \comm{B_h}{\chi^\w}\sqrt{\gamma}}{2\alpha}{(L^2(\R^d))} +\normSch{R_h\chi^\w\sqrt{\gamma}}{2\alpha}{(L^2(\R^d))}
		\\&\quad
		\leq \norm{\chi_0^\w \chi^\w}_{L^2(\R^d)\to L^2(\R^d)}\normSch{B_h \sqrt{\gamma} }{2\alpha}{(L^2(\R^d))}
		\\&\qquad+\left(\norm{\chi_0^\w \comm{B_h}{\chi^\w}}_{L^2(\R^d)\to L^2(\R^d)}+\norm{R_h\chi^\w}_{L^2(\R^d)\to L^2(\R^d)}\right)\normSch{\sqrt{\gamma}}{2\alpha}{(L^2(\R^d))}
		\\&\quad
		\lesssim \frac 1 h\normSch{P\gamma P}{\alpha}{(L^2(\R^d))}^{1/2} +\normSch{\gamma}{\alpha}{(L^2(\R^d))}^{1/2}
		.
	\end{align*}
	We only need to prove the Schatten estimates for the operators $U_\chi(t,r)$, $U_{\tilde{\chi}}(t,r)$ and $S$.

	\begin{prop}\label{prop:gene-matdens-term_nul}
		Recall that $S$ is the operator defined by \eqref{eq-def:gene-S}.
		Let $\beta\geq 2$. Then, we have the bound for all $W\in L^\beta(\R^{d+1})$
		\begin{equation*}
			\normSch{W(t,x)S}{\beta}{(L^2(\R^d),L^2(\R^{d+1}))}
			=\OR(h^\infty)\normLp{W}{\beta}{(I\times\R^d)}
			.
		\end{equation*}
	\end{prop}
	% fin \ref{prop:gene-matdens-term_nul}
	
%% cf DETAILS phd

	Before proving this proposition, let us conclude the proof of Theorem \ref{thm:ELp-gene-matdens}.
%% c.f. DETAILS PhD
	Given Assumption \ref{cond:gene}, we apply Theorem \ref{thm:SStrichartz-matdens} to
	$n=d$,
	$(x_0,\xi_0)\in\R^d\times\R^d$,
	$a=\chi_0 p\in\CR^\infty_t(\R,S_{(x,\xi)}(1))$ and $J=[-1,1]$. Thus, there exist $\delta>0$ and $\VR=U\times V$ neighborhood of $(x_0,\xi_0)$, so that we have for any $2\leq q\leq\tfrac{2(d+1)}{d-1}$
	\begin{align*}
		&
		\normSch{ W(x)U_\chi(t,0)}{2\alpha(q,d)'}{(L^2(\R^d),L^2(\R^{d+1}))}+\sup_{r\in J}\normSch{W(x)U_{\tilde{\chi}}(t,r)}{2\alpha(q,d)'}{(L^2(\R^d),L^2(\R^{d+1}))}
		\\&\quad
		\lesssim \log(1/h)^{t(q,d)} h^{-s(q,d)}\normLp{W}{2(q/2)'}{(\R^d)}\normLp{\psi}{\infty}{(\R)},
	\end{align*}
	for $\alpha(q,d)=\frac{2q}{q+2}$, $s(q,d)=\frac{d}{2}\big(\frac 12-\frac 1q\big)$ and $t(q,d) =0$ for $2\leq q<\frac{2(d+1)}{d-1}$ and $t(q,d) =\frac 1{d+1}$ if $q=\frac{2(d+1)}{d-1}$. 
	Hence, for any $2\leq q\leq\tfrac{2(d+1)}{d-1}$ 
	\begin{align*}
		&\normSch{\psi(t)W(x)\chi^\w\sqrt{\gamma}}{2}{(L^2(\R^d),L^2(\R^{d+1}))}
		\\&\quad
		\leq C_I \log(1/h)^{t(q,d)} h^{-s(q,d)}\normLp{W}{2(q/2)'}{(\R^d)} \times\\&\qquad\times
		\left(\normSch{\gamma}{\alpha(q,d)}{(L^2(\R^d))}^{1/2}+\frac 1 h\normSch{P\gamma P}{\alpha(q,d)}{(L^2(\R^d))}^{1/2} +\normSch{\gamma}{q/2}{(L^2(\R^d))}^{1/2}\right)
		\\&\quad
		\leq C'_I \log(1/h)^{t(q,d)} h^{-s(q,d)} \normLp{W}{2(q/2)'}{(\R^d)} \normSch{(1+P^2/h^2)^{1/2}\gamma(1+P^2/h^2)^{1/2}}{\alpha(q,d)}{(L^2(\R^d))}^{1/2}
		,
	\end{align*}
	which finishes the proof of Theorem \ref{thm:ELp-gene-matdens}
		(for all points between (4.1) and (4.3) of Table \ref{table:endpoint-gene}).

	Let us now prove Proposition \ref{prop:gene-matdens-term_nul}.
	
	\begin{proof}[\underline{Proof of Proposition \ref{prop:gene-matdens-term_nul}}]
	
	We only need to show the estimates
	\begin{equation}\label{eq-demo:gene-matdens-term_nul}
		\forall M\in\N,\quad
		\norm{S}_{L^2(\R^d)\to H_h^M(\R^{d+1})} =\OR(h^\infty).
	\end{equation}
	Note that $S=\indicatrice{t\in I} S$.
	Then, by Kato-Seiler-Simon Lemma \ref{lemma:Kato-Seiler-Simon_schatten} applied to $M\in\N$ such that $M\beta>d+1$
	\begin{align*}
		&\normSch{W(t,x)S}{\beta}{(L^2(\R^d),L^2(\R^{d+1}))}
		\\&\quad
		\leq \normSch{W(t,x)\indicatrice{t\in I}(1-h^2\Delta_{t,x})^{-M/2}}{\beta}{(L^2(\R^{d+1}))}\norm{(1-h^2\Delta_{t,x})^{M/2}S}_{L^2(\R^d)\to L^2(\R^{d+1})}
		\\&\quad
		= \OR(h^\infty)\normLp{W}{\beta}{(I\times\R^d)}
		.
	\end{align*}
	That gives us the desired estimates.
	
	Let us prove now the bounds \eqref{eq-demo:gene-matdens-term_nul}.	
	On the one hand, there exists $r\in\schwartz(\R^d\times\R^d)$ such that $(1-\tilde{\chi}^\w(x,hD))B_h\chi^\w(x,hD) =\mathfrak{r}^\w(x,hD)$.
	Since $\supp(1-\tilde{\chi})$ and $\supp\chi$ are disjoint
	\[ \mathfrak{r}^\w(x,hD) = \OR(h^\infty):\schwartzprime(\R^d)\to\schwartz(\R^d).\]
	On the other hand, by Lemma \ref{lemma:prop-F(t)}
	\begin{equation*}
		\forall M\in\N,\exists C>0\quad
		\sup_{r\in J}\norm{\psi(t)F(t)F(r)^* }_{H_h^M(\R^d)\to H_h^M(\R^{d+1})} \leq C.
	\end{equation*}
	Then, by composition we have for all $ M\in\N$
	\begin{equation*}
		\sup_{r\in J}\norm{\indicatrice{r\geq 0}\indicatrice{t\geq r} \:\psi(t)F(t)F(r)^*\mathfrak{r}^\w}_{L^2(\R^d)\to H_h^M(\R^{d+1})} =\OR(h^\infty),
	\end{equation*}
	and
	\begin{equation*}
		\sup_{r\in J}\norm{\indicatrice{r\leq 0}\indicatrice{t\leq r} \:\psi(t)F(t)F(r)^*\mathfrak{r}^\w}_{L^2(\R^d)\to H_h^M(\R^{d+1})} =\OR(h^\infty).
	\end{equation*}
	Finally, we get \eqref{eq-demo:gene-matdens-term_nul}, that ends the proof of Proposition \ref{prop:gene-matdens-term_nul}.
	\end{proof}
	% fin demo \ref{prop:gene-matdens-term_nul}

	\begin{rmk}\label{rmk:ELp-gene-matdens-bis}
		Notice that \eqref{eq:gene-matdens-sch-2} allows to treat the case $\beta\in\big[0,\tfrac{d-1}2\big)$ leading to better values $\alpha(q,d)$ and $t(q,d)$ but also a worse value of $s(q,d)$ for $d\geq 2$ and $\tfrac{2(d+1)}{d-1}<q\le\infty$:
		\begin{equation*}
			s(q,d)=\left(d+\frac12-\frac{2}{q-2}\right)\left(\frac12-\frac1q\right),\ t(q,d)=0,\ \alpha(q,d)=\frac{2q}{q+2}.
		\end{equation*}
		We discard these estimates because we always want to keep the same exponent $s(q,d)$ as in the one-body case (so that our many-body estimates imply the one-body estimates).
		However, we can discuss that for $\gamma=\sum_{j=1}^{N_h}\lambda_j\bra{u_j}\ket{u_j}$, how to write the best estimates (to choose the best exponent possible) for according to the range of $N_h$. For instance, for $N_h> (h^{-1})^{\min_{q}\frac{s(q)-s_{\rm gene(q)}}{1/\alpha_{\rm gene}(q)-1/\alpha(q)}}$, the concentration is better with the above estimated rather than the one of Theorem \ref{thm:ELp-gene-matdens}. But, the discussion is quite complex in the general case, without any specific given situation.
	\end{rmk}
	% fin \ref{rmk:ELp-gene-matdens-bis}

% c.f. phD	
	\begin{rmk}\label{rmk:TP:new-trans-point}
		Let us comment on why the many-body case has an additional transition point $q=2(d+1)/(d-1)$ compared to the one-body case. 
		Let $K_{h,\beta}$ be defined by
		\begin{equation*}
			\forall t\in\R\quad K_{h,\beta}(t):= h^{-d/2} \abs{t}^\beta(h+\abs{t})^{-d/2} .
		\end{equation*}
		We have in the one body case (which can be proved with the complex interpolation).
		\begin{equation*}
		\sup_{r\in J}\norm{T_0(r)}_{L^2_t L^{\left(\frac{2(\beta+1)}{\beta}\right)'}_x(\R^{d+1}) \to L^2_t L^{\frac{2(\beta+1)}{\beta}}_x(I\times\R^d)}
		\leq C \left(\int_{-1}^1K_{h,\beta}(t)dt\right)^{\frac 1{\beta+1}}
		,
		\end{equation*}
		while as the above proof shows, we have in the many body case
		\begin{equation*}
			\begin{split}
			\sup_{r\in J}&
			\normSch{W T_0(r) W}{2}{(L^2(\R^{d+1}))}
			\\&\quad
			\leq C 
			\norm{W}_{L_t^{\infty} L_x^{2 (\beta+1) }(I\times\R^d)}^2
			\left(\int_{-1}^1 K_{h,\beta}(t)^2 dt\right)^{\frac{1}{2(\beta+1)}}
			.
			\end{split}
		\end{equation*}
		Note that
		\begin{equation*}
			\int_{-1}^1K_{h,\beta}(t)dt
			= \begin{cases}
			h^{\beta-d+1} &\text{if}\  \beta<\frac{d-2}2,\\
			h^{-d/2}\log(1/h) &\text{if}\  \beta =\frac{d-2}2,\\
			h^{-d/2} &\text{if}\  \beta >\frac{d-2}2
			.
			\end{cases}
		\end{equation*}
		and
		\begin{equation*}
		\left(\int_{-1}^1K_{h,\beta}(t)^2dt\right)^{1/2}
			=\begin{cases}
				h^{\beta-d+1/2} &\text{if}\  \beta<\frac{d-1}2,\\
				h^{-d/2}\log(1/h)^{1/2} &\text{if}\  \beta =\frac{d-1}2,\\
				h^{-d/2} &\text{if}\  \beta >\frac{d-1}2,
			\end{cases}
		\end{equation*}
		so that the one-body and many-body constants coincide for $\beta>(d-1)/2$ (which corresponds to $2\leq q< 2(d+1)/(d-1)$) but differ for $\beta\leq(d-1)/2$ (which corresponds to $q>2(d+1)/(d-1)$). We expect that it is not a technical artefact of the proof, but rather that this transition point does appear in the many-body case. Indeed, a similar phenomenon exists for Strichartz estimates \cite{frank2014} where the existence of a transition is shown at this point $q=2(d+1)/(d-1)$. It is a challenging problem to adapt their result to our setting. A related problem would be to get rid of the logarithm in our many-body estimates at $q=2(d+1)/(d-1)$.
	% c.f. DETAILS phd
	\end{rmk}
	% end rmk \ref{rmk:TP:new-trans-point}
	
	%\newpage
	
	\section{Sogge's L$^p$ estimates}\label{sec:sogge-est}
	
	We now treat the case $p=0$ and $\nabla_\xi p\neq 0$. In the case of Schr\"odinger operators, it means that we are away from the turning point region $\{V=E\}$. This setting corresponds to the one of Sogge without potential on a compact manifold. In the one-body case ($\rk\gamma=1$), we recover \cite[Thm. 5]{koch2007semiclassical}. %\cite[Thm. 7.12 and 10.10]{zworski2012semiclassical}
	
	\subsection{Statement of the result}
	
	Let $d\geq 2$. For $x\in\R^d$, we denote by $x'$ the $d-1$ last variables of $x$
	\[ x' := (x_2,\ldots,x_d)\in\R^{d-1} .\]
	Let $m$ an order function on $\R^d\times\R^d$, $p\in S(m)$ be real-valued and $P:= p^\w(x,hD)$ (but the following theorems are true for any other quantization).
	
	\begin{assump}\label{cond:sogge}
		A point $(x_0,\xi_0)\in\R^d\times\R^d$ satisfies the \emph{Sogge non-degeneracy conditions} for the symbol $p$ if
		\begin{equation*}
			p(x_0,\xi_0)=0,\quad \nabla_\xi p(x_0,\xi_0)\neq 0,
		\end{equation*}
		and if
		\begin{equation}\label{cond:sogge-curv}
		\begin{split}
			&\text{the second fundamental form of }
			\{ \xi\in\R^d \: :\: p(x_0,\xi)=0 \}
			\\&\text{ is non-degenerate at } \xi_0
			.
		\end{split}
		\end{equation}
	\end{assump}
	% fin cond Sogge
	
	Fisrt recall the one-body result.
	
	\begin{thm}[Sogge one-body estimates, {\cite[Thm. 5]{koch2007semiclassical}}]\label{thm:ELp-sogge-1body}
		Let $(x_0,\xi_0)\in\R^d\times\R^d$ be a point satisfying Assumption \ref{cond:sogge}.
		Then, there exist a neighborhood $\VR$ of $(x_0,\xi_0)$ and $h_0>0$, such that for any $\chi\in\test{\R^d\times\R^d}$ with support contained in $\VR$, there  exists $C>0$ such that for any $0< h\leq h_0$ and for any $2\leq q\leq\infty$
		\begin{equation*}
			\normLp{\chi^\w u}{q}{(\R^d)} \leq C h^{-s(q,d)}\left(\normLp{u}{2}{(\R^d)}+\frac 1 h\normLp{Pu}{2}{(\R^d)}\right)
			,
		\end{equation*}
		where
		\begin{equation}\label{eq-def:s-ELp-sogge}
			s(q,d)=\begin{cases}
			\frac{d-1}{2}\left(\frac 1 2 -\frac 1 q\right) &\text{if}\ 2\leq q\leq\frac{2(d+1)}{d-1},\\
			d\left(\frac 1 2-\frac 1 q\right)-\frac 1 2 &\text{if}\ \frac{2(d+1)}{d-1}\leq q\leq\infty.
			\end{cases}
		\end{equation}
		Equivalently, one has for all $2\leq q\leq\infty$
		\begin{equation*}%\label{eq:ELp-sogge-1body}
			\chi^\w{(1+P^*P/h^2)^{-1/2}} =\OR (h^{-s(q,d)}) : L^2(\R^d)\to L^q(\R^d).
		\end{equation*}
	\end{thm}
	% fin \ref{thm:ELp-sogge-1body}
	
	\begin{rmk}\label{rmk:comp-s-sogge}
		The exponent $s_{\text{Sogge}}$, defined in \eqref{eq-def:s-ELp-sogge}, is always larger than the elliptic one $s_{\text{ellip}}$ for any $d\geq 2$ and $2\in[2,\infty]$. Moreover, it is strictly smaller than $s_{\text{gene}}$ for any $q\in(2,2d/(d-2))$ and they coincide when $q\in\{2\}\cup[2d/(d-2),\infty]$.  (c.f. Figure \ref{fig:comp-exp-s}).
	\end{rmk}

	\begin{thm}[Sogge many-body estimates]\label{thm:ELp-sogge-matdens} 
		Let $(x_0,\xi_0)\in\R^d\times\R^d$ be a point satisfying Assumption \ref{cond:sogge}.
		Then, there exist a neighborhood $\VR$ of $(x_0,\xi_0)$ and $h_0>0$, such that for any $\chi\in\test{\R^d\times\R^d}$ with support contained in $\VR$, there  exists $C>0$ such that for any $0< h\leq h_0$, for any $2\leq q\leq \infty$ and for any bounded self-adjoint non-negative operator $\gamma$ on $L^2(\R^d)$
		\begin{equation*}
			\normLp{\rho_{\chi^\w\gamma\chi^\w}}{q/2}{(\R^d)} \leq C h^{-2s(q,d)}\normSch{(1+P^*P/h^2)^{1/2}\gamma(1+P^*P/h^2)^{1/2}}{\alpha(q,d)}{}%{\left(L^2(\R^d)\right)}
		\end{equation*}
		where $s(q,d)$ is given by the formula \eqref{eq-def:s-ELp-sogge} and $\alpha(q,d)$ is given by
		\begin{equation}\label{eq-def:alpha-ELp-sogge-matdens}
			\alpha(q,d)=\begin{cases}
			\frac{2q}{q+2} &\text{if}\  2\leq q\leq\frac{2(d+1)}{d-1},\\
			\frac {q(d-1)}{2d} &\text{if}\  \frac{2(d+1)}{d-1}\leq q\leq\infty 
			.
			\end{cases}
		\end{equation}
	\end{thm}
	% fin \ref{thm:ELp-sogge-matdens} 

	\begin{rmk}\label{rmk:comp-alpha-sogge}
		For $d\geq3$, the exponent $\alpha_{\text{Sogge}}>\alpha_{\text{gene}}$ for $q\in(2d/(d-2),\infty)$ and $\alpha_{\text{Sogge}}=\alpha_{\text{gene}}$ for $q=[2,2d/d-2]$. It is stricly larger than the one $\alpha_{\text{ellip}}(q,d)=q/2$ in the elliptic estimates for any $q\in(2,\infty)$ and they coincide for $q=2$ or $q=\infty$ (c.f Figure \ref{fig:comp-exp-alpha}).
	\end{rmk}

		We write the different end-points with what we obtain the estimates by interpolation.
		\begin{table}[h!]
			\centering
			\begin{tabular}{|c|c|c|c|}
				\hline
				End-point $q$  & 2 & $\infty$ & $ \frac{2d}{(d-2)}$\\
				\hline
				Refering name  & (5.1) & (5.2) & (5.3)  \\
				\hline
			\end{tabular}
			\caption{References of the end-points' labels in the Sogge's case.}
			\label{table:endpoint-sogge}
		\end{table}
%\newpage
	
	% illustrations
%	\begin{figure}[!h]
	\begin{multicols}{2}
				%%%%%%%%%%%%%%%%%%%%%% DESSIN  exposant s(q,d)%%%%%%%%%%%%%%%%%%%%%%%%%%%%
		\begin{center}\begin{tikzpicture}[line cap=round,line join=round,>=triangle 45,x=4.0cm,y=4.0cm,scale=2.]
			\draw[->] (0.,0.) -- (0.6,0.);
			\draw (0.6,0.) node[right] {$\frac{1}{q}$};
			\draw[->] (0.,0.) -- (0.,0.7);
			\draw (0.,0.7) node[above] {$s(q,d)$};
			% graduation
			\draw (0.,0.) node[below] {$0$};
			\draw(0.5,0.) node[below] {$\frac{1}{2}$};
			\draw (0.5,0)  node {$\bullet$};
			\draw (0.,0.6) node[left] {$\frac{d-1}{2}$};
			\draw (0.,0.6)  node {$\bullet$};
			\draw (0.,0.2) node[left] {$\frac{d-1}{2(d+1)}$};
			\draw (0.2,0.) node[below] {$\frac{d-1}{2(d+1)}$};
			\draw (0.2,0.2) node{$\bullet$};
	%% etiquettes
	\draw (0.,0.6) node[right] {$(5.2)$};
	\draw (0.5,0) node[above] {$(5.1)$};
	\draw (0.2,0.3) node[above] {$(5.3)$};
			% lines
			\draw[dotted] (0., 0.2)-|(0.2, 0.01);
			\draw[line width=0.6pt] (0.,0.6)-- (0.2,0.2) -- (0.5,0.);
			%\draw[dashed, line width=0.6pt] (0.,0.6)--(0.5,0.);
			\end{tikzpicture}\end{center}
		%%%%%%%%%%%%%%%%%%%%%% FIN DESSIN %%%%%%%%%%%%%%%%%%%%%%%%%%%%
		
		%%%%%%%%%%%%%%%%%%%%%%  DESSIN alph(q,d) %%%%%%%%%%%%%%%%%%%%%%%%%%%%
		\begin{center}\begin{tikzpicture}[line cap=round,line join=round,>=triangle 45,x=4.0cm,y=2.5 cm,scale=2.]
			%d=4
			\draw[->] (0.,0.) -- (0.6,0.);
			\draw (0.6,0.) node[right] {$\dfrac{1}{q}$};
			\draw[->] (0.,0.) -- (0.,1.1);
			\draw (0.,1.1) node[above] {$\dfrac 1{\alpha(q,d)}$};
			%
			%% abscisses
			%
			\draw (0.,0) node[below]{0};
			\draw (0.3,0.) node[below]{$\frac{d-1}{2(d+1)}$};
			\draw (0.5,0.) node[below]{$\frac{1}{2}$};
			%
			%% etiquettes
		\draw (0.,-0.1) node[below] {$(5.2)$};
		\draw (0.5,1) node[right] {$(5.1)$};
		\draw (0.3,4/5) node[above] {$(5.3)$};
			%
			%% ordonnees
			%
			\draw (0.,0) node[left]{0};
			\draw (0.,1) node[left]{1};
			\draw (0.,4/5) node[left]{$\frac{d}{d+1}$};
			%
			%% points
			% point (1/2,1)
			\draw[dotted] (0, 1)-|(0.5, 0);
			\draw (0., 0) node {$\bullet$};
			\draw (3/10, 4/5) node {$\bullet$};
			\draw (0.5, 1) node {$\bullet$};
			\draw[dotted] (0, 4/5)-|(0.3,0);
			%
			%% courbes
			%
			% bourrin
			%\draw[dashed,line width=0.8pt] (0.,0) --(0.5,1);
			% sogge
			\draw[line width=0.8pt] (0.,0) --(3/10, 4/5)--(0.5,1);
			\end{tikzpicture}\end{center}
		%%%%%%%%%%%%%%%%%%%%%% FIN DESSIN alpha(q,d) %%%%%%%%%%%%%%%%%%%%%%%%%%%%
	\end{multicols}
%\caption{Exponents $s_{\rm Sogge}(q,d)$ and $\alpha_{\rm Sogge}(q,d)$ in dimension $d\geq 2$.}
%\label{fig:exp-s-alpha_sogge}
%\end{figure}
	% fin illustrations

	\subsection{Definitions and notation for the proof of Theorem \ref{thm:ELp-sogge-matdens}}\label{subsec:sogge-def}
	
	\begin{itemize}
	\item	
	Let $I_0, I\subset\R$ be open intervals which contain $\left(x_0\right)_1$ such that
	\begin{equation*}
	I_0\subset \bar{I_0} \subset I\subset \left[(x_0)_1-1,(x_0)_1+1\right].
	\end{equation*}
	
	\item
	Let $\psi_1\in\test{\R,{[0,1]}}$ such that $\psi_1=1$ on $\overline{I_0}$ and such that $\supp\psi_1\subset I$.
	
	\item
	Let $t_0\in I\setminus I_0$.
	
	\item
	Let $J\subset\R$ a bounded open interval which contains $\left(\xi_0\right)_1$.
	
	\item
	Let $R:=\abs{I}>0$.
%% c.f. DETAILS PhD PIMPIM
	Note that $R\leq 2$.
	
	\item
	We choose $\psi\in\test{\R}$ such that $\psi=1$ on $[-R,R]$ and $\psi=0$ outside $[-2R,2R]$.
	\\Note that $\supp\psi\subset[-4,4]$.
	%
%% c.f. DETAILS PhD	
	\item
	 Let us define $I_r := [(x_0)_1-5,(x_0)_1+5]$.
	
	\item
	Let $\VR_0'=U_0'\times V_0'$ and $\VR'=U'\times V'\subset\R^{d-1}\times\R^{d-1}$ be bounded open neighborhood of $(x_0',\xi_0')$ such that
	\begin{equation*}
		\VR_0' \subset \overline{\VR_0'} \subset \VR'.
	\end{equation*}
	
	\item
	Let $\varphi\in\test{\R^{2(d-1)},[0,1]}$ such that $\varphi=1$ on $\overline{\VR_0'} $ and $\supp\varphi\subset \VR'$.
	
	\item
	We define $\VR := I_0\times J\times \VR_0'$ and $\WR := I\times J\times\VR'$.
	We will add contraints on the size of $\WR$ along the proof.
	
	\item
	Let $\chi\in\test{\R^d\times\R^d}$ such that $\supp\chi\subset\VR$.	
\end{itemize}
	By construction this implies that
	\begin{equation*}
		\pi_{x_1}\supp\chi \subset I_0, \quad
		\pi_{(x',\xi')}\supp\chi \subset \VR_0',
	\end{equation*}
	then $\supp(1-\psi_1)\cap\pi_{x_1}\supp\chi=\emptyset$ and
	$\supp(1-\varphi)\cap\pi_{(x',\xi')}\supp\chi=\emptyset$.

	%%%%%%%%%%%%%%%%%%% DESSIN I %%%%%%%%%%%%%%%%%%%%%%%%%%
	\begin{center}
		\begin{tikzpicture}[scale =1.2]
		% axes
		\draw[->] (-3.5,0) -- (3.5,0);
		\draw (3.5,0) node[right] {$x_1$};
		\draw [->] (0,-0.1) -- (0,1.5);
		% point
		\draw (0.1,1.) node[right, above] {1};
		\draw (-1.2,0.) node{$+$};
		\draw (-1.2,0.) node[below] {$t_0$};
		% fonctions tests
		% psi_1
		\draw[ thick,  domain=-2.8:0]plot( \x,{ 0.5*(1+tanh(3.5*(\x+1.3))) });
		\draw[ thick,  domain=0:2.5]plot( \x,{ 0.5*(1-tanh(3.5*(\x-1.3))) })node[left,above]{$\psi_1$};
		\draw[<->] (-0.8,-0.15) -- (0.8,-0.15) node[midway, below] {$I_0$};
		\draw[<->] (-2.1,-0.5) -- (2.1,-0.5) node[midway, below] {$I$};
		\draw[<->,color=cristina] (-.5,-0.) -- (.3,-0.) node[midway, above] {$\pi_{x_1}\supp\chi$};
		\end{tikzpicture}
	\end{center}
	%%%%%%%%%%%%%%%%%% Fin DESSIN  I %%%%%%%%%%%%%%%%%%%%%%%
	
	%%%%%%%%%%%%%%%%%%% DESSIN II %%%%%%%%%%%%%%%%%%%%%%%%%%
	\begin{center}
		\begin{tikzpicture}[scale =1.2]
		% axes
		\draw[->] (-3.5,0) -- (3.5,0);
		\draw (3.5,0) node[right] {$(x',\xi')$};
		\draw [->] (0,-0.1) -- (0,1.5);
		% point
		\draw (0.1,1.) node[right, above] {1};
		% fonctions tests
		% phi
		\draw[ thick,  domain=-2.8:0]plot( \x,{ 0.5*(1+tanh(3.5*(\x+1.3))) });
		\draw[ thick,  domain=0:2.5]plot( \x,{ 0.5*(1-tanh(3.5*(\x-1.3))) })node[left,above]{$\varphi$};
		\draw[<->] (-0.8,-0.15) -- (0.8,-0.15) node[midway, below] {$\VR_0'$};
		\draw[<->] (-2.1,-0.6) -- (2.1,-0.6) node[midway, below] {$\VR'$};
		\draw[<->,color=cristina] (-.5,-0.) -- (.3,-0.) node[midway, above] {$\pi_{(x',\xi')}\supp\chi$};
		\end{tikzpicture}
	\end{center}
	%%%%%%%%%%%%%%%%%% Fin DESSIN  II %%%%%%%%%%%%%%%%%%%%%%%

	\subsection{Proof of Theorem \ref{thm:ELp-sogge-matdens}}
	
	First and foremost, the bounds at the two points $q=2$ and $q=\infty$ follow from Theorem \ref{thm:ELp-gene-matdens}.
	Since we have the bounds for $q=2$ and $q=\infty$, we now show it for $q=2(d+1)/(d-1)$, which implies the theorem by interpolation.
	
	Let us first explain why we will focus our proof on
	$\normLp{\rho_{\psi_1\chi^\w\gamma\chi^\w\psi_1}}{q/2}{(\R^d)}$.
	Recall with Mercer theorem (Remark \ref{rmk:mercer-thm}) that
	\begin{align*}
		\normLp{\rho_{\chi^\w\gamma\chi^\w}}{q/2}{(\R^d)}
		%
%% DETAILS cf PhD thesis
%		&
%		=\sup_{W\in L^{2(q/2)'}\cap\CR^0(\R^d)} \frac{\int_{\R^d}\rho_{\chi^\w\gamma\chi^\w}(x)\abs{W(x)}^2 dx}{\normLp{W}{2(q/2)'}{(\R^d)}^2}
%		%
%		\\&
		= \sup_{W\in L^{2(q/2)'}\cap\CR^0(\R^d)} \frac{\normSch{W(x)\chi^\w\sqrt{\gamma}}{2}{(\R^d)}^2}{\normLp{W}{2(q/2)'}{(\R^d)}^2}
		.
	\end{align*}
	By the triangle inequality, up to a multiplicative factor, it is bounded by
	\[ \sup_{W\in L^{2(q/2)'}(\R^d)} \frac{\normSch{W(x)\psi_1(x_1)\chi^\w\sqrt{\gamma}}{2}{(\R^d)}^2}{\normLp{W}{2(q/2)'}{(\R^d)}^2}
	 \]
	 and
	 \[\sup_{W\in L^{2(q/2)'}(\R^d)} \frac{\normSch{W(x)(1-\psi_1(x_1))\chi^\w\sqrt{\gamma}}{2}{(\R^d)}^2}{\normLp{W}{2(q/2)'}{(\R^d)}^2}.\]
	By construction $\supp(1-\psi_1)$ and $\pi_{x_1}\supp\chi$ are disjoint. Then
	\[ (1-\psi_1(x_1))\chi^\w=\OR(h^\infty):\schwartzprime(\R^d)\to\schwartz(\R^d) .\]
	By the Hölder and Kato-Seiler-Simon inequalities (Lemma \ref{lemma:Kato-Seiler-Simon_schatten}), with $M\in\N$ such that $2M(q/2)'>d$
	\begin{align*}
		&\normSch{W(x)(1-\psi_1(x_1))\chi^\w\sqrt{\gamma}}{2}{(\R^d)}
		\\&\quad
		\leq 	\normSch{W(x)(1-h^2\Delta)^{-M}}{2(q/2)'}{(L^2(\R^d))}\norm{(1-h^2\Delta)^M(1-\psi_1(x_1))\chi^\w}_{L^2(\R^d)\to L^2(\R^d)}\normSch{\sqrt{\gamma}}{q}{(L^2(\R^d))}
		\\&\quad
		=\OR(h^{N-d/2})\normLp{W}{2(q/2)'}{(\R^d)}\normSch{\gamma}{q/2}{(L^2(\R^d))}^{1/2}
		\quad\forall N\in\N
		.
	\end{align*}
	Hence, the crucial part of the proof relies on the estimation of
	\[ \sup_{W\in L^{2(q/2)'}(\R^d)} \frac{\normSch{W(x)\psi_1(x_1)\chi^\w\sqrt{\gamma}}{2}{(\R^d)}^2}{\normLp{W}{2(q/2)'}{(\R^d)}^2} = \normLp{\rho_{\psi_1\chi^\w\gamma\chi^\w\psi_1}}{q/2}{(\R^d)} .\]
	
	\bigskip
	The main idea now is to reduce the problem to an evolution equation in $d-1$ variables.
%% DETAILS PhD
%	By Lemma \ref{lemma:implicit-fct-sogge}, let $\UR$ a neighborhood of $(x_0,\xi_0)$, $e\in S(1)$ and $a\in\test{\R\times\R^{d-1}\times\R^{d-1}}$ such that we have:
%	$\inf\abs{e}>0$,  
%	\begin{equation*}
%		\forall(x,\xi)\in\UR,\quad p(x,\xi)=e(x,\xi)(\xi_1-a(x,\xi')),\quad
%		\partial^2_{\xi'}a(x_1,x',\xi') \text{ is non-degenerate} .
%	\end{equation*}
	Up to a permutation of coordinates, by the implicit functions theorem, there exist a neighborhood $\UR$ of $(x_0,\xi_0)$, functions $e\in S(1)$ and $a\in\test{\R\times\R^{d-1}\times\R^{d-1}}$,
	such that
	\begin{itemize}
		\item the Hessian
		\begin{equation}\label{cond:sogge-non-degen-hess}
		\partial_{\xi'}^2 a(x_0,\xi_0') \text{ is non-degenerate},
		\end{equation}
		\item $\inf\abs{e}>0$,
		\item for all $(x,\xi)\in\UR$
		\begin{equation*}\label{cond:sogge-decomp}
		p(x,\xi)=e(x,\xi)(\xi_1-a(x_1,x',\xi')).
		\end{equation*} 
	\end{itemize}
	We thus assume that $\WR\subset\UR$.
	Then, since $\supp\chi\subset\WR$, we have 
	\begin{equation*}
		\forall(x,\xi)\in\R^d\times\R^d,\quad
		p(x,\xi)\chi(x,\xi)=e(x,\xi)(\xi_1-a(x,\xi'))\chi(x,\xi).
	\end{equation*}
	We can write $P\chi^\w$ as 
	\begin{align*}
		P\chi^\w
		&= (p\chi)^\w(x,hD)+hr_1^\w(x,hD)
		\\&= e^\w (hD_{x_1}-a^\w(x_1,x',hD_{x'}))\chi^\w(x,hD) +hr_2^\w(x,hD) +hr_1^\w(x,hD)
		\\&=e^\w (hD_{x_1}-a^\w(x_1,x',hD_{x'}))\chi^\w(x,hD) +h r^\w(x,hD).
	\end{align*}
	By symbolic calculus $r\in\schwartz(\R^d\times\R^d)$.
	Let $B_h$ be the following pseudodifferential operator 
	\begin{equation*}
		B_h := \frac 1 h e^\w(x,hD_x)^{-1}\left(P\chi^\w(x,hD_x)-h r^\w(x,hD)\right)
		.
	\end{equation*}
	By definition $B_h$ satisfies $(hD_{x_1}-a^\w)\chi^\w = h B_h$.
	In other terms, we have for all $u\in L^2(\R^d)$
	\[ [hD_{x_1}-a^\w(x,hD_{x'})]\chi^\w u = h
	 B_h u
	.\]
	Furthermore, we have
	\begin{lemma}\label{lemma:sogge-prop-Bh}
		The operator $B_h$ satisfies 
		\begin{itemize}
			\item[(i)] the localization property
			\begin{equation}
			\label{eq:lemma-sogge-prop-Bh-eq-1}
			(1-\varphi^\w(x',hD_{x'}))B_h = \OR(h^\infty):\schwartzprime(\R^d)\to\schwartz(\R^d),
			\end{equation}
			\item[(ii)]
			\begin{equation}\label{eq:lemma-sogge-prop-Bh-eq-2}
			B_h(1+P^2/h^2)^{-1/2} = \OR(1): L^2(\R^d)\to L^2(\R^d).
			\end{equation}
		\end{itemize}
	\end{lemma}
	%
%	The proof of Lemma \ref{lemma:sogge-prop-Bh} is done in Appendix \ref{app-sec:tools-sogge}.
	\begin{proof}[\underline{Proof of Lemma \ref{lemma:sogge-prop-Bh}}]
		\item[(i)]
		Since $B_h$ satisfies
		\[ B_h = \frac 1h (hD_{x_1}-a^\w(x_1,x',hD_{x'}))\chi^\w(x,hD)  \]
		and  $\supp(1-\varphi)$ and $\pi_{(x',\xi')}\chi$ are disjoint, then
		\[
		(1-\varphi^\w(x',hD_{x'}))B_h =\OR(h^\infty):{\schwartzprime(\R^d)\to\schwartz(\R^d)}.\]
		
		\item[(ii)]
		Besides recalling the definition of $B_h$, $e\in S(1)$ and $r\in\schwartz(\R^d\times\R^d)$
		\begin{align*}
		B_h(1+P^2/h^2)^{-1/2}
		&= \left(\frac 1 h(e^\w)^{-1} P-(e^\w)^{-1}r^\w\right)(1+P^2/h^2)^{-1/2}
		\\&=  (e^\w)^{-1} \left(\frac 1 hP\right) (1+P^2/h^2)^{-1/2} \\&\quad
		-(e^\w)^{-1}r^\w(1+P^2/h^2)^{-1/2}.
		\end{align*}
		We obtain \eqref{eq:lemma-sogge-prop-Bh-eq-2} using that $(e^\w)^{-1}$, $r^\w$, $(P/h)(1+P^2/h^2)^{-1/2}$ and $(1+P^2/h^2)^{-1/2}$ are $\OR(1):L^2(\R^d)\to L^2(\R^d)$.
	\end{proof}
	% fin demo lemme \ref{lemma:sogge-prop-Bh}
	
	Let $r\in\R$. The following evolution equation
	\begin{equation*}\label{eq:sogge-propag}
		\begin{cases}
		[hD_t-a^\w(t,x',hD_{x'})] F(t,r)=0\ \quad t\in\R,
		\\ F(r,r) = \id,
		\end{cases}
	\end{equation*}
	is solved by a unique family of unitary operators $\{F(t,r)\}_{t\in\R}$ on $L^2(\R^{d-1})$ (we refer the reader to \cite[Thm. 10.1]{zworski2012semiclassical}).
	Recall the Duhamel's formula satisfied by all $u\in L^2(\R^d)$ and $t\in\R$
	\begin{equation*}
		\chi^\w u(t) = F(t,t_0)(\chi^\w u)(t_0)+i\int_{t_0}^t F(t,s)(B_h u)(s) ds \quad\text{ in }L^2_{x'}(\R^{d-1}) .
	\end{equation*}
	Defining by $\ev_{x_1=t_0}$ the operator of evaluation in $t_0\in\R$ of the first variable, which maps functions on $\R^d$ to fonctions on $\R^{d-1}$ 
	\begin{equation*}
		\ev_{x_1=t_0} u(x) = u(t_0,x')
	\end{equation*}
	and $U(t,r)$ the microlocalized operator on $L^2_{x'}(\R^{d-1})$
	\begin{equation*}\label{eq-def:sogge-U(t,r)}
		U(t,r) := \psi(t-r)F(t,r)\varphi^\w(x',hD_{x'}),
	\end{equation*}
	given the support property of $\psi_1$ and $\psi$, we get the decomposition
	\begin{equation*}
	\begin{split}
		\psi_1(t)
		\chi^\w(t,x',hD_{t,x'}) 
		&= 
		\psi_1(t) \left(
		F(t,t_0)\ev_{x_1=t_0}\chi^\w
		+i\left(\int_{t_0}^t F(t,s)(1-\varphi^\w)\ev_{x_1=s} 
		B_h
		ds\right)
		\right)
		\\&\quad+ i\psi_1(t)\left( \int_{t_0}^t U(t,s)\ev_{x_1=s} 
		B_h
		ds \right)
		.
	\end{split}
	\end{equation*}
	We notice that each term of this operator maps functions on $\R^d$ into functions on $\R^d$.
	Define $S$ by
	\begin{equation*}\label{eq:sogge-def-S}
		S := \psi_1(t) F(t,t_0)\ev_{x_1=t_0}\chi^\w
		+i\psi_1(t)\left(\int_{t_0}^t F(t,s)(1-\varphi^\w)\ev_{x_1=s} 
		B_h
		ds\right)
	\end{equation*}
	We introduce the operator $T_U$ which acts on functions on $\R^d$
	\begin{equation*}\label{eq:sogge-def-T_U}
		T_U f(t) := \psi_1(t)\int_{t_0}^t U(t,s)f(s) ds\quad \text{ in }L^2(\R^d)
		.
	\end{equation*}
	We then have the following results.
	
	\begin{prop}\label{prop:sogge-matdens-term_nul}
		Let $\beta\geq 2$. We have the bound
		\begin{equation*}
		\normSch{W S}{\beta}{(L^2(\R^d))} = \OR(h^\infty)\normLp{W}{\beta}{(\R^d)}.
		\end{equation*}
	\end{prop}
	% fait estimees avec S
	
	\begin{prop}\label{prop:sogge-matdens-term_cruc}
		If $\WR$ is a small enough  neighborhood of $(x_0,\xi_0)$, then the operator $T_U$ satisfies the dual estimates
		\begin{equation*}
			\normSch{W T_U}{2\alpha(q,d)'}{(L^2(\R^d))} \leq C h^{- s(q,d)}\normLp{W}{2(q/2)'}{(\R^d)}
		\end{equation*}
		for all $W\in L^{2(q/2)'}(\R^d)$, where $s(q,d)$ and $\alpha(q,d)$ are defined in the statement of Theorem \ref{thm:ELp-sogge-matdens}.
	\end{prop}
	% fin prop cruciale
	
	\begin{rmk}
		Recall that the operator $T_U$ depends on $\WR$ through the functions $\psi_1$ and $\varphi$.
	\end{rmk}
	
	% demo \ref{thm:ELp-sogge-matdens} en admettant lemmes
	Before proving the previous propositions, we use them to complete the proof of Theorem \ref{thm:ELp-sogge-matdens}.
	By the decomposition of the operator $\psi_1(t)\chi^\w$ and by the triangle inequality
	\begin{equation*}
		\normSch{W \psi_1(t)\chi^\w\sqrt{\gamma}}{2}{}
		\leq \normSch{W  S \sqrt{\gamma}}{2}{}+ \normSch{W T_U B_h\sqrt{\gamma}}{2}{}.
	\end{equation*}
	Using the H{\"o}lder inequality and Proposition \ref{prop:sogge-matdens-term_nul}, we have
	\begin{align*}
		\normSch{W S \sqrt{\gamma}}{2}{}
		&\leq \normSch{W  S}{2(q/2)'}{}\normSch{\sqrt{\gamma}}{q}{}
		\\&
		=\OR(h^\infty) \normLp{W}{2(q/2)'}{} \normSch{\gamma}{q/2}{}^{1/2}.
	\end{align*}
	By the H{\"o}lder inequality, Lemma \ref{lemma:sogge-prop-Bh} 
	 and Proposition \ref{prop:sogge-matdens-term_cruc}
	\begin{align*}
		&\normSch{W  T_U B_h\sqrt{\gamma}}{2}{}
		\\&\quad
		\leq \normSch{W  T_U }{2\alpha'}{} \norm{B_h(1+P^2/h^2)^{-1/2}}_{L^2\to L^2}\normSch{(1+P^2/h^2)^{1/2}\sqrt{\gamma}}{2\alpha}{}
		\\&\quad
		\lesssim h^{-s(q,d)} \normSch{(1+P^2/h^2)^{1/2}\gamma(1+P^2/h^2)^{1/2}}{\alpha}{}^{1/2}
		.
	\end{align*}
	Then, we have the same bound for the norm $\normSch{W(x)\psi_1(t)\chi^\w\sqrt{\gamma}}{2}{}$.
	
	\bigskip
	Finally, we obtain for $q=2(d+1)/d-1$
	\begin{equation*}
		\normLp{\rho_{\chi^\w\gamma\chi^\w}}{q/2}{(\R^d)} \lesssim h^{-2s(q,d)}\normSch{(1+P^2/h^2)^{1/2}\gamma(1+P^2/h^2)^{1/2}}{\alpha(q,d)}{} ,
	\end{equation*}
	which is the desired bound for the end-point (5.2).
	% fin demo \ref{thm:ELp-sogge-matdens}
	
	\subsubsection{Proof of Proposition \ref{prop:sogge-matdens-term_nul}}

		Let $\beta\geq 2$ and $W\in L^\beta(\R^d)$.
		In order to prove the desired inequality, we prove
		\begin{equation}\label{eq-demo:sogge-matdens-terme_nul}
		\forall k\in\N,\quad
			S=\OR(h^\infty): \schwartzprime(\R^d) \to H_h^k(\R^d).
		\end{equation}
		% details concl
		Assuming this result, we only need to choose $k\in\N$ so that $x\mapsto\crochetjap{x}^{-k}$ be $L^\beta(\R^d)$. By the Hölder, the previous inequality \eqref{eq-demo:sogge-matdens-terme_nul} and Kato-Seiler-Simon inequality (Lemma \ref{lemma:Kato-Seiler-Simon_schatten})
		\begin{align*}
			&\normSch{W(t,x')S}{\beta}{(L^2(\R^d))}
			\\&\quad
			\leq \normSch{W(t,x')(1-h^2\Delta_{t,x'})^{-k/2}}{\beta}{(L^2(\R^d))} \norm{(1-h^2\Delta_{t,x'})^{k/2} S}_{L^2(\R^d)\to L^2(\R^d)}
			\\&\quad
			\leq C_{d,k,N} h^{N-d/2}\normLp{W}{\beta}{(L^2(\R^d))} \quad\quad\forall N\in\N
			.
		\end{align*}
		Hence,
		\begin{equation*}
			\normSch{W S}{\beta}{(L^2(\R^d))} =\OR(h^\infty) \normLp{W}{\beta}{(L^2(\R^d))}.
		\end{equation*}
		% fin details concl
		
		\bigskip
		Let us now prove \eqref{eq-demo:sogge-matdens-terme_nul}.

		Since $t_0\not\in\pi_{x_1}\supp\chi$, we have
		\begin{equation*}\label{eq:sogge-eval}
			\ev_{x_1=t_0}\chi^\w(x,hD) = \OR(h^\infty):{\schwartzprime_x(\R^d)\to\schwartz_{x'}(\R^{d-1})},
		\end{equation*}
		which together with Lemma \ref{lemma:prop-F(t)}, implies that
		\begin{equation*}
			\forall k\in\N,\quad
			\psi_1(t)F(t,t_0)\ev_{x_1=t_0}\chi^\w =\OR(h^\infty):\schwartzprime(\R^d)\to H_h^k(\R^d).
		\end{equation*}
		
		Let us prove the same equality for the second term $\psi_1(t)\int_{t_0}^t F(t,s)(1-\varphi^\w(x,hD_{x'}))\ev_{x_1=s}B_h ds$.
		Recall $T_F$ is the operator which acts on functions in $\R^d$ defined by
		\begin{equation}\label{eq-def:sogge-T_F}
			T_F :
			u=u(t,x')\mapsto \psi_1(t)\int_{t_0}^t( F(t,s) u(s))(x') \: ds .
		\end{equation}
		By Lemma \ref{lemma:prop-F(t)} (to $n=d-1$, $t=x_1$ and $x=x'$), the operator $T_F$ which maps to $H_h^k(\R^d)$ into $H_h^k(\R^d)$ for all $k\in\N$.
		
		Then, since $1-\varphi^\w$ commutes with $\ev_{x_1=s}$ (because $\varphi^\w$ only acts on the variables $(x',\xi')$) and given Lemma \ref{lemma:sogge-prop-Bh} 
		\begin{align*}
			\psi_1(t) \: \int_{t_0}^t F(t,s) (1-\varphi^\w) \ev_{x_1=s}B_h \quad ds
			&= \psi_1(t) \: \int_{t_0}^t F(t,s)\ev_{x_1=s}(1-\varphi^\w)B_h \quad ds
			\\&= T_F\circ ((1-\varphi^\w)B_h)
			\\&= \OR(h^\infty): \schwartzprime(\R^d)\to H_h^k(\R^d)
			\quad\forall k\in\N.
		\end{align*}
		Finally, for all $k\in\N$
		\begin{equation*}
			S=\OR(h^\infty): \schwartzprime(\R^d)\to H_h^k(\R^d) ,
		\end{equation*}
		what is exactly \eqref{eq-demo:sogge-matdens-terme_nul}.

	\subsubsection{Proof of Proposition \ref{prop:sogge-matdens-term_cruc}}
	
		The operator $T_U$ can be split as $T_U=T_+-T_-$ with
		\begin{equation*}
			T_+ :=  \indicatrice{t\geq t_0} T_U
		\end{equation*}
		and
		\begin{equation*}
			T_- :=-\indicatrice{t\leq t_0}T_U.
%			\end{array}\right.
		\end{equation*}
		 Their dual operators' expressions are the following 
		 \begin{equation*}
			 T_+^* : 
			 \left\lbrace\begin{array}{lll}
			 L^2(\R^d)&\to& L^2(\R^d)\\
			 f&\mapsto& \indicatrice{r\geq t_0}	\int_\R \indicatrice{s\geq t_0}\indicatrice{s\geq r} U(s,r)^* \psi_1(s)f(s)  ds
			 \end{array}\right.
		 \end{equation*}
		 and
		 \begin{equation*}
			 T_-^* : 
			 \left\lbrace\begin{array}{lll}
			 L^2(\R^d)&\to& L^2(\R^d)\\
			 f&\mapsto& \indicatrice{r\leq t_0}	\int_\R \indicatrice{s\leq t_0}\indicatrice{t\leq r} U(s,r)^* \psi_1(s)f(s)  ds
			 .
			 \end{array}\right.
		 \end{equation*}
		 Then the operators $T_\pm T_\pm^*$ can be written as
		 \begin{equation*}
		 \begin{split}
			 T_+T_+^* f(t)
			 &= \int_{r\geq t_0}dr 
			 \int_\R ds \indicatrice{t\geq t_0}\indicatrice{t\geq r}	\indicatrice{s\geq t_0}\indicatrice{(s\geq r)}
			 \times\\&\quad\times
			 \psi_1(t)U(t,r) U(s,r)^* \psi_1(s)f(s)  
		 \end{split}
		 \end{equation*}
		 and
		 \begin{equation*}
		 \begin{split}
			 T_-T_-^* f(t)
			 &=  \int_{r\leq t_0}dr 
			 \int_\R ds \indicatrice{t\leq t_0}\indicatrice{t\leq r}	\indicatrice{s\leq t_0}\indicatrice{s\leq r}
			 \times\\&\quad\times
			 \psi_1(t)U(t,r) U(s,r)^* \psi_1(s)f(s)  .
		 \end{split}
		 \end{equation*}
		 Let $r\in\R$. Let us introduce the operators $A_{r,\pm}$
		 \begin{equation*}
		 	A_{r,+}:
		 	\left\lbrace\begin{array}{lll}
			 	L^2_{x'}(\R^{d-1})&\to& L^2_{t,x'}(\R^d)\\
			 	g&\mapsto& \indicatrice{t\geq t_0}\psi_1(t) \indicatrice{t\geq r} U(t,r)g
		 	\end{array}\right.
		 \end{equation*}
		 and
		 \begin{equation*}
			 A_{r,-}:
			 \left\lbrace\begin{array}{lll}
			 L^2_{x'}(\R^{d-1})&\to& L^2_{t,x'}(\R^d)\\
			 g&\mapsto& \indicatrice{t\leq t_0}	\psi_1(t)\indicatrice{t\leq r} U(t,r)g
			 .
			 \end{array}\right.
		 \end{equation*}
		 Their dual operators can be written as
		 \begin{equation*}
			 A_{r,+}^*:
			 \left\lbrace\begin{array}{lll}
			 L^2_{t,x'}(\R^d)&\to& L^2_{x'}(\R^{d-1})\\
			 f&\mapsto& 	\int_\R \indicatrice{s\geq t_0}\indicatrice{s\geq r}U(s,r)^* \psi_1(s)f(s)  ds
			 \end{array}\right.
		 \end{equation*}
		 and
		 \begin{equation*}
			 A_{r,-}^*:
			 \left\lbrace\begin{array}{lll}
			 L^2_{t,x'}(\R^d)&\to& L^2_{x'}(\R^{d-1})\\
			 f&\mapsto& \int_\R\indicatrice{s\leq t_0}\indicatrice{s\leq r}U(s,r)^*\psi_1(s)f(s)ds
			 .
			 \end{array}\right.
		 \end{equation*}
		 The operators $A_{r,\pm}A_{r,\pm}^*$ acts on $L^2(\R^d)$.
		 This gives
		 \begin{align*}
			  T_+T_+^* = \int_{r\geq t_0} A_{r,+}A_{r,+}^* dr
			  = \int_{I_r}A_{r,+}A_{r,+}^* dr
			  ,\\ T_-T_-^* = \int_{r\leq t_0} A_{r,-}A_{r,-}^* dr
			  = \int_{I_r}A_{r,-}A_{r,-}^* dr
			  .
		 \end{align*}
		 Moreover
		 \begin{equation*}
			  T_UT_U^* = T_+T_+^*+T_-T_-^* .
		 \end{equation*}
		 Thus, for any $\alpha\geq 1$
		 \begin{align*}
			 &\normSch{W T_U}{2\alpha'}{(L^2(\R^d))}^2
			 \\&\quad
			 = \normSch{W T_U T_U^* \bar{W}}{\alpha'}{(L^2(\R^d))}
			 \\&\quad
			 \leq \abs{I_r} \left(\sup_{r\in I_r} \normSch{W A_{r,+} A_{r,+}^* \bar{W}}{\alpha'}{(L^2(\R^{d}))}+\sup_{r\in I_r} \normSch{W A_{r,-} A_{r,-}^* \bar{W}}{\alpha'}{(L^2(\R^{d}))} \right)
			 \\&\quad
			 \leq \abs{I_r} \left(\sup_{r\in I_r}\normSch{W A_{r,+} }{2\alpha'}{(L^2(\R^{d-1}),L^2(\R^d))}^2+\sup_{r\in I_r}\normSch{W A_{r,-} }{2\alpha'}{(L^2(\R^{d-1}),L^2(\R^d))}^2\right).
		 \end{align*}
		 Now, notice that
		 \begin{equation*}
		 	\normSch{W A_{r,\pm} }{2\alpha'}{}\leq C\normSch{W U(t,r)}{2\alpha'}{}.
		 \end{equation*}
		 Given \eqref{cond:sogge-non-degen-hess}, we can apply Theorem \ref{thm:SStrichartz-matdens} to 
		 $n=d-1$,
		 $(x_0',\xi_0')\in\R^{d-1}\times\R^{d-1}$, $a\in\CR^\infty_{x_1}(\R,S_{(x',\xi')}(1))$ and $J=I_r$. This defines $\delta>0$ and $U_1'\times V_1'$ a neighborhood of $(x_0',\xi_0')$(which corresponds to the neighborhood $U\times V$ in Theorem \ref{thm:SStrichartz-bounds}) .
		 Thus, imposing the following constraints on $\WR$:
		 \begin{itemize}
		 	\item $\abs{I}<\frac\delta 2$,
		 	\item $\VR'\subset U_1'\times V_1'$.
		 \end{itemize}
		 we obtain 
		\begin{equation*} 			
			\sup_{r\in I_r}\normSch{WU(t,r)}{d+1}{(L^2(\R^{d-1}),L^2(\R^d))} \lesssim h^{-\frac{d-1}{2(d+1)}}\normLp{W}{d+1}{(\R^d)}.
		\end{equation*}
	That ends the proof of Proposition \ref{prop:sogge-matdens-term_cruc}.
		
	% fin demo \ref{prop:sogge-matdens-term_cruc}
	
	%%%%%%%%%%%%%%%%%%%%%%%%%%%%%%%%%%%%%%%%%%%%%%%

	%\newpage
			
	\section{L$^p$ estimates around turning points}\label{sec:TP-est}
	
	We now treat the turning point region $\{V=E\}$, under the assumption $\nabla_xV\neq 0$ on this set. In the one-body case ($\rk\gamma=1$), we recover \cite[Thm. 7]{zworski2012semiclassical}.
	
	\subsection{Statement of the result}
	
	Let $d\geq 2$.
	
	\begin{assump}\label{cond:TP}
		A point $(x_0,\xi_0)\in\R^d\times\R^d$ satisfies the following \emph{turning point conditions} for a symbol $p$ if
		\begin{equation*}
			p(x_0 , \xi_0 ) = 0 , \: \nabla_\xi p(x_0 , \xi_0 ) = 0 , \: \nabla_x p(x_0 ,\xi_0 )  \neq 0 ,\: \partial_\xi^2 p(x_0 , \xi_0 ) \text{ is positive definite}.
		\end{equation*}
	\end{assump}
	% fin cond PT
	
	\begin{rmk}
		For Schr\"odinger operators $p(x,\xi)=\xi^2+V(x)-E$ with $V\in\CR^\infty(\R^d,\R)$ bounded from below satisfying Definition \ref{cond:am-potential-pol-growth}, the previous assumption is equivalent to:
		\begin{equation*}\label{cond:TP-Sch}
			\xi_0=0,\quad V(x_0)=E,\quad \nabla_xV(x_0)\neq 0 .
		\end{equation*}
	\end{rmk}
	% fin cond \label{cond:TP-Sch}
	
	First recall the individual function result.
		
	\begin{thm}[Improved one-body estimates, {\cite[Thm. 7]{koch2007semiclassical}}]\label{thm:ELp-TP-1body}
		Let $V\in\CR^\infty(\R^d,\R)$ bounded and below and satisfying Definition \ref{cond:am-potential-pol-growth}, define $ p(x,\xi) :=\abs{\xi}^2+V(x)$ and $P:=p^\w(x,hD)$ (or any other quantization).
		Let $(x_0,\xi_0)\in\R^d\times\R^d$ be a point satisfying the Assumption \ref{cond:TP} for the symbol $p$. Then, there exist a neighborhood $\VR$ of $(x_0,\xi_0)$ and $h_0>0$, such that for any $\chi\in\test{\R^d\times\R^d}$ with support contained in $\VR$, there  exists $C>0$ such that for any $0< h\leq h_0$, for any $2\leq q\leq\infty$ and for any bounded self-adjoint non-negative operator $\gamma$ on $L^2(\R^d)$
		\begin{equation*}
			\normLp{\chi^\w u}{q}{(\R^d)} \leq C \log(1/h)^{t(q,d)} h^{-s(q,d)}\left(\normLp{u}{2}{(\R^d)}+\frac 1 h\normLp{Pu}{2}{(\R^d)}\right)
		\end{equation*}
		where $t(q,d)$ and $s(q,d)$ are given by the following formulas
		\begin{equation}\label{eq-def:t-ELp-TP}
			t(q,d) =\begin{cases}
			\frac{d+1}{2(d+3)}  &\text{if}\ q=\frac{2(d+3)}{d+1},\\
			0 &\text{otherwise}
			,
		\end{cases}
		\end{equation}
		and
		\begin{itemize}
			\item when $d=2$:
			\begin{equation}\label{eq-def:s-d=2-ELp-TP}
				s(q,2)=\begin{cases}
				\frac 1 4 -\frac 1 {2q} &\text{if}\ 2\leq q \leq \frac{10} 3,\\
				\frac 12 -\frac 43 \frac 1q &\text{if}\ \frac{10} 3\leq q\leq \infty,
				\end{cases}
			\end{equation}
			\item when $d\geq 3$:
			\begin{equation}\label{eq-def:s-d>2-ELp-TP}
				s(q,d)=\begin{cases}
				\frac{d-1}2\left(\frac 1 2 -\frac 1 q\right) &\text{if}\ 2\leq q \leq\frac{2(d+3)}{d+1},\\
				\frac{ 2d} 3\left(\frac 1 2 -\frac 1 q\right) -\frac 1 6 &\text{if}\ \frac{2(d+3)}{d+1}\leq q\leq \frac{2d}{d-2},
				\\ d\left(\frac 1 2-\frac 1q\right)-\frac 12 &\text{if}\ \frac{2d}{d-2}\leq q\leq\infty
				.
				\end{cases}
			\end{equation}
		\end{itemize}
		Equivalently, one has for all $2\leq q\leq\infty$
		\begin{equation*}%\label{eq:ELp-TP-1body}
			\chi^\w{(1+P^*P/h^2)^{-1/2}} =\OR (\log(1/h)^{t(q,d)}h^{-s(q,d)}) : L^2(\R^d)\to L^q(\R^d).
		\end{equation*}
	\end{thm}
	% thm \ref{thm:ELp-TP-1body}
	
	\begin{rmk}\label{rmk:comp-s-TP}
		The exponent $s_{\text{TP}}$, defined in Theorem \ref{thm:ELp-TP-1body} satisfies $s_{\text{Sogge}}\leq s_{\text{TP}}\leq s_{\text{gene}}$ for any $d\geq 1$ and $q\in[2,\infty]$. They are all equal for $q=\{2\}\cup[2d/(d-2),\infty]$. Furthermore $s_{\text{TP}}=s_{\text{Sogge}}$ when $q\in[2,2(d+3)/(d+1)]$. Otherwise, the inequalities are strict. (c.f. Figure \ref{fig:comp-exp-s}).
	\end{rmk}
	
	\begin{rmk}\label{rmk:ELp-TP-improv}
		Note that the previous result has been proved in \cite[Thm. 7]{koch2007semiclassical} for slightly more general symbols
		\begin{equation*}
			p(x,\xi)= \sum_{i,j=1}^d a_{ij}(x)\xi_i\xi_j + V(x),
		\end{equation*}
		where $\{a_{ij}\}_{1\leq i,j\leq d}\subset\CR^\infty(\R^d,\R)$ is a positive definite Riemannian metric on $\R^d$ and $V\in\CR^\infty(\R^d,\R)$ sfollows Definition \ref{cond:bound-from-below} and Definition \ref{cond:am-potential-pol-growth}. Our results can also be generalized to this case.
	\end{rmk}
	
	\begin{thm}[Improved many-body estimates]\label{thm:ELp-TP-matdens}
		Let $V\in\CR^\infty(\R^d,\R)$ bounded from below and satisfying Definition \ref{cond:am-potential-pol-growth}, define $ p(x,\xi) :=\abs{\xi}^2+V(x)$ and $P:=p^\w(x,hD)$.
		Let $(x_0,\xi_0)\in\R^d\times\R^d$ be a point satisfying the Assumption \ref{cond:TP} for the symbol $p$.
		Then, there exist a neighborhood $\VR$ of $(x_0,\xi_0)$ and $h_0>0$, such that for any $\chi\in\test{\R^d\times\R^d}$ with support contained in $\VR$, there  exists $C>0$ such that for any $0< h\leq h_0$, for any $2\leq q\leq \infty$ and for any bounded self-adjoint non-negative operator $\gamma$ on $L^2(\R^d)$
		\begin{equation*}
			\normLp{\rho_{\chi^\w\gamma\chi^\w}}{q/2}{(\R^d)} \leq C \log(1/h)^{2t(q,d)} h^{-2s(q,d)}\normSch{(1+P^2/h^2)^{1/2}\gamma(1+P^2/h^2)^{1/2}}{\alpha(q,d)}{}
		\end{equation*}
		where $t(q,d)$, $s(q,d)$ are given by the formulas  \eqref{eq-def:t-ELp-TP}, \eqref{eq-def:s-d=2-ELp-TP}, \eqref{eq-def:s-d>2-ELp-TP},
		and $\alpha(q,d)$ is given by the formula of \eqref{eq-def:alpha-ELp-sogge-matdens}.
	\end{thm}
	% thm \ref{thm:ELp-TP-matdens}
	
%	\begin{rmk}\label{rmk:comp-alpha-TP}
%		For $d\geq3$, the exponent $\alpha_{\text{TP}}(q,d)=\alpha_{Sogge}(q,d)$ is better than $\alpha_{\text{gene}}(q,d)$ for $2d/(d-2)\leq q<\infty$.
%	\end{rmk}

	\begin{rmk}\label{rmk:ELp-TP-matdens_improv-s}
		The proof of \cite{koch2007semiclassical} gives the exponent $q=\tfrac{2d}{d-2}$ as a threshold for the exponent $s(q,d)$. Actually, their proof shows that this threshold can be improved to $q=\tfrac{2d}{d-4}$ (because they only use a control in $H^1$ while their proof also provides a control in $H^2$). We choose to keep this weaker statement because it only applies to functions which are microlocalized around a turning point $(x_0,\xi_0=0)$ ($p(x_0,\xi_0)=0$). In our application to spectral clusters we will also need to deal with points such that $p(x_0,\xi_0)=0$ and $\xi_0\neq 0$, where only the Sogge estimates are available.
	\end{rmk}

	% d=2
	\begin{multicols}{2}
	%%%%%%%%%%%%%%%%%%%%%% FIN DESSIN s(q,d) d=2 %%%%%%%%%%%%%%%%%%%%%%%%%%%%
	\begin{center}\begin{tikzpicture}[line cap=round,line join=round,>=triangle 45,x=4. cm,y=3.5 cm,scale=2.6]
		d=2
		\draw[->] (0.,0.) -- (0.55,0.);
		\draw (0.55,0.) node[right] {$\dfrac{1}{q}$};
		\draw[->] (0.,-0) -- (0.,1/2+0.05);
		\draw (0,1/2+0.05) node[above] {$\displaystyle{s(q,d)}$};
		%
		%% abscisses
		%
		\draw (0+0.03,0.) node[below]{$\frac{d-2}{2d} =0$};
		\draw (3/10,0.) node[below]{$\frac{d+1}{2(d+3)}$};
		\draw (0.5,0.) node[below]{$\frac{1}{2}$};
		%
		%% etiquettes
	\draw (0+0.03,0.5) node[right]{$(6.2)$}; % inf
	\draw (3/10,1/10) node[above, right]{$(6.3)$}; % ep mil
	\draw (0.5,0) node[above]{$(6.1)$}; % 2
		%
		%% ordonnées
		\draw (0.,0.) node[left]{$0$};
		\draw (0,0.5) node[left]{$\frac{d-1}{2}=\frac{1}{2}$};
		\draw (0.,1/10) node[left]{$\frac{d-1}{2(d+3)}$};
		%
		%% points
		\draw (1/2,0) node{$\bullet$};
		\draw (0,1/2) node{$\bullet$}; % ep KT
		\draw (3/10,1/10) node{$\bullet$}; % ep mil
		\draw[dotted] (0.,1/10)-|(3/10, 0.); % ep mil
		%
		%% courbes
		% $\R^d\setminus\Omega_\varepsilon$
		\draw[line width=.8pt ] (0.,1/2)--(3/10,1/10)(3/10,1/10)--(0.5,0.);
		%
		% legende
		\draw (1/8,1/2+.1)  node[right]{For $d=2$};
		\end{tikzpicture}\end{center}
	%%%%%%%%%%%%%%%%%%%%%% FIN DESSIN s(q,d) d=2%%%%%%%%%%%%%%%%%%%%%%%%%%%%
	
	%%%%%%%%%%%%%%%%%%%%%%  DESSIN alph(q,d) d=2 %%%%%%%%%%%%%%%%%%%%%%%%%%%%
	\begin{center}\begin{tikzpicture}[line cap=round,line join=round,>=triangle 45,x=6. cm,y=2.5 cm,scale=2]
		%d=4
		\draw[->] (0.,0.) -- (0.6,0.);
		\draw (0.6,0.) node[right] {$\displaystyle{\frac{1}{q}}$};
		\draw[->] (0.,0.) -- (0.,1.1);
		\draw (0.,1.1) node[above] {$\displaystyle{\frac 1{\alpha(q,d)}}$};
		%
		%% abscisses
		%
		\draw (0.,0) node[below]{0};
		\draw (0.5,0.) node[below]{$\frac{1}{2}$};
		\draw (1/6,0.) node[below]{$\frac{d-1}{2(d+1)}$}; % sogge
		%\draw (0,0.) node[below]{$\frac{d-2}{2d}$}; % kt
		%
		%% etiquettes
	\draw (0+0.03,0.1) node[left]{$(6.2)$};
	\draw (1/6,2/3) node[above]{$(6.4)$}; % sogge
	\draw (0.5,1) node[above]{$(6.1)$};
		%
		%% ordonnees
		%
		\draw (0.,0) node[left]{0};
		\draw (0.,1) node[left]{1};
		\draw (0.,2/3) node[left]{$\frac{d}{d+1}$}; % sogge
		%\draw (0.,0) node[left]{$\frac{d-2}{d-1}$}; % kt
		%
		%% points
		\draw[dotted] (0, 1)-|(0.5, 0); % pt sogge
		\draw (0., 0) node {$\bullet$};
		\draw (1/6, 2/3) node {$\bullet$}; % pt sogge
		\draw (0.5, 1) node {$\bullet$};
		\draw[dotted] (0, 2/3)-|(1/6,0); % sogge
		%
		%% courbes
		%
		% sogge
		\draw[line width=0.8pt] (0.,0)-- 
		(0, 0)--(1/6, 2/3)--(0.5,1);
		%
		%% legende
		%
		\draw (1/8,1.2)  node[right]{For $d=2$};
		\end{tikzpicture}\end{center}
	%%%%%%%%%%%%%%%%%%%%%% FIN DESSIN alpha(q,d) d=2 %%%%%%%%%%%%%%%%%%%%%%%%%%%%	
	
	\end{multicols}
	% d=2

	% d>2
	\begin{multicols}{2}
	%%%%%%%%%%%%%%%%%%%%%% DESSIN s(q,d) d>2%%%%%%%%%%%%%%%%%%%%%%%%%%%%
	\begin{center}\begin{tikzpicture}[line cap=round,line join=round,>=triangle 45,x=6.5 cm,y=2. cm,scale=2.]
		% d=3
		\draw[->] (0.,0.) -- (0.56,0.);
		\draw (0.56,0.) node[right] {$\displaystyle{\frac{1}{q}}$};
		\draw[->] (0.,0.) -- (0.,3/2+0.1);
		\draw (0.,3/2+0.1) node[above] {$\displaystyle{s(q,d)}$};
		%\draw[->] (0.,0.) -- (0.,0.5+0.1);
		%\draw (0.,0.5+0.1) node[above] {$\displaystyle{s(q,d)}$};
		%
		%% abscisses
		%
		\draw (1/4,0.) node[below]{$\frac{d-2}{2d}$};
		\draw (5/14,0.) node[below]{$\frac{d+1}{2(d+3)}$};
		\draw (0.5,0.) node[below]{$\frac{1}{2}$};
		%
		%% ordonnées
		\draw (0.,0.) node[left]{$0$};
		\draw (0,0.5) node[left]{$\frac{1}{2}$};
		\draw (0.,3/14) node[left]{$\frac{d-1}{2(d+3)}$};
		\draw (0.,3/2) node[left]{$\frac{d-1}{2}$};
		%
		%% points
		%
		\draw (0.,3/2) node{$\bullet$}; % L^infty point
		\draw (1/4, 0.5) node{$\bullet$};
		\draw[dotted] (0., 0.5)-|(1/4, 0.); % ep KT
		\draw (5/14, 3/14) node{$\bullet$};
		\draw[dotted] (0., 3/14)-|(5/14, 0.); % ep sogge
		\draw (1/2, 0) node{$\bullet$};
		%
	%% etiquettes
	\draw (0.,3/2) node[right]{$(6.2)$}; % L^infty point
	\draw (1/4, 0.5) node[above, right]{$(6.5)$};% ep KT
	\draw (5/14, 3/14) node[above,right]{$(6.3)$};
	\draw (1/2, 0) node[above]{$(6.1)$};	
		%
		%% courbes
		% spec
		\draw[line width=0.9pt] (0.,3/2) -- 
		(1/4,0.5)--(5/14,3/14)--(0.5,0.);
		%
		% legend
		\draw
		(1/4,3/2)node[right]{For $d\geq 3$};
		\end{tikzpicture}\end{center}
	%%%%%%%%%%%%%%%%%%%%%% FIN DESSIN s(q,d) d>2 %%%%%%%%%%%%%%%%%%%%%%%%%%%%
	
	%%%%%%%%%%%%%%%%%%%%%%  DESSIN alph(q,d) d>2 %%%%%%%%%%%%%%%%%%%%%%%%%%%%
	\begin{center}\begin{tikzpicture}[line cap=round,line join=round,>=triangle 45,x=6.cm,y=2.5 cm,scale=2.2]
		%d=4
		\draw[->] (0.,0.) -- (0.55,0.);
		\draw (0.55,0.) node[right] {$\displaystyle{\frac{1}{q}}$};
		\draw[->] (0.,0.) -- (0.,1.1);
		\draw (0.,1.1) node[above] {$\displaystyle{\frac 1{\alpha(q,d)}}$};
		%
		%% abscisses
		%
		\draw (0.,0) node[below]{0};
		\draw (0.3,0.) node[below]{$\frac{d-1}{2(d+1)}$};
		\draw (0.5,0.) node[below]{$\frac{1}{2}$};
		\draw (1/4,0.) node[below]{$\frac{d-2}{2d}$};
		%
		%% ordonnees
		%
		%\draw (0.,0) node[left]{0};
		\draw (0.,1) node[left]{1};
		\draw (0.,4/5) node[left]{$\frac{d}{d+1}$}; % ep sogge
		\draw (0.,2/3) node[left]{$\frac{d-2}{d-1}$}; % kt
		%
		%% points
		\draw[dotted] (0, 1)-|(0.5, 0); % pt sogge
		\draw (0., 0) node {$\bullet$};
		\draw (1/4, 2/3) node {$\bullet$}; % pt kt
		\draw (3/10, 4/5) node {$\bullet$}; % pt sogge
		\draw (0.5, 1) node {$\bullet$};
		\draw[dotted] (0, 2/3)-|(1/4,0); % kt
		\draw[dotted] (0, 4/5)-|(0.3,0); % sogge
		%
	%% etiquettes
	\draw (0., -0.) node[above,right]{$(6.2)$};
	\draw (1/4, 2/3) node[above,left] {$(6.5)$}; % pt kt
	\draw (3/10, 4/5) node[above] {$(6.4)$}; % pt sogge
	\draw (0.5, 1) node[above] {$(6.1)$};
		%
		%% courbes
		%
		% sogge
		\draw[line width=0.8pt] (0.,0)-- 
		(1/4, 2/3)--(3/10, 4/5)--(0.5,1);
		%
		%% legende
		%
		\draw (1/8,1.2)  node[right]{For $d\geq 3$};
		\end{tikzpicture}\end{center}
	%%%%%%%%%%%%%%%%%%%%%% FIN DESSIN alpha(q,d) d>2 %%%%%%%%%%%%%%%%%%%%%%%%%%%%	
	\end{multicols}
	% d>2

	We write the different end-points with what we obtain the estimates by interpolation.
	\begin{table}[h!]
		\centering
		\begin{tabular}{|c|c|c|c|c|c|}
			\hline
			End-point $q$  & 2 & $\infty$ & $\frac{2(d+3)}{d+1}$ &$\frac{2(d+1)}{(d-1)}$ &$ \frac{2d}{(d-2)}$ \\
			\hline
			Refering name  & (6.1) & (6.2) & (6.3) & (6.4) & (6.5) \\
			\hline
		\end{tabular}
		\caption{References of the end-points' labels in the turning points case.}
		\label{table:endpoint-TP}
\end{table}

	\subsection{Proof of Theorem \ref{thm:ELp-TP-matdens}}
	
	As argued in \cite{koch2007semiclassical}, we may reduce the problem to the case
	\begin{equation*}
		p(x,\xi)= \xi_1^2 +\sum_{i,j=2}^d a_{ij}(x)\xi_i\xi_j +V(x),\quad V(x)=-c(x)x_1,
	\end{equation*}
	where $(a_{ij}(x))_{i,j}\subset\CR^\infty(\R^d)$ is positive definite uniformly, $c\in\CR^\infty(\R^d)$ with $c(0)>0$, and
	\begin{equation*}
		\forall\alpha\in\N^d,\:\exists C_\alpha>0,\: \forall i,j=\{2,\ldots,d\},\:\forall x\in\R^d,\quad \abs{\partial^\alpha a_{ij}(x)} + \abs{\partial^\alpha V(x)}\leq C_\alpha.
	\end{equation*}

	\paragraph{Notation}\label{subsec:notation-ELp-TP-matdens}
	
	\begin{itemize}
		\item Let $\delta>0$ such that $\inf_{x\in B_\delta}c(x)>0$, where $B_\delta:=\{ x\in\R^d \: :\: \abs{x}<\delta\}$.
		\item
		Let $\VR\subset\R^d\times\R^d$ a bounded open neighborhood of $(x_0,\xi_0)=(0,0)$ such that $\pi_x\VR$ is contained in $B_{\delta/4}$ and such that $\VR\subset\VR_0$ where $\VR_0$ is given by Corollary \ref{lemma:TP-matdens-3}.
		\item
		Let $\chi\in\test{\R^d\times\R^d}$ such that $\supp\chi\subset\VR$.
		\item
		Let $M\geq 1$ be larger than $M_0\geq 1$ given by Corollary \ref{lemma:TP-matdens-3}. An other constraint will be given in the proof.
		\item For all $\varepsilon>0$, let us define
		\begin{equation*}
		\Omega_\varepsilon := \{x\in\R^d \: : \: x_1<\varepsilon \}.
		\end{equation*}
		\item
		% def chi_epsi
		Let us define $\chi_\varepsilon:=\chi_0(\cdot/\varepsilon)$ where $\chi_0\in\CR^\infty(\R,[0,1])$ in a nonnegative function equal to 1 on $]-\infty,1]$ and equal to 0 on $[2,\infty[$.
		% fin del chi_epsi	
		%
		\item Let $s_{\text{Sogge}}$ and $\alpha_{\text{Sogge}}$ be given by the formulas in the statement of Theorem \ref{thm:ELp-sogge-matdens}. %\eqref{eq-def:s-ELp-sogge} and \eqref{eq-def:alpha-ELp-sogge-matdens}.
	\end{itemize}
	As in the proof of Theorem \ref{thm:ELp-gene-matdens}, we have by the one-body estimates (Theorem \ref{thm:ELp-TP-1body})
	\begin{equation*}
	\normLp{\rho_{\chi^\w\gamma\chi^\w}}{\infty}{(\R^d)}\lesssim h^{1-d}\norm{(1+P^2/h^2)^{1/2}\gamma(1+P^2/h^2)^{1/2}}_{L^2(\R^d)\to L^2(\R^d)}.
	\end{equation*}
	We prove the estimates of Theorem \ref{thm:ELp-TP-matdens} for low regime $2\leq q\leq 2d/(d-2)$. The remaining estimates for $2d/(d-2)<q<\infty$ are then obtained by interpolating between $q=2d/(d-2)$ and $q=\infty$. We now fix $2\leq q\leq 2d/(d-2)$.	The strategy to estimate $\normLp{\rho_{\chi^\w\gamma\chi^\w}}{q/2}{(\R^d)}$ is to estimate $\normLp{\rho_{\chi^\w\gamma\chi^\w}}{q/2}{(\Omega)}$ on various regions $\Omega$ that cover $\R^d$ (c.f. Figure \ref{fig:TP-regions-work}), and then sum the obtained estimates.

	Before going into the proof, let us recall some key estimates of Koch-Tataru-Zworski \cite{koch2007semiclassical}.

	\begin{lemma}[{\cite[Lem. 7.3 and Sec. 7]{koch2007semiclassical}}]
		\label{lemma:TP-3}
		Let $d\geq 1$.
		Then, there exist $M\geq 1$, $h_0>0$ and a bounded neighborhood $\VR\subset\R^d\times\R^d$ of $0$ such that, for any $\chi\in\test{\R^d\times\R^d}$ supported in $\VR$, any $h\in(0,h_0]$, and any $\varepsilon\geq Mh^{2/3}$ we have for all $\alpha\in\N^d $ such that $\abs{\alpha}\leq 2$
		\begin{equation}\label{eq:lemma-ELp-PT-3}
			\normLp{(hD)^\alpha \chi^\w u}{2}{(\Omega_\varepsilon)} = \OR\left(\varepsilon^{\frac{1}{4}+\frac{\abs{\alpha}}{2}}\right)
			\left(\normLp{u}{2}{(\R^d)}+\frac 1 h\normLp{Pu}{2}{(\R^d)}\right)
		.
		\end{equation}
	\end{lemma}
	% fin \ref{lemma:TP-3}
	
	\begin{rmk}\label{rmk:TP-3_ktz}
		In \cite{koch2007semiclassical}, the estimate \eqref{eq:lemma-ELp-PT-3} is proved only for $\abs{\alpha}\leq 1$ and for $\abs{\alpha}= 2$ when $d=2$ and $\varepsilon=Mh^{2/3}$. Their method allows to treat the case $\abs{\alpha}=2$ without the restrictions $d=2$ and $\varepsilon=Mh^{2/3}$.
	\end{rmk}

	\begin{rmk}\label{rmk:TP-3_opt}
		In the case $\varepsilon=Mh^{2/3}$, the estimate \eqref{eq:lemma-ELp-PT-3} reduces to 
		\begin{equation*}
			\sum_{\abs{\alpha}\leq 2}\normLp{(h^{2/3}D)^\alpha \chi^\w u}{2}{(\Omega_{Mh^{2/3}})} \leq Ch^{1/6}
			\left(\normLp{u}{2}{(\R^d)}+\frac 1 h\normLp{Pu}{2}{(\R^d)}\right)
			,
		\end{equation*}
		which, by Sobolev embeddings, imply that for all $2\leq q\leq \frac{2d}{(d-4)_+}$ (excluding $q=\infty$ for $d=4$)
		\begin{equation*}
			\normLp{\chi^\w u}{q}{(\Omega_{Mh^{2/3}})}\leq  Ch^{\frac 16-\frac {2d}3\left(\frac 12-\frac 1q\right)}
			\left(\normLp{u}{2}{(\R^d)}+\frac 1 h\normLp{Pu}{2}{(\R^d)}\right).
		\end{equation*}
		In dimension $d=1$, the above estimates match the $L^q$ norms of the normalized Hermite functions.  One can even get rid of the microlocalization $\chi^\w$ in the argument of \cite{koch2007semiclassical}. Indeed, the two normalized functions $u_h$ such that $-h^2u_h''+x^2u_h =u_h$ satisfy
		\begin{equation*}
			\normLp{u_h}{q}{(\R)}\lesssim \begin{cases}
				1 &\text{if}\ 2\leq q \leq 4,\\
				h^{-\frac 16+\frac{2}{3q}} &\text{if}\ 4\leq q \leq \infty.
			\end{cases}
		\end{equation*}
		These bounds are equivalent to \cite[Cor. 3.2]{koch2007semiclassical} for $(\phi,\lambda^2)$ the eigenfunction-eigenvalue pair of
		the Hermite operator $-\frac{d^2}{dx^2}+x^2$, by the scaling $u_h=h^{-1/4}\phi(h^{-1/2}\cdot)$ and $h=\lambda^{-2}$.
		For $\alpha=0$, this estimate is sharp for this potential $V(x)=x^2-1$ and $x_0=1$ (and $x_0=-1$), because Hermite functions behave like $h^{-1/6} {\rm Ai}(h^{-2/3}(x-1))$ close to $x_0= 1$ (c.f. also Figure \ref{fig:oh-hr-eigenfunction} for $E=1$).
		In particular, the contribution of the $L^q$ norm around in a a turning point's neighborhood of size $h^{2/3}$ is $h^{-\frac 16+\frac{2}{3q}}$, that saturates the bound for $q\geq 4$.
		More precisely, according to \cite[Chap. 6 and 11]{olver1997asymptotics}, there exist a normalization constant $C>0$ such that the normalized eigenfunction $u_h$ of $-h^2\frac{d^2}{dx^2}+x^2-1$ satisfy, for any $x\in\R$
		\begin{equation*}
		\begin{split}
		u_h(x)=
			C h^{-1/6}\frac{\zeta(x)^{1/4}}{(x^2-1)^{1/4}} {\rm Ai}&\left(h^{-2/3}\zeta(x)\right) 
			%\times\\&\quad\times
			\left(1+ r_h\left(h^{-2/3}\zeta(x)\right)\right)
			,
		\end{split}
		\end{equation*}
		with 
		\begin{equation*}\forall x\in\R^{\pm},\quad
		\zeta(x):=\pm\sgn\big(x\mp 1) \abs{ \tfrac 32 \int_{\pm 1}^x\sqrt{t^2-1}dt }^{2/3}
		,
		\end{equation*}
		and $r_h$ a smooth function such that for any $h>0$
		\begin{equation*}\forall x\in\R,\quad
		\abs{r_h(x)}\leq h(1+\abs{x}^{1/4})^{-1}\exp\left(-\tfrac 23(x_+)^{3/2}\right)
		.
		\end{equation*}
		We obtain the satured $L^q$ bounds using the asymptotics of the Airy function
		\begin{equation*}
		\begin{split}
			{\rm Ai}(x) & \underset{x>0}{=}
			\displaystyle{ \frac{\exp\left(-\frac{2}{3}x^{3/2}\right)}{2\pi^{1/2}x^{1/4}} \left[1+\OR\left(\frac{1}{x^{3/2}}\right)\right] }, \quad \\
			{\rm Ai}(x)&  \underset{x<0}{=}
			\displaystyle{  \frac{\cos\left(\frac{2}{3}(-x)^{3/2}-\frac{\pi}{4}\right)}{\pi^{1/2}(-x)^{1/4}}
				\left[1+\OR\left(\frac{1}{(-x)^{3/2}}\right)\right]  }.
		\end{split}
		\end{equation*}
		 and suitable changes of variables in the three types of regions. Note that \cite{larsson-cohn2002} provides also equivalents of $L^q$ norms of Hermite functions with respect to a gaussian weight.
		The sharpness in higher dimension seems open to us.
	\end{rmk}

	\begin{cor}\label{lemma:TP-matdens-3}
		Let $d\geq 1$.
		Then, there exist $M_0\geq 1$, $h_0>0$ and a bounded neighborhood $\VR_0\subset\R^d\times\R^d$ of 0, such that for any $\chi\in\test{\R^d\times\R^d}$ supported in $\VR_0$, for any $h\in(0,h_0]$ and any $\varepsilon\in [M_0h^{2/3},1]$
		\begin{equation*}
			(1-h^2\Delta)\chi_\varepsilon\chi^\w\left(1+P^2/h^2\right)^{-1/2}
			=\OR_{L^2(\R^d)\to L^2(\R^d)}\left(\varepsilon^{1/4}\right).
		\end{equation*}
	\end{cor}
	% fin  \ref{lemma:TP-matdens-3}

	%% c.f. PREUVE THESE
	
	%%%%%%%%%%%%%%%%%%%%%% DESSIN s(q,d) d=2 %%%%%%%%%%%%%%%%%%%%%%%%%%%%
	\begin{center}\begin{tikzpicture}[line cap=round,line join=round,>=triangle 45,x=7.7 cm,y=2.1 cm,scale=2.8]
		d=2
		\draw[->] (0.,0.) -- (0.52,0.);
		\draw (0.52,0.) node[right] {$\dfrac{1}{q}$};
		\draw[->] (0.,-1/6-0.1) -- (0.,1/2+0.1);
		\draw (0,1/2+0.1) node[above] {$\displaystyle{s(q,2)}$};
		%
		%% abscisses
		%
		\draw (3/10,0.) node[below]{$\frac 3{10}$};
		\draw (0.5,0.) node[below]{$\frac{1}{2}$};
		%
		%% etiquettes
	\draw (3/10,-1/6) node[below]{$(6.3)$};
	\draw (0,-1/6) node[below,right]{$(6.2)=(6.5)$};
	\draw (0.5,-1/6) node[below]{$(6.1)$};
		%
		%% ordonnées
		\draw (0.,0.) node[left]{$0$};
		\draw (0,0.5) node[left]{$\frac{1}{2}$};
		\draw (0.,1/10) node[left]{$\frac1{10}$};
		\draw (0,-1/6) node[left]{$-\frac 1 6$};
		%
		%% points
		\draw[color=stefou] (0,1/2) node{$\bullet$}; % ep KT
		\draw[dotted] (0.,1/10)-|(3/10, 0.); % ep sogge
		\draw (3/10,1/10) node{$\bullet$};
		\draw[dotted] (0., -1/6)-|(1/2, 0.); % -1/6
		\draw[color=stefou] (1/2,-1/6) node{$\bullet$};
		\draw (1/2,0) node{$\bullet$}; % ep q=2
		%
		%% courbes
		% $\R^d\setminus\Omega_\varepsilon$
		\draw[dashed, line width=1.2pt ] (0.,1/2)--(3/10,1/10)(3/10,1/10)--(0.5,0.);
		% $\Omega_\varepsilon$
		\draw[line width=0.9pt, color=stefou] (3/10,1/10)--(0.5,-1/6);
		%
		% legende
		\draw (1/8,1/2+.1)  node[right]{For $d=2$};
		\draw[line width=0.9pt, white] (1/8,1/2)  node[right]{\textcolor{black}{sum up of the estimates}};
		% $\R^d\setminus\Omega_\varepsilon$
		\draw[line width=0.9pt, dashed] (1/4,1/2-0.2)-|(1/4+.02,1/2-0.2)  node[right]{on $\R^d\setminus\Omega_{Mh^{2/3}}$};
		% $\Omega_\varepsilon$
		\draw[line width=0.9pt, color=stefou] (1/4,1/2-0.1)-|(1/4+.02,1/2-0.1)  node[right]{\textcolor{black}{on $\Omega_{Mh^{2/3}}$}};
		\end{tikzpicture}\end{center}
	%%%%%%%%%%%%%%%%%%%%%% FIN DESSIN s(q,d) d=2%%%%%%%%%%%%%%%%%%%%%%%%%%%%
	
	%%%%%%%%%%%%%%%%%%%%%%  DESSIN s(q,d) d>2 %%%%%%%%%%%%%%%%%%%%%%%%%%%%
	\begin{center}\begin{tikzpicture}[line cap=round,line join=round,>=triangle 45,x=7.7 cm,y=2.1 cm,scale=2.8]
		d=3
		\draw[->] (0.,0.) -- (0.52,0.);
		\draw (0.52,0.) node[right] {$\frac{1}{q}$};
		\draw[->] (0.,-1/6-0.1) -- (0.,1/2+0.1);
		\draw (0,1/2+0.1) node[above] {$\displaystyle{s(q,d)}$};
		%
		%% abscisses
		%
		\draw (1/4,0.) node[below]{$\frac{d-2}{2d}$};
		\draw (5/14,0.) node[below]{$\frac{d+1}{2(d+3)}$};
		\draw (0.5,0.) node[below]{$\frac{1}{2}$};
		%
		%% etiquettes
	\draw (1/4,-1/6) node[below]{$(6.5)$};
	\draw (5/14,-1/6) node[below]{$(6.3)$};
	\draw (0.5,-1/6) node[below]{$(6.1)$};
		%
		%% ordonnées
		\draw (0.,0.) node[left]{$0$};
		\draw (0,0.5) node[left]{$\frac{1}{2}$};
		\draw (0.,3/14) node[left]{$\frac{d-1}{2(d+3)}$};
		\draw (0,-1/6) node[left]{$-\frac 1 6$};
		%
		%% points
		%
		\draw[dotted] (0., 0.5)-|(1/4, 0.); % ep KT
		\draw (1/4,1/2) node{$\bullet$};
		\draw[dotted] (0., 3/14)-|(5/14, 0.); % ep ktz
		\draw (5/14,3/14) node{$\bullet$};
		\draw[dotted] (0., -1/6)-|(1/2, 0.); % -1/6
		\draw[color=stefou] (1/2,-1/6) node{$\bullet$};
		\draw (1/2,0) node{$\bullet$}; % 2
		%
		%% courbes
		%
		% $\R^d\setminus\Omega_\varepsilon$
		\draw[dashed, line width=1.2pt ] (1/4,0.5)--(5/14,3/14)--(0.5,0.);
		% $\Omega_\varepsilon$
		\draw[line width=0.9pt, color=stefou](1/4,0.5)--(5/14,3/14)--(0.5,-1/6);
		%
		%% legende
		%
		\draw (1/8,2/3+.1)  node[right]{For $d\geq 3$};
		\draw[line width=0.9pt, white] (1/8,2/3)  node[right]{\textcolor{black}{sum up of the estimates}};
		% $\R^d\setminus\Omega_\varepsilon$
		\draw[line width=1.pt, dashed] (1/4+0.02,2/3-0.2)-|(1/4+.04,2/3-0.2)  node[right]{on $\R^d\setminus\Omega_{Mh^{2/3}}$};
		% $\Omega_\varepsilon$
		\draw[line width=1.pt, color=stefou] (1/4+0.02,2/3-0.1)-|(1/4+.04,2/3-0.1)  node[right]{\textcolor{black}{on $\Omega_{Mh^{2/3}}$}};
		\end{tikzpicture}\end{center}
	%%%%%%%%%%%%%%%%%%%%%% FIN DESSIN s(q,d) d>2%%%%%%%%%%%%%%%%%%%%%%%%%%%%

%\newpage
	% alpha
	\begin{multicols}{2}
		
	%%%%%%%%%%%%%%%%%%%%%%  DESSIN alph(q,d) d=2 %%%%%%%%%%%%%%%%%%%%%%%%%%%%
	\begin{center}\begin{tikzpicture}[line cap=round,line join=round,>=triangle 45,x=6. cm,y=2.5 cm,scale=2]
		%d=4
		\draw[->] (0.,0.) -- (0.52,0.);
		\draw (0.52,0.) node[right] {$\displaystyle{\frac{1}{q}}$};
		\draw[->] (0.,0.) -- (0.,1.1);
		\draw (0.,1.1) node[above] {$\displaystyle{\frac 1{\alpha(q,d)}}$};
		%
		%% abscisses
		%
		%\draw (0.,0) node[below]{0};
		\draw (0.5,0.) node[below]{$\frac{1}{2}$};
		\draw (1/6,0.) node[below]{$\frac{d-1}{2(d+1)}=\frac 16$}; % sogge
		\draw (0,0.) node[below]{$0=\frac{d-2}{2d}$}; % kt
		%
		%% ordonnees
		%
		\draw (0.,0) node[left]{0};
		\draw (0.,1) node[left]{1};
		\draw (0.,2/3) node[left]{$\frac{d}{d+1}=\frac 23$}; % sogge
		%
		%% points
		\draw[dotted] (0, 1)-|(0.5, 0); % pt sogge
		\draw(0., 0) node {$\bullet$};
		\draw (1/6, 2/3) node {$\bullet$}; % pt sogge
		\draw (0.5, 1) node {$\bullet$};
		\draw[dotted] (0, 2/3)-|(1/6,0); % sogge
		%
	%% etiquettes
	\draw(0.,-0.1) node[below] {$(6.2)=(6.5)$};
	\draw (1/6, -0.2) node {$(6.4)$}; % pt sogge
	\draw(1/2,-0.1) node[below] {$(6.1)$};
		%
		%% courbes
		% sogge $\R^d\setminus\Omega_\varepsilon$
		\draw[line width=1pt, dashed] (0.,0)-- 
		(0, 0)--(1/6, 2/3)--(0.5,1);
		%
		% ellip $\Omega_\varepsilon$
		\draw[line width=1pt,color=stefou]
		(0, 0)--(0.5,1);
		%
		%% legende
		%
		\draw (1/8,1.2)  node[right]{For $d=2$};
		 node[right]{\textcolor{black}{sum up of the estimates}};
		% $\R^d\setminus\Omega_\varepsilon$
		\draw[line width=1.pt, dashed] (1/4+0.02,-0.1-0.2)-|(1/4+.04,-0.1-0.2)  node[right]{on $\R^d\setminus\Omega_{Mh^{2/3}}$};
		% $\Omega_\varepsilon$
		\draw[line width=1.pt, color=stefou] (1/4+0.02,-0.1-0.1)-|(1/4+0.04,-0.1-0.1)  node[right]{\textcolor{black}{on $\Omega_{Mh^{2/3}}$}};
		\end{tikzpicture}\end{center}
	%%%%%%%%%%%%%%%%%%%%%% FIN  alpha(q,d) d=2 %%%%%%%%%%%%%%%%%%%%%%%%%%%%	

	%%%%%%%%%%%%%%%%%%%%%%  DESSIN alph(q,d) d>2 %%%%%%%%%%%%%%%%%%%%%%%%%%%%
	\begin{center}\begin{tikzpicture}[line cap=round,line join=round,>=triangle 45,x=6.cm,y=3.5 cm,scale=2.2]
		%d=4
		\draw[->] (0.,0.) -- (0.55,0.);
		\draw (0.55,0.) node[right] {$\displaystyle{\frac{1}{q}}$};
		\draw[->] (0.,0.) -- (0.,1.1);
		\draw (0.,1.1) node[above] {$\displaystyle{\frac 1{\alpha(q,d)}}$};
		%
		%% abscisses
		%
		\draw (0.,0) node[below]{0};
		\draw (0.3,0.) node[below]{$\frac{d-1}{2(d+1)}$};
		\draw (0.5,0.) node[below]{$\frac{1}{2}$};
		\draw (1/4,0.) node[below]{$\frac{d-2}{2d}$};
		%
	% etiquettes
%	\draw (0.,-0.1) node[below,right]{$(6.2)$};
	\draw (0.3,-0.1) node[below,right]{$(6.4)$};
	\draw (0.5,-0.1) node[below]{$(6.1)$};
	\draw (1/4,-0.1) node[below,left]{$(6.5)$};
		%
		%% ordonnees
		%
		\draw (0.,0) node[left]{0};
		\draw (0.,1) node[left]{1};
		\draw (0.,4/5) node[left]{$\frac{d}{d+1}$}; % ep sogge
		\draw (0.,2/3) node[left]{$\frac{d-2}{d-1}$}; % ep kt, sogge
		\draw (0.,3/5) node[left]{$\frac{d-2}{d}$};% ep kt, ellip
		%
		%% points
		%
		\draw (0., 0) node {$\bullet$};
		\draw (0.5, 1) node {$\bullet$};
		\draw[dotted] (0, 1)-|(0.5, 0); % pt sogge
		\draw (1/4, 2/3) node {$\bullet$}; % pt kt, ellip
		\draw[dotted] (0, 2/3)-|(1/4,0); % kt, sogge
		\draw[dotted] (0, 3/5)-|(1/4, 0); 
		\draw (3/10, 4/5) node {$\bullet$}; % ep sogge, ellip
		\draw[dotted] (0, 4/5)-|(0.3,0);
		%
		%
		%% courbes
		%
		% sogge  $\R^d\setminus\Omega_\varepsilon$
		\draw[line width=0.8pt,dashed] 
		(1/4, 2/3)--(3/10, 4/5)--(0.5,1);
		%
		% ellip $\Omega_\varepsilon$
		\draw[line width=1pt,color=stefou] (1/4,3/5)--(0.5,1);
		\draw[color=stefou] (1/4,3/5) node{$\bullet$};
		%% legende
		%
		\draw (1/8,1.2)  node[right]{For $d\geq 3$};
		\end{tikzpicture}\end{center}
	%%%%%%%%%%%%%%%%%%%%%% FIN DESSIN alpha(q,d) d>2 %%%%%%%%%%%%%%%%%%%%%%%%%%%%	
	
	\end{multicols}
	% dessin alpha TP

	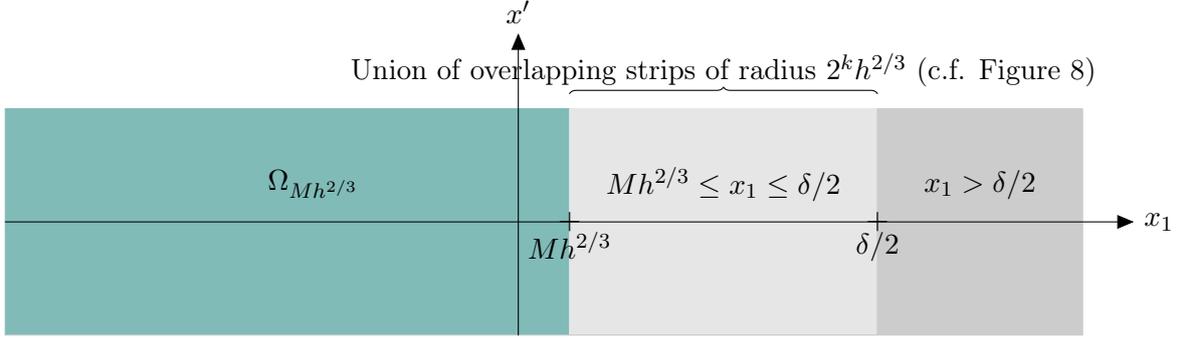
\begin{figure}[h!]
		%%%%%%%%%%%%%%%%%%%%%% DESSIN regions %%%%%%%%%%%%%%%%%%%%%%%%%%%%
		\begin{center}\begin{tikzpicture}[line cap=round,line join=round,>=triangle 45,x=1.35cm,y=1.0cm,scale=1]
			\usetikzlibrary{patterns}
			% x1<Mh^{2/3}, Omega_{Mh^{2/3}}
			\draw[color = stefou!50,fill=stefou!50] (-5,-1.5) -- (-5,1.5) -- (0.5,1.5) -- (0.5,-1.5) -- (-5,-1.5);
			\draw (-2,0.5) node{$\Omega_{Mh^{2/3}}$};
			% Mh^{2/3}<x1<delta/2
			\draw[color = gray!20,fill=gray!20](0.5,-1.5) -- (0.5,1.5) -- (3.5,1.5) -- (3.5,-1.5) -- (-5,-1.5);
			\draw (2,0.5) node{$Mh^{2/3}\leq x_1\leq \delta/2$};
			% x1>delta/2
			\draw[color = gray!40,fill=gray!40](3.5,-1.5) -- (3.5,1.5) -- (5.5,1.5) -- (5.5,-1.5) -- (-5,-1.5);
			\draw (4.5,0.5) node{$x_1> \delta/2$};
			\draw[decorate,decoration={brace,raise=0.2cm}]
			(0.5,1.5) -- (3.5,1.5) node[above=4.5,pos=0.5] {Union of overlapping strips of radius $2^{k}h^{2/3}$ (c.f. Figure \ref{fig:TP-strips})};
			%% points
			\draw (0.5,0) node{$+$};
			\draw (0.5,0) node[below]{$Mh^{2/3}$};
			\draw (3.5,0) node{$+$};
			\draw (3.5,0) node[below]{$\delta/2$};
			%% axes
			\draw[->] (-5.,0.) -- (6,0.);
			\draw (6,0.) node[right]{$x_1$};
			\draw[->] (0,-1.5)--(0,2.5);
			\draw (0.,2.5) node[above]{$x'$};
			\end{tikzpicture}\end{center}
		%%%%%%%%%%%%%%%%%%%%%% fin DESSIN regions %%%%%%%%%%%%%%%%%%%%%%%%%%%%
		%
		\caption{Different regions of $\R^d$.}
		\label{fig:TP-regions-work}
	\end{figure}

%	\subsubsection{Estimates on $\Omega_{-\delta/2}$}
%% included below

	\subsubsection{Estimates on $\Omega_{Mh^{2/3}}$}

	\begin{lemma}\label{lemma:TP-matdens-delta2<x1<Mh23}
		Let $d\geq 1$. For $2\leq q\leq  \frac{2d}{(d-2)_+}$ and any bounded self-adjoint operator $\gamma$ on $L^2(\R^d)$
		\begin{equation*}
			\normLp{\rho_{\chi^\w\gamma\chi^\w}}{q/2}{\left(\Omega_{Mh^{2/3}}\right)}\leq C h^{\frac 1 3-\frac{4d}{3}\left(\frac 1 2 -\frac 1 q\right)}\normSch{\left(1+\frac 1{h^2}P^2\right)^{1/2}\gamma\left(1+\frac 1{h^2}P^2\right)^{1/2}}{q/2}{\left(L^2(\R^d)\right)} 
			.
		\end{equation*}
	\end{lemma}

	\begin{rmk}
		The previous estimates are also true for all $2\leq q< \frac{2d}{(d-4)_+}$, but we will not need it (see Remark \ref{rmk:ELp-TP-matdens_improv-s}).
	\end{rmk}
	
	\begin{proof}[\underline{Proof of Lemma \ref{lemma:TP-matdens-delta2<x1<Mh23}}]   	
    By Corollary \ref{lemma:TP-matdens-3} with $\varepsilon=Mh^{2/3}$
    \begin{equation*}\label{eq:TP-matdens-bounded-alpha=2}
    		(1-h^{4/3}\Delta) \chi_{Mh^{2/3}}\chi^\w (1+P^2/h^2)^{-1/2} = \OR_{L^2(\R^d)\to L^2(\R^d)}(h^{1/6}),
    \end{equation*}
	On the one hand, by the Mercer theorem 
	\begin{align*}
		&\normLp{\rho_{\chi^\w\gamma\chi^\w}}{1}{\left(\Omega_{Mh^{2/3}}\right)}
		= \tr_{L^2}\left(\chi_{Mh^{2/3}}\chi^\w\gamma\chi^\w\chi_{Mh^{2/3}}\right)
		\\&\quad
		\leq \norm{(1-h^{4/3}\Delta)^{-1}}^2_{L^2\to L^2} \norm{(1-h^{4/3}\Delta) \chi_{Mh^{2/3}}\chi^\w (1+P^2/h^2)^{-1/2} }^2_{L^2\to L^2} 
		\times\\&\qquad\times
		\normSch{(1+P^2/h^2)^{1/2}\gamma(1+P^2/h^2)^{1/2}}{1}{}^{1/2}
		\\&\quad
		\lesssim h^{1/6}\normSch{(1+P^2/h^2)^{1/2}\gamma(1+P^2/h^2)^{1/2}}{1}{}^{1/2}
		.
	\end{align*}
	Let us consider now $2< q \leq \frac{2d}{(d-2)_+}$.
     By the Kato-Seiler-Simon bound applied to $m=2$, which is true if and only if $2<q<\frac{2d}{d-4}$ (see Lemma \ref{lemma:Kato-Seiler-Simon_dual})
	\begin{equation*}\label{eq:SA-matdens-bounded-ESob}
		\normSch{W (1-h^{4/3}\Delta)^{-1}}{2(q/2)'}{ (L^2(\R^d)) } \leq C h^{-\frac{2d}3\left(\frac 1 2 -\frac 1 q\right)}\normLp{W}{2(q/2)'}{ (L^2(\R^d)) } 
		,
	\end{equation*}
	we have for $2< q \leq \frac{2d}{(d-2)_+}$ and $W\in L^{2(q/2)'}(\R^d)$ 
	\begin{align*}
		&\normSch{W\chi_{Mh^{2/3}}\chi^\w\sqrt{\gamma}}{2}{}
		\\&\quad
		\leq \normSch{W(1-h^{4/3}\Delta)^{-1}}{2(q/2)'}{} \norm{(1-h^{4/3}\Delta)\chi_{Mh^{2/3}}\chi^\w(1+P^2/h^2)^{-1/2}}_{L^2\to L^2}
		\times\\&\qquad\times
		\normSch{(1+P^2/h^2)^{1/2}\sqrt{\gamma}}{q}{}
		\\&\quad
		\lesssim h^{\frac 16 -\frac{2d}3\left(\frac 12 -\frac 1q\right)} \normLp{W}{2(q/2)'}{(\R^d)} \normSch{(1+P^2/h^2)^{1/2}\gamma(1+P^2/h^2)^{1/2}}{q/2}{}^{1/2},
	\end{align*}
	which by duality (Remark \ref{rmk:mercer-thm}) ends the proof. 
	\end{proof}
	% end proof \ref{lemma:TP-matdens-delta2<x1<Mh23}

	\subsubsection{Estimates on $\{ x\in\R^d \: : \: x_1> \delta/2 \}$}
	
	Since $\pi_x\supp\chi$ and  $\{x\in\R^d \: :\: x_1>\delta/2\}$ are disjoint, we deduce as in the beginning of the proof of Theorem \ref{thm:ELp-sogge-matdens} that for any $2\leq q\leq\infty$ and any bounded self-adjoint operator $\gamma$ on $L^2(\R^d)$
	\begin{equation*}
	\normLp{\rho_{\chi^\w\gamma\chi^\w}}{q/2}{\left(\{x\in\R^d \: :\: x_1> \delta/2\}\right)}
	=\OR( h^\infty) \normSch{(1+P^2/h^2)^{1/2}\gamma(1+P^2/h^2)^{1/2}}{q/2}{}
	.
	\end{equation*}

	\subsubsection{Estimates on $\{x\in\R^d \: :\: Mh^{2/3} \leq x_1\leq \delta/2\}$}
	
	In the following, we assume $h_0<(\delta/(2M))^{3/2}$ to ensure $Mh^{2/3}<\delta/2$ for $M>0$ be defined later.
	
	\begin{lemma}\label{lemma:TP-matdens-crucial-zone}
		Let $d\geq 2$ and $2\leq q\leq\frac{2d}{d-2}$. Then, there exists $M>0$ such that for all bounded self-adjoint non-negative operator $\gamma$ on $L^2(\R^d)$ we have
		\begin{equation*}
			\normLp{\rho_{\chi^\w\gamma\chi^\w}}{q/2}{\left(\Omega_{\delta/2}\setminus\Omega_{Mh^{2/3}}\right)} \lesssim C_h^2 \normSch{\left(1+P^2/h^2\right)^{1/2}\gamma\left(1+P^2/h\right)^{1/2}}{\alpha_{\text{Sogge}}(q,d)}{(L^2(\R^d))}
			,
		\end{equation*}
		where 
		\begin{equation*}
			C_h
			:=\begin{cases}
				h^{-\frac{d-1}{2}\left(\frac 1 2-\frac 1 q\right)}&\text{if}\ 2\leq q<\frac{2(d+3)}{d+1},
				\\ \log^{\frac{d+1}{2(d+3)}}(1/h)h^{-\frac{d-1}{2(d+3)}}  &\text{if}\ q=\frac{2(d+3)}{d+1},
				\\h^{\frac 1 6 -\frac{2d}{3}\left(\frac 1 2 -\frac 1 q\right)} &\text{if}\ \frac{2(d+3)}{d+1}<q\leq\frac{2d}{d-2}
				.
			\end{cases}
		\end{equation*}
	\end{lemma}

	\begin{rmk}
		The previous estimates are also true for all $2\leq q< \frac{2d}{(d-4)_+}$.
	\end{rmk}

	\subsubsection*{Proof of Lemma \ref{lemma:TP-matdens-crucial-zone}.}
	
	We prove the result for $P$ which is the right quantization of the symbol $p$ (so that $P=h^2D_{x_1}^2 + \sum_{2\leq i,j\leq d}a_{ij}(x)(hD_{x_i})(hD_{x_j}) -c(x)x_1$ is a differential operator).
	While the final result does not depend on the choice of quantization, the proof will rely on several space localizations so that it will be useful that the operator $P$ is local.
	%% c.f. DETAILS PhD 
		
	Let us now fix some notation.
	Let $\varepsilon>0$ and let us define the strips by
	\begin{equation*}\label{eq-def:TP-A_eps}
		A_\varepsilon := \{x\in\R^d : \abs{x_1-\varepsilon}<\varepsilon/2 \} \quad\text{and}\quad
		\tilde{A}_\varepsilon:= \{x\in\R^d \: \: \abs{x_1-\varepsilon}<3\varepsilon/4 \}
		.
	\end{equation*}
	Each strip $A_\varepsilon$ or $\tilde{A}_\varepsilon$ can be decomposed into an union of boxes of size $\varepsilon$ (see for instance Figure \ref{fig:TP-boxes})
	\begin{equation*}
		A_\varepsilon =\bigcup_{k\in\Z^{d-1}}A_\varepsilon^k
		,\quad
		\tilde{A}_\varepsilon =\bigcup_{k\in\Z^{d-1}}\tilde{A}_\varepsilon^kn
	\end{equation*}
	defined by
	 \begin{equation*}\label{eq-def:TP-A_eps^k}
	 \begin{split}
		 A_\varepsilon^k &:=\{x\in\R^d : \abs{x_1-\varepsilon}<\varepsilon/2,\quad \abs{x'-\varepsilon k}_{l^\infty}<\varepsilon/2 \} \\
		 \tilde{A}_\varepsilon^k &:=\{x\in\R^d : \abs{x_1-\varepsilon}<3\varepsilon/4,\quad \abs{x'-\varepsilon k}_{l^\infty}<3\varepsilon/4 \},\\
		 \tilde{\tilde{A}}_\varepsilon^k &:=\{x\in\R^d : \abs{x_1-\varepsilon}<4\varepsilon/5,\quad \abs{x'-\varepsilon k}_{l^\infty}<4\varepsilon/5 \}.
	 \end{split}
	 \end{equation*}
	
	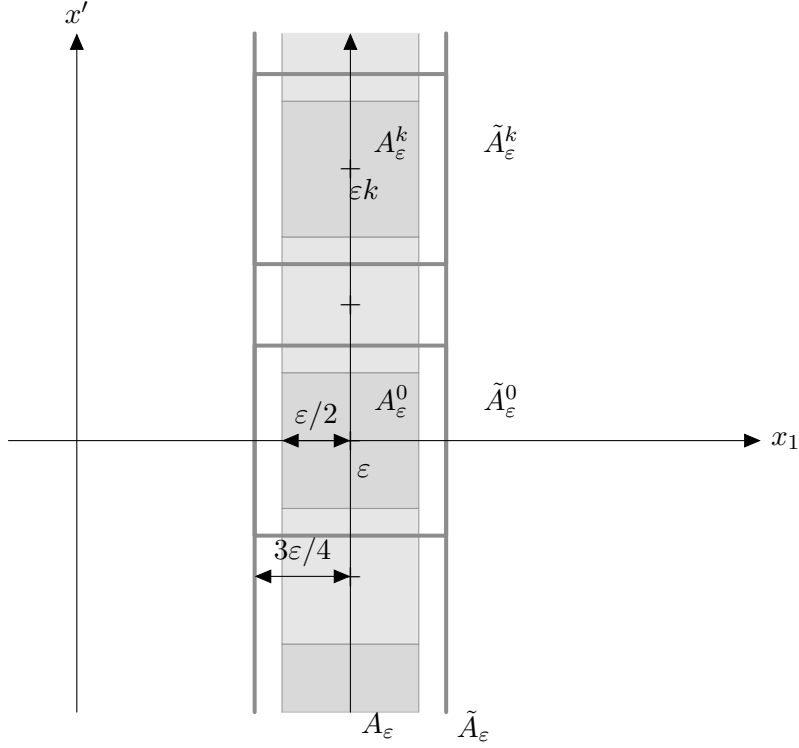
\begin{figure}[!h]
	%%%%%%%%%%%%%%%%%%%% Dessin %%%%%%%%%%%%%%%%%%%%		
	\begin{center}\begin{tikzpicture}[line cap=round,line join=round,>=triangle 45,x=2.0cm,y=2.0cm,scale=.9]
		% cylindres
		%
		\draw[color = gray!80,fill=gray!30](-0.5,-0.5) -- (-0.5,0.5) -- (0.5,0.5) -- (0.5,-0.5) -- (-0.5,-0.5);
		\draw[color = gray!80,fill=gray!20](-.5,0.5) -- (-0.5,1.5) -- (.5,1.5) -- (.5,0.5) -- (-.5,0.5);
		\draw[color = gray!80,fill=gray!30](-0.5,1.5) -- (-0.5,2.5) -- (0.5,2.5) -- (0.5,1.5) -- (-0.5,1.5);
		%
		%\draw[color = gray!20](-1.,3) -- (1,3.);
		\draw[color = gray!80, fill=gray!20](-0.5,3) -- (-0.5,2.5) -- (0.5,2.5) -- (0.5,3);
		\draw[color = gray!20](-0.5,3) -- (0.5,3.);
		\draw[color = gray!80,fill=gray!20](-.5,-1.5) -- (-.5,-0.5) -- (.5,-0.5) -- (.5,-1.5) -- (-.5,-1.5);
		\draw[color = gray!80,fill=gray!30](-.5,-2) -- (-.5,-1.5) -- (.5,-1.5) -- (.5,-2.) -- (-.5,-2.);
		\draw[color = gray!30](-.5,-2) -- (0.5,-2);
		%
		% gros cyl
		%
		\draw[color = gray!90,line width=1.5pt](-0.7,-0.7) -- (-0.7,0.7) -- (0.7,0.7) -- (0.7,-0.7) -- (-0.7,-0.7);
		\draw[color = gray!90,line width=1.5pt](-0.7,1.3) -- (-0.7,2.7) -- (0.7,2.7) -- (0.7,1.3) -- (-0.7,1.3);
		\draw[color = gray!90,line width=1.5pt](-0.7,-2)-- (-0.7,3);
		\draw[color = gray!90,line width=1.5pt](0.7,-2)-- (0.7,3);
		%
		% gradu
		%
		\draw[->] (-2.5,0.) -- (3,0.);
		\draw  (3.,0.) node[right]{$x_1$};
		\draw[->] (-2.,-2.) -- (-2.,3);
		\draw  (-2.,3.) node[above]{$x'$};
		\draw[->] (0.,-2.) -- (0.,3);
		\draw  (0,0.) node{$+$};
		\draw  (0.1,-0.1) node[below] {$\varepsilon$};
		\draw  (0,1.) node{$+$};
		\draw  (0,2.) node{$+$};
		\draw  (0,-1.) node{$+$};
		%
		%% noms des cyl
		%
		\draw  (0.3,0.1) node[left, above]{$A_\varepsilon^0$};
		\draw  (1.1,0.1) node[left, above]{$\tilde{A}_\varepsilon^0$};
		\draw  (0.4,-2.1) node[left]{$A_\varepsilon$};
		\draw  (1.1,-2.1) node[left]{$\tilde{A}_\varepsilon$};
		\draw  (0.1,2.) node[right, below] {$\varepsilon k$};
		\draw  (1.1,2) node[left, above]{$\tilde{A}_\varepsilon^k$};
		\draw  (0.3,2) node[left, above]{$A_\varepsilon^k$};
		%
		% mesures
		\draw[<->](-0.5,0)--(0.,0) node[above, midway]{$\varepsilon/2$};
		\draw[<->](-0.7,-1)--(0.,-1) node[above, midway]{$3\varepsilon/4$};
		\end{tikzpicture}\end{center}
	%%%%%%%%%%%%%%%%%%%% Fin Dessin %%%%%%%%%%%%%%%%%%%%
	\caption{Boxes $A_{\varepsilon}^k$ and $\tilde{A}_{\varepsilon}^k$.}
	\label{fig:TP-boxes}
	\end{figure}
	
	Finally, the set $\Omega_{\delta/2}\setminus\Omega_{Mh^{2/3}}=\{x\in\R^d \: :\: Mh^{2/3}\leq x_1\leq \delta/2 \}$ can be covered by an union of $\sim\log(1/h)$ overlapping strips
	$ \bigcup_{k=\lfloor\log_2(M)\rfloor}^{K(h)} A_{2^kh^{2/3}} $ where $K(h)=\lceil\log_2(\delta h^{-2/3}/2)\rceil$ (see for instance Figure \ref{fig:TP-strips}).
	
	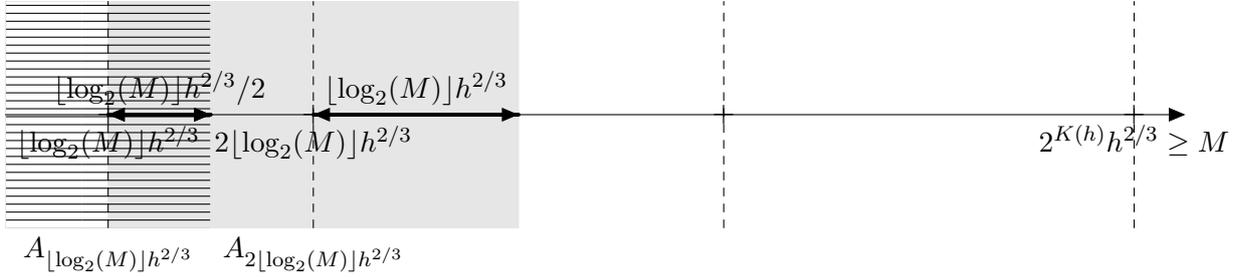
\begin{figure}[!h]
	%%%%%%%%%%%%%%%%%%%%%%%%%%%%%%% DESSIN %%%%%%%%%%%%%%%%%%%%%%%%%%%%%%%
	\begin{center}\begin{tikzpicture}[line cap=round,line join=round,>=triangle 45,x=1.35cm,y=1.0cm,scale=1]
		\usetikzlibrary{patterns}
		% volume
		\draw[color = gray!20,fill=gray!20](-8,-1.5) -- (-8,1.5) -- (-4,1.5) -- (-4,-1.5) -- (-8,-1.5);
		\draw[color = gray!20,pattern=horizontal lines](-9.,-1.5) -- (-9.,1.5) -- (-7.,1.5) -- (-7.,-1.5) -- (-9,-1.5);
		%% gradu
		\draw[->] (-9.,0.) -- (2.5,0); %(8,0.);
		\draw (-8,0.)node{$+$};
		\draw (-8., -0.01) node[below] {$ \lfloor\log_2(M)\rfloor h^{2/3}$};
		\draw[dashed](-8., -1.5)--(-8., 1.5);
		\draw[<->](-8.,0)--(-7,0);
		\draw[line width=1.5pt](-8.,0)--(-7,0);
		\draw[<->](-8.,0)--(-7,0) node[midway, above]{$\lfloor\log_2(M)\rfloor h^{2/3}/2$};
		\draw(-8,-1.5) node[below]{$A_{ \lfloor\log_2(M)\rfloor h^{2/3}}$};
		\draw (-6,0.)node{$+$};
		\draw (-6., -0.01) node[below] {$2 \lfloor\log_2(M)\rfloor h^{2/3}$};
		\draw[dashed](-6., -1.5)--(-6., 1.5);
		\draw[<->](-6,0.)--(-4.,0.);
		\draw[line width=1.5pt](-6,0.)--(-4.,0.);
		\draw (-5.,0.) node[above]{$\lfloor\log_2(M)\rfloor h^{2/3}$};
		\draw(-6,-1.5) node[below]{$A_{2 \lfloor\log_2(M)\rfloor h^{2/3}}$};
		\draw (-2,0.)node{$+$};
		\draw[dashed](-2., -1.5)--(-2., 1.5);
		\draw (2.,0.)node{$+$};
		\draw (2.,0.)node[below] {$2^{K(h)}h^{2/3}\geq M$};
		\draw[dashed](2., -1.5)--(2., 1.5);%(-6., -1.5)--(6., 1.5);
		\end{tikzpicture}\end{center}
	%%%%%%%%%%%%%%%%%%%%%%%%%%%% FIN DESSIN %%%%%%%%%%%%%%%%%%%%%%%%%%%%%%%
	\caption{Strips $A_\varepsilon$.}
	\label{fig:TP-strips}
	\end{figure}
	
	%% plan d'action
	The main steps for the proof of $L^q$ estimates in $\Omega_{\delta/2}\setminus\Omega_{Mh^{2/3}}$ are the following:
	\begin{itemize}
		\item[\quad 0.] obtain the estimates on a box $A$ of size $1$,
		\item[\quad 1.] obtain the estimates on the boxes $A_\varepsilon^k$ of size $\varepsilon$ by scaling the previous one,
		\item[\quad  2.] obtain the estimates on the strips $A_\varepsilon$ by summing the estimates on the $\varepsilon$-boxes,
		\item[\quad  3.] conclude by summing the estimates on the strips.
	\end{itemize}
	%% fin plan d'action
	
	%\newpage

	\paragraph[Step 0]{Step 0. Estimates on a box of size 1.}
	
	We prove estimates on the boxes 
	\begin{equation*}\label{eq-def:TP-A^k}
	\begin{split}
		A &:= \{ \tilde{x}\in\R^d \: : \: \abs{\tilde{x}_1-1}<1/2,\quad\abs{\tilde{x}'}_{l^\infty}<1/2 \} 
		,\\
		\tilde{A} &:= \{ \tilde{x}\in\R^d \: : \: \abs{\tilde{x}_1-1}<3/4,\quad\abs{\tilde{x}'}_{l^\infty}<3/4  \}
		,\\
		\tilde{\tilde{A}} &:=\{ x\in\R^d \: :\: \abs{\tilde{x}_1-1}<4/5, \quad \abs{\tilde{x}'}_{l^\infty}<4/5 \}
		.
	\end{split}
	\end{equation*}
	
	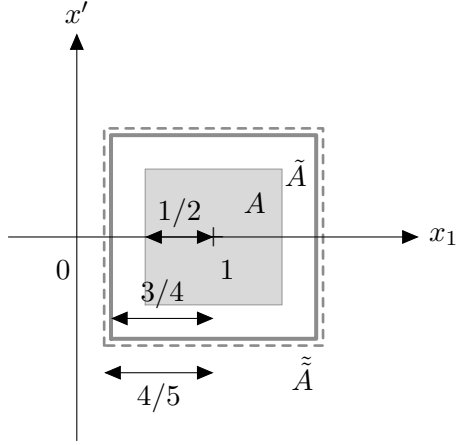
\begin{figure}[h!]
	%%%%%%%%%%%%%%%%%%%% Dessin %%%%%%%%%%%%%%%%%%%%		
	\begin{center}\begin{tikzpicture}[line cap=round,line join=round,>=triangle 45,x=2.0cm,y=2.0cm,scale=.9]
		% cyl A
		\draw[color = gray!80,fill=gray!30](-0.5,-0.5) -- (-0.5,0.5) -- (0.5,0.5) -- (0.5,-0.5) -- (-0.5,-0.5);
		%
		% gros cyl
		\draw[color = gray!90,line width=1.5pt](-0.75,-0.75) -- (-0.75,0.75) -- (0.75,0.75) -- (0.75,-0.75) -- (-0.75,-0.75);
		% cyl limite
		\draw[color = gray!90,line width=1.pt,dashed](-0.8,-0.8) -- (-0.8,0.8) -- (0.8,0.8) -- (0.8,-0.8) -- (-0.8,-0.8);
		%
		%% noms des cyl
		\draw  (0.3,0.1) node[left, above]{$A$};
		\draw  (.6,0.3) node[left, above]{$\tilde{A}$};
		\draw  (.65,-1.2) node[left, above]{$\tilde{\tilde{A}}$};
		%
		% mesures
		\draw[<->](-0.5,0)--(0.,0) node[above, midway]{$1/2$};
		\draw[<->](-0.75,-0.6)--(0.,-0.6) node[above, midway]{$3/4$};
		\draw[<->](-0.8,-1)--(0.,-1) node[below, midway]{$4/5$};
		%
		% gradu
		%
		\draw[->] (-1.5,0.) -- (1.5,0);
		\draw  (1.5,0.) node[right]{$x_1$};
		\draw[->] (-1.,-1.5) -- (-1,1.5);
		\draw  (-1.,1.5) node[above]{$x'$};
		\draw  (0,0.) node{$+$};
		\draw  (0.1,-0.1) node[below] {$1$};
		\draw  (-1.1,-0.1) node[below] {$0$};
		\end{tikzpicture}\end{center}
	%%%%%%%%%%%%%%%%%%%% Fin Dessin %%%%%%%%%%%%%%%%%%%%
	\caption{Boxes of size 1}
	\label{fif:TP-box-size1}
	\end{figure}

	\begin{lemma}\label{lemma:TP-matdens-crucial-zone-rescal-boxes}
		Let $d\geq 2$ and $2\leq q\leq\frac{2d}{d-2}$. Let $\varepsilon_0>0$.
		Let $\tilde{p}_\varepsilon(\tilde{x},\tilde{\xi}):=\prodscal{\tilde{\xi}}{B_\varepsilon(\tilde{x})\tilde{\xi}}+V_\varepsilon(\tilde{x})$ with  $\{(B_\varepsilon,V_\varepsilon)\}_{\varepsilon\in[0,\varepsilon_0]}\subset\CR^\infty(\R^d,\R^{d\times d}_{\rm sym}\times\R)$ and $(\varepsilon,\tilde{x})\mapsto (B_\varepsilon(\tilde{x}),V_\varepsilon(\tilde{x}))\in\CR([0,\varepsilon_0]\times\R^d)$ 
		such that
	\begin{itemize}
		\item
		for any $\alpha\in\N^d$, there exists $C_\alpha>0$ such that for any $\varepsilon\in[0,\varepsilon_0]$
		\begin{equation}\label{cond:TP-Veps-growth}
			\forall x\in\R^d \quad
			  \abs{\partial^\alpha B_\varepsilon(\tilde{x})}+\abs{\partial^\alpha V_\varepsilon(\tilde{x})}\leq C_\alpha
			 ,
		\end{equation}
		\item
		there exists $c>0$ such that for any $\varepsilon\in[0,\varepsilon_0]$
		\begin{equation}\label{cond:TP-Veps-ellip}
		B_\varepsilon\geq c,\,\abs{V_\varepsilon}\geq c \quad\text{ on }\quad \{ \tilde{x}\in\R^d \: : \: \abs{\tilde{x}_1-1}<9/10,\quad \abs{\tilde{x}'}_{l^\infty} <9/10  \}.
		\end{equation} 
	\end{itemize}
		Let $\tilde{P}:=\tilde{p}_\varepsilon^\w(\tilde{x},\tilde{h}D_{\tilde{x}})$ (or any other quantization) and let $\chi_A\in\test{\R^d\times\R^d,[0,1]}$ such that $\chi_A=1$ on $A$ and $\supp\chi_A\subset\tilde{A}$.
		Then, 
		there exist $C>0$ and $\tilde{h}_0>0$  such that 
		for all $0<\tilde{h}\leq \tilde{h}_0$, for all $0\leq\varepsilon\leq\varepsilon_0$ and all bounded  self-adjoint non-negative operator $\tilde\gamma$ on $L^2(\R^d)$
		\begin{equation*}
			\normLp{\rho_{\tilde\gamma}(\tilde{x})}{q/2}{(A)} \leq C \tilde{h}^{-2s_{\text{Sogge}}(q,d)} 
			\normSch{ \big(1+\tilde{P}^*\tilde{P}/\tilde{h}^2\big)^{1/2} \chi_A\tilde\gamma\chi_A \big(1+\tilde{P}^*\tilde{P}/\tilde{h}^2\big)^{1/2} }{\alpha_{\text{Sogge}}(q,d)}{}
			.
		\end{equation*}
	\end{lemma}
	% fin lemme \ref{lemma:TP-matdens-crucial-zone-rescal-boxes}
	
	\begin{proof}[\underline{Proof of Lemma \ref{lemma:TP-matdens-crucial-zone-rescal-boxes}}]
	
		Let $\chi_{\tilde{A}}\in\test{\R^d,[0,1]}$ be a cut-off function equal to 1 on $\tilde{A}$ and supported into $\tilde{\tilde{A}}$.
		It follows from the equality 
		$\chi_A^2(\tilde{x})\rho_{\tilde{\gamma}}(\tilde{x},\tilde{x})=\rho_{\chi_A\tilde{\gamma}\chi_A}(\tilde{x},\tilde{x})$ that
		\begin{equation*}
			\normLp{\rho_{\tilde{\gamma}}}{q/2}{(A)} \leq \normLp{\chi_A^2\rho_{\tilde{\gamma}}}{q/2}{(\R^d)} =\normLp{\rho_{\chi_A\tilde{\gamma}\chi_A}}{q/2}{(\R^d)} .
		\end{equation*}
		Let
		\begin{equation*}\label{def-demo:TP-tildeK}
			\tilde{K} := \bigcup_{0\leq \varepsilon\leq\varepsilon_0}
			\left\lbrace (\tilde{x},\tilde{\xi})\in\bar{\tilde{\tilde{A}}}\times\R^d \: :\: \prodscal{\tilde{\xi}}{B_\varepsilon(\tilde{x})\tilde{\xi})}+V_\varepsilon(\tilde{x})\leq \frac 12\big(1+\prodscal{\tilde{\xi}}{B_\varepsilon(\tilde{x})\tilde{\xi})}\big) \right\rbrace
			.
		\end{equation*}
		Note that each $\tilde{K}$ is a closed set of $\R^d$ (here we use the continuity of $(B_\varepsilon(\tilde{x}),V_\varepsilon(\tilde{x}))$ in $(\varepsilon,\tilde{x})$). It is also bounded, since it is contained into the bounded set 
		\begin{equation*}
			\overline{\tilde{\tilde{A}}}\times \left\lbrace \tilde{\xi}\in\R^d \: :\: \abs{\tilde{\xi}}\leq c^{-1/2} \sqrt{1+2 \sup_{0\leq\varepsilon\leq\varepsilon_0}\normLp{V_\varepsilon}{\infty}{(\R^d)}} \right\rbrace .
		\end{equation*}
		Let $\tilde{\psi} \in\test{\R^d\times\R^d}$ be a function
		$0\leq \tilde{\psi}\leq 1$, such that $\tilde{\psi}=1$ in the compact $\tilde{K}$.
		The operator is composed of three parts
		\begin{equation*}
			\chi_A\tilde{\gamma}\chi_A=\tilde{\gamma}_1 +\tilde{\gamma}_2 + \tilde{\gamma}_3,
		\end{equation*}
		defined by
		\begin{align*}
			\tilde{\gamma}_1 &:= \tilde{\psi}^\w\chi_A\tilde{\gamma}\chi_A\tilde{\psi}^\w, \\
			\tilde{\gamma}_2 &:= 	(1-\tilde{\psi}^\w)\chi_A\tilde{\gamma}\chi_A(1-\tilde{\psi}^\w),  \\
			\tilde{\gamma}_3 &:= 
			\underset{ =: \tilde{\gamma}_{3,1}}{\underbrace{
					(1-\tilde{\psi}^\w)\chi_A\tilde{\gamma}\chi_A 	\tilde{\psi}^\w
			}}
			+ \underset{ =: \tilde{\gamma}_{3,2}}{\underbrace{
						\tilde{\psi}^\w\chi_A\tilde{\gamma}\chi_A(1-\tilde{\psi}^\w)
			}}
			.
		\end{align*}
		Let us prove that for any $i\in\{1,2,3\}$
		\begin{equation*}
			 \normLp{\rho_{\tilde{\gamma}_i}}{q/2}{(\R^d)} \lesssim \tilde{h}^{-2s_{\text{Sogge}}(q,d)} \normSch{(1+\tilde{P}^*\tilde{P}/\tilde{h}^2)^{1/2}\chi_A\tilde{\gamma}\chi_A(1+\tilde{P}^*\tilde{P}/\tilde{h}^2)^{1/2}}{\alpha_{\text{Sogge}}(q,d)}{} .
		\end{equation*}
		These above bounds together with the triangle inequality prove the lemma.
		
		\subparagraph{Estimate of the term $\tilde{\gamma}_1$.}
		
		Note that
		\begin{equation*}
			\supp\tilde{\psi} \cap \tilde{\tilde{A}}\times\R^d \subset S=\{ \tilde{x}\in\R^d \: : \: \abs{\tilde{x}_1-1}<9/10,\quad \abs{\tilde{x}'}_{l^\infty} <9/10  \}\times\R^d .
		\end{equation*}
		Every point $(\tilde{x},\tilde{\xi})$ of $S$ satisfies  $\tilde{p}_\varepsilon(\tilde{x},\tilde{\xi})\neq 0$ or Sogge curvature conditions (Assumption \ref{cond:sogge}).
		Indeed, if $(\tilde{x},\tilde{\xi})$ satisfies $\tilde{p}_\varepsilon(\tilde{x},\tilde{\xi})= 0$, then uniformly in $\varepsilon$, by \eqref{cond:TP-Veps-ellip}
		\begin{equation*}
			\abs{\nabla_{\tilde{\xi}}\tilde{p}_\varepsilon(\tilde{x},\tilde{\xi})}=2\abs{B_\varepsilon(\tilde{x})\tilde{\xi}} = 2\sqrt{c}\sqrt{\abs{V_\varepsilon(\tilde{x})}}\geq 2 c .
		\end{equation*}
		Furthermore, for all $\varepsilon$ and $\tilde{x}$
		\begin{equation*}
			\{ \tilde{\xi}\in\R^d \: :\: \tilde{p}_\varepsilon(\tilde{x},\tilde{\xi}) =0 \} = \{  \tilde{\xi}\in\R^d \: :\: \prodscal{\tilde{\xi}}{B_\varepsilon(\tilde{x})\tilde{\xi}}= \abs{V_\varepsilon(\tilde{x})} \} 
		\end{equation*}
		has a positive curvature which is bounded and bounded away from 0 uniformly in $\varepsilon$.
		Hence, by Theorem \ref{thm:ELp-elliptic-matdens} and Theorem \ref{thm:ELp-sogge-matdens}, Assumption \ref{ass:microloc} holds for $q$, $s=s_{\text{Sogge}}(q,d)$, $t= 0$ and $\alpha=\alpha_{\text{Sogge}}(q,d)$ on the set $S$. 
		By Remark \ref{rmk:abstract-loc-space} applied to $S$, to $\Omega=\tilde{A}$, to $\chi=\tilde{\psi}$, to $(q,s_{\text{Sogge}},0,\alpha_{\text{Sogge}})$ and to the operator  $\gamma=\chi_A\tilde{\gamma}\chi_A
		$
		\begin{equation*}
		\begin{split}
			\normLp{\rho_{\tilde{\gamma}_1}}{q/2}{(\R^d)}
			&= \normLp{\rho_{\tilde\psi^\w\chi_A\tilde{\gamma}\chi_A\tilde\psi^\w}}{q/2}{(\R^d)}
			%
			% detail
			= \normLp{\rho_{\tilde\psi^\w\chi_{\tilde{A}} \: \chi_A\tilde{\gamma}\chi_A \: \chi_{\tilde{A}}\tilde\psi^\w}}{q/2}{(\R^d)}
			% fin detail
			%
			\\&\lesssim \tilde{h}^{-2s_{\text{Sogge}}(q,d)}\normSch{(1+\tilde{P}^*\tilde{P}/\tilde{h}^2)^{1/2}\chi_A\tilde{\gamma}\chi_A(1+\tilde{P}^*\tilde{P}/\tilde{h}^2)^{1/2}}{\alpha_{\text{Sogge}}(q,d)}{}
			.
		\end{split}
		\end{equation*}

		\subparagraph{Estimate of the term $\tilde{\gamma}_2$.}
		
		We write with the Mercer theorem (Remark \ref{rmk:mercer-thm}) that
		\begin{equation*}
			\normLp{\rho_{\tilde{\gamma}_2}}{q/2}{(\R^d)}
			= 
			\begin{cases}
				\normSch{(1-\tilde{\psi}^\w)\chi_A\sqrt{\tilde{\gamma}}}{2}{}^2 &\text{if}\
				q=2,\\\displaystyle
				\sup_{W\in L^{2(q/2)'}\cap\CR^0(\R^d)}\frac{\normSch{W(1-\tilde{\psi}^\w)\chi_A\sqrt{\tilde{\gamma}}}{2}{}^2}{\normLp{W}{2(q/2)'}{(\R^d)}^2} &\text{if}\
				2\leq q\leq \frac{2d}{d-2}.
			\end{cases}
		\end{equation*}
		By the Hölder and Kato-Seiler-Simon inequalities (Lemma \ref{lemma:Kato-Seiler-Simon_dual}) then applied successively, we have for any $2\leq q\leq\frac{2d}{d-2}$
		(since $\frac{2d}{d-2}< \frac{2d}{(d-4)_+}$)
		\begin{align*}
			\normSch{W(1-\tilde{\psi}^\w)\chi_A\sqrt{\tilde{\gamma}}}{2}{}
			&\leq \normSch{W (1-\tilde{h}^2\Delta)^{-1}}{2(q/2)'}{}
			\normSch{ (1-\tilde{h}^2\Delta)(1-\tilde{\psi}^\w)\chi_A\sqrt{\tilde{\gamma}}}{q/2}{}^{1/2}
			\\&\lesssim h^{-d\left(\frac 12-\frac 1q\right)}\normLp{W}{2(q/2)'}{(\R^d)}\normSch{ (1-\tilde{h}^2\Delta)(1-\tilde{\psi}^\w)\chi_A\sqrt{\tilde{\gamma}}}{q/2}{}^{1/2}
			.
		\end{align*}
		Implicitly, when $q=2$, we write $W=1$ and the operator norm instead of the $\schatten^{2(q/2)'}$ norm.
		Now let us give a estimate of the right side of the previous bound.
		\begin{fait}\label{fact:TP-matdens-g2qw}
			There exist $C>0$ and $\tilde{h}_0>0$ such that for any $0<\tilde{h}\leq \tilde{h}_0$ and any $0\leq\varepsilon\leq\varepsilon_0$, one has for all $\alpha\geq 1$
			\begin{equation*}%\label{eq:TP-matdens-g2qw}
			\normSch{ (1-\tilde{h}^2\Delta)(1-\tilde{\psi}^\w)\chi_A\sqrt{\tilde{\gamma}}}{ \alpha}{} \leq
			C \left( \normSch{\tilde{P}\chi_A\sqrt{\tilde{\gamma}}}{\alpha}{}+ \tilde{h} \normSch{\chi_A\sqrt{\tilde{\gamma}}}{\alpha}{} \right)
			.
			\end{equation*}
		\end{fait}
		% fin \ref{fact:TP-matdens-g2qw}
		%
		\begin{proof}[\underline{Proof of Fact \ref{fact:TP-matdens-g2qw}}]
			Let $\alpha\geq 1$.
			\item[\quad 1)] Let us first show that it is enough to do everything with right quantization by replacing $(1-\tilde{\psi}^\w)$ by $(1-\tilde{\psi}^\qr)$ into Fact \ref{fact:TP-matdens-g2qw}.
			By Proposition \ref{cor:SA-quantif-change} applied to the symbol $\tilde{\psi}$, there exists $\tilde{r}\in\schwartz(\R^d\times\R^d)$ such that
			\begin{equation*}
				(1-\tilde{\psi}^\w)-(1-\tilde{\psi}^\qr)=\tilde{\psi}^\qr-\tilde{\psi}^\w =\tilde{h}\tilde{r}^\w .
			\end{equation*}
			We can now write
			\begin{align*}
			(1-\tilde{h}^2\Delta)\: (1-\tilde{\psi}^\w)\chi_A\sqrt{\tilde{\gamma}}
			%
			%\\&\qquad 
			&= (1-\tilde{h}^2\Delta)\:\left(
			(1-\tilde{\psi}^\qr)
			\chi_A\sqrt{\tilde{\gamma}} + \op_{\tilde{h}}^{1/2}\left(\OR_\schwartz(\tilde{h})\right) \chi_A\sqrt{\tilde{\gamma}}
			\right)
			\\&=
			(1-\tilde{h}^2\Delta)(1-\tilde{\psi}^\qr)\chi_A\sqrt{\tilde{\gamma}}
			+ \OR_{L^2\to L^2}(\tilde{h}) \chi_A\sqrt{\tilde{\gamma}}
			.
			\end{align*}
			
			%% point 2
			\item[\quad 2)] It remains to show fact \ref{fact:TP-matdens-g2qw}
			by replacing $\tilde{\psi}^\w$ by $\tilde{\psi}^\qr$.
			\item  Let $\tilde{m}_0(\tilde{x},\tilde{\xi})=\langle\tilde{\xi}\rangle^2$. Let us first notice that the operator $\tilde{p}_\varepsilon$ is elliptic on the support of $(1-\tilde{\psi})\chi_{\tilde{A}}$.
			Indeed, the function $(1-\tilde{\psi})\chi_{\tilde{A}}$ is supported into $\tilde{K}^c\cap(\bar{\tilde{\tilde{A}}}\times\R^d)$ and  by definition of $\tilde{K}$, for all $\varepsilon\in[0,\varepsilon_0]$ and for all $(\tilde{x},\tilde{\xi})\in \tilde{K}^c\cap(\bar{\tilde{\tilde{A}}}\times\R^d)$ we have
			\begin{equation*}
				\tilde{p}_\varepsilon(\tilde{x},\tilde{\xi})> \frac {\min(1,c)}2\tilde{m}_0(\tilde{x},\tilde{\xi}) .
			\end{equation*}
			Thus, we get the ellipticity (uniform in $\varepsilon$) of $\tilde{p}_\varepsilon$ on $\tilde{K}^c\cap(\bar{\tilde{\tilde{A}}}\times\R^d)$ and thus on $\supp((1-\tilde{\psi})\chi_{\tilde{A}})$.
			By Lemma \ref{lemma:SA-inversion} applied to $(1-\tilde{\psi})\chi_{\tilde{A}} \in S(1)$ and to $p\in S(\tilde{m}_0)$, there exist $\tilde{b}\in S(\tilde{m_0}^{-1})$ and $\tilde{r}=\OR_{S(1)}(\tilde{h}^\infty)$ such that
			\begin{equation*}
			(1-\tilde{\psi})^\qr \chi_{\tilde{A}}= \tilde{b}^\qr \tilde{P} \: (1-\tilde{\psi}^\qr)\chi_{\tilde{A}} +  \tilde{r}^\qr .
			\end{equation*}
			Note that $\chi_A=\chi_{\tilde{A}}\chi_A$.
			We compose by $(1-\tilde{h}^2\Delta)$ on the left and $\chi_{A}\sqrt{\tilde{\gamma}}$ on the right to infer
			\begin{align*}
				&(1-\tilde{h}^2\Delta)(1-\tilde{\psi})^\qr\chi_A\sqrt{\tilde{\gamma}}
				\\&\qquad = \underset{=\OR_{L^2\to L^2}(1)}{\underbrace{
						(1-\tilde{h}^2\Delta)\tilde{b}^\qr
				}}
				\tilde{P}(1-\tilde{\psi}^\qr)\chi_A\sqrt{\tilde{\gamma}}
				+ 
				\underset{=\OR_{L^2\to L^2}(\tilde{h}^\infty)}{\underbrace{
						(1-\tilde{h}^2\Delta) \tilde{r}^\qr (1-\tilde{h}^2\Delta)^{-1}
				}}
				\: (1-\tilde{h}^2\Delta)\chi_A\sqrt{\tilde{\gamma}}
				.
			\end{align*}
			Let us write $\tilde{P}(1-\tilde{\psi}^\qr)\chi_A\sqrt{\tilde{\gamma}}$ in two parts
			\begin{equation*}
				\tilde{P}(1-\tilde{\psi}^\qr)\chi_A\sqrt{\tilde{\gamma}}
				= 
				\underset{=\OR_{L^2\to L^2}(1)}{\underbrace{
						(1-\tilde{\psi}^\qr)
				}}
				\tilde{P} \chi_A\sqrt{\tilde{\gamma}}
				-
				\underset{=\OR_{L^2\to L^2}(\tilde{h})}{\underbrace{
						\comm{\tilde{P}}{\tilde{\psi}^\qr}
				}}
				\chi_A\sqrt{\tilde{\gamma}}.
			\end{equation*}
			In addition, the term $(1-\tilde{h}^2\Delta)\chi_A\sqrt{\tilde{\gamma}}$ is also divided into two parts
			\begin{align*}
				(1-\tilde{h}^2\Delta)\chi_A\sqrt{\tilde{\gamma}}
				&= \underset{=\OR_{L^2\to L^2}(1)}{\underbrace{
						(1-\tilde{h}^2\Delta)\tilde{\psi}^\qr 
				}}  
				\chi_A\sqrt{\tilde{\gamma}}
				+ (1-\tilde{h}^2\Delta)(1-\tilde{\psi}^\qr ) \chi_A\sqrt{\tilde{\gamma}} .
			\end{align*}
			Putting everything together and moving to the Schatten norm
			\begin{align*}
				\normSch{(1-\tilde{h}^2\Delta)(1-\tilde{\psi}^\qr)\chi_A\sqrt{\tilde{\gamma}}}{\alpha}{}
				&\lesssim \normSch{\tilde{P}\chi_A\sqrt{\tilde{\gamma}}}{\alpha}{} +
				\tilde{h} \normSch{\chi_A\sqrt{\tilde{\gamma}}}{\alpha}{} \\&\quad +\OR(\tilde{h}^\infty)
				\normSch{(1-\tilde{h}^2\Delta)(1-\tilde{\psi}^\qr)\chi_A\sqrt{\tilde{\gamma}}}{\alpha}{}.
			\end{align*}
			Then, we conclude by taking $\tilde{h}$ small enough
			\begin{equation*} 				\normSch{(1-\tilde{h}^2\Delta)(1-\tilde{\psi}^\qr)\chi_A\sqrt{\tilde{\gamma}}}{\alpha}{} \lesssim \normSch{\tilde{P}\chi_A \sqrt{\tilde{\gamma}}}{\alpha}{} + \tilde{h} \normSch{\chi_A\sqrt{\tilde{\gamma}}}{\alpha}{} .
			\end{equation*}
			That ends the proof of Fact \ref{fact:TP-matdens-g2qw}.
		\end{proof}
		% fin demo \ref{fact:TP-matdens-g2qw}
		%
		We apply it to $\alpha=q/2$ and we apply Lemma \ref{lemma:equiv-norm-P-matdens}.
		Hence, we get for any $2\leq q\leq\frac{2d}{d-2}$
		\begin{align*}
			\normLp{\rho_{\tilde{\gamma}_2}}{q/2}{(\R^d)}
			&\lesssim \tilde{h}^{2-2d\left(\frac 1 2 -\frac 1 q\right)} \left(\normSch{\chi_A\tilde{\gamma}\chi_A}{q/2}{}+\frac 1{\tilde{h}^2}\normSch{\tilde{P}^*\chi_A\tilde{\gamma}\chi_A\tilde{P}}{q/2}{}\right) 
			\\& \lesssim \tilde{h}^{2-2d\left(\frac 1 2 -\frac 1 q\right)} \normSch{(1+\tilde{P}^*\tilde{P}/\tilde{h}^2)^{1/2}\chi_A\tilde{\gamma}\chi_A(1+\tilde{P}^*\tilde{P}/\tilde{h}^2)^{1/2}}{q/2}{}
			.
		\end{align*}
		
		\subparagraph{Estimate of the crossed terms $\tilde{\gamma}_{3,1}$ and $\tilde{\gamma}_{3,2}.$}
		We deduce the estimates on the crossed terms $\tilde{\gamma}_{3,1}$ and $\tilde{\gamma}_{3,2}$ from those of $\tilde{\gamma}_1$ and $\tilde{\gamma}_2$.
		For example for $\tilde{\gamma}_{3,1}$, noting $C := \sqrt{\tilde{\gamma}}\chi_A(1-\tilde{\psi}^\w)$ and $B := \sqrt{\tilde{\gamma}}\chi_A\tilde{\psi}^\w$
		\begin{align*}
			&\abs{\tr_{L^2}(W\tilde{\gamma}_{3,1}W)}
			\\&\quad 
			= \abs{ \tr_{L^2}\left(W(1-\tilde{\psi}^\w)\chi_A\tilde{\gamma}\chi_A\tilde{\psi}^\w W\right) }
			= \abs{\tr_{L^2}(WC^*BW)}
			\\&\quad\leq \abs{\tr_{L^2}\left(WC^*CW\right)}^{1/2}\abs{\tr_{L^2}\left(WB^*BW\right)}^{1/2}
			\\&\quad\leq \left(\normLp{W}{2(q/2)'}{(\R^d)}^2\normLp{\rho_{C^*C}}{q/2}{(\R^d)}\right)^{1/2}\left(\normLp{W}{2(q/2)'}{(\R^d)}^2\normLp{\rho_{B^*B}}{q/2}{(\R^d)}\right)^{1/2}
			\\&\quad
			\leq \normLp{W}{2(q/2)'}{(\R^d)}^2
			\normLp{\rho_{(1-\tilde{\psi}^\w)\chi_A\tilde{\gamma}\chi_A(1-\tilde{\psi}^\w)}}{q/2}{(\R^d)}^{1/2} \normLp{\rho_{\tilde{\psi}^\w\chi_A\tilde{\gamma}\chi_A\tilde{\psi}^\w}}{q/2}{(\R^d)}^{1/2}
			\\&\quad
			\lesssim \tilde{h}^{1-d\left(\frac 1 2-\frac 1 q\right)-s_{\text{Sogge}}(q,d)}\normLp{W}{2(q/2)'}{(\R^d)}^2  \normSch{(1+\tilde{P}^*\tilde{P}/\tilde{h}^2)^{1/2}\chi_A\tilde{\gamma}\chi_A(1+\tilde{P}^*\tilde{P}/\tilde{h}^2)^{1/2}}{\alpha_{\text{Sogge}}(q,d)}{}
			.
		\end{align*}
		For any $2\leq q\leq\frac{2d}{d-2}$
		\begin{equation*}
			 \normLp{\rho_{\tilde{\gamma}_3}}{q/2}{(\R^d)}
			\lesssim \tilde{h}^{1-d\left(\frac 1 2-\frac 1 q\right)-s_{\text{Sogge}}(q,d)} \normSch{(1+\tilde{P}^*\tilde{P}/\tilde{h}^2)^{1/2}\chi_A\tilde{\gamma}\chi_A(1+\tilde{P}^*\tilde{P}/\tilde{h}^2)^{1/2}}{\alpha_{\text{Sogge}}(q,d)}{} .
		\end{equation*}
	\end{proof}
	% fin demo lemme \ref{lemma:TP-matdens-crucial-zone-rescal-boxes}

	\paragraph[Step 1]{Step 1. The scaling.}

	Let us deduce from Lemma \ref{lemma:TP-matdens-crucial-zone-rescal-boxes} the same kind of result but on the boxes $A^k_\varepsilon$ by a scaling argument.
	The following lemma controls the $L^{q/2}$ norm of the density on the boxes $A^k_\varepsilon$.
	
	\begin{lemma}\label{lemma:TP-matdens-crucial-zone-boxes}
		Let $d\geq 2$ and $2\leq q\leq\frac{2d}{d-2}$.
		Then, there exists $C>0$
		such that for any $0<h\leq h_0$, for any $\varepsilon\in [M h^{2/3},\delta/2]$ and for any $k\in\Z^{d-1}$ such that $\abs{k}_{l^\infty}<\delta/\varepsilon-1$,
		there exists $\chi_{\tilde{A}_\varepsilon^k}\in\test{\R^d\times\R^d,[0,1]}$  supported into $\tilde{\tilde{A}}^k_\varepsilon$,
		such that any bounded self-adjoint non-negative operator $\gamma$ on $L^2(\R^d)$
		\begin{equation*}\label{eq:TP-matdens-crucial-zone-boxes}
			\normLp{\rho_\gamma}{q/2}{(A_\varepsilon^k)} \leq C
			h^{-2s_{\text{Sogge}}(q,d)}\varepsilon^{-2\mu(q,d)}
			\left(\normSch{\chi_{\tilde{A}_\varepsilon^k}\gamma\chi_{\tilde{A}_\varepsilon^k}}{\alpha_{\text{Sogge}}(q,d)}{}+\frac{\varepsilon}{h^2}\normSch{\chi_{\tilde{A}_\varepsilon^k}P\gamma P^*\chi_{\tilde{A}_\varepsilon^k}}{\alpha_{\text{Sogge}}(q,d)}{}\right)
			,
		\end{equation*}
		where $\mu(q,d)$ is given by the formula
		\begin{equation}\label{def:TP-mu}
			\mu(q,d):=d\left(\frac 1 2 -\frac 1 q\right)-\frac{3s_{\text{Sogge}}(q,d)}{2} .
		\end{equation}
	\end{lemma}
	% fin lemme \ref{lemma:TP-matdens-crucial-zone-boxes}
	
	\begin{proof}[\underline{Proof of Lemma \ref{lemma:TP-matdens-crucial-zone-boxes}}]
	Recall that $h_0$ and $M$ were already fixed above. We will make additional constraints on them along this proof. Let $h\in(0,h_0]$,  $\varepsilon\in [M h^{2/3},\delta/2]$ and $k\in\Z^{d-1}$ such that $\abs{k}_{l^\infty}<\delta/\varepsilon-1$. Let $\chi_{\tilde{A}}\in\test{\R^d,[0,1]}$ be a cut-off function equal to 1 on $\tilde{A}$ and supported into $\tilde{\tilde{A}}$. Let $\chi_{\tilde{\tilde{A}}}\in\test{\R^d,[0,1]}$ be a cut-off function equal to 1 on $\supp\chi_{\tilde{A}}$ and supported into $\tilde{\tilde{A}}$. 
	Define for any $\tilde{x}\in\R^d$
	\begin{equation*}
		B_\varepsilon(\tilde{x}):= \begin{pmatrix}
		1 & \vline &0 \cdots 0 \\ \hline 
		0 & \vline &\\
		\vdots & \vline&   (a_{ij}(\varepsilon\tilde{x}_1,\varepsilon\tilde{x}'+\varepsilon k))_{2\leq i,j\leq d}& \\
		0  & \vline& 
		\end{pmatrix}
		,
		\quad
		V_\varepsilon(\tilde{x}) := -\tilde{x}_1c(\varepsilon\tilde{x_1},\varepsilon\tilde{x}'+\varepsilon k)\chi_{\tilde{\tilde{A}}}(\tilde{x})
		.
	\end{equation*}
	By our assumptions on $(a_{ij})_{2\leq i,j\leq d}$ and $c$, $\{(B_\varepsilon,V_\varepsilon)\}_\varepsilon$ satisfies the assumptions of Lemma \ref{lemma:TP-matdens-crucial-zone-rescal-boxes} with $\varepsilon_0=\delta/2$ (notice that for our choice of $\varepsilon$ and $k$, $\abs{(\varepsilon\tilde{x_1},\varepsilon\tilde{x}'+\varepsilon k)}< \delta$ for all $\tilde{x}\in\tilde{\tilde{A}}$).
	Hence, let $C>0$ and $\tilde{h}_0>0$ such that 
	for all $0<\tilde{h}\leq \tilde{h}_0$
	and all bounded  self-adjoint non-negative operator $\tilde\gamma$ on $L^2(\R^d)$
	\begin{equation*}
		\normLp{\rho_{\tilde\gamma}(\tilde{x})}{q/2}{(A)} \leq C \tilde{h}^{-2s_{\text{Sogge}}(q,d)} 
		\normSch{ \big(1+\tilde{P}^*\tilde{P}/\tilde{h}^2\big)^{1/2} \chi_A\tilde\gamma\chi_A \big(1+\tilde{P}^*\tilde{P}/\tilde{h}^2\big)^{1/2} }{\alpha_{\text{Sogge}}(q,d)}{}
		,
	\end{equation*}
	where $\tilde{P}=p_\varepsilon^\qr(\tilde{x},\tilde{h}D_{\tilde{x}})$ and $p_\varepsilon(\tilde{x},\tilde{\xi})=\prodscal{\tilde{\xi}}{B_\varepsilon(\tilde{x})\tilde{\xi}} +V_\varepsilon(\tilde{x})$. 
	Moreover, we have the following fact (which proof is given below).
	\begin{fait}\label{fact:TP-matdens-crucial-zone-inv-P&chi_A}
		Let $\alpha\geq 1$. Then, there exists $C>0$ such for all $\tilde{h}\in(0,\tilde{h}_0]$, for all $\varepsilon\in[0,\varepsilon_0]$ and for all bounded self-adjoint non-negative operator $\tilde{\gamma}$
		\begin{equation*}
		\normSch{(1+\tilde{P}^*\tilde{P}/\tilde{h}^2)^{1/2}\chi_A\tilde{\gamma}\chi_A(1+\tilde{P}^*\tilde{P}/\tilde{h}^2)^{1/2}}{\alpha}{} \leq C \left( \normSch{\chi_{\tilde{A}}\tilde{\gamma}\chi_{\tilde{A}}}{\alpha}{}+\frac 1{\tilde{h}^2} \normSch{\chi_{\tilde{A}}\tilde{P}\tilde{\gamma}\tilde{P}^*\chi_{\tilde{A}}}{\alpha}{}\right).
		\end{equation*}
	\end{fait}
	% fin fait \ref{fact:TP-matdens-crucial-zone-inv-P&chi_A}
	%
	We deduce that for all bounded  self-adjoint non-negative operator $\tilde\gamma$ on $L^2(\R^d)$, we have
	\begin{equation*}
		\normLp{\rho_{\tilde\gamma}(\tilde{x})}{q/2}{(A)} \leq C \tilde{h}^{-2s_{\text{Sogge}}(q,d)} \left( \normSch{\chi_{\tilde{A}}\tilde{\gamma}\chi_{\tilde{A}}}{\alpha_{\text{Sogge}}(q,d)}{}+\frac 1{\tilde{h}^2} \normSch{\chi_{\tilde{A}}\tilde{P}\tilde{\gamma}\tilde{P}^*\chi_{\tilde{A}}}{\alpha_{\text{Sogge}}(q,d)}{}\right)
		.
	\end{equation*}
	Since $\chi_{\tilde{A}}\chi_{\tilde{\tilde{A}}}=\chi_{\tilde{A}}$, the same bound holds when there is no factor $\chi_{\tilde{\tilde{A}}}$ in $V_\varepsilon$. We still denote by $\tilde{P}$ the resulting operator.
	Now, let $\gamma$ be a  bounded  self-adjoint non-negative operator on $L^2(\R^d)$. 
	We apply the above bound to $\tilde{\gamma}=\big(U_\varepsilon^k\big)^* \gamma U_\varepsilon^k$, where
	$U_\varepsilon^k$ is the unitary transformation defined by
	\begin{equation*}\label{def:TP-matdens-unit-transf}
		\begin{array}{lll}
		U_\varepsilon^k : \:&	L^2(\R^d)&\to L^2(\R^d)\\
		&f &\mapsto \left( x\mapsto \varepsilon^{-d/2}f\left(\frac{x_1}{\varepsilon},\frac{x'-\varepsilon k}{\varepsilon}\right) \right)
		.
		\end{array} 
	\end{equation*}
	Since we have
	\begin{equation*}
		\normLp{\rho_{\tilde\gamma}}{q/2}{(A)}=\varepsilon^{2d(1/2-1/q)} \normLp{\rho_\gamma}{q/2}{(A^k_\varepsilon)} ,
	\end{equation*}
	and $U_\varepsilon^k\tilde{P}\big(U_\varepsilon^k\big)^*=P/\varepsilon$, we deduce that
	\begin{align*}
		&\normLp{\rho_\gamma}{q/2}{(A^k_\varepsilon)} \\&\quad \leq C \tilde{h}^{-2s_{\text{Sogge}}(q,d)}  \varepsilon^{-2d(1/2-1/q)} \left(
		\normSch{\chi_{\tilde{A}^k_\varepsilon}\gamma \chi_{\tilde{A}^k_\varepsilon}}{\alpha_{\text{Sogge}}(q,d)}{}+\frac 1{(\varepsilon\tilde{h})^2}\normSch{\chi_{\tilde{A}^k_\varepsilon}P\gamma P^*\chi_{\tilde{A}^k_\varepsilon}}{\alpha_{\text{Sogge}}(q,d)}{} \right)
		,
	\end{align*}
	with $\chi_{\tilde{A}^k_\varepsilon}=U_\varepsilon^k\chi_{\tilde{A}}\big(U_\varepsilon^k\big)^*$, i.e.\ $\chi_{\tilde{A}^k_\varepsilon}(x)=\chi_{\tilde{A}}\left(\frac{x_1}{\varepsilon},\frac{x'-\varepsilon k}{\varepsilon}\right)$.
	Now assume that $M\geq \tilde{h}_0^{-2/3}$ and let $h\in(0,h_0]$ (where we recall that $h_0>0$ was chosen such that $Mh_0^{2/3}<\delta/2$). We apply the above bound to $\tilde{h}=h/\varepsilon^{3/2}$, which indeed satisfies $\tilde{h}\leq \tilde{h}_0$ since $\varepsilon\geq Mh^{2/3}$ implies $\tilde{h}\leq M^{-3/2}\leq \tilde{h}_0$. Finally, we obtain
	\begin{align*}
		&\normLp{\rho_\gamma}{q/2}{(A_\varepsilon^k)}
		\\&\quad\leq C (h/\varepsilon^{3/2})^{-2s_{\text{Sogge}}(q,d)}\varepsilon^{-2d(1/2-1/q)}  \left(
		\normSch{\chi_{\tilde{A}^k_\varepsilon}\gamma \chi_{\tilde{A}^k_\varepsilon}}{\alpha_{\text{Sogge}}(q,d)}{}+\frac \varepsilon{h^2}\normSch{\chi_{\tilde{A}^k_\varepsilon} P\gamma P^*\chi_{\tilde{A}^k_\varepsilon}}{\alpha_{\text{Sogge}}(q,d)}{} \right)
		\\&\quad= C h^{-2s_{\text{Sogge}}(q,d)}\varepsilon^{-2\mu(q,d)}  \left(
		\normSch{\chi_{\tilde{A}^k_\varepsilon}\gamma \chi_{\tilde{A}^k_\varepsilon}}{\alpha_{\text{Sogge}}(q,d)}{}+\frac \varepsilon{h^2}\normSch{\chi_{\tilde{A}^k_\varepsilon} P\gamma P^*\chi_{\tilde{A}^k_\varepsilon}}{\alpha_{\text{Sogge}}(q,d)}{} \right)
		.
	\end{align*}
	\end{proof}
	% fin demo \ref{lemma:TP-matdens-crucial-zone-boxes}
	
	We now give the missing proof of Fact \ref{fact:TP-matdens-crucial-zone-inv-P&chi_A}.
	
	\begin{proof}[\underline{Proof of Fact \ref{fact:TP-matdens-crucial-zone-inv-P&chi_A}}]
		% comprehension pour 1 fct
		\item[\qquad$\bullet$] It is essentially enough to understand why this inequality is true for the one-body case. The bound to prove is 
		\begin{equation*}
			\normLp{\tilde{P}\chi_A \tilde{u}}{2}{(\R^d)} \lesssim
			\normLp{\chi_{\tilde{A}}\tilde{P}\tilde{u}}{2}{(\R^d)} +\tilde{h}\normLp{\chi_{\tilde{A}}\tilde{u}}{2}{(\R^d)}.
		\end{equation*}
		We already notice that
		\begin{align*}
			\tilde{P}\chi_A
			&= \chi_A\tilde{P} -\tilde{h}^2(\partial_{\tilde{x_1}}^2\chi_A)
			-2\tilde{h}(\partial_{\tilde{x}_1}\chi_A)(\tilde{h}\partial_{\tilde{x}_1})
			% +\comm{\tilde{\lambda}_\varepsilon^\qr}{\chi_A}
			\\&\quad
			 -\tilde{h^2} \sum_{2\leq i,j\leq d} (\partial_{\tilde{x}_i}\partial_{\tilde{x}_j}\chi_{A}) a_{ij}(\varepsilon\tilde{x}_1,\varepsilon x'+\varepsilon k)
			 -\tilde{h} \prodscal{\nabla\chi_A}{(a_{ij}(\varepsilon\tilde{x}_1,\varepsilon\tilde{x}'+\varepsilon k))_{i,j}\tilde{h}\nabla}
			 .
		\end{align*}	
		On the one hand, since $a_{ij}\in L^\infty(\R^d)$
		\begin{align*}
			&\normLp{\tilde{h}^2(\partial_{\tilde{x_1}}^2\chi_A)\:\tilde{u}}{2}{(\R^d)}
			+\normLp{\tilde{h^2} \sum_{2\leq i,j\leq d} (\partial_{\tilde{x}_i}\partial_{\tilde{x}_j}\chi_{A}) a_{ij}(\varepsilon\tilde{x}_1,\varepsilon x'+\varepsilon k)\tilde{u}}{2}{(\R^d)}
			\\&\qquad\lesssim  \tilde{h^2}\normLp{\chi_{\tilde{A}}\tilde{u}}{2}{(\R^d)}
			.
		\end{align*}
		On the other hand, since $\tilde{P}$ is elliptic in the sense of \cite[Lem. 2.6]{koch2007semiclassical}, we have 
		\begin{align*}
			\sum_{\abs{\alpha}=1}\normLp{(\tilde{h}D)^\alpha \tilde{u}}{2}{(A)}
			&\lesssim  \normLp{\tilde{u}}{2}{\left(\tilde{A}\right)}+\normLp{\tilde{P}\tilde{u}}{2}{\left(\tilde{A}\right)},
		\end{align*}
		(by using again that $a_{ij}\in L^\infty(\R^d)$) so that 
		\begin{align*}
			&\normLp{\tilde{h}(\partial_{\tilde{x}_1}\chi_A)(\tilde{h}\partial_{\tilde{x}_1})}{2}{(\R^d)}^2
			+\normLp{\tilde{h} \prodscal{\nabla\chi_A}{(a_{ij}(\varepsilon\tilde{x}_1,\varepsilon\tilde{x}'+\varepsilon k))_{i,j}(\tilde{h}\nabla\tilde{u})}}{2}{(\R^d)}^2
			\\&\quad
			\lesssim
			\tilde{h}^2\normLp{\tilde{h}D_{\tilde{x}_1}\tilde{u}}{2}{(A)}^2
			+\tilde{h}^2 \normLp{(a_{ij}(\varepsilon\tilde{x}_1,\varepsilon\tilde{x}'+\varepsilon k))_{i,j}(\tilde{h}\nabla\tilde{u})}{2}{(\tilde{A})}
			\\&\quad
			\lesssim \tilde{h}^2 \left(\normLp{\tilde{u}}{2}{\left(\tilde{A}\right)}^2+\normLp{\tilde{P}\tilde{u}}{2}{\left(\tilde{A}\right)}^2\right)
			\\&\quad
			\lesssim \tilde{h}^2 \left(\normLp{\chi_{\tilde{A}}\tilde{u}}{2}{(\R^d)}^2+\normLp{\chi_{\tilde{A}}\tilde{P}\tilde{u}}{2}{(\R^d)}^2\right).
		\end{align*}
		Hence,
		\begin{align*}
			\normLp{\tilde{P}\chi_A \tilde{u}}{2}{(\R^d)} 
			&\lesssim \normLp{\chi_{\tilde{A}}\tilde{P}\tilde{u}}{2}{(\R^d)} + \tilde{h}\left(\normLp{\chi_{\tilde{A}}\tilde{P}\tilde{u}}{2}{(\R^d)} + \normLp{\chi_{\tilde{A}}\tilde{u}}{2}{(\R^d)}\right)
			\\&\lesssim \normLp{\chi_{\tilde{A}}\tilde{P}\tilde{u}}{2}{(\R^d)} +\tilde{h}\normLp{\chi_{\tilde{A}}\tilde{u}}{2}{(\R^d)}
			.
		\end{align*}
		% fin comprehension pour 1 fct
		%
		\item[$\qquad\bullet$] Let us now extend the result to density matrices. 
		We have shown the inequality of operators
		\begin{equation*}
			\left(\tilde{P}\chi_A\right)^*\left(\tilde{P}\chi_A\right) \lesssim \tilde{h}^2
			\chi_{\tilde{A}}^2 + \left(\chi_{\tilde{A}}\tilde{P}\right)^*\chi_{\tilde{A}}\tilde{P}
			.
		\end{equation*}
		In other words
		\begin{equation*}
			 \chi_A\tilde{P}^*\tilde{P}\chi_A \lesssim\tilde{h}^2 \chi_{\tilde{A}}^2+\tilde{P}^*\chi_{\tilde{A}}^2\tilde{P} .
		\end{equation*}
		Then for all $\alpha\geq 1$ and all bounded self-adjoint non-negative operator $\tilde{\gamma}$ on $L^2(\R^d)$
		\begin{equation*}
			\normSch{\sqrt{\tilde{\gamma}}\chi_A(1+\tilde{P}^*\tilde{P}/\tilde{h}^2)\chi_A\sqrt{\tilde{\gamma}}}{\alpha}{}\lesssim \normSch{\sqrt{\tilde{\gamma}}\chi_{\tilde{A}}^2\sqrt{\tilde{\gamma}}}{\alpha}{} +\frac 1{\tilde{h}^2}\normSch{\sqrt{\tilde{\gamma}}\tilde{P}^*\chi_{\tilde{A}}^2\tilde{P}\sqrt{\tilde{\gamma}}}{\alpha}{}.
		\end{equation*}
		This concludes the proof of Fact \ref{fact:TP-matdens-crucial-zone-inv-P&chi_A}.
	\end{proof}
	% fin demo fait inversion P and chi_A \ref{fact:TP-matdens-crucial-zone-inv-P&chi_A}

	\paragraph[Step 2]{Step 2. The summation of the boxes.}
	
	With the results on boxes of size $\varepsilon$, we will now have the following result.
	
	\begin{lemma}\label{lemma:TP-matdens-crucial-strip}
		Let $d\geq 2$ and $2\leq q\leq \frac{2d}{d-2}$. Then, there exists $C>0$ such that for any $\varepsilon\in [M h^{2/3},\delta/2]$, for any $h\in(0,h_0]$, and for any bounded self-adjoint non-negative operator $\gamma$ on $L^2(\R^d)$ we have
		\begin{equation*}\label{eq:TP-matdens-crucial-strip}
			\normLp{\rho_{\chi^\w\gamma\chi^\w}}{q/2}{(A_\varepsilon)} \leq C
			h^{-2s_{\text{Sogge}}(q,d)}\varepsilon^{1/2-2\mu(q,d)} \normSch{\left(1+\frac{1}{h^2}P^*P\right)^{1/2}\gamma\left(1+\frac{1}{h^2}P^*P\right)^{1/2}}{\alpha_{\text{Sogge}}(q,d)}{}
			,
		\end{equation*}
		where $\mu(q,d)$ is given by \eqref{def:TP-mu}.
	\end{lemma}
	% fin lemme \ref{lemma:TP-matdens-crucial-strip}
	
	\begin{proof}[\underline{Proof of Lemma \ref{lemma:TP-matdens-crucial-strip}}]
		Since $\pi_x\supp\chi\subset B_{\delta/2}$, we have for all $2\leq q\leq\infty$
		\begin{equation*}
			\normLp{\rho_{\chi^\w\gamma\chi^\w}}{q/2}{(B_\delta^c)}=\OR(h^\infty) \normSch{\gamma}{q/2}{} .
		\end{equation*}
		By Lemma \ref{lemma:TP-matdens-crucial-zone-boxes}, we have for non-negative operators $\gamma$
		\begin{align*}&
			\normLp{\rho_{\chi^\w\gamma\chi^\w}}{q/2}{(A_\varepsilon\cap B_{\delta})}
			%
			%\\&\quad
			= \left(\sum_{|k|_{l^\infty}\leq\delta/\varepsilon-1}\normLp{\rho_{\chi^\w\gamma\chi^\w}}{q/2}{(A_\varepsilon^k)}^{q/2}\right)^{2/q}
			\\&\quad´
			\lesssim h^{-2s_{\text{Sogge}}(q,d)} \varepsilon^{-2\mu(q,d)}
			\times\\&\quad\quad\times
			\underset{ 
				= \norm{ \chi_{\tilde{A}_\varepsilon^k}\chi^\w\gamma\chi^\w \chi_{\tilde{A}_\varepsilon^k} }_{l_k^{q/2}\schatten^{\alpha_\text{Sogge}(q,d)}}
			 }{\underbrace{
					\left( \sum_{k\in\Z^{d-1}} \normSch{ \chi_{\tilde{A}_\varepsilon^k}\chi^\w\gamma\chi^\w \chi_{\tilde{A}_\varepsilon^k}}{\alpha_\text{Sogge}(q,d)}{}^{q/2}
					\right)^{2/q}
			}}
			+ \frac{\varepsilon}{h^2}
			\underset{
				= \norm{ \chi_{\tilde{A}_\varepsilon^k}P\chi^\w\gamma\chi^\w P^*\chi_{\tilde{A}_\varepsilon^k} }_{l_k^{q/2}\schatten^{\alpha_\text{Sogge}(q,d)}}
			}{\underbrace{
				\left( \sum_{k\in\Z^{d-1}}\normSch{ \chi_{\tilde{A}_\varepsilon^k}P\chi^\w\gamma \chi^\w P^* \chi_{\tilde{A}_\varepsilon^k}}{\alpha_\text{Sogge}(q,d)}{}^{q/2}\right)^{2/q}
			}}
			.
		\end{align*}
		\item[$\triangleright$]
		The first step consists in showing 	that for $q\geq 2$, for any bounded self-adjoint operator $\gamma$
		\begin{equation*}
		\begin{split}
			&\norm{ \chi_{\tilde{A}_\varepsilon^k}\chi^\w\gamma\chi^\w\chi_{\tilde{A}_\varepsilon^k} }_{l_k^{q/2}\schatten^{\alpha_{\text{Sogge}}(q,d)}} +\frac\varepsilon{h^2}\norm{ \chi_{\tilde{A}_\varepsilon^k}P\chi^\w\gamma\chi^\w P^*\chi_{\tilde{A}_\varepsilon^k} }_{l_k^{q/2}\schatten^{\alpha_{\text{Sogge}}(q,d)}} 
			\\&\qquad\quad\lesssim \normSch{ \chi_{2\varepsilon}\chi^\w\gamma\chi^\w\chi_{2\varepsilon} }{\alpha_{\text{Sogge}}(q,d)}{} +\frac\varepsilon{h^2}\normSch{ \chi_{2\varepsilon}P\chi^\w\gamma\chi^\w P^*\chi_{2\varepsilon} }{\alpha_{\text{Sogge}}(q,d)}{}
			.
		\end{split}
		\end{equation*}
		% but
		In fact, we prove more precisely that for all non-negative compact operator $\Gamma$
		\begin{equation*}
			\sum_{k\in\Z^{d-1}} \normSch{\chi_{\tilde{A}_\varepsilon^k}\Gamma\chi_{\tilde{A}_\varepsilon^k}}{\alpha_{\text{Sogge}}(q,d)}{}^{q/2} \leq C\normSch{\chi_{2\varepsilon}\Gamma\chi_{2\varepsilon}}{\alpha_{\text{Sogge}}(q,d)}{}^{q/2} .
		\end{equation*}
		% simplification
		Recall that $\supp\chi_{\tilde{A}^k_\varepsilon}\subset\tilde{\tilde{A}}^k_\varepsilon\times\R^d$, 
		$\chi_{2\varepsilon}(x_1)=1$ when $x_1\leq 2\varepsilon$.
		Hence, we have
		\begin{equation*}
			\chi_{\tilde{A}^k_\varepsilon}=\chi_{\tilde{A}^k_\varepsilon}\chi_{2\varepsilon}
		\end{equation*}
		and thus we only have to show that 
		\begin{equation*}
			 \sum_{k\in\Z^{d-1}} \normSch{\chi_{\tilde{A}_\varepsilon^k}\Gamma\chi_{\tilde{A}_\varepsilon^k}}{\alpha_{\text{Sogge}}(q,d)}{}^{q/2} \leq C\normSch{\Gamma}{\alpha_{\text{Sogge}}(q,d)}{}^{q/2} .
		\end{equation*}
		%%%%% l^1G^1 < G^1 %%%%%
		%
		\item[\quad 1)]
		Let us check first that
		\begin{equation*}
			\norm{ \chi_{\tilde{A}_\varepsilon^k}\Gamma \chi_{\tilde{A}_\varepsilon^k} }_{l_k^1\schatten^1}\lesssim
			\normSch{\Gamma}{1}{}
			.
		\end{equation*}
		Since $\chi_{\tilde{A}^k_\varepsilon}\in\CR^\infty(\R^d,[0,1])$ and $\supp\chi_{\tilde{A}^k_\varepsilon}\subset\tilde{\tilde{A}}^k_\varepsilon$, there exists $C>0$ such that for all $\varepsilon>0$,
		\begin{equation*}
			\sum_{k\in\Z^{d-1}}\chi_{\tilde{A}^k_\varepsilon}^2 \leq C
			.
		\end{equation*}
		We recall that:
		\begin{itemize}
			\item[$\bullet$] if we have two non-negative operators $A$ and $B$ such that $A\leq B$ then $\sqrt{A}\leq\sqrt{B}$,
			\item[$\bullet$] for all trace-class operators $A$ and $B$, $A\leq B \: \Longrightarrow \: \tr(A)\leq\tr(B)$
			.
		\end{itemize}
		So if $\Gamma$ is a non-negative operator
		\begin{align*}
			\norm{\chi_{A_\varepsilon^k}\Gamma\chi_{A_\varepsilon^k}}_{l_k^1\schatten^1}
			&= \sum_{k\in\Z^{d-1}}\tr_{L^2}\left(\chi_{A_\varepsilon^k}\Gamma\chi_{A_\varepsilon^k}\right)
			\\&= \sum_{k\in\Z^{d-1}} \tr_{L^2}\left(\sqrt{\Gamma}\chi_{\tilde{A}_\varepsilon^k}^2\sqrt{\Gamma}\right)
			\\&\leq \tr_{L^2}\left(\sqrt{\Gamma} \sum_{k\in\Z^{d-1}}\chi_{\tilde{A}_\varepsilon^k}^2\sqrt{\Gamma}\right)
			\\&\leq C \tr_{L^2}\Gamma
			=C\normSch{\Gamma}{1}{}
			.
		\end{align*}
		We can pass to a general trace-class $\Gamma$ by decomposing $\Gamma= \Gamma_+-\Gamma_-$ with $\Gamma_+,\Gamma_-\geq 0$ and we obtain
		\begin{align*}
			\norm{\chi_{A_\varepsilon^k}\Gamma\chi_{A_\varepsilon^k}}_{l_k^1\schatten^1}
			&\leq \norm{\chi_{A_\varepsilon^k}\Gamma_+\chi_{A_\varepsilon^k}}_{l_k^1\schatten^1}+\norm{\chi_{A_\varepsilon^k}\Gamma_-\chi_{A_\varepsilon^k}}_{l_k^1\schatten^1}
			\\&\lesssim \normSch{\Gamma_+}{1}{} + \normSch{\Gamma_-}{1}{}
			\\&\lesssim \normSch{\Gamma}{1}{}
			.
		\end{align*}
		%
		%%%%% l^2G^2 < G^2 %%%%%
		%
		\item[\quad 2)]
		Notice that we always have
		\begin{equation*}
			\norm{  \chi_{\tilde{A}_\varepsilon^k}\Gamma \chi_{\tilde{A}_\varepsilon^k} }_{l_k^\infty\schatten^\infty}\lesssim
			\norm{\Gamma}_{L^2\to L^2}
			.
		\end{equation*}
		%
		%%%% l^{d+1/d-1}G^{d+1 /d} < G^{d+1 /d} %%%%%%
		%
		\item[\quad 3)]
		The interpolation of
		\begin{equation*}
			\begin{cases}
			\norm{  \chi_{\tilde{A}_\varepsilon^k}\Gamma \chi_{\tilde{A}_\varepsilon^k} }_{l_k^1\schatten^1} &\lesssim \normSch{\Gamma}{1}{}
			\\
			\norm{  \chi_{\tilde{A}_\varepsilon^k}\Gamma \chi_{\tilde{A}_\varepsilon^k} }_{l_k^\infty\schatten^\infty} &\lesssim \norm{\Gamma}_{L^2\to L^2}
			\end{cases} 
		\end{equation*}
		gives
		\begin{equation*}
			\forall 1\leq\alpha\leq\infty \quad
			\norm{  \chi_{\tilde{A}_\varepsilon^k}\Gamma \chi_{\tilde{A}_\varepsilon^k} }_{l_k^{\alpha}\schatten^{\alpha}} \lesssim \normSch{\Gamma}{\alpha}{} .
		\end{equation*}
		Hence, for any $2\leq q\leq\infty$ and any $1\leq\alpha\leq q/2$ 
		\begin{align*}
			\norm{  \chi_{\tilde{A}_\varepsilon^k}\Gamma \chi_{\tilde{A}_\varepsilon^k} }_{l_k^{q/2}\schatten^\alpha} 
			& \leq 	\norm{  \chi_{\tilde{A}_\varepsilon^k}\Gamma \chi_{\tilde{A}_\varepsilon^k} }_{l_k^\alpha\schatten^\alpha}
			\\& \lesssim \normSch{\Gamma}{\alpha}{}.
		\end{align*}
		Now for any $2\leq q\leq\infty$, we have  $1\leq \alpha_{\text{Sogge}}(q,d)\leq q/2$, hence  
		for any $2\leq q\leq \frac{2d}{d-2}$ we have
		\begin{align*}
			&\normLp{\rho_{\chi^\w\gamma\chi^\w}}{q/2}{(A_\varepsilon\cap B_\delta)} \\&\quad
			\lesssim
			h^{-2s_{\text{Sogge}}(q,d)}\varepsilon^{-2\mu(q,d)}
			\left(\normSch{ \chi_{2\varepsilon}\chi^\w\gamma\chi^\w\chi_{2\varepsilon} }{\alpha_{\text{Sogge}}(q,d)}{} +\frac\varepsilon{h^2}\normSch{ \chi_{2\varepsilon}P\chi^\w\gamma\chi^\w P^*\chi_{2\varepsilon} }{\alpha_{\text{Sogge}}(q,d)}{}\right)
			.
		\end{align*}
		\item[$\triangleright$]
		By the Hölder inequality and Corollary \ref{lemma:TP-matdens-3}, we have for any $\alpha\geq1$,
		\begin{align*}
			\normSch{\chi_{2\varepsilon}\chi^\w\gamma\chi^\w\chi_{2\varepsilon}}{\alpha}{}
			&\lesssim \varepsilon^{1/2} \normSch{\left(1+P^*P/h^2\right)^{1/2}\gamma\left(1+P^*P/h^2\right)^{1/2}}{\alpha}{}.
		\end{align*}
		Besides for any $\alpha\geq 1$,
		\begin{align*}
			\normSch{\chi_{2\varepsilon}P\chi^\w\gamma \chi^\w P^*\chi_{2\varepsilon}}{\alpha_{\text{Sogge}}(q,d)}{}
			&\lesssim \normSch{P\gamma P^*}{\alpha_{\text{Sogge}}(q,d)}{}
			\\&\lesssim  \normSch{\left(1+P^*P/h^2\right)^{1/2}\gamma\left(1+P^*P/h^2\right)^{1/2}}{\alpha}{}.
		\end{align*}
		Thus, with the triangle inequality for any $2\leq q\leq \frac{2d}{d-2}$ 
		\begin{align*}&
			\normLp{\rho_{\chi^\w\gamma\chi^\w}}{q/2}{(A_\varepsilon\cap B_\delta)}
			\\&\quad \lesssim
			h^{-2s_{\text{Sogge}}(q,d)}\varepsilon^{-2\mu(q,d)} \left( \normSch{\chi_{2\varepsilon}\chi^\w\gamma\chi^\w\chi_{2\varepsilon}}{\alpha_{\text{Sogge}}(q,d)}{}
			+
			\frac{\varepsilon}{h^2}\normSch{\chi_{2\varepsilon}P\chi^\w\gamma \chi^\w P^*\chi_{2\varepsilon}}{\alpha_{\text{Sogge}}(q,d)}{}\right)
			\\&\quad\lesssim h^{-2s_{\text{Sogge}}(q,d)}\varepsilon^{1/2-2\mu(q,d)} \normSch{\left(1+P^*P /h^2\right)^{1/2}\gamma\left(1+P^*P/h^2\right)^{1/2}}{\alpha_{\text{Sogge}}(q,d)}{}
			.
		\end{align*}
		That finishes the proof of Lemma \ref{lemma:TP-matdens-crucial-strip}.
	\end{proof}
	% fin demo lemme \ref{lemma:TP-matdens-crucial-strip}
	
	\paragraph[Step 3]{Step 3. The final summation.}
	
	Finally, by Lemma \ref{lemma:TP-matdens-crucial-strip} we are in position to obtain the estimates
	\begin{align*}
		\normLp{\rho_{\chi^\w\gamma\chi^\w}}{q/2}{(\Omega_{\delta/2}\setminus\Omega_{Mh^{2/3}})}
		&\leq \normLp{\rho_{\chi^\w\gamma\chi^\w}}{q/2}{\left( \cup_k A_{2^kh^{2/3}} \right)}
		\\&\leq \left( \sum_{k=\lfloor\log_2(M)\rfloor}^{K(h)} 	\normLp{\rho_{\chi^\w\gamma\chi^\w}}{q/2}{\left(A_{2^kh^{2/3}}\right)}^{q/2} \right)^{2/q}
		\\&\lesssim \left( \sum_{k=\lfloor\log_2(M)\rfloor}^{K(h)} C_{h,2^kh^{2/3}}^{q} \right)^{2/q}
		\normSch{\left(1+P^*P/h^2\right)^{1/2}\gamma\left(1+P^*P/h^2\right)^{1/2}}{\alpha_{\text{Sogge}}(q,d)}{}
		,
	\end{align*}
	where 
	 $C_{h,\varepsilon}=
	h^{-s_{\text{Sogge}}(q,d)}\varepsilon^{1/4-\mu(q,d)}$. Hence
	\begin{equation*}
		 \normLp{\rho_{\gamma}}{q/2}{(\Omega_{\delta/2}\setminus\Omega_{Mh^{2/3}})} \lesssim C_h^2 \normSch{\left(1+P^*P/h^2\right)^{1/2}\gamma\left(1+P^*P/h^2\right)^{1/2}}{\alpha_{\text{Sogge}}(q,d)}{}
	\end{equation*}
	where
	\begin{align*}
		C_h
		&=  \left( \sum_{k=\lfloor\log_2(M)\rfloor}^{K(h)} C_{h,2^kh^{2/3}}^q \right)^{1/q}
		\\&= \left( \sum_{k=\lfloor\log_2(M)\rfloor}^{K(h)} h^{-q(s_{\text{Sogge}}(q,d)+2/3(\mu(q,d)-1/4)} 2^{qk(1/4-\mu(q,d))} \right)^{1/q}
		\\&\lesssim \left\lbrace\begin{array}{lll}
		h^{-\frac{(d-1)}{2}\left(\frac 1 2-\frac 1 q\right)}&\text{ if } 2\leq q<\frac{2(d+3)}{d+1},
		\\ \log^{\frac{d+1}{2(d+3)}}(1/h)h^{-\frac{d-1}{2(d+3)}}  &\text{ if } q=\frac{2(d+3)}{d+1},
		\\h^{\frac 1 6 -\frac{2d}{3}\left(\frac 1 2 -\frac 1 q\right)} &\text{ if } \frac{2(d+3)}{d+1}<q \leq\frac{2d}{d-2}.
		\end{array}\right.
	\end{align*}
	That ends the proof of Lemma \ref{lemma:TP-matdens-crucial-zone}.
	
	%\newpage
	\section{Applications to spectral clusters}\label{sec:app-spectral-clusters}
	
	In this section, we apply the results of the preceding sections on microlocalized quasimodes to spectral clusters. As we will see, this allows to get rid the microlocalization and leads to global estimates.
	
	\subsection{Notation}
	
	\begin{itemize}
		\item
		Let $\varepsilon\in(0,1)$ and $h_0\in(0,\varepsilon/2)$.
		\item
		Let $E\in\R$ and $h\in(0,h_0]$.
		\item
		Let  $I_{h,E}:=[E-h,E+h]$.
		\item
		Let $t_{\text{gene}}$, $s_{\text{gene}}\geq 0$ and $\alpha_{\text{gene}}\geq 1$ be given by the formulas in the statement of Theorem \ref{thm:ELp-gene-matdens}.
		\item
		Let $s_{\text{Sogge}}\geq 0$ and $\alpha_{\text{Sogge}}\geq 1$ be given by the formulas in the statement of Theorem \ref{thm:ELp-sogge-matdens}.
		\item
		Let $t_{\text{TP}}$, $s_{\text{TP}}\geq 0$ and $\alpha_{\text{TP}}\geq 1$ be given by the formulas in the statement of Theorem \ref{thm:ELp-TP-matdens}.
	\end{itemize}

	\subsection{Statement of the results}
	
	Let us add an addition assumption on the potential $V$, that will implies that the operator $-h^2\Delta+V$ has a compact resolvent. Moreover, it also ensures that $V$ is bounded from below.
	
	\begin{defi}[Polynomial growth]\label{cond:potential-pol-growth}
		A potential $V\in\CR^\infty(\R^d,\R)$ has a \emph{polynomial growth} if it satisfies Definition \eqref{cond:am-potential-pol-growth} and if there exist $k\in\N^*$ and $R>0$ such that
		\begin{equation}
			\label{eq:cond-potential-conf}
			\forall x\in\R^d,\:\forall\abs{x}\geq R,\quad V(x) \geq c\crochetjap{x}^k
		\end{equation}
	\end{defi}

	\begin{thm}[Spectral cluster upper bounds]\label{thm:spectral-clusters}
		\begin{itemize}
			\item[(i)]
		Let $d\geq 1$.
		Let $p(x,\xi)=\abs{\xi}^2+V(x)$ with $V\in\CR^\infty(\R^d,\R)$ with a polynomial growth (Definition \ref{cond:potential-pol-growth}).
		For $h>0$ and $E\in\R$, let us define $P:=p^\w(x,hD)$ and the spectral projector $\Pi_h$ by
		\begin{equation*}
		\Pi_h := \indicatrice{P\in I_{h,E}}.
		\end{equation*}
		Let $E\in\R$.
		Then, there exist $C>0$ and $h_0>0$ such that, for any $0< h\leq h_0$, any $2\leq q\leq\infty$ and any bounded self-adjoint non-negative operator $\gamma$ on $L^2(\R^d)$ 
		\begin{equation}\label{eq:spectral-clusters}
			\normLp{\rho_{\Pi_h\gamma\Pi_h}}{q/2}{(\R^d)} \leq C \log(1/h)^{2t_{\text{gene}}(q,d)}h^{-2s_{\text{gene}}(q,d)}\normSch{\gamma}{\alpha_{\text{gene}}(q,d)}{(L^2(\R^d))}
		,
		\end{equation}
			\item[(ii)]
		Let $d\geq 2$. There exist $C>0$ and $h_0>0$ such that, for any $0< h\leq h_0$, any $2\leq q\leq\infty$ and any bounded self-adjoint non-negative operator $\gamma$ on $L^2(\R^d)$
		\begin{equation}\label{eq:spectral-clusters-loc-far-tp}
			\normLp{\rho_{\Pi_h\gamma\Pi_h}}{q/2}{\left(\{x\in\R^d \: :\: \abs{V(x)-E}>\varepsilon\}\right)} \leq Ch^{-2s_{\text{Sogge}}(q,d)}\normSch{\gamma}{\alpha_{\text{Sogge}}(q,d)}{} 
			,
		\end{equation}
			\item[(iii)]
		Let $d\geq 2$.
		Under the additional assumption that
		\begin{equation}\label{cond:TP-potent}
			\forall x\in\R^d,\quad V(x)=E \quad\Longrightarrow\quad \nabla_x V(x)\neq 0,
		\end{equation}
		there exist $C>0$ and $h_0>0$ such that, for any $0< h\leq h_0$, any $2\leq q\leq\infty$ and any bounded self-adjoint non-negative operator $\gamma$ on $L^2(\R^d)$
		\begin{equation}\label{eq:spectral-clusters-loc-tp}
			\normLp{\rho_{\Pi_h\gamma\Pi_h}}{q/2}{\left(\{x\in\R^d \: :\: \abs{V(x)-E}\leq\varepsilon\}\right)}
			%\\&\quad
			\leq C \log(1/
			h)^{2t_{\text{TP}}(q,d)} h^{-2s_{\text{TP}}(q,d)}  \normSch{\gamma}{\alpha_{\text{TP}}(q,d)}{} 
			.
		\end{equation}
		\end{itemize}
	\end{thm}
	% fin \ref{thm:spectral-clusters-loc}
	
	\begin{rmk}\label{rmk:spectral-clusters_unif-C}
		The constant $C$ appearing in the above theorem can be chosen to be uniform in the energy level $E$ when it varies in a compact set.
	\end{rmk}
	
	\begin{rmk}\label{rmk:spectral-clusters_unif-forbidden-zone}
		One can split the $L^{q/2}$-norm of $\rho_{\Pi_h\gamma\Pi_h}$ on $\{x\in\R^d \: :\: \abs{V(x)-E}>\varepsilon\}$ into two parts: on the the classically allowed region
		$\{x\in\R^d \: :\: V(x)< E-\varepsilon\}$ and on the classically forbidden region $\{x\in\R^d \: :\: V(x)> E+\varepsilon\}$.
		When $d\geq 2$, there exist $C>0$ and $h_0>0$ such that for any $0<h\leq h_0$, any $2\leq q\leq\infty$ and bounded self-adjoint non-negative operator $\gamma$ on $L^2(\R^d)$,
		\begin{equation}\label{eq:spectral-clusters-loc-allow}
			\normLp{\rho_{\Pi_h\gamma\Pi_h}}{q/2}{\left(\{x\in\R^d \: : \: V(x)-E< -\varepsilon\}\right)} \leq C h^{-2s_{\text{Sogge}}(q,d)}\normSch{\gamma}{\alpha_{\text{Sogge}}(q,d)}{(L^2(\R^d))}
			.
		\end{equation}
		We will see below that, as one can expect, for any $d\geq 1$, there exists $C,c>0$ such that for any $h\in(0,h_0]$, any $2\leq q\leq\infty$ and any bounded operator $\gamma$ 
		\begin{equation}\label{eq:spectral-clusters-loc-forbidd_improv}
			\normLp{\rho_{\Pi_h\gamma\Pi_h}}{q/2}{\left(\{x\in\R^d \: : \: V(x)-E> \varepsilon\}\right)} =\OR(h^\infty )(e^{-c/h})^\frac{1}{q}\norm{\gamma}_{L^2\to L^2}.
		\end{equation}
		We do not know if one can improve the $\OR(h^\infty(e^{-c/h})^{1/q})$ to $\OR( e^{-c/h})$.
	%% c.f. DETAILS PhD
	\end{rmk}
	
%% CORRECTION RMk 6, %% VIRER	
	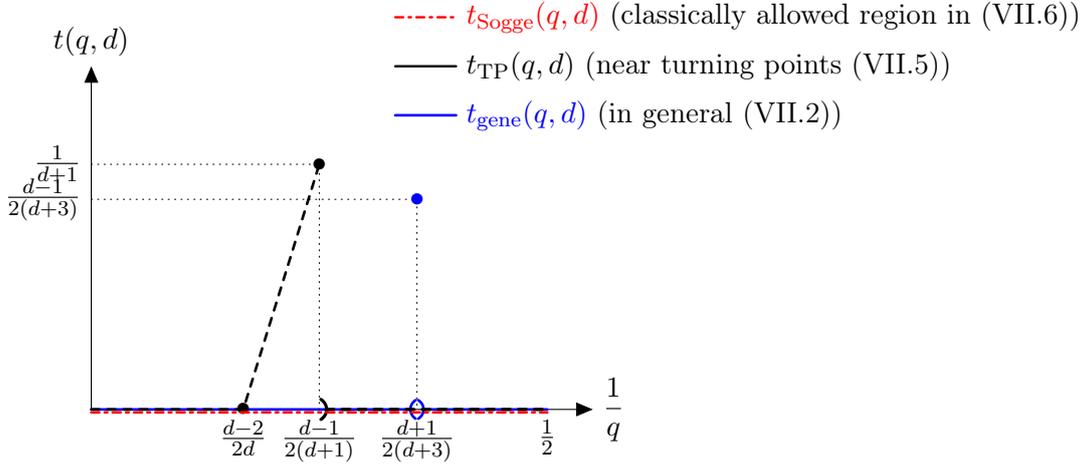
\begin{figure}[!h]
		\begin{center}\begin{tikzpicture}[line cap=round,line join=round,>=triangle 45,x=6 cm,y=6.5 cm,scale=2.]
			d=3
			\draw[->] (0.,0.) -- (0.55,0.);
			\draw (0.55,0.) node[right] {$\displaystyle{\frac{1}{q}}$};
			\draw[->] (0.,0.) -- (0.,0.35);
			\draw (0.,0.35) node[above] {$\displaystyle{t(q,d)}$};
			%
			% t(q,d) log TP
			\draw (5/14,0.) node[below]{$\frac{d+1}{2(d+3)}$};
			\draw (0.,3/14) node[left]{$\frac{d-1}{2(d+3)}$};
			\draw[dotted] (0., 3/14)-|(5/14, 0.);
			\draw[color=blue] (5/14,3/14) node {$\bullet$};
			\draw[-(,color=blue, line width=1.pt] (0,0)--(5/14,0);
			\draw[)-,color=blue, line width=1.pt] (5/14,0)--(1/2,0);
			%
			% t(q,d) gene
			\draw (1/6,0.) node[below]{$\frac{d-2}{2d}$};
			\draw (0.5,0.) node[below]{$\frac{1}{2}$};
			\draw (1/6,0.) node {$\bullet$};
			\draw[line width=1.pt, dashed] (0,0)--(1/6,0.)--(1/4,1/4);
			\draw (1/4,0.) node[below]{$\frac{d-1}{2(d+1)}$};
			\draw (0.,1/4) node[left]{$\frac{1}{d+1}$};
			\draw (1/4,1/4) node {$\bullet$};
			\draw[dotted] (0., 1/4)-|(1/4, 0.);
			\draw[)-,line width=1.2pt, dashed] (1/4,0)--(0.5,0.);
			\draw[line width=1.pt, dash dot,color=red] (0,-0.003)--(0.5,-0.003);
			%
			%% legendes
			% sogge
			\draw[line width=0.9pt,red, dash dot] (1/3,0.4) -- (1/3+1/15,0.4);
			\draw (1/3+1/15,0.4) node[right]{\textcolor{red}{$t_{\text{Sogge}}(q,d)$} (classically allowed region in \eqref{eq:spectral-clusters-loc-allow})};
			% TP
			\draw[line width=0.9pt] (1/3,0.35) -- (1/3+1/15,0.35);
			\draw (1/3+1/15,0.35) node[right]{$t_{\text{TP}}(q,d)$ (near turning points \eqref{eq:spectral-clusters-loc-tp})};
			% gene
			\draw[color=blue,line width=1.pt] (1/3,0.3) -- (1/3+1/15,0.3);
			\draw (1/3+1/15,0.3) node[right]{\textcolor{blue}{$t_{\text{gene}}(q,d)$} (in general \eqref{eq:spectral-clusters})};
			\end{tikzpicture}\end{center}
		%%%%%%%%%%%%%%%%%%%%%% $t(q,d)$ d>2 %%%%%%%%%%%%%%%%%%%%%%%%%%%%
		\caption{Concentration logarithm exponent $t(q,d)$ when $d\geq 3$.}
		\label{fig:comp-exp-t}
	\end{figure}

	\subsection{Proof of Theorem \ref{thm:spectral-clusters}}
	
	\paragraph{Microlocalization of $\gamma$.}
	
	We first explain why we only need to consider the microlocalized density matrix $\chi^\w\Pi_h\gamma\Pi_h\chi^\w$ for some $\chi\in\test{\R^d\times\R^d}$ instead of the full one $\Pi_h\gamma\Pi_h$.
	
	Let $f\in\test{\R,[0,1]}$ such that $f=1$ on $I_{h_0,E}$ and $\supp f\subset I_{\varepsilon/2,E}$.
	If we assume that $h\in(0, h_0]$, we thus have 
	\begin{equation*}
	 	\Pi_h = f(P)\Pi_h.
	\end{equation*}
	By fonctional calculus Theorem \ref{thm:funct-calculus}, $f(P)$ can be written as the Weyl quantization of a symbol $\tilde{\chi}\in\schwartz(\R^d\times\R^d)$. Furthermore, for all $N\in\N$, there exists $\tilde{\chi}_N\in\test{\R^d\times\R^d}$ and $r_N\in \schwartz(\R^d\times\R^d)$, such that $\supp\tilde{\chi}_N\subset\supp f\circ p$ (note that $\supp f\circ p$ is compact since $p(x,\xi)\to\infty$ when $\abs{(x,\xi)}\to\infty$) and 
	\begin{equation*}
		\tilde{\chi}(x,\xi) =\tilde{\chi}_N(x,\xi)+h^N r_N(x,\xi).
	\end{equation*}
	Let us write the decomposition
	\begin{equation}\label{eq:spectral-loc}
		\Pi_h = \tilde{\chi}_N^\w \Pi_h +h^N r_N^\w\Pi_h.
	\end{equation}
	One the one hand, by the fact that $r_N^\w=\OR_{\schwartzprime\to\schwartz}(1)$ and Kato-Seiler-Simon  applied to $k\in\N$ when $q>2$ (resp. just an estimation of the operator norm of $(1-h^2\Delta)^{-k/2}$ and $W=1$ when $q=2$) such that $k(q/2)'>d$, for any $W\in L^{2(q/2)'}(\R^d)$
	\begin{align*}
		&
		\normSch{h^NWr_N^\w\Pi_h\sqrt{\gamma}}{2}{(L^2(\R^d))}
		\\&\quad
		\leq h^N \normSch{W(1-h^2\Delta)^{-k/2}}{2(q/2)'}{(L^2(\R^d))}\norm{(1-h^2\Delta)^{k/2}r_N^\w\Pi_h}_{L^2(\R^d)\to L^2(\R^d)}\normSch{\sqrt{\gamma}}{q}{(L^2(\R^d))}
		\\&\quad
		\lesssim h^{N-d/2}\normLp{W}{2(q/2)'}{(\R^d)}\normSch{\gamma}{q/2}{(L^2(\R^d))}^{1/2}
		.
	\end{align*}
	By duality and Mercer theorem, we deduce that for all $N\in\N$
	\begin{equation*}
		\normLp{\rho_{(1-\tilde{\chi}_N^\w)\Pi_h\gamma\Pi_h(1-\tilde{\chi}_N^\w)}}{q/2}{(\R^d)} \leq C h^{2N-d} \normSch{\gamma}{q/2}{(L^2(\R^d))}	,
	\end{equation*}
	hence, it remains to estimate $\rho_{\tilde{\chi}_N^\w\Pi_h\gamma\Pi_h\tilde{\chi}_N^\w}$ on various regions with perhaps additional assumptions on $V$.
	
	\paragraph{(o) Microlocalized estimates in the classically forbidden region.}
	Notice that since 
	\begin{equation*}
		\supp (f\circ p) \cap (\{x\in\R^d \: : \: V(x)-E> \varepsilon\}\times\R^d) =\emptyset,
	\end{equation*}
	we have
	\begin{equation*}
		\normLp{\rho_{\tilde{\chi}_N^\w\Pi_h\gamma\Pi_h\tilde{\chi}_N^\w}}{q/2}{\left(x\in\R^d \: : \: V(x)-E> \varepsilon\right)} =\OR(h^\infty)\normSch{\gamma}{q/2}{}.
	\end{equation*}
	By Weyl's law (Proposition \ref{prop:int-weyl-law}), we have
	\begin{align*}
		\normSch{\Pi_h\gamma\Pi_h}{q/2}{} 
		&\leq \norm{\gamma}_{L^2\to L^2}\normSch{\Pi_h}{q/2}{}
		\leq C h^{-2d/q} \norm{\gamma}_{L^2\to L^2}
		,
	\end{align*}
	and thus, we deduce the bound
	\begin{equation}\label{eq:spectral-clusters-loc-forbidd}
		\normLp{\rho_{\Pi_h\gamma\Pi_h}}{q/2}{\left(\{x\in\R^d \: : \: V(x)-E> \varepsilon\}\right)} =\OR(h^\infty)\norm{\gamma}_{L^2\to L^2}.
	\end{equation}
	Let a write how we can improve this bound in forbidden region. Let us write $\Pi_h=\sum_{j\in J_h}\bra{u_j}\ket{u_j}$ with $\{u_j\}_{1\leq j\leq N_h}$ the orthonormal basis of eigenfunctions of $P$ associated to eigenfunctions in $[E-h,E+h]$. Here, $N_h=\tr(\Pi_h)=\rk(\Pi_h(L^2(\R^d)))$. By definition, there exists $\{\nu_j\}_{1\leq j\leq N_h}\subset\C$ such that $\Pi_h\gamma\Pi_h =\sum_{1\leq j\leq N_h}\nu_j\bra{u_j}\ket{u_j} $.
	Let us introduce the Agmon distance
	\begin{equation*}
		d_{\delta,E}:x\mapsto \delta\dist(x,\{x\in\R^d \: :\: V(x)<E+\delta\}).
	\end{equation*}
	By the Agmon estimates ((c.f. for instance \cite[Chap. 6]{dimassi1999spectral} or \cite[Prop. 2.3]{deleporte2021universality}), there exists $C=C(\delta,E)>0$  such that for any normalized eigenfunction $v$ of $P$ less or equal than $E_0$
	\begin{equation*}
	\normLp{e^{d_{\varepsilon,E}(x)/h}v}{2}{(\R^d)}\leq C
	.
	\end{equation*}
	Recall that $h_0\in(0,\varepsilon/2)$. One has, by the triangle inequality, the Agmon estimates and the Weyl's law (Theorem \ref{prop:int-weyl-law}), there exists $C'=C'(h_0,E)>0$
	\begin{align*}
	&\normLp{\rho_{\Pi_h\gamma\Pi_h}}{1}{\left(x\in\R^d \: : \: V(x)>E+\varepsilon\right)}
	\\&\qquad
	\leq \sum_{j=1}^{N_h}\abs{\nu_j} \normLp{u_j}{2}{\left(x\in\R^d \: : \: V(x)>E+\varepsilon\right)}
	\\&\qquad
	\leq  \sup_{1\leq j\leq N_h}\abs{\nu_j}\sum_{j=1}^{N_h} \normLp{u_j}{2}{\left(x\in\R^d \: : \: V(x)>E+\varepsilon\right)}
	\\&\qquad
	\leq \norm{\Pi_h\gamma\Pi_h}_{L^2\to L^2} \normLp{e^{-d_{h_0,E+h_0}/h}}{\infty}{\left(x\in\R^d \: : \: V(x)>E+\varepsilon\right)}\sum_{j=1}^{N_h} \normLp{e^{d_{h_0,E+h_0}/h}u_j}{2}{(\R^d)}
	\\&\qquad
	\leq C'e^{-c/h}h^{-d}\norm{\gamma}_{L^2\to L^2}
	.
	\end{align*}
	Here $c=\min_{x\in\{V\leq E+\varepsilon\}} \dist({x,\{V\leq E+h_0\}})=\dist({\{V\leq E+\varepsilon\},\{V\leq E+h_0\}})>0$, given the choice of $h_0>0$.
	We interpolate then this estimate with the $L^\infty$ estimate \eqref{eq:spectral-clusters-loc-forbidd}
	\begin{equation*}
	\normLp{\rho_{\Pi_h\gamma\Pi_h}}{\infty}{\left(x\in\R^d \: : \: V(x)>E+\varepsilon\right)}=\OR(h^\infty)\norm{\gamma}_{L^2\to L^2}
	.
	\end{equation*}
	One has for any $q\in[2,\infty]$, for any $N>0$, there exists $C>0$ such that
	\begin{align*}
	\normLp{\rho_{\Pi_h\gamma\Pi_h}}{q/2}{\left(x\in\R^d \: : \: V(x)>E+\varepsilon\right)}
	&\leq C (e^{- c/h})^{\frac 1q} h^N\norm{\gamma}_{L^2\to L^2}.
	\end{align*}
	That proves \eqref{eq:spectral-clusters-loc-forbidd_improv}.
	
	\paragraph{(i) General microlocalized estimates.}
	
	Note that all $(x_0,\xi_0)\in\R^d\times\R^d$ satisfy the non-degeneracy Assumption \ref{cond:gene} for the symbol $p_E:= p-E$. By Theorem \ref{thm:ELp-gene-matdens}, Assumption \ref{ass:microloc} is thus satisfied for $S=\R^d\times\R^d$, $q\in[2,\infty]$, $s=s_{\text{gene}}(q,d)$, $t=0$ and $\alpha=\alpha_{\text{gene}(q,d)}$.
	We apply Theorem \ref{thm:abstract-microloc-extend} to these parameters $(q,s_{\text{gene}}(q,d),0,\alpha_{\text{gene}}(q,d))$ and to $\chi=\tilde{\chi}_N$, that gives us
	\begin{align*}
		\normLp{\rho_{\tilde{\chi}_N^\w\Pi_h\gamma\Pi_h\tilde{\chi}_N^\w}}{q/2}{(\R^d)}
		&\leq C\log(1/h)^{2t_{\text{gene}}(q,d)}h^{-2s_{\text{gene}}(q,d)}
		\times\\&\quad\times
		\left(\normSch{\Pi_h\gamma\Pi_h}{\alpha_{\text{gene}}(q,d)}{}+\frac{1}{h^2} \normSch{(P-E)\Pi_h\gamma\Pi_h (P-E)}{\alpha_{\text{gene}}(q,d)}{} \right)
		\\&\leq  C'\log(1/h)^{2t_{\text{gene}}(q,d)}h^{-2s_{\text{gene}}(q,d)}\normSch{\gamma}{\alpha_{\text{gene}}(q,d)}{}
		,
	\end{align*}
	which is exactly \eqref{eq:spectral-clusters}.

	\paragraph{(ii) Microlocalized estimates in the classically allowed region.}
		 
	Let $d\geq 2$.
	Let $S=\{(x,\xi)\in\R^d\times\R^d \: : \: \abs{V(x)-E}>\varepsilon\}$.
	Any $(x_0,\xi_0)\in S$ satisfies either the ellipticity condition $p_E(x_0,\xi_0)\neq 0$ or Assumption \ref{cond:sogge} for the symbol $p_E$.
	By Theorem \ref{thm:ELp-elliptic-matdens} and Theorem \ref{thm:ELp-sogge-matdens}, the set $S$ thus satisfies Assumption \ref{ass:microloc} for all $q\in[2,\infty]$, $s=s_{\text{Sogge}}(q,d)$, $t=0$ and $\alpha=\alpha_{\text{Sogge}}(q,d)$.
	We apply Theorem \ref{thm:abstract-loc-space} to these parameters,  $\Omega = \{x\in\R^d \: :\: \abs{V(x)-E}>\varepsilon\}$) and $\chi=\tilde{\chi}_N$. Then, there exist $C>0$ and $h_0>0$ such that for any $0< h\leq h_0$ and $2\leq q\leq\infty$
	\begin{equation*}
		\normLp{\rho_{\tilde{\chi}_N^\w\Pi_h\gamma\Pi_h\tilde{\chi}_N^\w}}{q/2}{(\{\abs{V-E}>\varepsilon \})}\leq C h^{-2s_{\text{Sogge}}(q,d)} \normSch{\gamma}{\alpha_{\text{Sogge}}(q,d)}{}  ,
	\end{equation*}
	where we got rid of the operator $P-E$ in the Schatten norm by the same method as in the previous step. We thus get  \eqref{eq:spectral-clusters-loc-far-tp}.
	
	\paragraph{(iii) Microlocalized estimates near the turning points.}
	Let $d\geq 2$.
	Let $S=\{(x,\xi)\in\R^d\times\R^d \: : \: \abs{V(x)-E}\leq\varepsilon\}$.
	Any $(x_0,\xi_0)\in S$ satisfies either the ellipticity condition $p_E(x_0,\xi_0)\neq 0$, Assumption \ref{cond:sogge}, or Assumption \ref{cond:TP-Sch} for the symbol $p_E$.
	By Theorem \ref{thm:ELp-elliptic-matdens}, Theorem \ref{thm:ELp-sogge-matdens}, and Theorem \ref{thm:ELp-TP-matdens}, the set $S$ thus satisfies Assumption \ref{ass:microloc} for all $q\in[2,\infty]$, $s=s_{\text{TP}}(q,d)$, $t=t_{\text{TP}}(q,d)$ and $\alpha=\alpha_{\text{TP}}(q,d)$.
	We apply Theorem \ref{thm:abstract-loc-space} to these parameters,  $\Omega = \{x\in\R^d \: :\: \abs{V(x)-E}\leq\varepsilon\}$) and $\chi=\tilde{\chi}_N$. As above, there exist $C>0$ and $h_0>0$ such that for any $0< h\leq h_0$ and $2\leq q\leq\infty$
	\begin{equation*}
		\normLp{\rho_{\tilde{\chi}_N^\w\Pi_h\gamma\Pi_h\tilde{\chi}_N^\w}}{q/2}{(\{\abs{V-E}\leq\varepsilon\})} \leq C \log(1/h)^{2t_{\text{TP}}(q,d)} h^{-2s_{\text{TP}}(q,d)} \normSch{\gamma}{\alpha_{\text{TP}}(q,d)}{} ,
	\end{equation*}
	showing \eqref{eq:spectral-clusters-loc-far-tp}.

	%\newpage
	\section{Optimality}\label{sec:optim}
	
	In this section, we discuss about the optimality of the concentration exponents $s(q,d)$ and $\alpha(q,d)$, appearing in the estimates of Theorem \ref{thm:spectral-clusters}. We first explain why the exponents $s(q,d)$ are sharp for all values of $q$ and $d$. To do so, we will see that it is enough to consider the one-body case $\rk\gamma=1$ for which only the exponent $s(q,d)$ appears. In many cases, this optimality was known in the literature, but we provide some details here. On the contrary, the optimality of the exponent $\alpha(q,d)$ is only proved in a restricted range of cases.

	\subsection{One-body optimality}

	Similarly as \cite{tacy2009notes}, that treats the optimality of the Laplacian in the manifolds and submanifolds, we write a survey of the one-body optimality in the Euclidean case. Actually, we will see that the understanding of the proof in Section \ref{subsec:zonal-quasimode} allows us to prove the sharpness of the spectral bound in the bulk. There are reasons to hope that other examples can help to prove the optimality of the exponent $\alpha$ in the other cases.

	Let $V$ satisfying Definition \ref{cond:potential-pol-growth}, $h>0$, $P=-h^2\Delta+V$ and $E\in\R$. In this section, we explain several concentration scenarii of functions $u_h$, which saturate the various one-body $L^q$ bounds. All the saturating scenarii happen in the bulk $\{V-E<-\varepsilon\}$, meaning that they satisfy lower bounds of the type
	\begin{equation}\label{eq:optim-1body_glob}
		\normLp{u_h}{q}{(\{V-E<-\varepsilon\})}\geq Ch^{-s(q,d)}\left(\normLp{u_h}{2}{(\R^d)}+\frac 1h\normLp{(P-E)u_h}{2}{(\R^d)}\right).
	\end{equation} 
	The different saturation scenarii according to the values of $s(q,d)$ are summarized in Figure \ref{fig:conc-opt-1body} and Figures \ref{fig:conc_gaussian-gs}, \ref{fig:conc_zonal-type-quasim} and  \ref{fig:conc_oh-gaussian-beams}. 
	The optimality of the estimates in the turning point region $\{\abs{V-E}\leq \varepsilon\}$ is much more delicate. In \cite{koch2005p}, the optimality in the case $V(x)=\abs{x}^2$ is proved for the estimates in the dyadic regions of size $2^jh^{2/3}$ (see the proof of Theorem \ref{thm:ELp-TP-matdens}) for fixed $j$.
	The optimality in the full region $\{\abs{V-E}\leq \varepsilon\}$ (that is, when we sum over $j$), as well as the optimality in the small neighborhood $\Omega_{Mh^{2/3}}$ of a turning point seem open to us in dimension $d\geq 2$ (see Remark \ref{rmk:TP-3_opt} for the discussion in dimension 1).
	Indeed, for small regime $2\leq q\leq 2(d+3)/(d-1)$, the $L^q$ bounds are saturated by Gaussiam beams (see Proposition \ref{prop:optim-lowr:ELp-sogge-1body}); for high regime $2d/(d-2)\leq q\leq\infty$, they are saturated by zonal-type quasimodes (see Proposition \ref{prop:optim-highr:ELp-gene-1body}). This is less obvious for the intermediate regime $2(d+3)/(d-1)\leq q\leq 2d/(d-2)$. By some raw calculations, the estimates cannot be saturated by simple direct products of $d$ scalar gaussian grounds states $\varphi_0(t)=(2\pi\sqrt{h})^{-1/2}e^{-\frac{t^2}{2h}}$ and Hermite functions $\varphi_n$ associated to the eigenvalue in a neighborhood of $E$ (as in Figure \ref{fig:oh-hr-eigenfunction}). At best, one could expect that  well-chosen linear combinaison of direct product scalar Hermite functions can saturate the bounds. 
	
		%%%%%%%%%%%%%%% DESSIN GAUSSIENNE %%%%%%%%%%%%%%%
	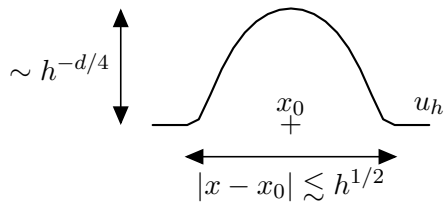
\begin{figure}[!h]	\centering
		\begin{tikzpicture}[line cap=round,line join=round,>=triangle 45,x=2.0cm,y=6.0cm,scale=0.7]
		\draw (0.,0.) node{$+$};
		\draw (0.,0.) node[above]{$x_0$};
		\draw[color=black, thick,  domain=-1.3:1.3]plot( \x,{exp(1/((\x)^2-1)))*(abs(\x)<1)})node[left,above]{$u_h$};
		\draw[<->,line width=0.6pt,color=black] (-1.,-0.1) -- (1.,-0.1) node[midway, below]{$\abs{x-x_0}\lesssim h^{1/2}$};%{$\sim h^{d/2}$};
		\draw[<->,line width=0.6pt,color=black] (-1.6,0.) -- (-1.6,{exp(-1)}) node[midway, left]{$\sim h^{-d/4}$};
		\end{tikzpicture}
		\caption{Concentration of a gaussian groundstate.}
		\label{fig:conc_gaussian-gs}
	\end{figure}
	%%%%%%%%%%%%%%% FIN DESSIN GAUSSIENNE %%%%%%%%%%%%%%%
	
	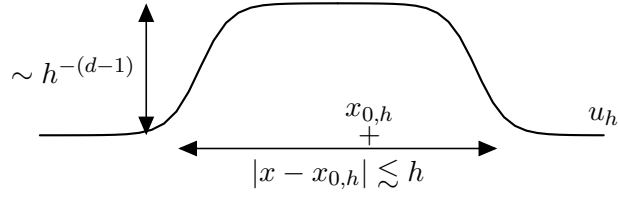
\begin{figure}[!h]	\centering
		%%%%%%%%%%%%%%% DESSIN SOGGE ZONAL %%%%%%%%%%%%%%%
		\begin{tikzpicture}[line cap=round,line join=round,>=triangle 45,x=2.0cm,y=2.5cm,scale=0.7]
		\draw (0.3,0.) node{$+$};
		\draw (0.3,0.) node[above]{$x_{0,h}$};
		\draw[ thick,  domain=-2.8:0]plot( \x,{ 0.5*(1+tanh(3.5*(\x+1.3))) });
		\draw[ thick,  domain=0:2.5]plot( \x,{ 0.5*(1-tanh(3.5*(\x-1.3))) })node[left,above]{$u_h$};
		\draw[<->,line width=0.6pt,color=black] (-1.5,-0.1) -- (1.5,-0.1) node[midway, below]{$\abs{x-x_{0,h}}\lesssim h$};%{$\sim h^{d}$};
		\draw[<->,line width=0.6pt,color=black] (-1.8,0.) -- (-1.8,1) node[midway, left]{$\sim h^{-(d-1)}$};
		\end{tikzpicture}
		\caption{Concentration of a zonal-type quasimode.}
		\label{fig:conc_zonal-type-quasim}
	\end{figure}
	%%%%%%%%%%%%%%% FIN SOGGE ZONAL %%%%%%%%%%%%%%%

	%%%%%%%%%%%%%%% GAUSSIAN BEAMS %%%%%%%%%%%%%%%
	\begin{figure}[h!]
		\centering
		\includegraphics[scale=0.2]{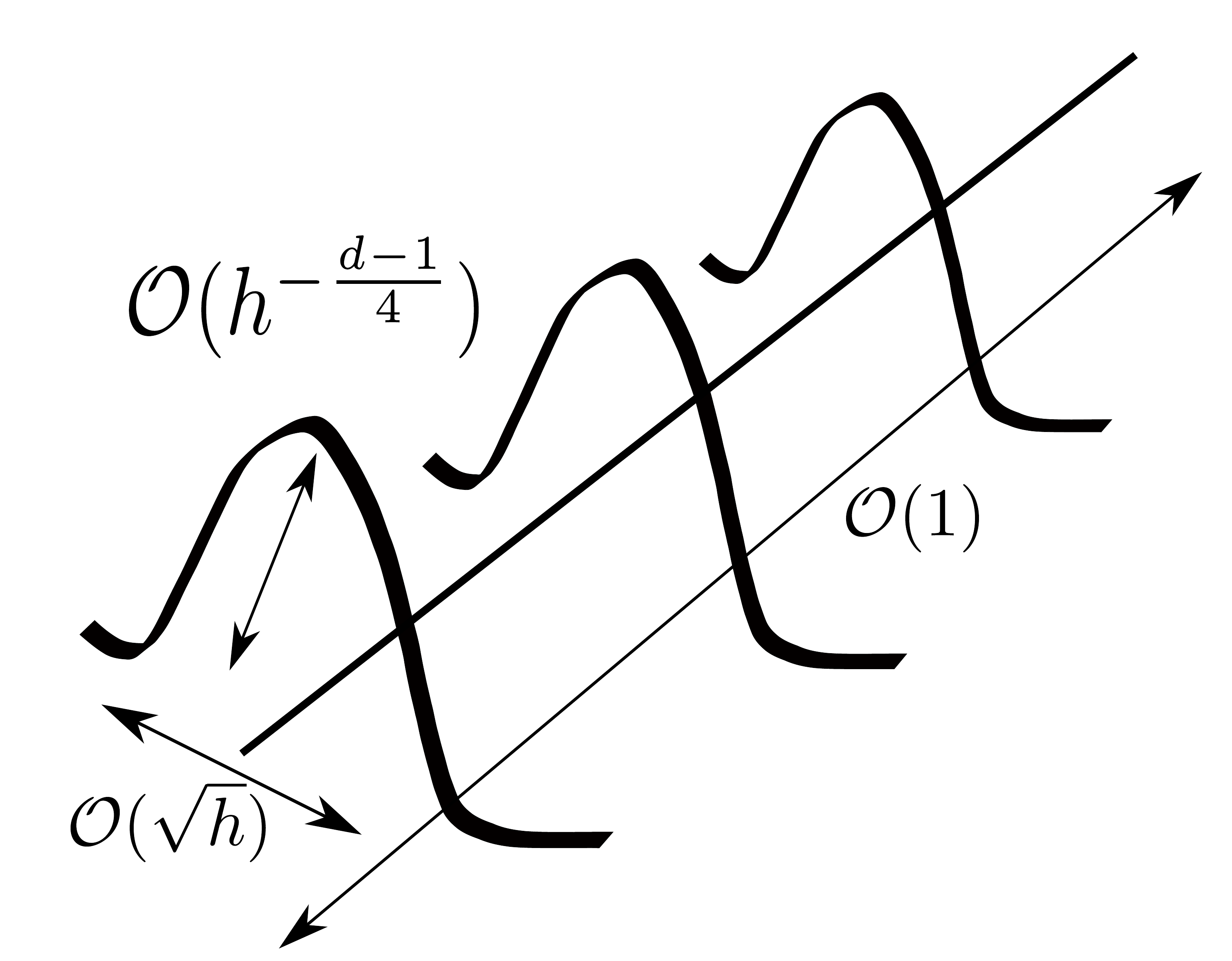}
		\caption{Gaussian beams: concentration around a curve.}
		\label{fig:conc_oh-gaussian-beams}
	\end{figure}
	%%%%%%%%%%%%%%% GAUSSIAN BEAMS %%%%%%%%%%%%%%%
	
	\newpage
	
	%%%%%%%%%%%%%%%%%%%%%% DESSIN optim s(q,d) %%%%%%%%%%%%%%%%%%%%%%%%%%%%
	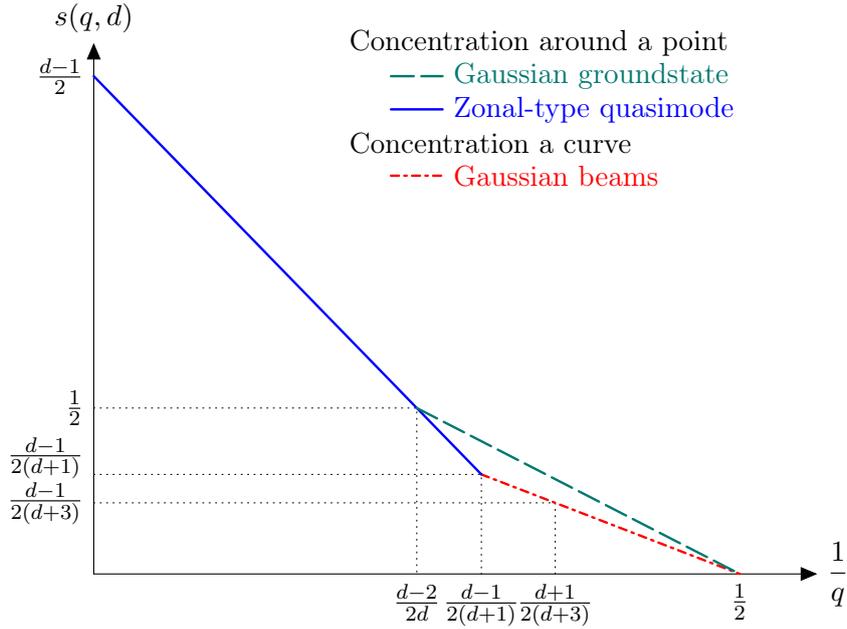
\begin{figure}[!h]
	\begin{center}\begin{tikzpicture}[line cap=round,line join=round,>=triangle 45,x=8.5 cm,y=2.2 cm,scale=2]
		d=3
		\draw[->] (0.,0.) -- (0.56,0.);
		\draw (0.56,0.) node[right] {$\displaystyle{\frac{1}{q}}$};
		\draw[->] (0.,0.) -- (0.,3/2+0.1);
		\draw (0.,3/2+0.1) node[above] {$\displaystyle{s(q,d)}$};
		%
		%% abscisses
		%
		\draw (1/4,0.) node[below]{$\frac{d-2}{2d}$};
		\draw (3/10,0.) node[below]{$\frac{d-1}{2(d+1)}$};
		\draw (5/14,0.) node[below]{$\frac{d+1}{2(d+3)}$};
		\draw (0.5,0.) node[below]{$\frac{1}{2}$};
		%
		%% ordonnées
		\draw (0,0.5) node[left]{$\frac{1}{2}$};
		\draw (0,3/10+0.05) node[left]{$\frac{d-1}{2(d+1)}$};
		\draw (0.,3/14) node[left]{$\frac{d-1}{2(d+3)}$};
		\draw (0.,3/2) node[left]{$\frac{d-1}{2}$};
		%
		%% points
		%
		\draw[dotted] (0., 0.5)-|(1/4, 0.); % ep KT
		\draw[dotted] (0., 3/10)-|(3/10, 0.); % ep sogge
		\draw[dotted] (0., 3/14)-|(5/14, 0.); % ep sogge
		%
		%% courbes
		% gaussian beams
		\draw[ line width=0.9pt, dash dot, color=red] (3/10,3/10)--(0.5,0.);
		% quasimode Sogge, zonal harm
		\draw[ line width=0.9pt, color=blue] (0.,3/2) -- (1/4,0.5)--(3/10,3/10);
		% gaussian
		\draw[ line width=0.9pt, color=stefou,dash pattern=on 3mm off 1mm](1/4,0.5)--(0.5,0.);
		%
		%% legende
		%%
		\draw (1/4-0.06,3/2+0.1) node[right]{Concentration around a point};
		% sogge quasimode
		\draw[line width=0.9pt, color= stefou,dash pattern=on 3mm off 1mm] (1/4-0.02,3/2)-|(1/4+.02,3/2)   node[right]{Gaussian groundstate};
		% gaussian
		\draw[line width=0.9pt, color=blue]
		(1/4-0.02,3/2-0.1)-|(1/4+.02,3/2-0.1)  node[right]{Zonal-type quasimode};
		\draw (1/4-0.06,3/2-0.2) node[right]{Concentration a curve};
		%%
		% hermite
		\draw[line width=0.9pt, dash dot, color=red]
		(1/4-0.02,3/2-0.3)-|(1/4+.02,3/2-0.3)   node[right]{Gaussian beams};
		%
		%\draw (1/4,-0.3) node{$d\geq 3$} ;
		\end{tikzpicture}\end{center}
	\caption{Saturation of $s(q,d)$ for $d\geq 3$.}
	\label{fig:conc-opt-1body}
	\end{figure}
	%%%%%%%%%%%%%%%%%%%%%% FIN DESSIN optim s(q,d) %%%%%%%%%%%%%%%%%%%%%%%%%%%%

	\subsubsection{Gaussian groundstate}
	
	We begin with explaining why the exponent $s_{\text{gene}}(q,d)$ is sharp (that is, it cannot be lowered). We do so, by exhibiting a family of functions that saturates the inequalities in which $s_{\text{gene}}(q,d)$ appears. We will see that in this case, the saturation scenario happens for functions concentrating around a non-degenerate local minimum of the potential, exactly like the ground state of a harmonic oscillator. Such a construction is well-known (see for instance \cite[Chap. 2]{helffer2006semi}, \cite[Thm 4.23]{dimassi1999spectral} and \cite[Exemple 2]{koch2007semiclassical}), but we recall here for completeness.
	
	\begin{prop}\label{prop:optim-lowr:ELp-gene-1body}
		Let $d\geq 1$ and $2\leq q\leq\infty$.
		Let $p(x,\xi):=\abs{\xi}^2+V(x)$ where $V\in\mathcal{C}^\infty(\R^d,\R)$ is as in Definition \ref{cond:potential-pol-growth}.
		For any $h>0$, define $P=p^\w(x,hD)$.
		Let $E\in\R$ such that there exists $x_0\in\R^d$ such that
		\begin{equation*}
			V(x_0)=E,\quad \nabla_x V(x_0)=0,\quad \partial_x^2V(x_0) \text{ definite positive }
			.
		\end{equation*}
		Then, there exist $C>0$ and $h_0>0$, such that the normalized groundstate of the operator $-h^2\Delta+\prodscal{x-x_0}{\sqrt{\partial_x^2 V(x_0)}(x-x_0)}$	
		\begin{equation*}
			u_h(x) := (2\pi\sqrt{h})^{-d/2}\det(\partial_x^2 V(x_0))^{-1/4}e^{-\frac{\prodscal{x-x_0}{\sqrt{\partial_x^2 V(x_0)}(x-x_0)}}{2h}} 
		\end{equation*}
		 associated  to the eigenfunction $\lambda_h := h\tr_{\R^d}\sqrt{\partial_x^2 V(x_0)}$ satisfies the bound for any $h\in(0,h_0]$
		\begin{equation}\label{eq:prop:optim-lowr:ELp-gene-1body}
			\normLp{u_h}{2}{(\R^d)}=1,\quad		
			\normLp{(P-E)u_h}{2}{(\R^d)}\leq C h,\quad
			\normLp{u_h}{q}{(\R^d)}\geq  (1/C) h^{-\frac d 2\left(\frac 1 2 -\frac 1 q\right) }.
		\end{equation}
	\end{prop}
	% fin \ref{prop:optim-lowr:ELp-gene-1body}
	
	The function $u_h$ concentrates around $x_0$ at a scale $\sqrt{h}$ with a height $\sim h^{-d/4}$ (c.f. Figure \ref{fig:conc_gaussian-gs}).

	\begin{rmk}\label{rmk:optim-lowr:ELp-gene-1body_1}
		The previous proposition gives the optimality of the exponent $s_{\text{gene}}(q,d)$
		\begin{itemize}
			\item when $d=1$: for $2\leq q\leq\infty$,
			\item when $d=2$: for $2\leq q<\infty$,
			\item when $d\geq 3$: for $2\leq q\leq 2d/(d-2)$.
		\end{itemize}
		Indeed, recall that in these cases we have $s_{\text{gene}}(q,d)=\tfrac d 2\left(\frac 1 2 -\frac 1 q\right)$, which is the exponent appearing in \eqref{eq:prop:optim-lowr:ELp-gene-1body}. More precisely, this proves that one cannot take a smaller $s(q,d)$ in Theorem \ref{thm:ELp-gene-1body}. Notice that this theorem applies to microlocalized functions, which is not the case for our quasimode $u_h$ above. However, since $u_h$ is an eigenfunction of a Schr\"odinger operator with a quadratic potential, there exists $\chi\in\test{\R^d\times\R^d}$ such that $u_h = \chi^\w u_h +\OR(h^\infty)$ as in the proof of Theorem \ref{thm:spectral-clusters}. Any point $(x_0,\xi_0)$ in the support of $\chi$ satisfies either $p(x_0,\xi_0)\neq 0$ or Assumption \ref{cond:gene}. Hence, it shows that $\{u_h\}_{h\in(0,h_0]}$ actually saturates the bound of Theorem \ref{thm:ELp-gene-1body} but in its version of Theorem \ref{thm:abstract-microloc-extend}.
	\end{rmk}
	
	%% c.f. DETAILS phd
	
	\begin{rmk}\label{rmk:optim-lowr:ELp-gene-1body_2}
		In the special case where $V$ is quadratic (meaning that $\partial_x^2 V$ is constant), $u_h$ (more precisely $\gamma_h=\bra{u_h}\ket{u_h}$) saturates also the bound \eqref{eq:spectral-clusters} in Theorem \ref{thm:spectral-clusters} because in this case it satisfies $\Pi_h u_h=u_h$. Notice that in this one-body setting, no logarithm appears in this estimate as we recall in Section \ref{sec:gene}.
	\end{rmk}

	\begin{rmk}\label{rmk:optim-lowr:ELp-gene-1body_3}
		Note that the assumptions of Proposition \ref{prop:optim-lowr:ELp-gene-1body} hold when the potential $V\in\CR^\infty(\R^d,\R)$ is a Morse function, i.e.\ such that all its critical points has a non-degenerate Hessian. Actually, the trapping assumption of polynomial growth forces this critical points to be global minima. Besides, in this case the normalized ground states functions $u_h$ of the resulting Schr\"odinger operator $P=-h^2\Delta+V$ admits a WKB expansion $h^{-d/4}a_he^{-\varphi/h}:=h^{-d/4}\sum_{j=1}^\infty h^j a_je^{-\varphi/h}$ (see for instance \cite[Sec. 3]{helf-sjost1984}), which is a quasimode of $P$ for its lower eigenvalue $\lambda_h$. In particular, one has
		\begin{equation*}
			 e^{\varphi/h}(P-\lambda_h)\left(h^{-d/4}a_he^{-\varphi/h}\right) =\OR_{L^2}(h^\infty)
			.
		\end{equation*}
		and there exists a neighborhood $K$ of the minimum $x_0$ of $V$ such that one has uniformly in $K$
		\begin{equation*}
			(P-\lambda_h)\left(h^{-d/4}a_he^{-\varphi/h}\right) =\OR(h^\infty)e^{-\varphi/h}
			.
		\end{equation*}
		Here, the phase function $\varphi$ is smooth on $\R^d$ and positive in $K$ such that for any $x\in K$
		\begin{equation*}
			\varphi(x)=\prodscal{x-x_0}{\sqrt{\partial_x^2 V(x_0)}(x-x_0)}+\OR(\abs{x-x_0}^3).
		\end{equation*}
		One has also the functions $\{a_j\}_{j\in\N}\subset S$ and $a_0(x_0)=(2\pi)^{-d/2}\det(\partial_x^2 V(x_0))^{-1/4}$. Furthermore, $\lambda_h=E+\OR(h)$.
		In particular, this ansatz $h^{-d/4}a_he^{-\varphi/h}$ behaves like the normalized ground state of $-h^2\Delta+\prodscal{x-x_0}{\sqrt{\partial_x^2 V(x_0)}(x-x_0)}$ in Proposition \ref{prop:optim-lowr:ELp-gene-1body} and then satisfies \eqref{eq:prop:optim-lowr:ELp-gene-1body}. Then, by the triangle inequality, it is also the case for the eigenfunction $u_h$.
	\end{rmk}

	\begin{proof}[\underline{Proof of Proposition \ref{prop:optim-lowr:ELp-gene-1body}}]
		Let $2\leq q\leq\infty$ and $h>0$.
		Let $\varphi_1$ the normalized gaussian on $L^2(\R^d)$
		\begin{equation*}
			\varphi_1(x):=(2\pi)^{-d/2}e^{-\abs{x}^2/2}.
		\end{equation*}
		When we replace the potential $\abs{x}^2$ by $V_{x_0}(x):=\frac1 2 \prodscal{x-x_0}{\partial_x^2V(x_0)(x-x_0)}$ in the harmonic oscillator, the normalized ground state associated to the eigenvalue $\lambda_h=h\tr_{\R^d}\sqrt{\partial_x^2 V(x_0)}$ of $P_0:= -h^2\Delta +V_{x_0}(x)$ is given by the formula
		\begin{equation*}
			u_h(x) = h^{-d/4}(\det(\partial_x^2 V(x_0))^{-1/4}\varphi_1(h^{-1/2}(\partial_x^2 V(x_0))^{1/4}(x-x_0)) . 
		\end{equation*}
		Moreover, for any compact $K$ neighborhood of $x_0$, there exists $C(d,K,h_0)>0$ such that for $0< h\leq h_0$ and any $2\leq q\leq\infty$
		\begin{align*}
			\normLp{u_h}{q}{(K)}
			&= (\det(\partial_x^2 V(x_0)))^{-1/4} h^{-d/4}\normLp{\varphi_1(h^{-1/2})(\partial_x^2 V(x_0))^{1/4}(x-x_0)}{q}{(K)}
			\\&\geq C(d,K) h^{-\frac d2\left(\frac 1 2-\frac 1 q\right)} \underset{>0}{ \underbrace{ \normLp{\varphi_1}{q}{\big(h_0^{-1/2}\:(\partial_x^2 V(x_0))^{-1/4}(K+x_0)\big)} }}
			\\&\geq C(d,K,h_0) h^{-\frac d2\left(\frac 1 2-\frac 1 q\right)}.
		\end{align*}
		We have for any $h_0>0$, $C>0$ such that for any $h\in(0,h_0]$ and any $2\leq q\leq\infty$
		\begin{equation*}
			\normLp{u_h}{q}{(\R^d)} \geq C h^{-\frac d2\left(\frac 1 2-\frac 1 q\right)}\left(\normLp{u_h}{2}{(\R^d)}+\frac 1h\normLp{(P_0-\lambda_h)u_h}{2}{(\R^d)}\right) .
		\end{equation*}
		 It remains to show that $u_h$ is a quasimode of the operator $P-E$
		\begin{equation*}
			(P-E)u_h = \OR_{L^2}(h)
			.
		\end{equation*}
		By Taylor formula of $V$ at $x=x_0$, we have for any $x\in\R^d$
		\begin{equation*}
			V(x)= E +V_{x_0}(x)+ \OR( \abs{x-x_0}^3 )
			.
		\end{equation*}
		Thus, one can estimate $(V-E-V_{x_0})u_h$
		\begin{align*}
			\normLp{(V-E-V_{x_0})u_h}{2}{(\R^d)}^2
			&= 
			\int_{\R^d}\abs{(V-E-V_{x_0})(x) u_h(x)}^2 dx 
			\\&
			= 
			\int_{\R^d}\abs{(V-E-V_{x_0})(x_0+\sqrt{h}x) u_1(x_0+\sqrt{h}x)}^2 dx 
			.
		\end{align*}
		On the one hand, since $\{ y\in\R^d \: : \: y= x_0+\sqrt{h}x \text{ and } x\in\R^d, \:  \abs{x}\leq 1/\sqrt{h} \}
		\subset\{ y\in\R^d \: :\: \abs{y}\leq \abs{x_0}+2 \}$ is a compact set of $\R^d$, we have by Taylor formula and by $\varphi_1\in\schwartz(\R^d)$
		\begin{align*}
			&\int_{\abs{x}\leq\frac1{\sqrt{h}}}\abs{(V-E-V_{x_0})(x_0+\sqrt{h}x) u_1(x_0+\sqrt{h}x)}^2 dx
			\\&\quad
			\leq C h^3\int_{\abs{x}\leq\frac1{\sqrt{h}}}\abs{x}^6 \abs{u_1(x_0+\sqrt{h}x)}^2 dx
			\\&\quad 
			\leq C h^3\int_{\R^d}\abs{x}^6 \abs{\varphi_1((\partial_x^2 V(x_0))^{1/4}x)}^2 (\det((\partial_x^2 V(x_0)))^{-1/2} dx
			\\&\quad
			\leq 	C' h^3 
			\underset{<\infty}{\underbrace{
				\int_{\R^d}\abs{x}^6 \abs{\varphi_1(x)}^2 dx
				}}
			.
		\end{align*}
		On the other hand, by the growth assumption \eqref{eq:cond-potential-symb} on the potential $V$ 
		\begin{align*}
			&\int_{\abs{x}>\frac1{\sqrt{h}}}\abs{(V-E-V_{x_0})(x_0+\sqrt{h}x) u_1(x_0+\sqrt{h}x)}^2 dx
			\\&\quad
			= \int_{\abs{x}>\frac1{\sqrt{h}}} \abs{V(x_0+\sqrt{h}x)-E+  h\prodscal{x}{\partial_x^2 V(x_0)x}}^2 \abs{ u_1(x_0+\sqrt{h}x)}^2  dx
			\\&\quad
			\leq C \int_{\abs{x}>\frac1{\sqrt{h}}} (1+\abs{x_0+\sqrt{h}x}^k+ h\prodscal{x}{\partial_x^2 V(x_0)x})^2 \abs{ u_1(x_0+\sqrt{h}x)}^2  dx
			\\&\quad
			\leq C' \int_{\abs{x}>\frac1{\sqrt{h}}}  \abs{\varphi_1((\partial_x^2 V(x_0))^{1/4}x)}^2 (\det((\partial_x^2 V(x_0)))^{-1/2} dx
			\\&\quad
			\leq C''  \int_{\abs{x}>\frac{c'}{\sqrt{h}}}  \abs{\varphi_1(x)}^2  dx = \OR\big( e^{-c/h} \big) =\OR(h^\infty)
			.
		\end{align*}
		Thus,
		\begin{equation*}
			\normLp{(V-E-V_{x_0})u_h}{2}{(\R^d)} =\OR(h^{3/2}) .
		\end{equation*}
		We finally get
		\begin{align*}
			(P-E)u_h &= (P_0-\lambda_h) u_h + \lambda_h u_h  +(V-E-V_{x_0}) u_h
			% details
			%\\&= 0 + \OR_{L^2}(h) +\OR(h^{3/2})
			%
			\\&= \OR_{L^2}(h).
		\end{align*}
		This concludes the proof.
	\end{proof}
	% end proof \ref{prop:optim-lowr:ELp-gene-1body}
	
	\subsubsection{Zonal-type quasimode}\label{subsec:zonal-quasimode}
	
	We now prove the sharpness of the exponent $s_{\text{Sogge}}(q,d)$ for large values of $q$. We will see that the saturation phenomenon is obtained for a sequence of functions concentrating around a point (with a rate different from the one of the gaussian groundstate of the preceding section).
	For the harmonic potential $V(x)=\abs{x}^2$, such functions have already been constructed in \cite[Sec. 5.2]{koch2005p}. Our example which relies on the Weyl law is valid for a more general class of potentials. It is inspired by the construction in \cite[Eq. (5.1.12)]{sogge2017fourier}.
	
	\begin{prop}\label{prop:optim-highr:ELp-gene-1body}
		Let $d\geq 2$.
		Let $E_0>0$ such that $E_0>\min V$.
		Let $\varepsilon_0 := E_0-\min V$.
		Assume that
		\begin{equation}\label{eq:prop:optim-highr:ELp-gene-1body}
			\lim_{\varepsilon\to 0}\abs{\{x\in\R^d :\: \: \abs{V(x)-\lambda}\leq \varepsilon\}} = 0
		\end{equation}
		uniformly for $\lambda$ in a neighborhood of $E_0$.
		Then, there exist $h_0>0$, energies $\{ E_h\}_{h\in(0,h_0]} \subset[E_0-\varepsilon_0/2,E_0+\varepsilon_0/2]$, $\varepsilon\in(0,\varepsilon_0/4)$ and points $\{x_{0,h}\}_{h\in(0,h_0]}\subset \{V\leq E_h-\varepsilon\} $ such that the family of functions defined by
		\begin{equation*}
			u_h := \Pi_h(x_{0,h},\cdot)
		\end{equation*}
		where $\Pi_h$ denotes the spectral projector
		\begin{equation*}
			\Pi_h := \indicatrice{P\in I_{h,E_h}}
			= \indicatrice{P\in\left[E_h-h,E_h+h\right]}
			,
		\end{equation*}
		satisfies along a sequence $h_n\to 0$ that for any $2\leq q\leq \infty$
		\begin{equation*}
			\normLp{u_h}{2}{(\R^d)} \leq C h^{-\frac{(d-1)}2},\quad
			\normLp{u_h}{q}{(\{x\in\R^d \: :\: V(x)-E_h\leq \varepsilon\})} \geq (1/C) h^{-(d-1)+\frac d q}
			.
		\end{equation*}
	\end{prop}
	% end \ref{prop:optim-highr:ELp-gene-1body}
	
	\begin{rmk}\label{cor:optim-highr:ELp-gene-1body}
		As a consequence, we get that
		\begin{equation*}
			\limsup_{h\to 0} h^{d\left(1/2-1/q\right)-1/2} \frac{ \normLp{u_h}{q}{(x\in\R^d \: :\: V(x)-E_h\leq \varepsilon)} }{ \normLp{u_h}{2}{(\R^d)}  } >0 ,
		\end{equation*}
		which proves the optimality of the exponent $s_{\text{Sogge}}(q,d)$ (which is equal to $d\left(1/2-1/q\right)-1/2$ when $2(d+1)/(d-1)\leq q\leq \infty$) in the estimate \eqref{eq:spectral-clusters-loc-far-tp}.
	\end{rmk}

	\begin{rmk}\label{rmk:optim-highr:ELp-gene-1body_hyp-V}
		Assumption \eqref{eq:prop:optim-highr:ELp-gene-1body} is satisfied for instance when:
		\begin{itemize}
			\item[$\bullet$] $\nabla_x V(x_0)$ does not vanish for the points $x_0\in\{V=\lambda\}$,
			\item[$\bullet$] or when the Hessian $\partial_x^2 V(x_0)$ is non-degenerate for the points $x_0\in\{V=\lambda\}$,
		\end{itemize}
		for $\lambda$ in a neighborhood of $E_0$.
		In the first case, we have $\abs{\{\abs{V-\lambda}\leq\varepsilon\}}=\OR(\varepsilon)$. We even expect that the condition holds if for any $x_0$ such that $V(x_0)$ belongs to a neighborhood of $E_0$, the Taylor expansion of $V$ at $x_0$ is not identically zero. On the contrary, if $V$ is constant in a neighborhood of $x_0$, then Assumption \eqref{eq:prop:optim-highr:ELp-gene-1body} is not satisfied.
		% c.f. DETAILS phd examples
	\end{rmk}
	
	\begin{rmk}\label{rmk:optim-highr:ELp-gene-1body_conc}
		The lower bound on the $L^q$ norm of $u_h$ comes the pointwise estimate
			\begin{equation}\label{eq:prop:optim-highr:ELp-gene-1body_res1}
			\forall x\in\R^d, \: \abs{x-x_{0,h}}\leq c h \quad\implies\quad
			\abs{u_h(x)}\geq Ch^{1-d}.
		\end{equation}
		In other words, the function $u_h$ concentrates around $x_{0,h}$ at a scale $h$ with a height $\sim h^{-(d-1)}$ (such that in Figure \ref{fig:conc_zonal-type-quasim}). This motivates the name \emph{zonal-type quasimode} because they concentrate similarly to the zonal harmonics, which are known to saturate the Sogge $L^q$ estimates in the same regime of $q$. This originally appeared in \cite{sogge1985oscillatory}.
	\end{rmk}

	Before proving Proposition \ref{prop:optim-highr:ELp-gene-1body}, we first provide a lemma which is an easy consequence of the integrated Weyl law, which is well known in the high energy regime and that we state here in the semiclassical setting. This result gives an interval of size $h$ with the maximal number of eigenvalues inside. It will also be useful for the many-body optimality.
	
	For any $I\subset\R$, recall that $N_h(I)$ denotes the number of eigenvalues of $P$ in $I$.
	\begin{lemma}\label{lemma:cor-weyl-law}
		Let $a<b$ such that $\abs{ p^{-1}([a,b]) }>0$. Then
		\begin{equation*}
			\limsup_{h\to 0}\sup_{J_h\subset[a,b] \: ,\: \abs{J_h}=2h} h^{d-1}N_h(J_h) >0 .
		\end{equation*}
	\end{lemma}
	\begin{proof}[Proof of Lemma \ref{lemma:cor-weyl-law}]
		Assume by contradiction that
		\begin{equation*}
			\limsup_{h\to 0}\sup_{J_h\subset[a,b] \: ,\: \abs{J_h}=2h} h^{d-1}N_h(J_h) =0
		\end{equation*} i.e.
		for all $\varepsilon>0$, there exists $\tilde{h}>0$ such that for all $h\in(0,\tilde{h}]$ and for all interval $ J_h\subset[a,b]$ such that $\abs{J_h}=2h$
		\begin{equation*}
				h^{d-1}N_h(J_h) \leq\varepsilon   .
		\end{equation*}
		One can cover $[a,b]$ by a finite set of intervals of length $2h$ :
		\begin{equation*}
			\{ J_j \}_{j=1}^{M_h}\subset [a,b], \quad [a,b] \subset \bigcup_{j=1}^{M_h} J_j .
		\end{equation*}
		For example, for $2h<\abs{b-a}$, one can take $M_h = \left\lfloor\frac{\abs{b-a}}{2h}\right\rfloor -1$, define the $M_h-1$ first intervals by $J_j := [a+2(j-1)h,a+2jh]$ and define the $M_h$-th one $J_{M_h} := [b-2h,b]$.
		%%%%%%%%%%%%%%%%%%% DESSIN %%%%%%%%%%%%%%%%%%%%%%%%%%%%%%%
		\begin{center}\begin{tikzpicture}[scale= 2.5]%[line cap=round,line join=round,>=triangle 45,x=4.0cm,y=4.0cm,scale=0.9]
			\draw[->] (0.,0.) -- (5.0,0.);
			\draw  (0.5,0.) node{$+$};
			\draw (0.5, -0.1) node[below] {$a$};
			\draw  (1,0.) node{$+$};
			\draw  (1.5,0) node{$+$};
			\draw  (2.0,0) node{$+$};
			\draw  (2.5,0) node{$+$};
			\draw  (3.0,0) node{$+$};
			\draw  (3.5,0) node{$+$};
			\draw  (4.0,0) node{$+$};
			\draw  (4.5,0) node{$+$};
			\draw  (4.75,0) node{$+$};
			\draw (4.75, -0.1) node[below] {$b$};
			\draw[<->,very thick] (0.5,0.1)-- (1.0,0.1);
			\draw (.75, 0.1) node[above] {$J_1$};
			\draw[<->,very thick] (1,0.1) -- (1.5,0.1);
			\draw (1.25, 0.1) node[above] {$J_2$};
			\draw[<->,very thick] (2,0.1) -- (2.5,0.1)node[midway,above] {$\ldots$};
			\draw[<->,very thick] (4,0.1) -- (4.5,0.1)node[midway,above] {$J_{M_h-1}$};
			\draw[<->,color=djgreen,very thick] (4.25,-0.1) -- (4.75,-0.1)node[midway,below] {$J_{M_h}$};
			\end{tikzpicture}\end{center}
		%%%%%%%%%%%%%%%%%%% FIN DESSIN %%%%%%%%%%%%%%%%%%%%%
		Finally,
		\begin{equation*}
			N_h([a,b])\leq\sum_{j=1}^{M_h}N_h(I_j) \lesssim M_h \varepsilon h^{-(d-1)} \sim \varepsilon \abs{b-a} h^{-d}.
		\end{equation*}
		When $\varepsilon$ goes to $0$, we get
		\begin{equation*}
			\lim_{h\to 0} h^d N_h([a,b])=0 .
		\end{equation*}
		However, by the Weyl law (Proposition \ref{prop:int-weyl-law}) $h^d N_h([a,b])$ is equivalent to $ (2\pi)^{-d}\abs{p^{-1}([a,b])}$, which is positive.
		That is in contradiction with the original assumption.
	\end{proof}
	% end proof \ref{lemma:cor-weyl-law}

	\begin{proof}[\underline{Proof of Proposition \ref{prop:optim-highr:ELp-gene-1body}}]
		Let $R>0$ such that for any $\lambda\in[E_0-R,E_0+R]$, we have \eqref{eq:prop:optim-highr:ELp-gene-1body}.
		Let $c_0 := \min\left(\frac 14 , \frac{R}{2\varepsilon_0}\right)$.
		Define the intervals $I_0 := \left[E_0-c_0\varepsilon_0,E_0+c_0\varepsilon_0\right]$.
		Given the definition of $\varepsilon_0$ and $c_0$, we have $I_0\subset (\min V,\infty)$.
		Thus, $\abs{p^{-1}(I_0)}>0$.
		We begin to take $h_0\in\left(0,c_0\varepsilon_0/2\right)$ (of course, we will lower it afterwards).
		By Lemma \ref{lemma:cor-weyl-law} (up to a sequence $\{h_n\}_{n\in\N}\subset \R^*_+$ with $h_n\to 0$ when $n\to\infty$) there exists $h_0>0$, $\{ E_h\}_{h\in(0,h_0]}\subset I_0$ and $C'>0$ such that for any $h\in(0,h_0]$
		\begin{equation*}
			\normLp{\rho_{\Pi_h}}{1}{(\R^d)} \geq C' h^{-(d-1)}.
		\end{equation*}
		Let $\varepsilon>0$ such that $\varepsilon <\dist(E_0-c_0\varepsilon_0,\min V)/2 = (1-c_0)\varepsilon_0/2$.
		By the $L^\infty$ estimates (see for instance Theorem \ref{thm:spectral-clusters}), there exists $C>0$ such that for any $h\in(0,h_0]$
		\begin{align*}
			\normLp{\rho_{\Pi_h}}{1}{\left(\{\abs{V- E_h}\leq\varepsilon\}\right)}
			&\leq \abs{\{\abs{V-E_h}\leq \varepsilon\}}\normLp{\rho_{\Pi_h}}{\infty}{(\R^d)}
			\\&\leq C h^{-(d-1)} \abs{\{\abs{V-E_h}\leq \varepsilon\}}
			.
		\end{align*}
		By \eqref{eq:prop:optim-highr:ELp-gene-1body}, let us fix $\varepsilon\in(0,(1-c_0)\varepsilon_0/2)$ such that $C\abs{\{\abs{V-E_h}\leq \varepsilon\}}\leq C'/2$.
		Besides, by \eqref{eq:spectral-clusters-loc-forbidd}, there exists $C_\varepsilon>0$ such that for any $h\in(0,h_0]$
		\begin{equation*}
			\normLp{\rho_{\Pi_h}}{1}{(\{V\geq E_h+\varepsilon\})}\leq C_\varepsilon h^{-(d-2)}.
		\end{equation*}
		Finally, by the triangle inequality and the previous estimates there exists $h_0>0$ such that for any $h\in(0,h_0]$
		\begin{equation}\label{eq:optim-lowbound-normL1-Pih}
		\begin{split}
			\normLp{\rho_{\Pi_h}}{1}{\left(\{V\leq E_h-\varepsilon\}\right)}
			&= \normLp{\rho_{\Pi_h}}{1}{(\R^d)}-\normLp{\rho_{\Pi_h}}{1}{\left(\{\abs{V- E_h}<\varepsilon\}\right)}- \normLp{\rho_{\Pi_h}}{1}{\left(\{V\geq E_h+\varepsilon\}\right)}
			\\&
			\geq \frac 14C' h^{-(d-1)}.
		\end{split}
		\end{equation}
		For any $h\in (0,h_0]$, let us define $x_{0,h}$ as a maximizer of the function $\rho_{\Pi_h}$ on the compact set $\{V\leq E_h-\varepsilon\}$.
	We thus have for any $h\in(0,h_0]$
	\begin{align*}\label{eq-demo:prop:optim-highr:ELp-gene-1body_ord0}
		\Pi_h(x_{0,h},x_{0,h})
		&=\rho_{\Pi_h}(x_{0,h}) 
		\geq \frac{1}{\abs{\{V\leq E_h-\varepsilon\}}}\normLp{\rho_{\Pi_h}}{1}{\left(\{V\leq E_h-\varepsilon\}\right)}
		\\&\geq \frac{C'/4}{\abs{\{V\leq E_0+ c_0\varepsilon_0\}}} h^{-(d-1)} = C''h^{-(d-1)}
		.
	\end{align*}
	By the definition of $u_h$ and the $L^\infty$ estimate (see for instance Theorem \ref{thm:ELp-gene-matdens}) there exists $C>0$ such that for any $h\in(0,h_0]$
	\begin{align*}
		\normLp{u_h}{2}{(\R^d)} =\rho_{\Pi_h}(x_{0,h}) 
		\leq \normLp{\rho_{\Pi_h}}{\infty}{(\R^d)}^{1/2} \leq Ch^{-{(d-1)}/2}.
	\end{align*}
	Let us now prove that there exists $C_1>0$ and $h_0>0$ such that for any $h\in(0,h_0]$
	\begin{equation}\label{eq-demo:prop:optim-highr:ELp-gene-1body_ord1}
		\sup_{x,y\in B(x_{0,h},h)}\abs{\nabla_y\Pi_h(x,y)} \leq C_1 h^{-d}
		.
	\end{equation}
		This implies that for any $h\in(0,h_0]$ and for any $x\in B\big(x_{0,h}, \min(1,\frac {C''}{2C_1})  h\big)$
		\begin{align*}
			\abs{u_h(x)} &= \abs{\Pi_h(x_{0,h},x)}
			\\&\geq \Pi_h(x_{0,h},x_{0,h}) -\abs{\Pi_h(x_{0,h},x)-\Pi_h(x_{0,h},x_{0,h})}
			\\&\geq  \Pi_h(x_{0,h},x_{0,h}) - \normLp{\nabla_y\Pi_h(x_{0,h},\cdot)}{\infty}{(B(x_{0,h},h))}\abs{x-x_{0,h}}
			\\&\geq \frac{C''}2 h^{-(d-1)}
			.
		\end{align*}
		That gives us \eqref{eq:prop:optim-highr:ELp-gene-1body_res1}.
		Let us prove now the estimate \eqref{eq-demo:prop:optim-highr:ELp-gene-1body_ord1}.
		Denote
		\begin{equation*}
			\Pi_h^{(0)} := \indicatrice{-h^2\Delta+V(x_{0,h})\in I_{h,E_h}}.
		\end{equation*}
		For all $x,y\in\R^d$, we have
		\begin{align*}
			\Pi_h^{(0)}(x,y) &= \frac{1}{(2\pi h)^d} \int_{E_h-V(x_{0,h})-h\leq \abs{\xi}^2\leq E_h-V(x_{0,h})+h} e^{\frac i h\prodscal{\xi}{x-y}} d\xi .
		\end{align*}
		Hence,
		\begin{align*}
			\abs{\nabla_y\Pi_h^{(0)}(x,y)} 
			&\leq \abs{y}Ch^{-d-1}\left((E_h-V(x_{0,h})+h)^{d/2}-(E_h-V(x_{0,h})-h)_+^{d/2}\right)
			\\&\leq \abs{y}C_\varepsilon h^{-d}
			.
		\end{align*}
		We next prove that
		\begin{equation*}
			\sup_{x,y\in B(x_{0,h},h)}\abs{\nabla_y\Pi_h(x,y)-\nabla_y\Pi_h^{(0)}(x,y)} \leq C h^{-d}
		\end{equation*}
		by showing that
		\begin{equation*}
			\sup_{x,y\in B(x_{0,h},h)}\abs{\nabla_y\Pi_{h,\leq E_h\pm h}(x,y)-\nabla_y\Pi_{h,\leq E_h\pm h}^{(0)}(x,y)} \leq C h^{-d},
		\end{equation*}
		where for any $E\in\R$, $\Pi_{h,\leq E} := \indicatrice{P\leq E}$ and $\Pi^{(0)}_{h,\leq E} := \indicatrice{-h^2\Delta+V(x_{0,h})\leq E}$.
		Introducing
		\begin{equation*}
			K_{h,\leq E}(x,y) := \Pi_{h,\leq E}(x_{0,h}+hx,x_{0,h}+hy),\quad K_{h,\leq E}^{(0)}(x,y) := \Pi_{h,\leq E}^{(0)}(x_{0,h}+hx,x_{0,h}+hy),
		\end{equation*}
		it is enough to show that
		\begin{equation*}
			\sup_{x,y\in B(0,1)}\abs{\nabla_yK_{h,\leq E_h\pm h}(x,y)-\nabla_yK_{h,\leq E_h\pm h}^{(0)}(x,y)} \leq C h^{-(d-1)}.
		\end{equation*}
		This is given by \cite[Rem. 1.2]{deleporte2021universality}.
		Notice that in this work, the parameter $x_{0,h}$ and $E_h$ are independent of $h$. However, the result still holds when $x_{0,h}$ and $E_h$ depend on $h$ in such a way that they belong to a $h$-independent compact set and are such that $V(x_{0,h})\leq E_h-\delta$ for some $h$-independent $\delta>0$.
	\end{proof}
	% end proof \ref{prop:optim-highr:ELp-gene-1body}

	\subsubsection{Gaussian beams}
	
	We now explain the optimality of the exponent $s_{\text{Sogge}}(q,d)$ for low values of $q$, in the case $V(x)=\abs{x}^2$. The saturating functions already appeared in \cite[Sec. 5.1]{koch2005p}, and we just provide the computational details here. This example only works for the harmonic oscillator since the argument relies on separation of variables. However, we expect that this exponent is also sharp for more general $V$, using for instance the argument of \cite{sogge89remarks}. Here, the saturation phenomenon happens for a family of functions concentrating around a curve, similarly to Gaussian beams on spheres \cite{sogge1985oscillatory}.
	
	\begin{prop}\label{prop:optim-lowr:ELp-sogge-1body}
		Let $d\geq 2$ and $2\leq q\leq\infty$. Let $p(x,\xi)=\abs{\xi}^2+\abs{x}^2$. Let $E_{\text{exc}}>0$. For any $h>0$, define $P:=p^\w(x,hD)$ and $E_h := E_{\text{exc}}+(d-1)h$. 
		For any $n\in\N$, let us denote $h_n:=E_{\text{exc}}/(2n+1)$ and $\varphi_n$ the normalized eigenfunction associated to the eigenvalue $E_{\text{exc}}$ of the scalar harmonic oscillator $-h_n^2\frac{d^2}{dy^2}+y^2$ on $L^2(\R)$.
		For any $n\in\N$, define
		\begin{equation*}
			u_n(x) := (2\pi)^{-(d-1)/2}h_n^{-(d-1)/4}e^{-\frac{\abs{x'}^2}{2h_n}}\varphi_n(x_1) ,\quad x=(x_1,x')\in\R\times\R^{d-1}.
		\end{equation*}
		Then, for all $n\in\N$, we have $\normLp{u_n}{2}{(\R^d)}=1$, $Pu_n = E_{h_n}u_n$, and there exists $\varepsilon>0$ and $C>0$ such that for $n$ large enough
		\begin{equation*}\label{eq:prop:optim-lowr:ELp-sogge-1body}
			 \normLp{u_n}{q}{\left(\{x\in\R^d :\:\:\abs{x}^2-E_{h_n}<\varepsilon\}\right)}\geq C h_n^{-\frac{d-1}{2}\left(\frac 12-\frac 1q\right)}.
		\end{equation*}
	\end{prop}

	\begin{rmk}\label{rmk:optim-lowr:ELp-sogge-1body_conc}
		These eigenfunctions concentrate along the curve $\{x'=0\}$ at scale $\sqrt{h}$ in the orthogonal direction (c.f. Figure \ref{fig:conc_oh-gaussian-beams}). Indeed, due to the well-known asymptotics of the Hermite functions, $\varphi_n(x_1)$ is essentially constant (up to oscillations which average out when taking the $L^q$ norm) for $x_1$ in a neighborhood of $0$.
	\end{rmk}

	%%%%%%%%%%%%%%%%%%%% Dessin tube %%%%%%%%%%%%%%%%%%%%		
	\begin{center}\begin{tikzpicture}[line cap=round,line join=round,>=triangle 45,x=1.0cm,y=1.0cm,scale=.9]
		% tube
		\draw[color = gray!90,fill=gray!30,line width=1.5pt](-3,-0.5) -- (-3,0.5) -- (3,0.5) -- (3,-0.5) -- (-3,-0.5);
		%
		% mesures
		\draw[<->](-3,-1.)--(3.,-1.) node[below, midway]{$\sim 1$};
		\draw[<->](-3.5-0.2,-.5)--(-3.5-0.2,.5) node[left, midway]{$\sim h^{1/2}$};
		%
		% gradu
		%
		\draw[->] (-3.5,0.) -- (3.5,0);
It is stricly larger than the one		\draw  (3.5,0.) node[right]{$x_1$};
		\draw[->] (0.,-1.5) -- (0.,1.5);
		\draw  (0.,1.5) node[above]{$x'$};
		\draw  (0,0.) node[below,right] {$0$};
	\end{tikzpicture}\end{center}
	%%%%%%%%%%%%%%%%%%%% Fin Dessin tube %%%%%%%%%%%%%%%%%%%%
	
	\begin{rmk}\label{cor:optim-lowr:ELp-sogge-1body}
		This proves the optimality of the exponent $s_{\text{Sogge}}(q,d)$ for low regime $2\leq q\leq 2(d+1)/(d-1)$ in the classically allowed region and also the optimality of the exponent $s_{\text{TP}}(q,d)=s_{\text{Sogge}}(q,d)$ for lower regime $2\leq q\leq 2(d+3)/(d+1)$ around the turning points $\{ V= E_{h_n} \}$ on Theorem \ref{thm:spectral-clusters} when $V(x)=\abs{x}^2$, by taking $\gamma_{h_n} = \bra{u_n}\ket{u_n}$.		
	\end{rmk}
	
	\begin{proof}[\underline{Proof of Proposition \ref{prop:optim-lowr:ELp-sogge-1body}}]
		Let $\varepsilon=E_{\text{exc}}/2$. First, defining for any $h>0$
		\begin{equation*}
			\varphi_{h,\text{ground}}(x'):= (2\pi)^{-(d-1)/2}h^{-(d-1)/4}e^{-\frac{\abs{x'}^2}{2h}} ,
		\end{equation*}
		 there exists $C_d>0$ such that for any $h>0$
		\begin{equation*}
			\normLp{\varphi_{h,\text{ground}}}{q}{(\{x'\in\R^{d-1} \: :\: \abs{x'}\leq h^{1/2}\})} = C_d h^{-\frac{d-1}4  +\frac{d-1}{2q}} .
		\end{equation*}
		On the other hand, by Liouville-Green asymptotics (see for instance Chapter 6 in \cite{olver1997asymptotics}), we have for any $x_1\in[-\sqrt{E_{\text{exc}}}/2,\sqrt{E_{\text{exc}}}/2]$
		\begin{align*}
			\varphi_n(x_1)
			&
			=  \frac{C_{h_n}}{(E_{\text{exc}}-x_1^2)^{1/4}}
			\begin{cases}
				\cos\left(h_n^{-1}\int_0^{x_1}\sqrt{E_{\text{exc}}-t^2} dt\right)(1+e_{h_n}(x_1)) &\text{ if }\ n\text{ is even},\\
				\sin\left(h_n^{-1}\int_0^{x_1}\sqrt{E_{\text{exc}}-t^2} dt\right)(1+e_{h_n}(x_1)) &\text{ if }\ n\text{ is odd},
			\end{cases}
		\end{align*}
		where $C_{h_n}\in\R$ be a normalization constant such that $\lim_{n\to\infty}C_{h_n}>0$ and
		\begin{equation*}
			\normLp{e_{h_n}}{\infty}{([-\sqrt{E_{\text{exc}}}/2,\sqrt{E_{\text{exc}}}/2])} =\OR(h_n) .
		\end{equation*}
	%% details
		Let us now show that $\liminf_{n\to\infty}\normLp{\varphi_n}{q}{([0,\delta])}>0$ for $\delta=\sqrt{E_{\text{exc}}}/2$, which then proves the result since for $n$ large enough
		\begin{align*}
			\normLp{u_n}{q}{\left(x\in\R^d \: :\: V(x)-E_{h_n}\leq\varepsilon\right)}
			&\geq \normLp{u_n}{q}{\left([0,\delta]\times B_{\R^{d-1}}(0,\sqrt{h_n})\right)}
			\\&\geq\normLp{\varphi_n}{q}{([0,\delta])} \normLp{\varphi_{h_n,\text{ground}}}{q}{( B_{\R^{d-1}}(0,\sqrt{h_n}) )}
			\\&\gtrsim  h_n^{-\frac{d-1}4 +\frac{d-1}{2q} }.
		\end{align*}
		Suppose for instance that $n$ is even (the cas $n$ odd is similar). For any $n\in\N$ and any $x_1\in[0,\delta]$, define
		\begin{equation*}
			f_{h_n}(x_1) := (E_{\text{exc}}-x_1^2)^{-1/4}\cos\left(h_n^{-1}\int_0^{x_1}\sqrt{E_{\text{exc}}-t^2} dt\right) .
		\end{equation*}
		Given the estimate on the error term $e_{h_n}$, we have
		\begin{equation*}
			\normLp{\varphi_n-f_{h_n}}{q}{([0,\delta])}=\OR(h_n) .
		\end{equation*}
		Let us denote
		\begin{equation*}
			S(x_1):= \int_0^{x_1}\sqrt{E_{\text{exc}}-t^2} dt .
		\end{equation*}
		Since its derivative $S'$ is positive on $[0,\delta]$, we have
		\begin{align*}
			\normLp{f_{h_n}}{q}{([0,\delta])}^q
			&= \int_0^\delta S'(x_1)^{-q/2}\abs{\cos\left(h_n^{-1}S(x_1)\right)}^q  dx_1
			\\&\gtrsim  \int_0^\delta \abs{\cos\left(h_n^{-1}S(x_1) \right)}^q S'(x_1) dx_1
			\\&\quad = \int_0^{S(\delta)} \abs{\cos\left(h_n^{-1} y\right)}^q dy
			= \int_0^{S(\delta)} \left(1+\cos(2h_n^{-1} y)\right)^{q/2} dy
			.
		\end{align*}
		Since $q/2\geq 1$, the function $g=\left(1+\cos(2 \cdot)\right)^{q/2}$ is $\pi$-periodic and $\CR^1$. Hence, its Fourier coefficients 
		\begin{equation*}
			c_k(g) := \frac 1{\pi}\int_0^{h_n\pi} e^{-2i k t} g(t) dt
		\end{equation*}
		are summable over $k\in\Z$. As a consequence,
		\begin{equation*}
			\int_0^{S(\delta)} g(h_n^{-1}y) dy = \sum_{k\in\Z} c_k(g)\int_0^{S(\delta)} e^{2ik y/h_n} dy = S(\delta)c_0(g)+\OR(h_n) ,
		\end{equation*}
		which finishes the proof.
	\end{proof}
	% end proof \ref{prop:optim-lowr:ELp-gene-1body}

	\subsection{Many-body optimality}\label{subsec:optim-manybody}
	
	We now turn to the optimality of the Schatten exponent $\alpha$. We show that the exponent $\alpha=\alpha_{Sogge}$ is sharp in the estimates where it appears together with the exponent $s=s_{\text{Sogge}}$. The saturation scenario here is a family of operators $\gamma$ such that $\rho_\gamma$ is delocalized in the bulk region of the potential. We will see that it happens for the maximal family $\gamma=\Pi_h$, in the same spirit as in \cite[Rem. 11]{frank2017spectral}. The optimality of the other Schatten exponents $\alpha_{\text{gene}}$ or $\alpha_{\text{TP}}$ (when they are not equal to $\alpha_{\text{Sogge}}$) is a very challenging problem since $\gamma=\Pi_h$ does not saturate the inequalities where they appear.
	
	\begin{prop}[Many-body optimality of the Sogge exponent]\label{prop:optim:ELp-sogge-manybody}
		Let $d\geq 2$.
		Let $p(x,\xi)=\abs{\xi}^2+V(x)$ with $V\in\CR^\infty(\R^d,\R)$ satisfying Definition \ref{cond:potential-pol-growth}.
		Let $E_0>\min V$ such that we have \eqref{eq:prop:optim-highr:ELp-gene-1body}.
		Then, there exist $h_0>0$, $\{E_h\}_{h\in(0,h_0])}$ in a compact neighborhood of $E_0$ on $(\inf V,\infty)$, $\varepsilon>0$ and $C>0$ such that for any $h\in(0,h_0]$ and any $2\leq q\leq\infty$ (along a sequence $h_n\to 0$ when $n\to\infty$)
		\begin{equation}\label{eq:optim-highr:ELp-TP-manybody}
			\normLp{\rho_{\Pi_h}}{q/2}{\left(\{ x\in\R^d \ :\: V(x)\leq E_h-\varepsilon \}\right)} \geq C h^{-(d-1)} 
			,
		\end{equation}
		where $\Pi_h$ denotes the spectral projector $\Pi_h := \indicatrice{P\in I_{h,E_h}}$.
	\end{prop}

	\begin{rmk}
		Due to the $L^\infty$ estimates \eqref{eq:spectral-clusters} in the case $q=\infty$, we always have $\normLp{\rho_h}{\infty}{(\R^d)}\leq Ch^{-(d-1)}$. Together with \eqref{eq:optim-highr:ELp-TP-manybody}, this shows that all the $L^{q/2}$ norms of $\rho_{\Pi_h}$ are of the order $h^{-(d-1)}$ in the bulk region $\{V<E_h-\varepsilon\}$, indicating that $\rho_{\Pi_h}$ behaves like a (large) constant in this region.
	\end{rmk}

	\begin{rmk}
		This result also proves that the Schatten exponent $\alpha_{Sogge}(q,d)$ is optimal for instance in the estimate \eqref{eq:spectral-clusters-loc-far-tp}. Indeed, for $\gamma=\Pi_h$, \eqref{eq:optim-highr:ELp-TP-manybody} shows that the left side of \eqref{eq:spectral-clusters-loc-far-tp} is of order $h^{-(d-1)}$, while the right side is order \[h^{-2s_{\text{Sogge}}(q,d)}\normSch{\Pi_h}{\alpha_{Sogge}(q,d)}{}\lesssim h^{-2s_{\text{Sogge}}(q,d)}h^{-(d-1)/\alpha_{\text{Sogge}}(q,d)}=h^{-(d-1)}.\]
		Here, we used that $\rk(\Pi_h)\leq Ch^{-(d-1)}$ which follows from the fact that
		\begin{equation*}
			\rk(\Pi_h)=\normLp{\rho_{\Pi_h}}{1}{(\R^d)}=\normLp{\rho_{\Pi_h}}{1}{(\{V\leq E_h+\varepsilon\})} + \normLp{\rho_{\Pi_h}}{1}{(\{V> E_h+\varepsilon\})}
		\end{equation*}
		and the estimates \eqref{eq:spectral-clusters} in the case $q=\infty$ and \eqref{eq:spectral-clusters-loc-forbidd}.
	\end{rmk}

	\begin{proof}[\underline{Proof of Proposition \eqref{prop:optim:ELp-sogge-manybody}}]
		Let $\varepsilon_0:=E_0-\min V$.
		Here, we take $\{E_h\}_{h\in(0,h_0]} \subset[E_0-\varepsilon_0/2,E_0+\varepsilon_0/2]$ and $\varepsilon\in(0,\varepsilon_0/4)$ as in Proposition \ref{prop:optim-highr:ELp-gene-1body}.
		They are chosen to satisfy the lower bound \eqref{eq:optim-lowbound-normL1-Pih}. Thus, there exists $C'>0$ such that for any $h\in(0,h_0]$ and any $2\leq q\leq\infty$
		\begin{align*}
			\normLp{\rho_{\Pi_h}}{q/2}{(\{V\leq E_h-\varepsilon\})} 
			&\geq \abs{\{V\leq E_h-\varepsilon\}}^{-1/(q/2)'} \normLp{\rho_{\Pi_h}}{1}{(\{V\leq E_h-\varepsilon\})}
			\\&\geq \abs{\{V\leq E_0+3\varepsilon_0/4\}}^{-1/(q/2)'} \normLp{\rho_{\Pi_h}}{1}{(\{V\leq E_h-\varepsilon\})}
			\\&\geq C' h^{-(d-1)}
			.
		\end{align*}
	\end{proof}
	% end proof \ref{prop:optim:ELp-sogge-manybody}

	%%%%%%%%%%%%%%%%% BIBLIOGRAPHIE %%%%%%%%%%%%%%%%%%%%%%

	\bibliographystyle{siam}
	{\small
		\bibliography{biblioNhi.bib}
	}

\end{document}